\documentclass[11pt, reqno]{amsart} 
 \usepackage[margin=1in]{geometry}
\usepackage{tikz-cd}
\usetikzlibrary{cd}
\usepackage{xypic}
\usepackage{hyperref}

\usepackage{wasysym,stmaryrd,color}
\usepackage{bm,soul}

\usepackage{amssymb, amsmath, amscd, amsthm, color, epsfig,url, graphicx}
\usepackage[font={footnotesize}]{caption}

\usepackage{enumitem}

\makeatletter
\renewcommand\l@subsection{\@tocline{2}{0pt}{2pc}{5pc}{}}
\makeatother

\newcommand{\vvcenteredinclude}[2]{\begingroup
	\setbox0=\hbox{\includegraphics[scale=#1]{#2}}%
	\parbox{\wd0}{\box0}\endgroup}

\newcommand{\R}{{\mathbb R}}
\newcommand{\C}{{\mathbb C}}
\newcommand{\Z}{{\mathbb Z}}
\newcommand{\Q}{{\mathbb Q}}

\newcommand{\Map}{\operatorname{Map}}

\newcommand{\Emb}{\operatorname{Emb}}
\newcommand{\BrEmb}{\operatorname{BrEmb}}

\newcommand{\Ch}{\operatorname{C}}
\newcommand{\Ba}{\operatorname{B}}
\newcommand{\ChdR}{\operatorname{C}_{dR}}

\newcommand{\Conf}{\operatorname{Conf}}

\newcommand{\Chen}{\operatorname{Ch}}

\newcommand{\D}{{\mathcal{D}}}
\newcommand{\Dm}{\overline{\mathcal{D}}}
\newcommand{\LDm}{\overline{\mathcal{L}\mathcal{D}}}
\newcommand{\TTheta}{\widetilde{\Theta}}

\newcommand{\LD}{{\mathcal{LD}}}

\newcommand{\x}{\times}

\newcommand{\no}{\noindent}
\newcommand{\lrarrow}{\leftrightarrow}

\newcommand{\del}{{\partial}}

\newtheorem*{rep@theorem}{\rep@title}
\newcommand{\newreptheorem}[2]{%
	\newenvironment{rep#1}[1]{%
		\def\rep@title{#2 \ref{##1}}%
		\begin{rep@theorem}}%
		{\end{rep@theorem}}}

\theoremstyle{plain}
\newtheorem{thm}{Theorem}[section]
\newreptheorem{thm}{Theorem}

\newtheorem{lem}[thm]{Lemma}
\newreptheorem{lem}{Lemma}
\newtheorem{cor}[thm]{Corollary}
\newreptheorem{cor}{Corollary}
\newtheorem{conj}[thm]{Conjecture}

\theoremstyle{definition}
\newtheorem{defin}[thm]{Definition}
\newtheorem{example}[thm]{Example}
\newtheorem{def/ex}[thm]{Definition/Example}

\theoremstyle{remark}
\newtheorem{rem}[thm]{Remark}

\newcommand{\refD}[1]{Definition~\ref{D:#1}}

\def\phi{\varphi}

\graphicspath{{cochains/pict/},{pict/},{./../pict/},{./../../pict/}}

\AtBeginDocument{%
   \def\MR#1{}
}

\begin{document}


\title[Diagrams for primitive cycles in spaces of pure braids]{Diagrams for primitive cycles in spaces of\\ pure braids and string links}



\author{Rafal Komendarczyk}
\address{Department of Mathematics, Tulane University, 6823 St. Charles Ave, New Orleans, LA 70118}
\email{rako@tulane.edu}
\urladdr{dauns01.math.tulane.edu/\textasciitilde rako}

\author{Robin Koytcheff}
\address{Department of Mathematics, University of Louisiana at Lafayette, PO Box 43568, Lafayette, LA 70504}
\email{koytcheff@louisiana.edu}
\urladdr{userweb.ucs.louisiana.edu/\textasciitilde C00401634}

\author{Ismar Voli\'c}
\address{Department of Mathematics, Wellesley College, 106 Central Street, Wellesley, MA 02481}
\email{ivolic@wellesley.edu}
\urladdr{ivolic.wellesley.edu}

\dedicatory{Dedicated to the memory of Fred Cohen (1945--2022)}

\makeatletter
\@namedef{subjclassname@2020}{%
  \textup{2020} Mathematics Subject Classification}
\makeatother


\begin{abstract}
The based loop space of a configuration space of points in a Euclidean space can be viewed as a space of pure braids in a Euclidean space of one dimension higher.  We continue our study of such spaces in terms of Kontsevich's CDGA of diagrams and Chen's iterated integrals.  We construct a power series connection which yields a Hopf algebra isomorphism between the homology of the space of pure braids and the cobar construction on diagrams.  It maps iterated Whitehead products to trivalent trees modulo the IHX relation.  As an application, we establish a correspondence between Milnor invariants of Brunnian spherical links and certain Chen integrals.  Finally we show that graphing induces injections of a certain submodule of the homotopy of configuration spaces into the homotopy of many spaces of high-dimensional string links.  We conjecture that graphing is injective on all rational homotopy classes.
\end{abstract}



\maketitle

\vspace{-1pt}

\tableofcontents

\newpage

\section{Introduction}\label{S:Intro}

We study the space $\Omega \Conf(m,\R^n)$ of smooth based loops
in the space of $m$-point configurations in $\R^n$ for $n \geq 3$.  One may view this as the space of pure braids in $\R^{n+1}$.
Its integral homology $H_\ast(\Omega \Conf(m,\R^n); \Z)$ and its rational homotopy groups 
$\pi_\ast(\Omega \Conf(m,\R^n))\otimes \Q$ were determined by Cohen and Gitler \cite{Cohen-Gitler}. 
For $n=2$, i.e.~the setting of classical pure braids, an analogous result was obtained by  Kohno \cite{Kohno:1985, Kohno:2000} using a result of Arnold \cite{Arnold:1969} and Chen's iterated integrals \cite{Chen:1973}.  
In \cite{KKV:2020} we studied an isomorphism of algebras $\Phi$ given by composing the formality integration map of Kontsevich \cite{Kontsevich:1999, Lambrechts-Volic:2014} with Chen's integrals, given as follows.
Let $(\D(m),\delta)$ be Kontsevich's (graded-)commutative differential graded algebra (CDGA) of diagrams, which we consider as an algebra over $\R$.  
It fits into a zig-zag of quasi-isomorphisms 
\[
H^\ast_{dR} (\Conf(m,\R^n)) \overset{\overline{I}}{\longleftarrow} \D(m) \overset{I}{\longrightarrow} \Ch^*_{dR}(\Conf(m,\R^n)) 
\]
that establishes the formality of $\Conf(m,\R^n)$, in the sense of Sullivan \cite{Deligne-Griffiths:1975, Sullivan:1977}, over 
$\R$.
 (It also leads to the formality of the little $n$-disks operad.) 
Let $(\Ba(\D(m)),d_{\Ba})$ be the bar complex on $\D(m)$, and let $H^\ast_{dR}(\Omega \Conf(m,\R^n))$ be the analogue of de Rham cohomology defined via Chen's iterated integrals. 
Then the composition of formality integration with Chen's integrals is 
\begin{equation}\label{eq:Phi-cohomology-iso-D(m)}
\Phi = \int_{\mathrm{Chen}} \circ \Ba(I): H^\ast(\Ba(\D(m)))\longrightarrow H^\ast_{dR}(\Omega \Conf(m,\R^n)),
\end{equation}
where we use the same symbol for the map on cochains and the induced map on cohomology.  
Here and throughout, we work over $\R$ unless specified otherwise.  
In this paper, we continue to study spaces of pure braids and this map $\Phi$, but we focus on homology and homotopy.

 Briefly, our main results can be organized as follows:
\begin{itemize}
\item[$(i)$]
We describe the isomorphism\footnote{In more detail, we construct 
$\Theta: H_\ast(\Omega \Conf(m,\R^n)) 
\longrightarrow H_\ast(\Ba^\ast(\Dm(m)))$ such that $\Theta=\Phi^\ast\circ i$, where $i$ is the dual of the de Rham homomorphism.  See Diagram \eqref{eq:Theta-Phi-dual-diagram}.}
\[
\Phi^* \colon H_\ast(\Omega \Conf(m,\R^n)) \longrightarrow H_\ast(\Ba^\ast(\Dm(m)))
\] 
in terms of Chen's formal power series connection (Theorem \ref{thm:Phi-pwr-series}), where $\Dm(m)$ is a slightly bigger CDGA than $\D(m)$.
\item[$(ii)$]
We recursively describe $\Theta$ on iterated Samelson (or Whitehead) products, and we show that $\Theta$ induces an injection of $\pi_{*}(\Omega \Conf(m,\R^n)) \otimes \R \cong \pi_{*+1}(\Conf(m,\R^n))\otimes \R$ into the space $\mathcal{T}^n(m)$ of  trivalent trees with leaves labeled by $1,\dots,m$, modulo the graded AS and IHX relations (Theorem \ref{thm:theta-trees}); the correspondence is particularly simple for non-repeating monomials.
This leads to an injection of link homotopy classes of certain Brunnian spherical links in $\R^n$ into $\mathcal{T}^n(m)$ (Corollary \ref{cor:brunnian-spherical-links}).
\item[$(iii)$]
We produce an injection of the subspace of $\pi_*(\Conf(m, \R^n))$ corresponding to trees with distinctly labeled leaves
into homotopy groups $\pi_\ell$ of spaces of $m$-component, $k$-dimensional string links in $\R^{n+k}$ for various values of $n \geq 3$ and $k \geq 1$ (Theorem \ref{T:braids-as-links}).  Certain values of $m$, $k$, and $n$ yield $\ell=0$, i.e.~isotopy classes of high-dimensional string links.  We conjecture that this result can be extended to all of $\pi_*(\Conf(m, \R^n))\otimes \Q$.
\end{itemize}

Parts (i) and (ii) are higher-dimensional analogues of results of Kohno \cite{Kohno:2000} and Habegger and Masbaum \cite{Habegger-Masbaum:2000} on Vassiliev invariants and the Kontsevich integral for classical braids and string links.  Our power series connection is an alternative to one provided by Kohno \cite{Kohno:2010} for pure braids in $\R^{n+1}$. 
Part (iii) generalizes a previous result of ours \cite{KKV:2020} for $k=1$ and $n \geq 3$, which in turn generalizes a result of Artin for $k=1$ and $n=2$ on isotopy classes of classical braids and string links \cite{Artin:Braids}.
A more detailed summary, including some auxiliary results, is given below.

\subsection{Overview of results}\label{S:summary}

In this paper we primarily study the dual to the map $\Phi$ in \eqref{eq:Phi-cohomology-iso-D(m)}.
To study homotopy groups, it is convenient to use the CDGA $\Dm(m)$ which is obtained from $\D(m)$ by allowing diagrams with multiple edges.  The quotient $\Dm(m) \twoheadrightarrow \D(m)$ by diagrams with multiple edges is a quasi-isomorphism.
The map $\Phi$ in \eqref{eq:Phi-cohomology-iso-D(m)} trivially extends to $H^\ast(\Ba(\Dm(m)))$, because the extension of the formality integration map to $\Dm(m)$ vanishes on diagrams with multiple edges.  Therefore the dual $\Theta=\Phi^\ast$ is valued in $H_\ast(\Ba^\ast(\Dm(m)))\cong H_\ast(\Ba^\ast(\D(m)))$, and we can write the dual map to \eqref{eq:Phi-cohomology-iso-D(m)} as
\begin{equation}\label{eq:Theta-homology-iso-D(m)}
\Theta: H_\ast(\Omega \Conf(m,\R^n))\longrightarrow 
H_\ast(\Ba^\ast(\Dm(m))),
\end{equation}
where $(\Ba^\ast(\Dm(m)),d^\ast_{\Ba})$ is the cobar complex of the dual coalgebra $(\Dm(m)^\ast,\delta^\ast)$.

We picture monomials $\gamma \in \Ba(\Dm(m))$ as diagrams $\Gamma$ stacked horizontally from left to right, based on $m$ strands with free vertices of valence $\geq 3$.  
There are functionals $\Gamma^*$ indexed by diagrams $\Gamma$ which are defined by the pairing $\langle \Gamma_1, \Gamma_2^*\rangle = \pm|\mathrm{Aut}(\Gamma_1)| \delta_{\Gamma_1, \Gamma_2}$ where $\mathrm{Aut}(\Gamma)$ is the group of automorphisms of $\Gamma$ and the sign depends on certain orientations on the $\Gamma_i$; see Section \ref{S:Diagrams}.
We picture dual diagrams $\Gamma^*$ and dual monomials $\gamma^*\in\Ba^\ast(\Dm(m))$ by  ``transposing'' the picture of $\gamma\in \Ba(\Dm(m))$ so that the monomial $\gamma^*$ is read from bottom to top.
For example
\begin{equation}\label{eq:sum-in-BD(m)}
	\begin{split}
  \gamma_1 & =\raisebox{-1.6pc}{\includegraphics[scale=0.12]{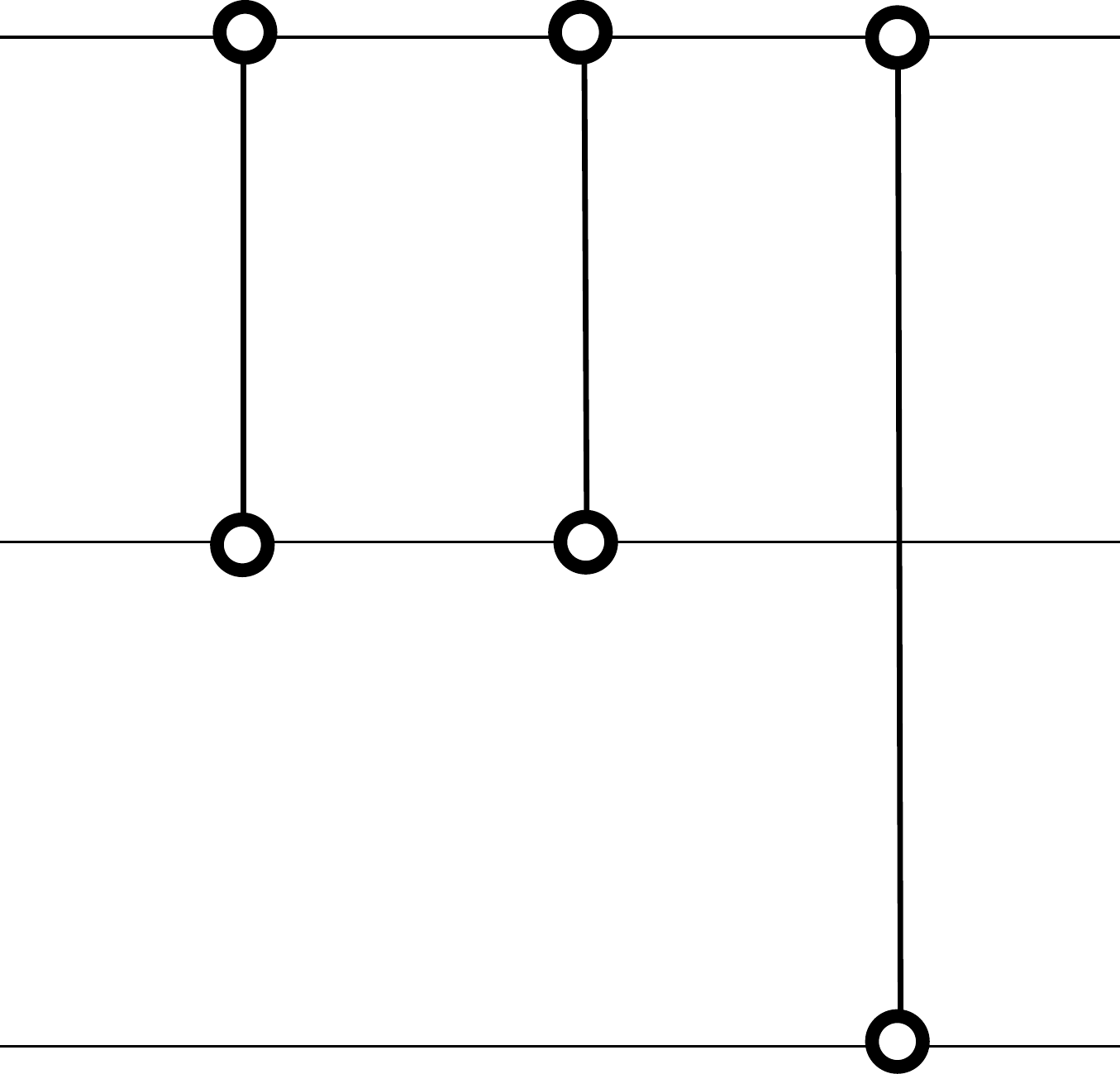}}, \quad
\gamma_1^\ast  =\raisebox{-1.6pc}{\includegraphics[scale=0.12]{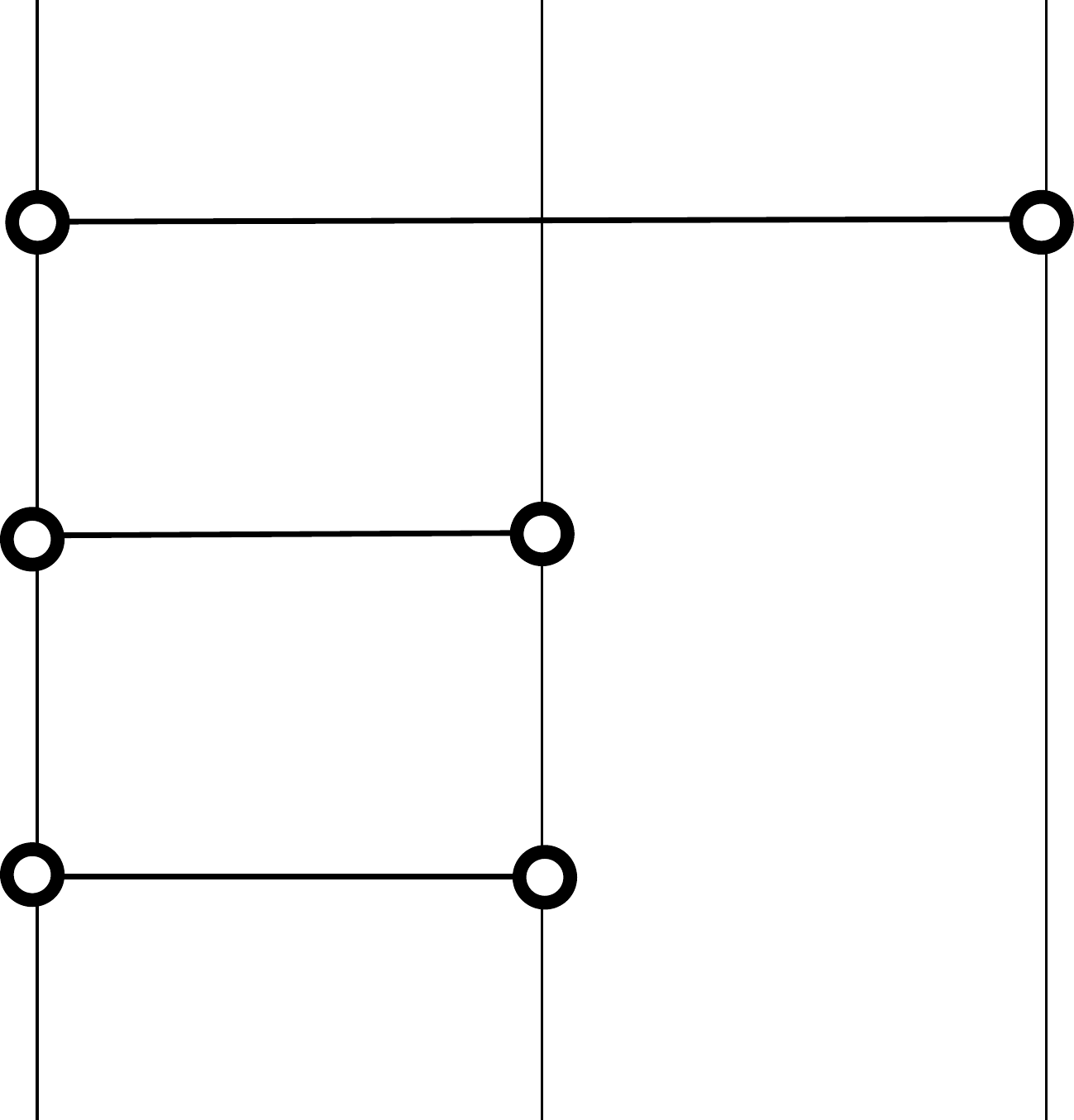}},  \quad 
  \gamma_2  =\raisebox{-1.6pc}{\includegraphics[scale=0.12]{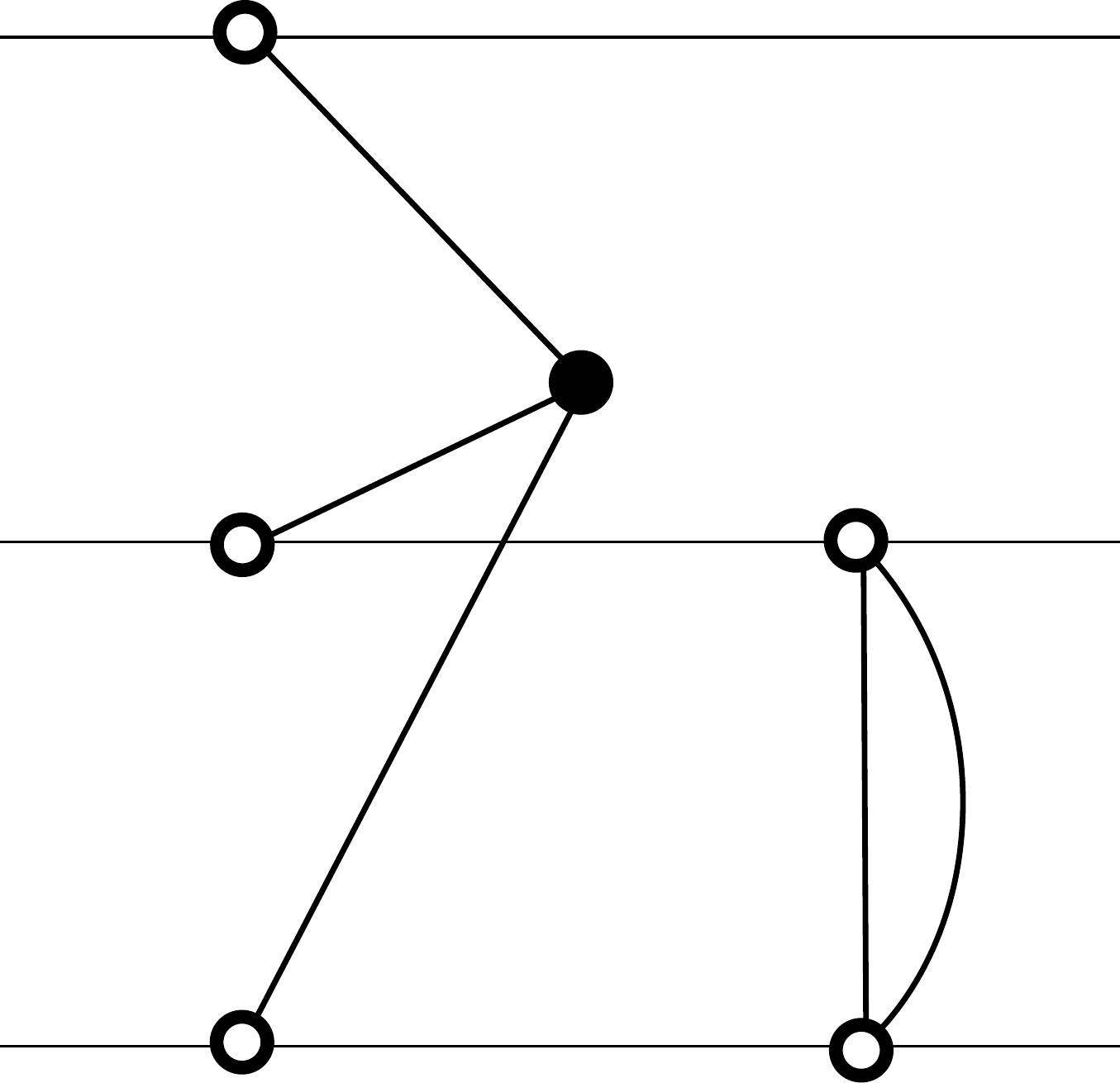}}, \quad
    \gamma_2^*  =\raisebox{-1.6pc}{\includegraphics[scale=0.12]{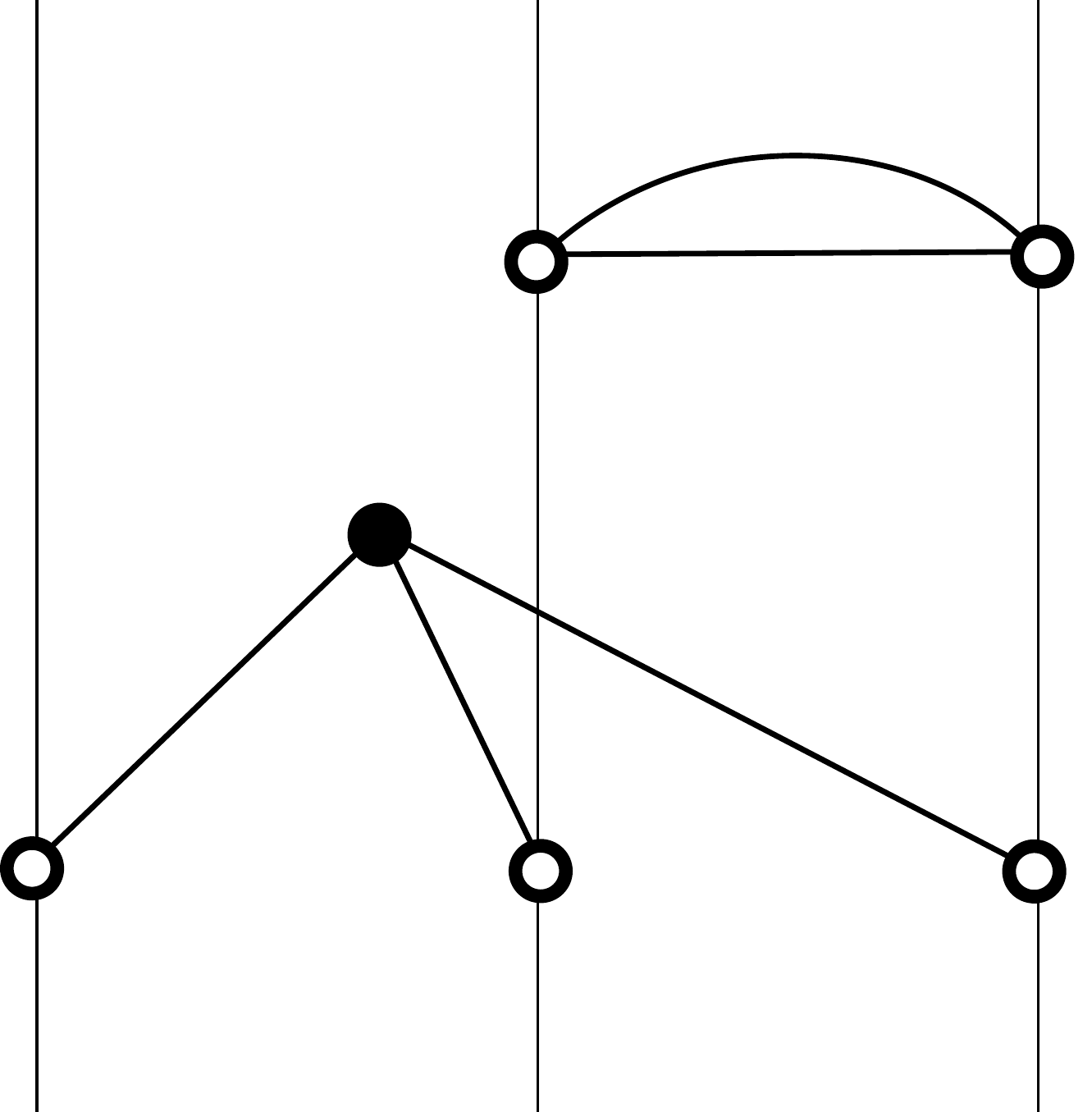}},\\
    \notag  
   \gamma_3 & = \raisebox{-1.4pc}{\includegraphics[scale=0.12]{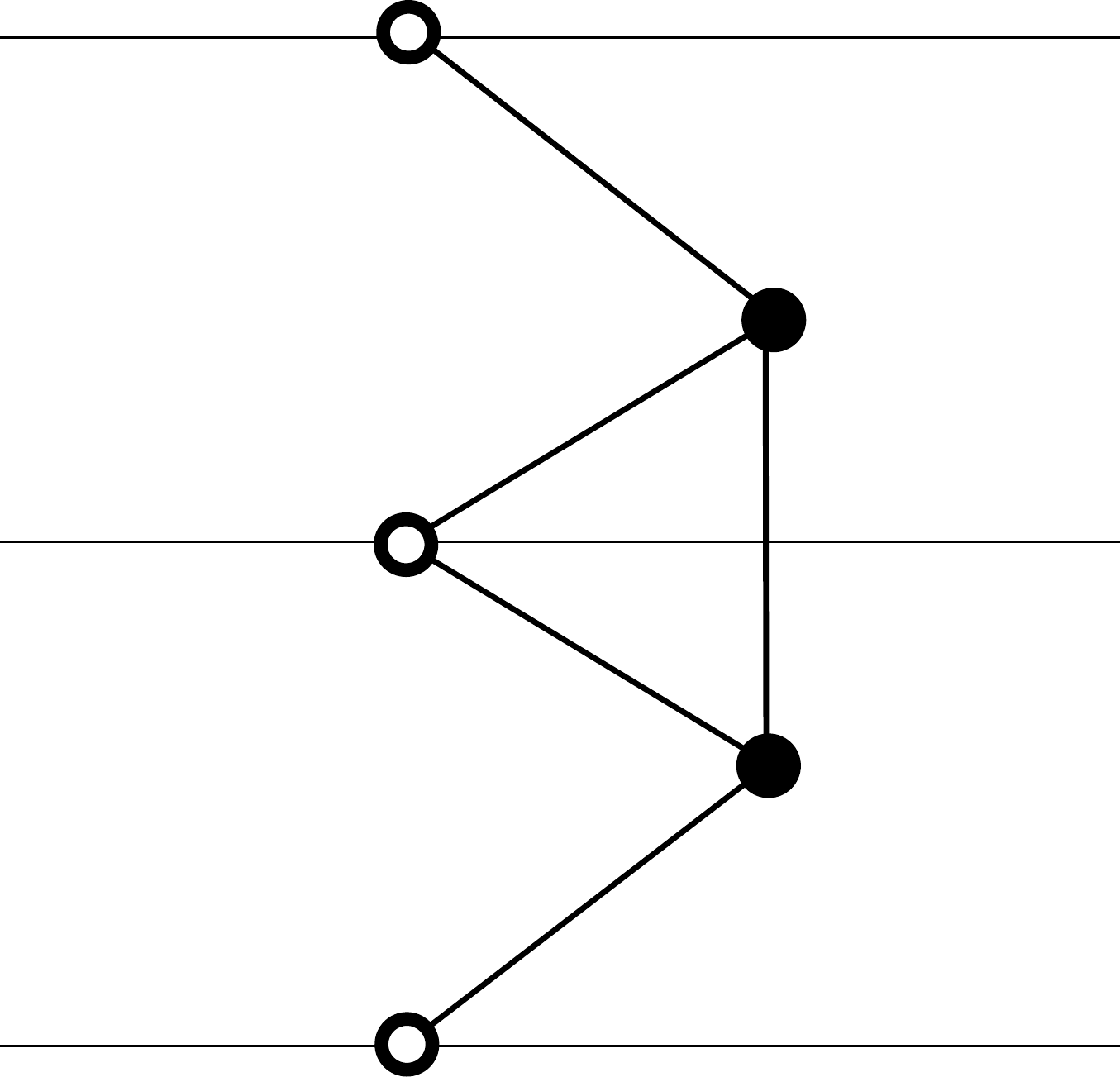}},   \quad
   \gamma_3^*  = \raisebox{-1.4pc}{\includegraphics[scale=0.12]{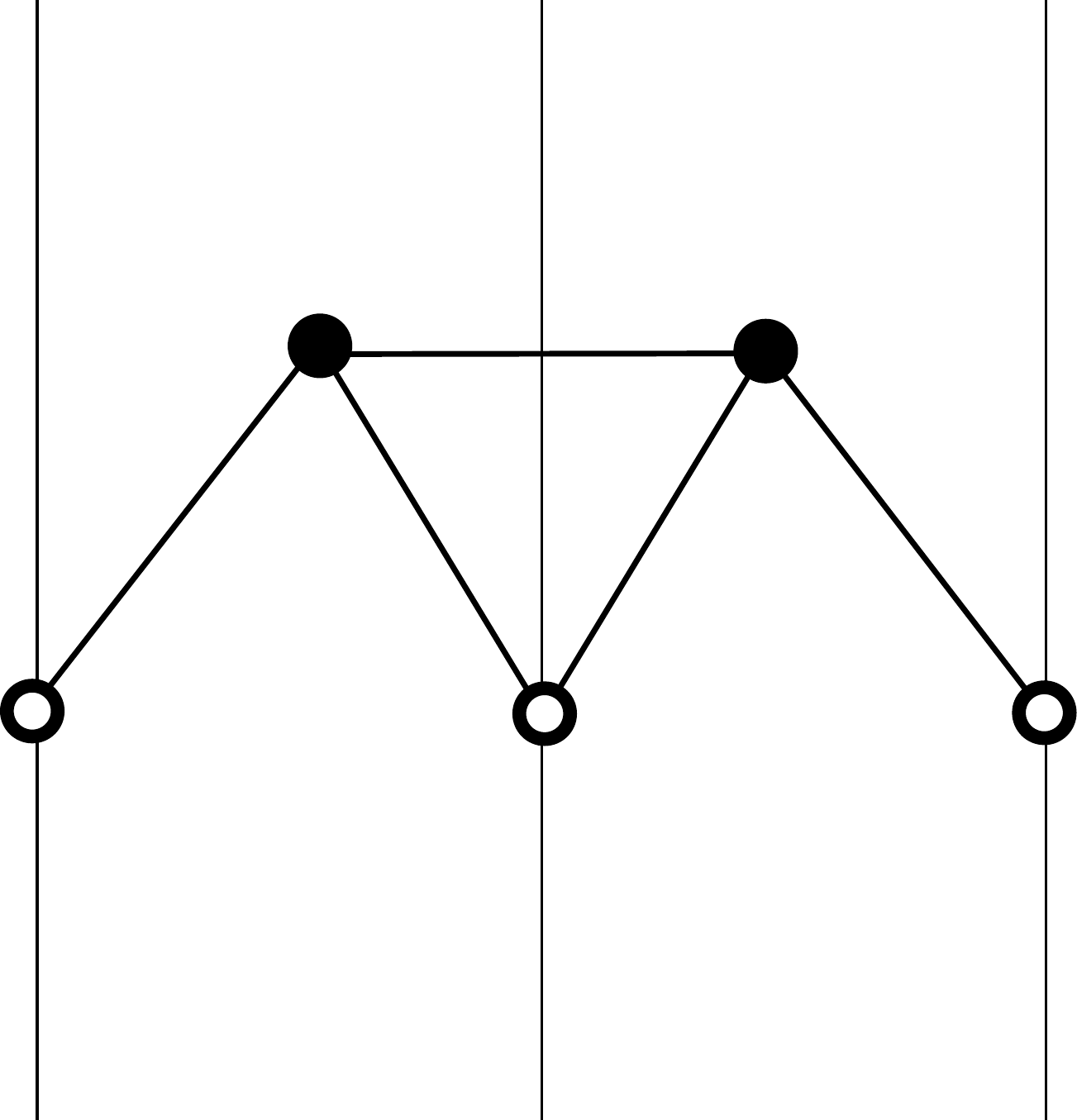}}\ .
   \end{split}
\end{equation}

The first improvement on our previous work \cite{KKV:2020} is to show that  \eqref{eq:Theta-homology-iso-D(m)} can be obtained via the power series connection method  \cite{Chen:1973}.  By work of Hain \cite{Hain:1984}, this illuminates the fact that $\Phi$ (or $\Theta$) is a map of Hopf algebras, which was mentioned in \cite{KKV:2020} but not explained in detail.
\begin{thm}\label{thm:Phi-pwr-series} Let $n \geq 3$.
	\begin{itemize}
	\item[(a)]  
	Let
	$\Theta:C^\textrm{sing}_\ast(\Omega \Conf(m,\R^n)) \longrightarrow \Ba^\ast(\Dm(m))$
	be the map induced by the transport of the formal power series connection $\omega$, valued in $\Ba^\ast(\D(m))\subset \Ba^\ast(\Dm(m))$ and defined by 
	\begin{equation}\label{eq:omega-form-0}
	\omega=\sum_{\Gamma\in \mathcal{B}(m)} I(\Gamma)\otimes \frac{[\Gamma^\ast]}{|\mathrm{Aut}(\Gamma)|} \ \in \ \Ch^*_{dR}(\Conf(m,\R^n)) \otimes \Dm(m)^*,
	\end{equation}
	where $I$ is the formality integration map and $\mathcal{B}(m)$ is the basis of diagrams  (which is well defined up to signs) for the subspace of $\D(m)$ spanned by nonempty diagrams.  
	Then at the level of homology, $\Theta$ agrees with the map $\Phi^\ast$ in \eqref{eq:Theta-homology-iso-D(m)}.
	\item[(b)] At the level of homology,  $\Theta$ is a Hopf algebra isomorphism.
	\end{itemize}
\end{thm}
\no Since $I(\Gamma)=0$ for every $\Gamma$ with multiple edges, the sum in \eqref{eq:omega-form-0} is really just over the basis $\mathcal{B}(m)$ of the subspace of nonempty diagrams in $\D(m)$. The general construction of the power series connection for the loop space $\Omega M$ (where $M$ is a manifold) proposed in \cite{Chen:1973, Hain:1984} involves an inductive procedure yielding a connection form $\omega$ which is valued in the tensor algebra of $H_{\ast-1}(M;\R)$. As a result, $\omega$ is generally not available via a direct formula. In Theorem \ref{thm:Phi-pwr-series},  we obtain such a formula thanks to the existence of the formality integration map $I$.  See Section \ref{S:power-series} for further information on Chen's formal power series connections.

Using just the fact that $\Theta$ preserves the Hopf algebra structure, we deduce that it maps the real homotopy groups 
$\pi_\ast(\Omega \Conf(m,\R^n))\otimes \R$ to primitive diagrams.  This result is similar to one obtained by Lambrechts and Turchin \cite{Lambrechts-Turchin}; the work of  Conant \cite{Conant:AJM} is also related.

\newpage

\begin{thm}
\label{thm:primitive-diagrams} \ 
\begin{itemize}
\item[(a)]
	The map $\Theta$ restricts to an isomorphism 
	\[
	\pi_\ast(\Omega \Conf(m,\R^n)) \otimes \R \overset{\Theta}{\longrightarrow} 
	PH_\ast(\Ba^\ast(\D(m)))
	\cong H_\ast(P\Dm(m)^\ast,\delta^\ast),
	\]
	where $P(-)$ denotes the subspace of primitive elements in a coalgebra.
\item[(b)]
The dual $\Phi$ of $\Theta$ induces an isomorphism 
\[
H^\ast(I\Dm(m),\widetilde{\delta}) 
\cong IH^\ast(\Ba(\D(m))) \overset{\Phi}{\longrightarrow} 
\mathrm{Hom}(\pi_\ast(\Omega \Conf(m,\R^n)), \, \R), 
\]
\end{itemize}
where $I(-)$ denotes the subspace of indecomposable elements in an algebra, $\widetilde{\delta}=\pi_{I\Dm}\circ\delta$, and $\pi_{I\Dm}$ is the projection onto $I\Dm$.
\end{thm}	

We prove this result using an explicit formula relating the Samelson product to the boundary operator on $\Dm(m)^\ast$ (Lemma \ref{lem:bracket-recursive}).  We also use that formula to deduce part of our next main result, Theorem \ref{thm:theta-trees}. 
Before stating it, we provide some context.

Theorem \ref{thm:primitive-diagrams} and the Milnor--Moore Theorem \cite{Milnor-Moore:1965} lead to a diagrammatic description of homology: 

\begin{cor}\label{thm:homology-prim}
The homology  $H_\ast(\Omega \Conf(m,\R^n))$ is generated by diagrams in $H_\ast(P\Dm(m)^\ast)$, i.e.,
  \[
   H_\ast(\Omega \Conf(m,\R^n))\cong  U(H_\ast(P\Dm(m)^\ast)),
  \]
  where $U$ is the universal enveloping algebra with the bracket on 
  $H_\ast(P\Dm(m)^\ast)$ as defined in \eqref{eq:PD-bracket}.
  \qed
\end{cor}	

Corollary \ref{thm:homology-prim} can be viewed as a diagrammatic description of $H_\ast(\Omega \Conf(m,\R^n))$, alternative to the horizontal chord diagram algebra of Cohen and Gitler \cite{Cohen-Gitler} and Kohno \cite{Kohno:2002}
(see also Fadell and Husseini \cite{Fadell-Husseini:2001}). To explain, let $\mathcal{L}_{m}(n-2)$ be the graded Lie algebra 
over $\R$ 
on generators $B_{j,i}, 1\leq i<j\leq m$, each of degree $n-2$, modulo the {\em Yang--Baxter relations} \cite[p.~1711]{Cohen-Gitler}:
\begin{equation}\label{eq:YB-relations}
	\begin{split}
[ B_{j,i},B_{k,\ell} ] & =0\ \ \text{for}\ i, j, k, \ell \ \text{distinct},\\
 B_{i,j} & = (-1)^n B_{j,i},  \\
[ B_{i,j},B_{i,t}+(-1)^n B_{t,j}] & =0 \ \ \text{for}\ 1\leq j < t < i\leq m, \\
[ B_{t,j},B_{i,j}+B_{i,t}] & =0\ \ \text{for}\ 1\leq j<t<i\leq m.
	\end{split}
\end{equation}
Note that the last identity follows from the previous two, and is stated only for convenience. For brevity, we sometimes denote $\mathcal{L}_{m}(n-2)$ simply by $\mathcal{L}_m$ when the ambient dimension $n$ is understood.

\begin{thm}[Cohen and Gitler \cite{Cohen-Gitler}]\label{thm:cohen-gitler}
	If $n \geq 3$, the integral homology of $\Omega \Conf(m,\R^n)$ is torsion-free, and there is an isomorphism of Lie algebras:
	\begin{equation}\label{eq:L_m}
	\mathcal{L}_{m}(n-2)\longrightarrow PH_\ast(\Omega \Conf(m,\R^n))\cong \pi_\ast(\Omega \Conf(m,\R^n))\otimes \R.
	\end{equation}
	Furthermore, the induced map on the level of universal enveloping algebras,
	\begin{equation}
	U\mathcal{L}_{m}(n-2)\longrightarrow H_\ast(\Omega \Conf(m,\R^n)),
	\end{equation}
	is an isomorphism of Hopf algebras.\footnote{Cohen and Gitler actually prove the analogous result over $\Z$ \cite{Cohen-Gitler}, though in that setting the primitives in homology cannot be identified with a homotopy group.	We instead work over $\R$.} 
\end{thm}

\begin{rem}
\label{R:classical-lie-alg}
For $n=2$, the resulting Lie algebra $\mathcal{L}_m(0)$  is called the Drinfeld--Kohno Lie algebra.  
Kohno \cite{Kohno:1985} showed that it is the associated graded Lie algebra of the pure braid group $\mathcal{PB}_m=\pi_0 (\Omega\Conf(m,\R^2))$ and that the degree completion of $\mathcal{L}_m(0)$ is the Malcev Lie algebra of $\mathcal{PB}_m$.  
See for example the work of Suciu and Wang \cite{Suciu-Wang:Forum2019, Suciu-Wang:2020} for definitions of these Lie algebras associated to groups.
The resulting filtration agrees with the Vassiliev filtration on braids, and Kohno's application of power series connections in \cite{Kohno:1987} shows that Vassiliev invariants separate braids.
Similar algebras of diagrams appear for Vassiliev invariants of knots in $\R^3$, though it is not known whether they separate knots.  
Bar-Natan \cite{Bar-Natan:1995} described these invariants as dual both to spaces of chord diagrams, as in Theorem \ref{thm:cohen-gitler}, and to spaces of (trivalent) graphs with free vertices, as in Corollary \ref{thm:homology-prim}.
The Kontsevich integral is the analogue of the power series connection for knots \cite{Bar-Natan:1995, Lin:PowerSeries}.
\end{rem}

\begin{rem}\label{rem:B_ji}
	The generators $B_{j,i}$ of $\mathcal{L}_m(n-2)$ are represented by spherical cycles \cite{Fadell-Husseini:2001}
	\begin{equation}\label{eq:B_ji}
	B_{j,i}: {S}^{n-2}\longrightarrow \Omega\Conf(m, \mathbb{R}^n),\qquad 1\leq j<i\leq m,
	\end{equation}
	which are the adjoints of 
	\[
	\begin{split}
	b_{j,i}: {S}^{n-1} & \longrightarrow \Conf(m, \mathbb{R}^n),\qquad b_{j,i}(\xi) = (y_1,\ldots,y_{i-1},y_j-\xi,y_{i+1},\ldots, y_m).
	\end{split}
	\]
	Here $(y_1,..., y_m)$ is the basepoint of $\Conf(m, \mathbb{R}^n)$.
\end{rem}

We will use the fact that as a vector space, $\mathcal{L}_m \cong \pi_* (\Omega\Conf(m,\R^n))\otimes \R$ is the direct sum \cite{Fadell-Husseini:2001}
\begin{equation}\label{eq:L_m(n-2)}
 \mathcal{L}_m\cong \bigoplus^m_{j=2} \mathcal{L}(B_{j,1},B_{j,2},\ldots, B_{j,j-1}), 
\end{equation}
where $\mathcal{L}(B_{j,1},B_{j,2},\ldots, B_{j,j-1})$ is the free graded Lie algebra generated by $B_{j,1},B_{j,2},\ldots, B_{j,j-1}$.
This decomposition holds because the projections which forget points have sections up to homotopy.
It is analogous to the splitting of the pure braid group as an iterated semi-direct product of free groups.


Our next main result is Theorem \ref{thm:theta-trees}, in which we show that $\Theta$ maps iterated Samelson products in $\pi_* (\Omega\Conf(m,\R^n))\otimes \R \cong  \mathcal{L}_m$ 
to certain trivalent trees.  
First, we find it convenient to use a spanning set of $\mathcal{L}_m$ consisting of certain bracket expressions. 
Namely, any graded Lie algebra is spanned by left-normed monomials in its generators (Lemma \ref{lem:left-normed-span}).  For 
$\mathcal{L}_m$, these are of the form 
\begin{equation}
\label{eq:left-normed-brackets}
\begin{split}
B_{j;I}&:=[ B_{j,i_1},B_{j,i_2},B_{j,i_3}\ldots, B_{j,i_k} ] :=[[\ldots [[B_{j,i_1},B_{j,i_2}],B_{j,i_3}],\ldots ], B_{j,i_k}],\\
& 2 \leq j \leq m, \qquad I=(i_1,\ldots,i_k), \qquad i_1, \dots, i_k <j.
 \end{split}
\end{equation}
Our choice to use left-normed rather than right-normed monomials is arbitrary.
Paring this spanning set down to a basis is not straightforward,\footnote{The monomials in the well known Lyndon basis \cite{Chen-Fox-Lyndon:1958} or Hall bases are generally not left-normed, though there is relatively recent work on bases of left-normed \cite{Walter:2010arXiv} and right-normed \cite{Chibrikov:2005} monomials for a free Lie algebra.}
but a certain subspace that will be important for us admits an easily described basis of left-normed monomials.

Namely, recall that the \emph{Lie module} $Lie(m-1)$ is the submodule of the free (graded) Lie algebra on $m-1$ generators spanned by brackets where each generator appears exactly once. 
It is well known that $\dim Lie(m-1) = (m-2)!$, with a basis of left-normed monomials in which the leftmost generator is  the first generator \cite{Cohen:2016aa, Sinha:Graphs-Trees}.  In our setting, these are the monomials 
\begin{equation}\label{eq:Lie(m-1)-basis}
 B_{m;1,\sigma(2),\ldots,\sigma(m-1)}=[ B_{m,1}, B_{m,\sigma(2)},\ldots  B_{m,\sigma(m-1)} ]
\end{equation}
 where $\sigma$ is a permutation of $\{2,\dots, m-1\}$.
(Koschorke \cite{Koschorke:1997} observed that geometrically, $Lie(m-1)$ is isomorphic to the subspace of $m$-component Brunnian spherical links in the space of spherical link maps; see Section \ref{S:brunnian-spherical-links} for further details.)

Any monomial $B$ in generators of degree $n$ of a graded Lie algebra can be represented by a planar rooted trivalent tree with leaves labeled by the generators; see for example the first pictures in \eqref{eq:bracket-to-trees} and \eqref{eq:bracket-to-trees-2}.
In the setting of such trees, the anti-symmetry and Jacobi relations are often called the (graded) AS and IHX relations.  
The planar embedding of this tree can be used to determine an equivalence class of labeling called an \emph{orientation}, defined essentially as in Definition \ref{D:DmOrientations}.  For $n$ even, a labeling consists of an ordering of its edges, while for $n$ odd, it consists of an orientation of each edge and an ordering of the non-leaf vertices.  A tree with an orientation is called an \emph{oriented tree}.
In either parity of $n$, the orientation relations that yield equivalences of labelings encode the (graded) AS relations.

Now let $B=B_{j;I}$ be a monomial in $\pi_*(\Omega \Conf(m,\R^n))\otimes \R$ as in \eqref{eq:left-normed-brackets}.
Let $T(B)$ be the result of replacing each leaf label $B_{j,i}$ by $i$ and labeling the root by $j$.
Via the decomposition \eqref{eq:L_m(n-2)}, the assignment $B \mapsto T(B)$ determines an injection $T$ from $\pi_*(\Omega\Conf(m,\R^n))\otimes \R$ to the $\R$-vector space $\mathcal{T}^n(m)$ of (\emph{unrooted}) oriented trivalent trees with leaves labeled by $\{1,...,m\}$ modulo the (graded) IHX relations.  

Further, a representative of any element of $\mathcal{T}^n(m)$ can be viewed as an element of $P\overline{\D}(m)^*$, by identifying leaves with the same label.  
Below are examples with $m=4$ of the injection $B \mapsto T(B)$ and the identification of the given representative of $T(B)$ with an element of $P\overline{\D}(m)^*$.  The labelings which determine orientations are not shown.
\begin{equation}\label{eq:bracket-to-trees}
B=[[B_{4,1},B_{4,2}],B_{4,3}]
\lrarrow \vvcenteredinclude{0.15}{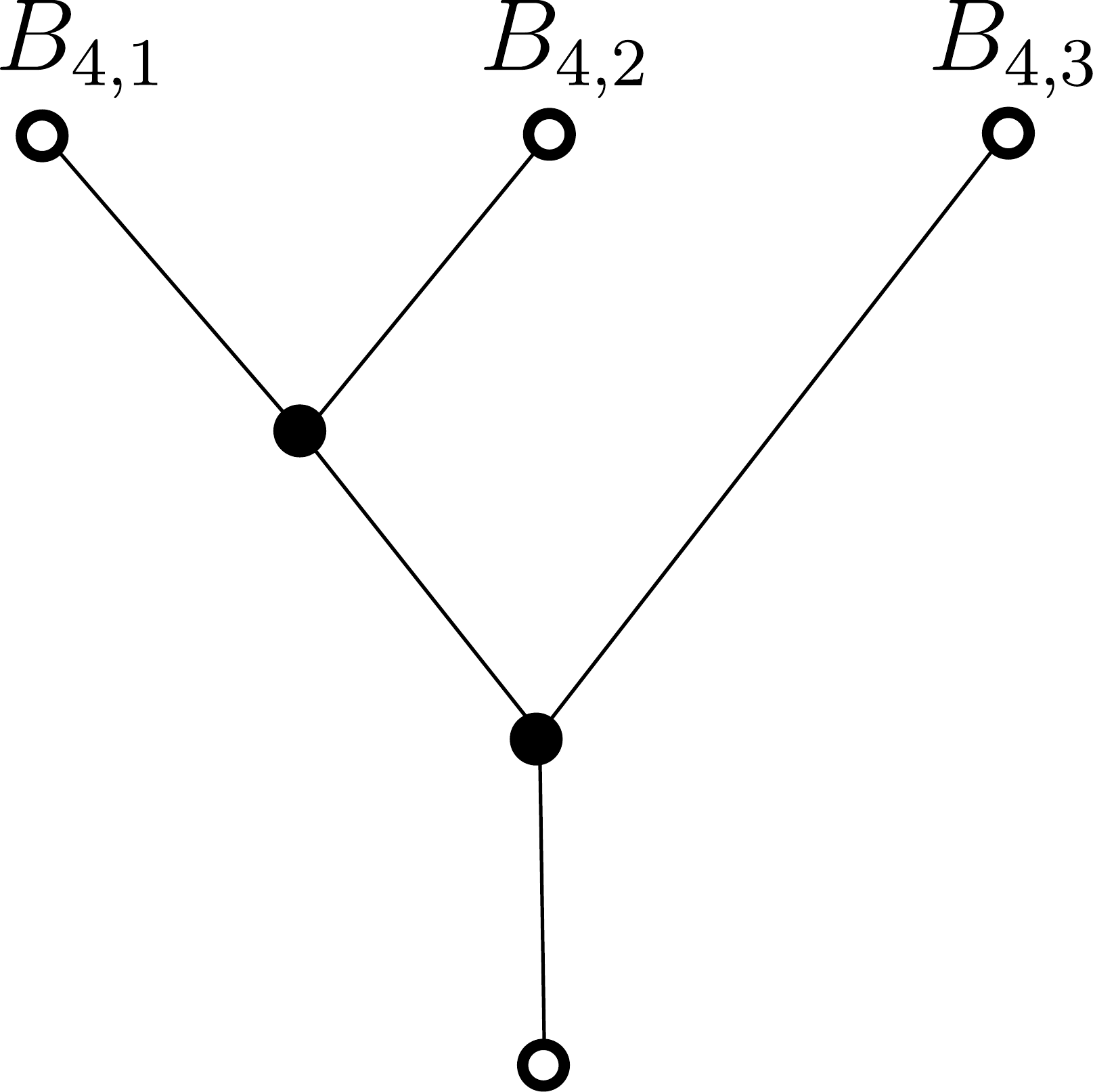}
\mapsto T(B)=\vvcenteredinclude{0.15}{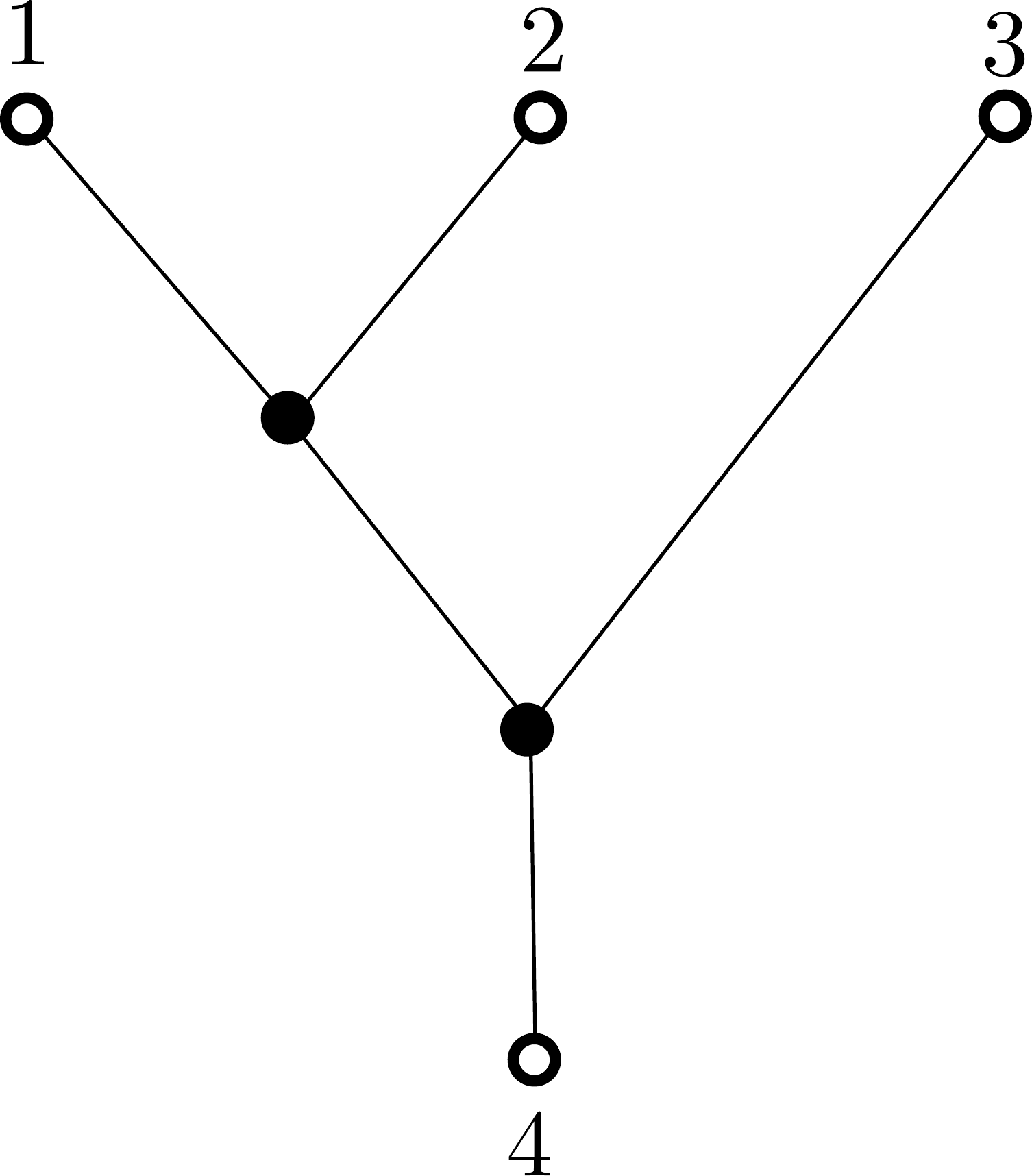} 
\lrarrow  \vvcenteredinclude{0.15}{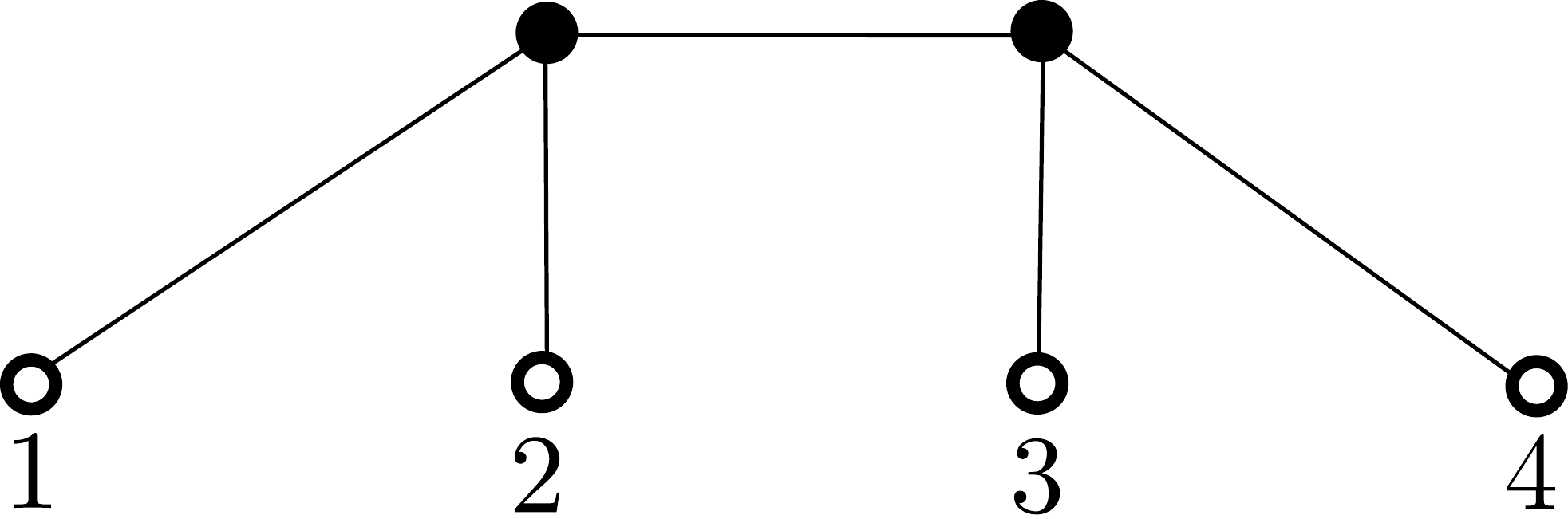} 
\end{equation}
 \begin{equation}\label{eq:bracket-to-trees-2}
 B=[[B_{4,2},B_{4,1}],B_{4,2}]
 \lrarrow \vvcenteredinclude{0.15}{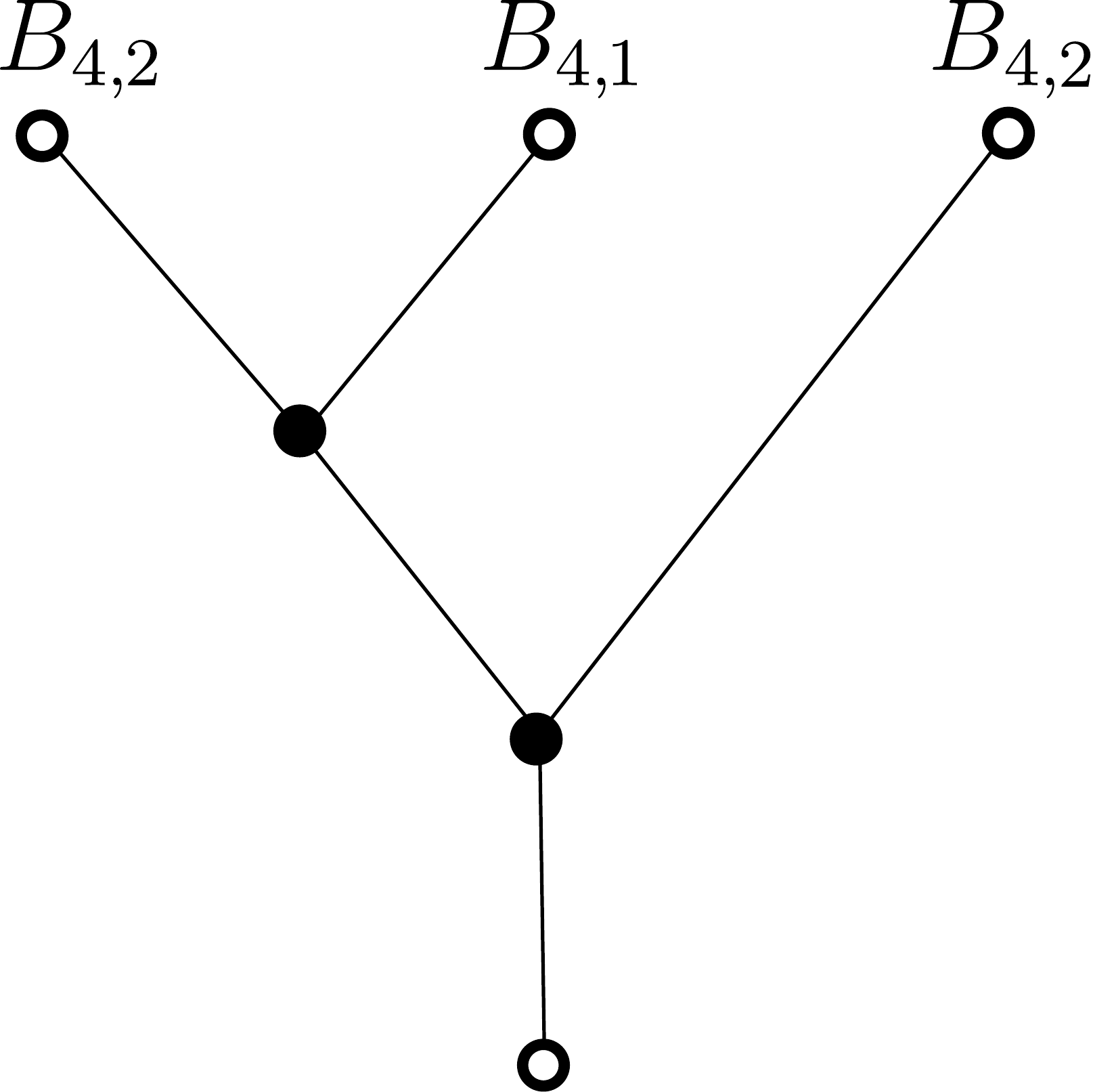}
 \mapsto T(B) =\vvcenteredinclude{0.15}{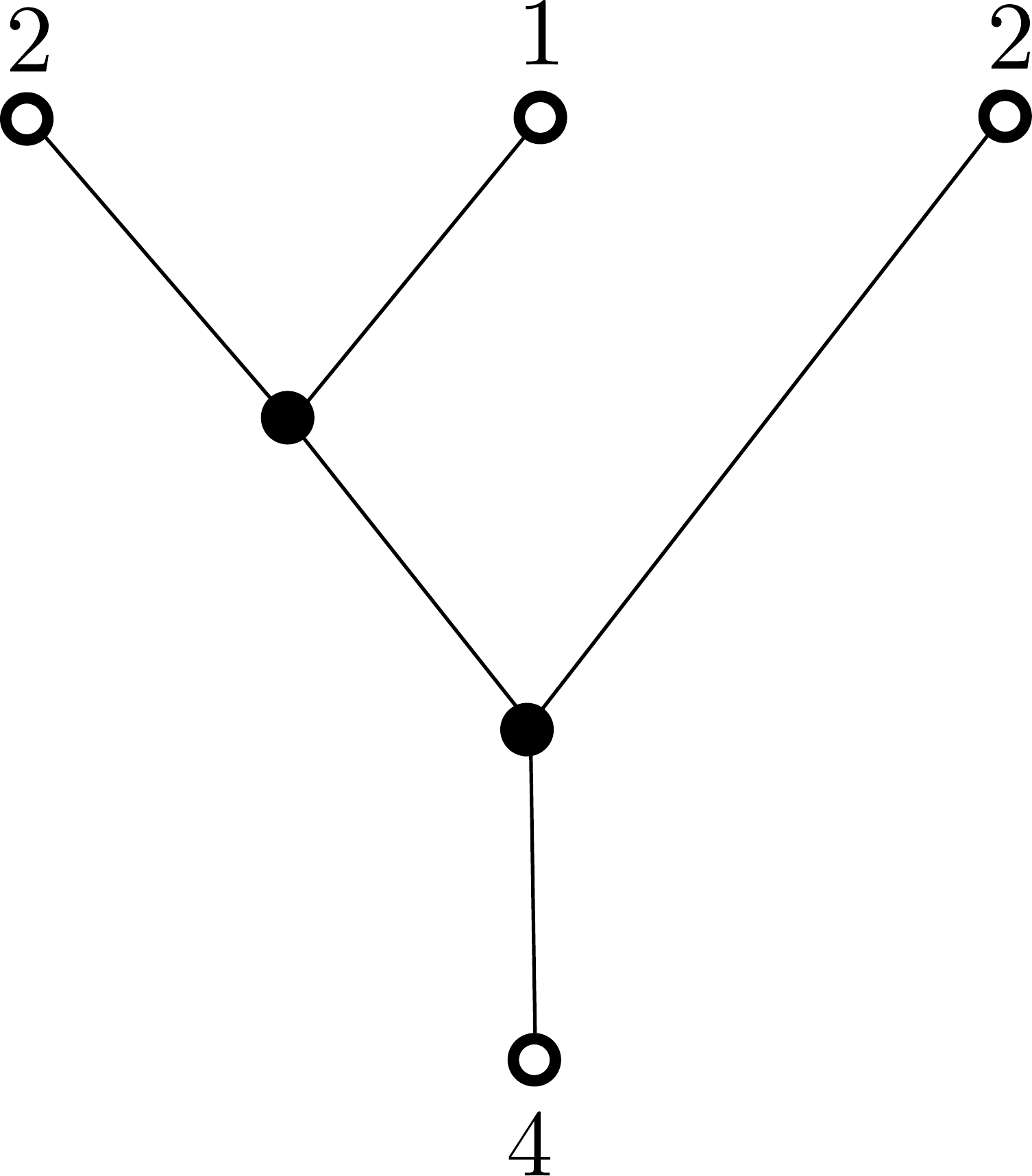}
 \lrarrow \vvcenteredinclude{0.15}{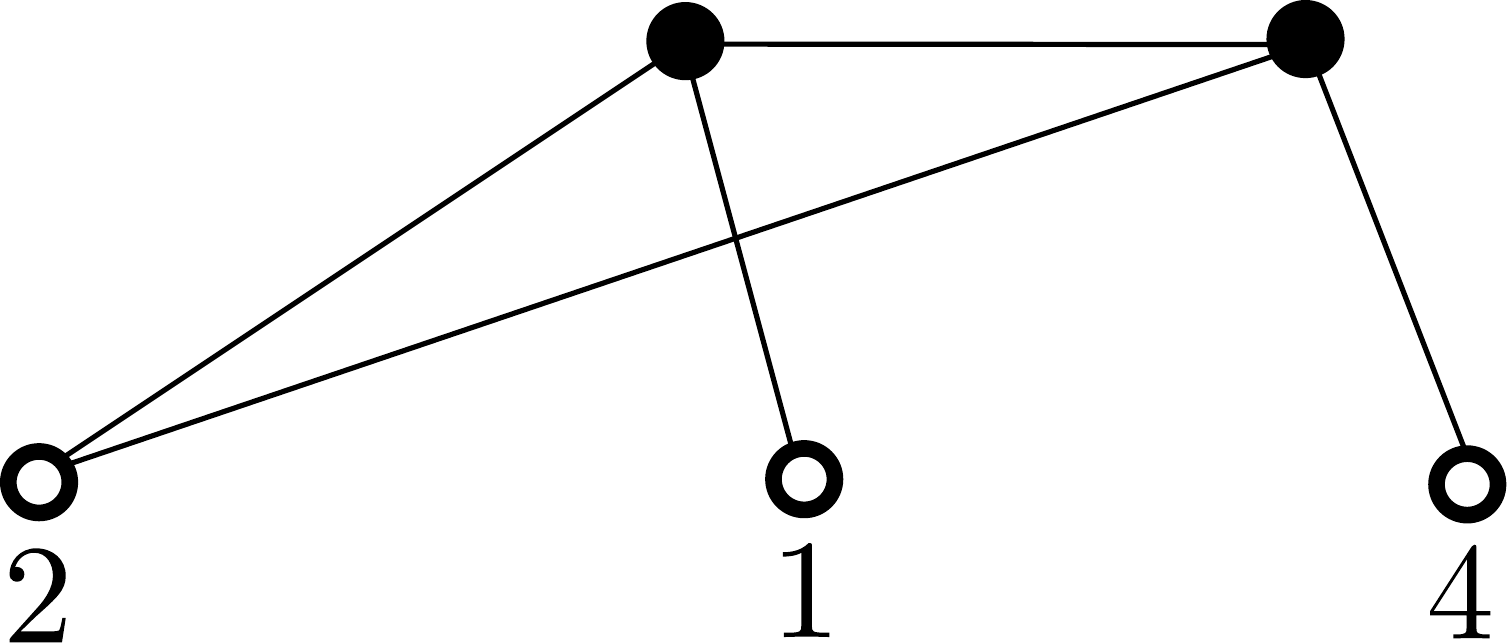}.
 \end{equation}
If $B$ is a left- or right-normed monomial, then the class of $T(B)$ can be represented by trees of the type as in the examples above, sometimes called ``tall trees'' or ``caterpillar trees.''

Part (b) of Theorem \ref{thm:theta-trees} below says that there is also a map $\widetilde{\Theta}: \pi_*(\Omega \Conf(m,\R^n))\otimes \R \to \mathcal{T}^n(m)$ induced by $\Theta$.
Although $\TTheta(B)$ is not always equal to $T(B)$, part (d) says that if the monomial $B$ has no repeated indices, then it is, at least up to sign.  For example, the last picture in \eqref{eq:bracket-to-trees} is (up to a sign) $\Theta(B_{4;1,2,3})$.   
Part (c) says that in general $\TTheta(B)$ is obtained from $T(B)$ by adding some extra terms.  
For example, the last picture in \eqref{eq:bracket-to-trees-2} is (up to a sign) one of the terms in $\Theta(B_{4;2,1,2})$.

\begin{thm}\label{thm:theta-trees} Below, assume $j, j_1, \dots, j_k \leq m$ and $I = (i_1, i_2, \dots, i_k)$ with $1 \leq i_1, \dots, i_k < j$.
	\begin{itemize}
		\item[(a)] If $C$ is a bracket expression containing generators $B_{j_1,i_1}, B_{j_2,i_2},\ldots, B_{j_k,i_k}$ 
		(with possible repeats), 
		then $\Theta(C)$ is represented in $H_\ast (P\Dm(m)^\ast)$ by a linear combination with $\pm 1$
		coefficients of all successive vertex blow-ups of the diagram 
		$(\Gamma_{j,i_1}\cdot \Gamma_{j,i_2}\cdot\ldots  \cdot \Gamma_{j,i_k})^*$,
		performed in the order given by the parenthesization in $C$, from innermost to outermost brackets.
		\item[(b)]  The map $\Theta$ gives rise to an injection $\widetilde{\Theta}:\pi_* (\Omega \Conf(m,\R^n)) \otimes \R$ into the space $\mathcal{T}^n(m)$ of trivalent trees with leaves labeled by $\{1,\dots, m\}$, modulo the (graded) AS and IHX relations.
		
		\item[(c)]  For any such multi-index $I$  (with repeated indices allowed), 
		$\langle \TTheta(B_{j;I}), T(B_{j;I}) \rangle =\pm 1$, where $\langle \cdot, \cdot \rangle$ is any Kronecker pairing determined by a basis of trees for $\mathcal{T}^n(m)$.
		
		\item[(d)] 
		If $I$ has no repeated indices, then $\TTheta(B_{j;I})=\pm T(B_{j;I})$.
		In particular, for any $k\leq j$,
		\begin{equation}\label{eq:borromean-tree}
		B_{j;1,2,\ldots,k-1} \overset{\Theta}{\longmapsto} \pm\vvcenteredinclude{0.12}{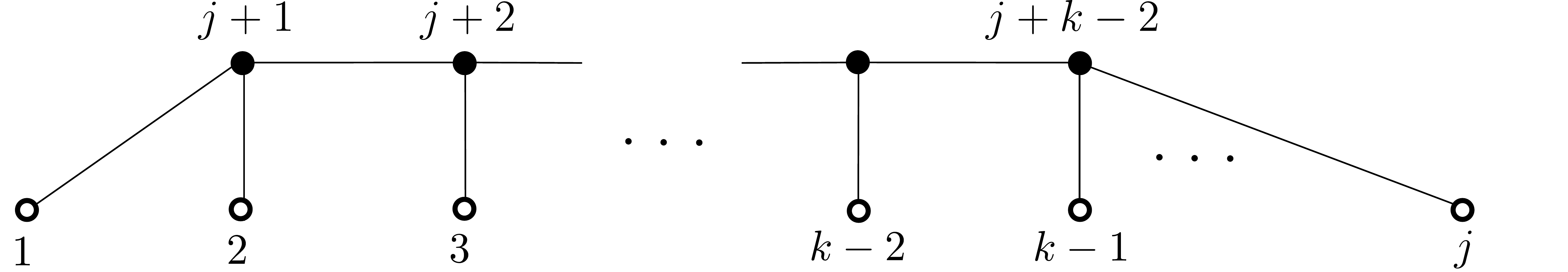} \quad \in \quad H_\ast(P\Dm(m)^\ast).
		\end{equation}
		Thus $\Theta$ maps a basis element \eqref{eq:Lie(m-1)-basis} of $Lie(m-1)$ to the diagram obtained from \eqref{eq:borromean-tree} by setting $k=j=m$ and permuting the segment vertices by $\sigma\in \Sigma(2,\ldots,m-1)$.  
		
	\end{itemize}
\end{thm}

Part (a) is a translation of  formula \eqref{eq:Gamma_C} from Lemma \ref{lem:bracket-recursive} into a diagrammatic procedure.
Examples \ref{Ex:B3121} and \ref{Ex:B31222} illustrate this procedure in the less straightforward setting of brackets with repeated indices.
Part (d) implies that if $B$ is a monomial without repeated indices, and $\gamma \in \Ba(\Dm(m))$ is any cocycle, then the evaluation of the integral associated to $\gamma$ on $B$ can be computed purely graphically.  Indeed, 
$\langle \Phi(\gamma), B \rangle=
\langle \Theta(B), \gamma \rangle$, and $\Theta(B)$ is essentially the tree $T(B)$ that can be written down from the bracket expression.
The full statement of Theorem \ref{thm:theta-trees} in Section \ref{S:brackets-primitives-trees} includes a final part (e), which gives the sign on the diagram in part (d) corresponding to a basis element of $Lie(m-1)$.

Theorem \ref{thm:theta-trees}  resembles Habegger and Masbaum's result on Milnor invariants of long links in $\R^3$ via the Kontsevich integral \cite{Habegger-Masbaum:2000}.\footnote{Habegger and Masbaum use $C^t(m)$ to denote the space that we call $\mathcal{T}^\mathrm{even}(m)$.}  It also relates to the last two authors' work on Milnor homotopy invariants via configuration space integrals and the associated cocycles of diagrams \cite{Koytcheff-Volic:2019}.

A high-dimensional $m$-component Brunnian link $L$ in $\R^n$ gives rise to class in $\pi_*(\Conf(m,\R^n))$ and hence a class in the homotopy groups $\pi_*(\Omega \Conf(m,\R^n))$ of the space of braids in one dimension higher.  
Using Theorem \ref{thm:Phi-pwr-series} and Theorem \ref{thm:theta-trees}, we deduce in the following Corollary that the Milnor invariants of $L$ can be recovered by applying $\Theta$ to the corresponding class in $\pi_*(\Omega \Conf(m,\R^n))$.

\begin{cor}\label{cor:brunnian-spherical-links}
Let $L: S^{p_1} \sqcup \dots \sqcup S^{p_m} \to \R^n$ be a Brunnian link such that $\sum_{i=1}^m p_i= mn-2m-n+3$.  
Let  $\overline{\kappa}(L)$ be the adjoint of the map $\widetilde{\kappa}(L)$ defined in \eqref{eq:borr-diagram}. 
Then 
\[
\Theta(\overline{\kappa}(L)) =\sum_{I \in \Sigma_{m-2}} \mu_{I;m}(L)\Gamma_{m;I}^*\qquad  \in Lie(m-1)\subset H_*(P\D(m)^*),
\]
where  $\Gamma_{m;I}^*=\Theta(B_{m;I})$ can be viewed as a trivalent tree and each $\mu_{I;m}(L)$ is a Milnor invariant of $L$ in the sense of Koschorke \cite{Koschorke:1997}.
 \qed
\end{cor}

We may regard the above result as an analogue of Habegger and Masbaum's result \cite{Habegger-Masbaum:2000} for spherical links because $\Theta$ is a transport of the power series connection in \eqref{eq:omega-form-0}, which is an analogue of the partition function in \cite{Habegger-Masbaum:2000}. 
Consequently, $\mu_{I;j}(L)$ can be obtained as Chen's  iterated integral over $\overline{\kappa}(L)$, which ties into the results of \cite{DeTurck:2013,Komendarczyk:2009, Komendarczyk:2010}.

For context for our last main result, recall that a pure braid in $\R^{n+1}$ can be viewed as a long (a.k.a.~string) 1-dimensional link by graphing the loop of configurations in $\R^n$.  
We previously showed \cite[Corollary 5.21]{KKV:2020} that this inclusion is surjective on real cohomology and hence  injective on real homology for any $n\geq 3$.  Since the real homotopy of the space of pure braids injects into its real homology, this inclusion is also injective on real homotopy.
Our last main result concerns $\Omega^k \Conf(m, \R^n)$, which can be viewed as the space of $k$-dimensional braids in $\R^{n+k}$.  
Indeed, there is a graphing map 
\begin{align*}
	G: \Omega^k \Conf(m,\R^n) &\longrightarrow \Emb_c \left(\coprod_m \R^k, \ \R^{n+k}\right) \\
	f &\longmapsto G(f)
\end{align*}
where $\Emb_c(-)$ stands for the space of embeddings that are fixed outside the unit box and, for each $i=1, \dots, m$, the component $G(f)_i$ is given by 
\[
(G(f)_i)(t_1, \dots, t_k) := (f_i(t_1, \dots, t_k), t_1, \dots, t_k) \in\R^{n+k}.
\]
The codomain is called the space of $k$-dimensional string links in $\R^{n+k}$.  Using Koschorke's map $\kappa$ and his generalized Hopf invariants \cite{Koschorke:1997}, we show that $G$ induces a monomorphism  on the subspace of homotopy consisting of bracket expressions where each generator appears at most once.  

\begin{thm} 
	\label{T:braids-as-links}
	Fix any $m\geq 2$, $k\geq 1$, and $n\geq 3$.  Let $\mathcal{H}_{m}^{n,k}$ be the submodule	 of $\pi_{*}(\Conf(m,\R^n))$ spanned by iterated Whitehead brackets of distinct generators $b_{j,i_i}, \dots, b_{j,i_p}$ where $2\leq j \leq m$ and where $p (n-2) \geq {n+k-3}$.  Then the following composition is injective:
		\[
		\mathcal{H}_{m}^{n,k} \hookrightarrow \pi_{\ast+k} (\Conf(m,\R^n)) \overset{\cong}{\to} \pi_{\ast} (\Omega^k \Conf(m,\R^n)) \overset{G_*}{\longrightarrow} \pi_{\ast} \Emb_c \left(\coprod_m \R^k, \ \R^{n+k}\right).
		\]
\end{thm}

The assumptions on $p, m, n,$ and $k$ imply that the subspace $Lie(m-1)$ of length-$(m-1)$ brackets on distinct generators lies inside $\mathcal{H}_{m}^{n,k}$.  
In accordance with our strategy of proof, we divide the full statement of Theorem \ref{T:braids-as-links} in Section \ref{S:high-dim-braids} into part (a) on $Lie(m-1)$ and part (b) on $\mathcal{H}_{m}^{n,k}$.
	
The map $\kappa$ is essentially the map to the $(1,1,\dots,1)$-stage of the multivariable Taylor tower for link maps, and the subspace of brackets without repeats is a higher-dimensional analogue of braids up to link homotopy.  We conjecture that the graphing map is injective on all of real homotopy and real homology and that this injectivity can be detected by higher stages of the Taylor tower.

\subsection{Organization of the paper}
In Section \ref{S:diagrams-bar}, we review the definition of the diagram complex $(\D(m),\delta)$ from \cite{Lambrechts-Volic:2014}, the bar (respectively cobar) complex on a CDGA (respectively CDGC), and the resulting Hopf algebra structure. 

In Section \ref{S:power-series}, we start with the isomorphism $\Phi$ shown in formula \eqref{eq:Phi-cohomology-iso-D(m)} and use Chen's power series connection method \cite{Chen:1973} to construct the dual map $\Theta$ shown in formula \eqref{eq:Theta-homology-iso-D(m)}.

In Section \ref{S:brackets-primitives-trees}, we prove Theorems \ref{thm:primitive-diagrams} and \ref{thm:theta-trees}, which give diagrammatic descriptions of homotopy classes.  
We then provide the connection the module $Lie(m-1)$ in the context of Brunnian spherical link maps. 
The only fact from Section \ref{S:power-series} needed in Section \ref{S:brackets-primitives-trees} is that $\Theta$ is a Hopf algebra map; its explicit formula is not needed.

In Section \ref{S:high-dim-braids}, we prove Theorem \ref{T:braids-as-links}, which says that the subspace of non-repeating bracket expressions injects into the homotopy of many spaces of high-dimensional (equidimensional) string links.  
Section \ref{S:high-dim-braids} is somewhat independent of the previous sections.  The main connection is that  $\Theta$ provides a somewhat geometric description of the bracket expressions for those high-dimensional braids in terms of trivalent trees with non-repeating leaf labels.  The proof also uses maps introduced at the end of Section \ref{S:brackets-primitives-trees}.

Appendix \ref{apx:brackets} reviews some basic facts about the Whitehead and Samelson products.  
In Appendix \ref{A:cocycles}, we remark on methods of computing cocycles for iterated integrals, which can be applied to computations of the Milnor invariants as stated in Corollary \ref{cor:brunnian-spherical-links}.
\subsection{Acknowledgments}
The first author acknowledges the partial support by Louisiana Board of Regents Targeted Enhancement Grant 090ENH-21.
The second author was supported by the Louisiana Board of Regents grant LEQSF(2019-22)-RD-A-22. 
The third author was supported by the Simons Foundation.
We thank Victor Turchin for valuable input on a preprint of this article, and we thank the referee for their report and comments.  

\section{The CDGA of diagrams and the bar construction}\label{S:diagrams-bar}

In Section \ref{S:Diagrams}, we review Kontsevich's CDGA of diagrams.  In Section \ref{S:bar-and-cobar}, we review the bar construction on a differential graded algebra and its dual, the cobar construction on the dual differential graded coalgebra, as well as the Hopf algebra structure on these objects.

\subsection{The CDGA of diagrams}
\label{S:Diagrams}
The commutative differential graded algebra (CDGA) of \emph{admissible diagrams} $\D(m)$
originally appears in \cite[Section 3.3.3]{Kontsevich:1999} and is discussed in detail in \cite[Chapter 6]{Lambrechts-Volic:2014}.  
Here we also consider a larger algebra $\Dm(m)$ of diagrams where we do not quotient by multiple edges, defined in the following paragraphs.
The quasi-isomorphism $\Dm(m) \to C^*(\Conf(m,\R^n))$ given by the formality integration map is reviewed in Section \ref{S:power-series}.

\begin{defin}
	\label{D:Diagrams}
	Fix $n\geq 2$ and $m\geq 1$.  A \emph{diagram $\Gamma$ (in $\R^n$ on $m$ vertices)} consists of 
	\begin{itemize}
		\item
		a set of vertices $V(\Gamma)$, partitioned into $m$ \emph{segment vertices} $V_{\mathrm{seg}}(\Gamma)$ with distinct labels from $\{1,...,m\}$, and any number of \emph{free vertices} $V_{\mathrm{free}}(\Gamma)$.  Thus 
		\[
		V(\Gamma) = V_{\mathrm{seg}}(\Gamma) \sqcup V_{\mathrm{free}}(\Gamma);
		\]
		\item
		a set of edges $E(\Gamma)$ joining vertices of $\Gamma$, where an edge between two segment vertices is called a \emph{chord}. 
		
	\end{itemize}
	satisfying the following conditions:
	\begin{itemize}
		\item
		each free vertex of $\Gamma$ has valence at least 3;
		\item
		each free vertex of $\Gamma$ is joined to some segment vertex by a path of edges.
	\end{itemize}
\end{defin}

In \cite{Lambrechts-Volic:2014}, segment vertices are called \emph{external vertices}, and free vertices are called \emph{internal vertices}.  Six examples of diagrams in $\Dm(3)$ are shown below.  
A segment vertex is drawn as a hollow vertex, whereas a free vertex is drawn as a solid vertex. We align the segment vertices vertically rather than horizontally, as was done in \cite{Lambrechts-Volic:2014}, since this will be more convenient when we  picture elements of the bar construction on $\Dm(m)$ later.
\[
\includegraphics[scale=0.16]{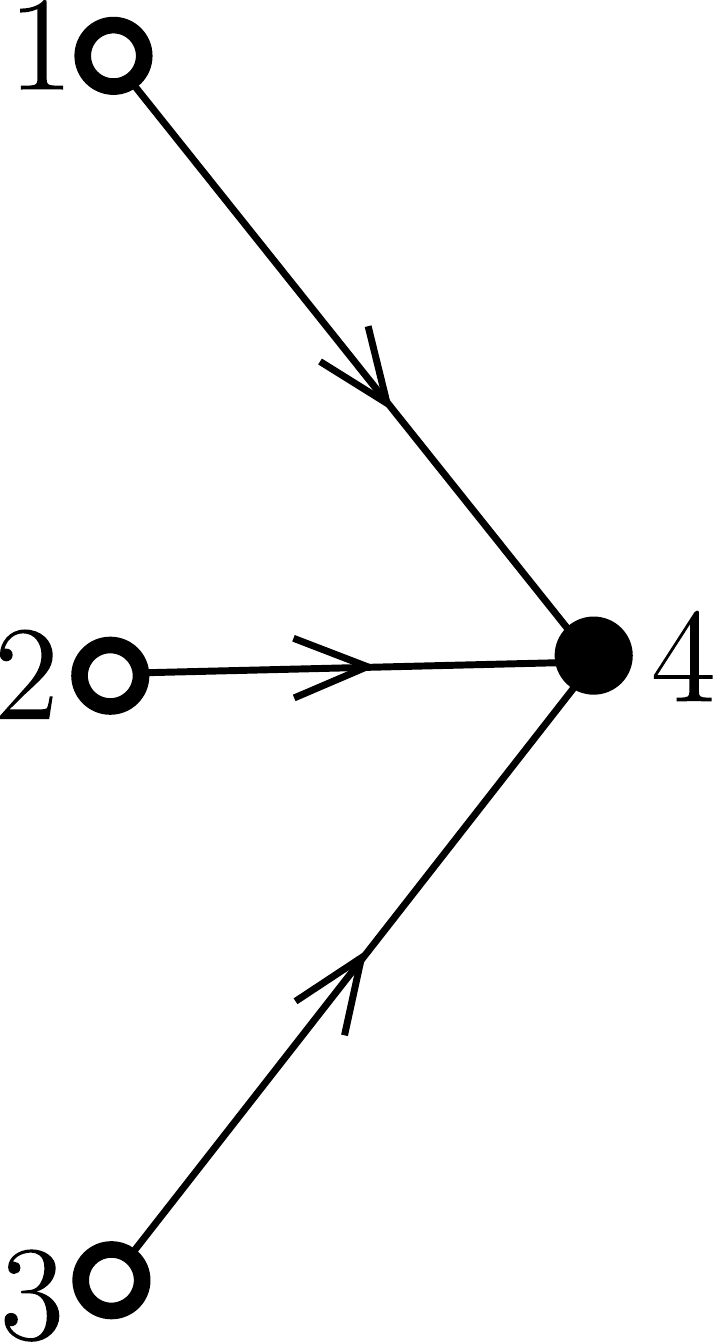}
\qquad \quad
\includegraphics[scale=0.16]{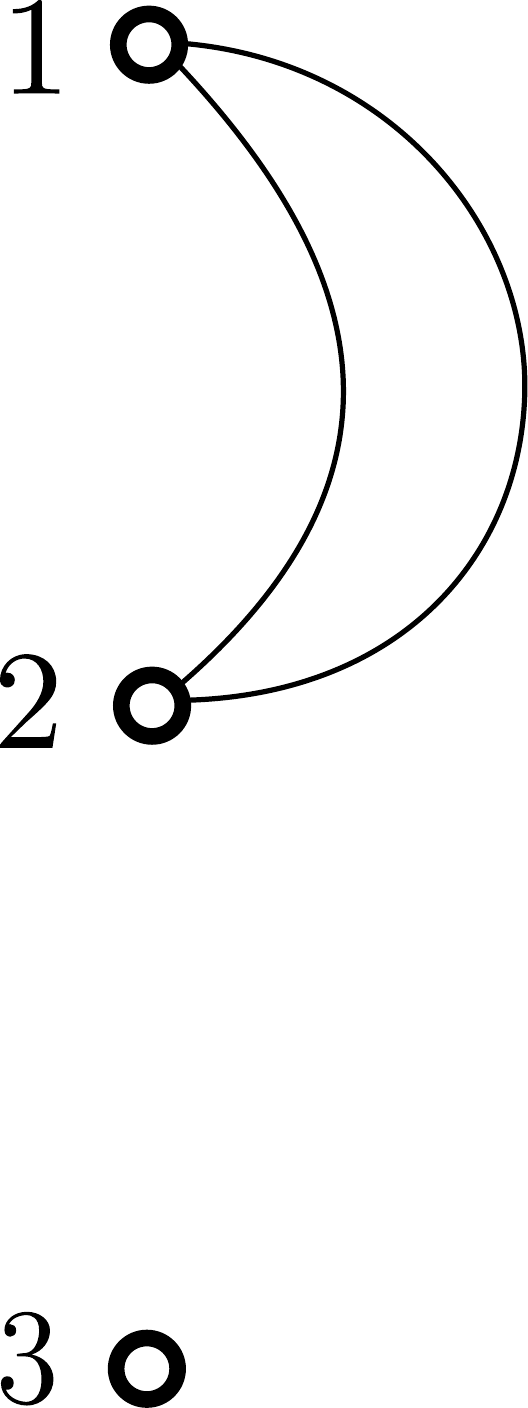}
\qquad \quad
\includegraphics[scale=0.16]{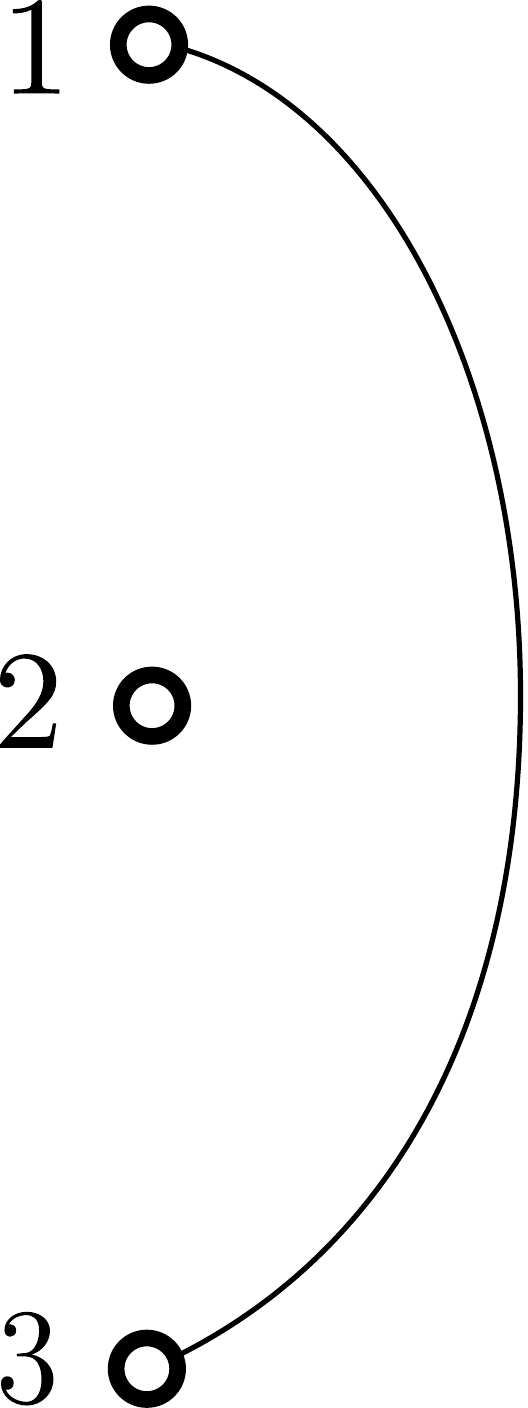}
\qquad \quad
\includegraphics[scale=0.16]{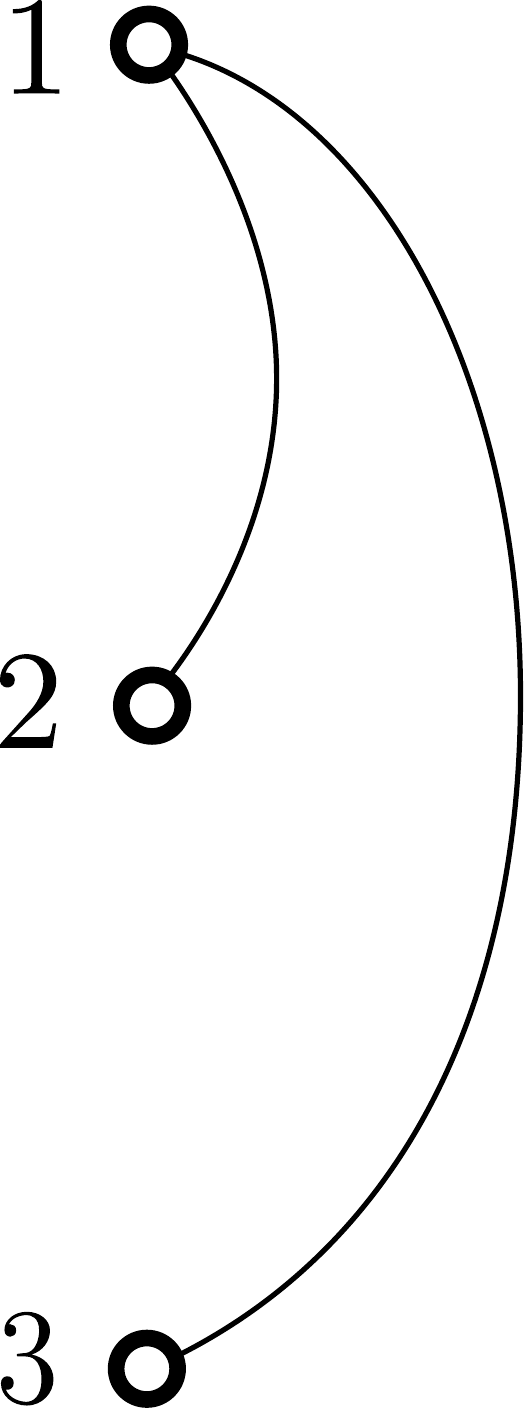}
\qquad \quad
\includegraphics[scale=0.16]{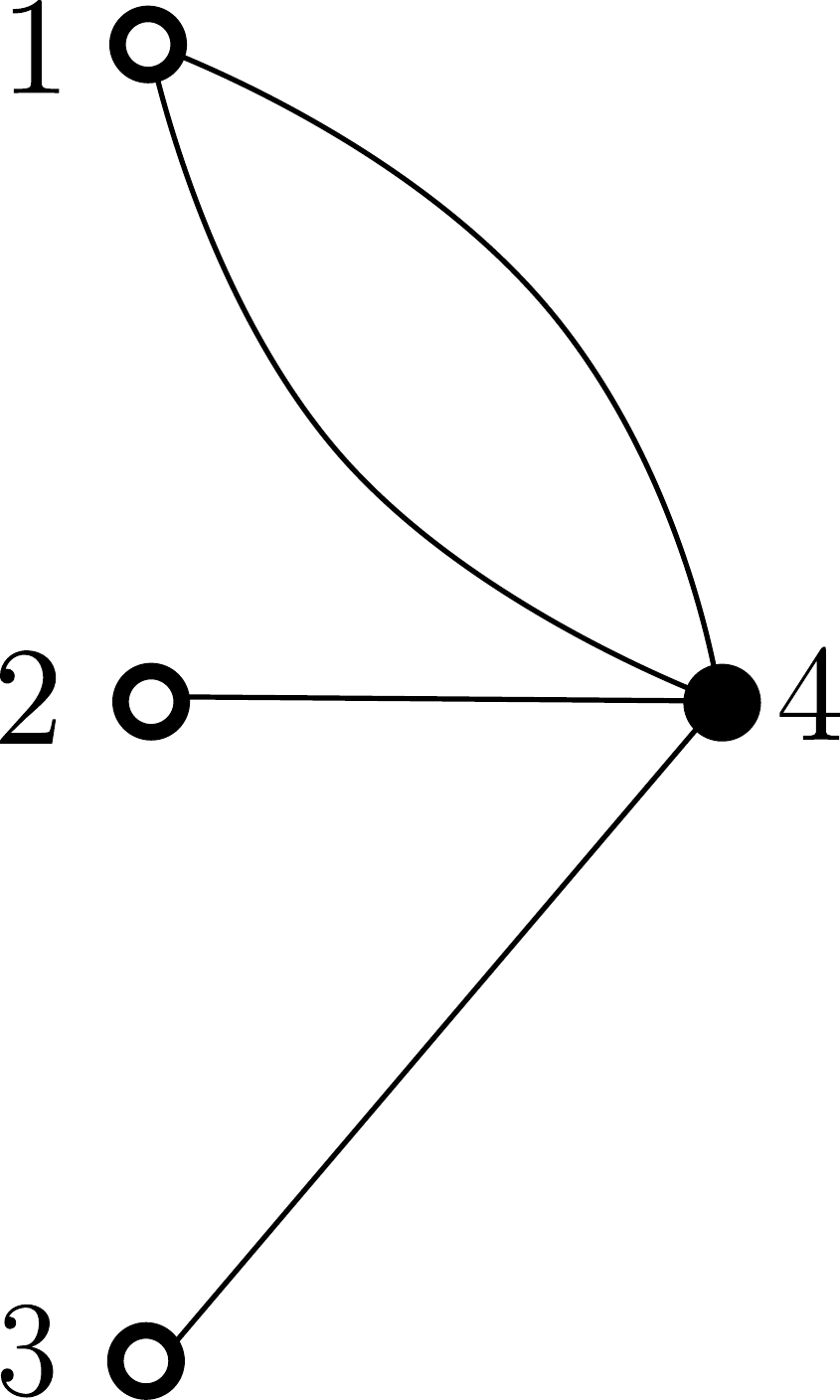}
\qquad \quad
\includegraphics[scale=0.16]{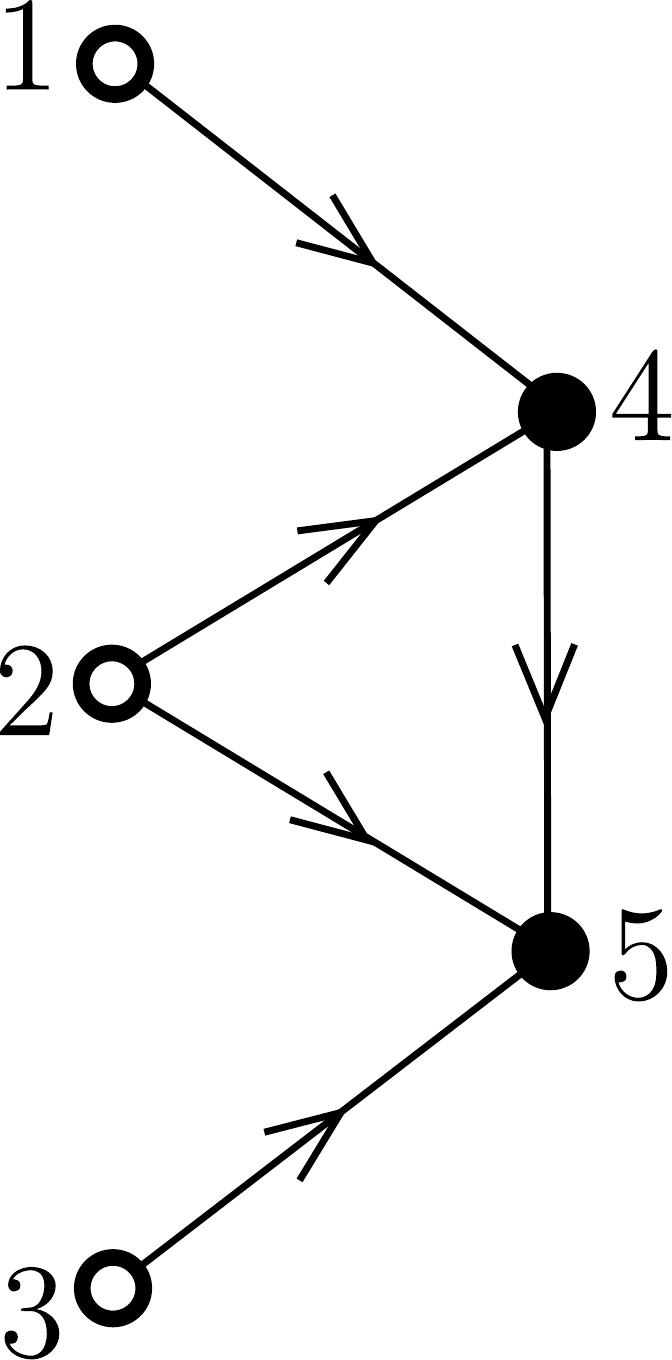}
\]

\begin{defin}
	\label{D:DmOrientations}
	A \emph{labeling} of a diagram $\Gamma$ consists of the following data:
	\begin{itemize}
		\item
		for odd $n$, a labeling of $V_{\mathrm{free}}(\Gamma)$ by integers $\{m+1,\ldots,|V(\Gamma)|\}$ and an orientation of each edge;
		\item 
		for even $n$, an ordering of the set of edges.
	\end{itemize}
	We call a diagram $\Gamma$ together with a labeling a \emph{labeled diagram}; these are considered up to graph isomorphisms which respect the labelings.   
	We then define \emph{orientation relations} on labeled diagrams:
	\begin{itemize}
	\item
	for odd $n$, $\Gamma \sim - \Gamma'$ if $\Gamma$ and  $\Gamma'$ differ by a transposition of two free vertex labels;
	\item
	for odd $n$, $\Gamma \sim -\Gamma'$ if $\Gamma$ and $\Gamma'$ differ by an orientation-reversal of an edge;
	\item
	for even $n$, $\Gamma \sim - \Gamma'$ if $\Gamma$ and  $\Gamma'$ differ by a transposition of two edge labels.
\end{itemize}
	The resulting equivalence class of a labeled diagram is called an \emph{oriented diagram}.
Let $\Dm(m)$ be the $\R$-vector space of oriented diagrams, modulo the relation that any diagram with a self-loop is zero.
\end{defin}

The space $\D(m)$, which we used in our previous work \cite{KKV:2020}, is the quotient of $\Dm(m)$ obtained by further setting to zero any diagram $\Gamma$ that has more than one edge between a pair of vertices.

We view  $\Dm(m)$ as a graded vector space where the \emph{degree} of a diagram $\Gamma$ is
\begin{equation}\label{eq:diag-degree}
|\Gamma| = (n-1)|E(\Gamma)| - n |V_{\mathrm{free}}(\Gamma)|.
\end{equation}
When we consider the bar complex on $\Dm(m)$, this will be the internal degree.

There is a differential $\delta$ on $\Dm(m)$ given by 
\begin{equation}\label{eq:d_D(m)}
\delta\Gamma  = \sum \varepsilon(\Gamma, e) \Gamma/e
\end{equation}
where the sum is taken over {\em contractible} edges $e$ in $\Gamma$, and $\Gamma/e$ is the result of contracting $e$ to a point.  Recall that a contractible edge is an edge that is not a chord.
Write $e=i \longrightarrow j$ to indicate that $e$ is an edge with endpoints $i$ and $j$, oriented from $i$ to $j$.
  If $e$ has endpoints $i<j$, the sign $\varepsilon(\Gamma, e)$ is 
\begin{equation}\label{eq:edge-sign}
\begin{split}
\text{for $n$ odd,}\ \varepsilon(e,\Gamma) & =\begin{cases}
(-1)^{j-m} 
& 
\text{if}\ e=i \longrightarrow j \\
(-1)^{j-m+1}
& 
\text{if}\ e=i \longleftarrow j \\
\end{cases}\\
\quad \text{and for $n$ even,}\ \varepsilon(e,\Gamma) & =(-1)^{e},
\end{split}
\end{equation}
where for $n$ even, $e$ also denotes the label on this edge.
This differential $\delta$ makes $\Dm(m)$ into a cochain complex.  For further details on the differential, see \cite[Section 6.4]{Lambrechts-Volic:2014}.

There is also a product of diagrams 
\begin{equation}\label{eq:diag-product}
\begin{split}
\Dm(m) \otimes \Dm(m) &\longrightarrow \Dm(m)\\
(\Gamma_1, \Gamma_2) & \longmapsto \Gamma_1 \cdot \Gamma_2
\end{split}
\end{equation}  
given by superposition of the two diagrams along segment vertices in the following sense: $\Gamma_1 \cdot \Gamma_2$ has $|V_{\mathrm{free}}(\Gamma_1 \cdot \Gamma_2)| = |V_{\mathrm{free}}(\Gamma_1)| + |V_{\mathrm{free}}(\Gamma_2)|$ and 
$|E(\Gamma_1 \cdot \Gamma_2)| = |E(\Gamma_1)| + |E(\Gamma_2)|$, but it still $m$ segment vertices.  
The orientation is given by appropriately raising the labels of vertices or edges from $\Gamma_2$.
The degree of a product is the sum of degrees,
\[
|\Gamma_1 \cdot \Gamma_2|=|\Gamma_1|+ |\Gamma_2|,
\]
and the product is graded-commutative:  
\begin{equation}\label{eq:diag-product-graded-commutative}
\Gamma_1 \cdot \Gamma_2=(-1)^{|\Gamma_1| |\Gamma_2|}\Gamma_2 \cdot \Gamma_1.
\end{equation}
This product makes $\Dm(m)$ into a CDGA, where the diagram with no edges and no free vertices is the unity 1.  
More details on this product can be found in \cite[Sections 6.3 and 6.5]{Lambrechts-Volic:2014}.
The quotient $\Dm(m) \to \D(m)$ is a CDGA map.
Moreover, it is a quasi-isomorphism because all the arguments used to prove Theorem 8.1 in \cite{Lambrechts-Volic:2014} apply just as well to $\Dm(m)$ as to $\D(m)$.

The graded dual $\Dm(m)^*$ of $\Dm(m)$ is a (graded-)cocommutative differential graded coalgebra (CDGC).  
The space $\Dm(m)$ has a basis of isomorphism classes of diagrams, which is canonical up to the orientation and therefore the sign of each diagram.
For a diagram $\Gamma \in \Dm(m)$, define an element $\Gamma^*\in \Dm(m)^*$ via the pairing 
\begin{equation}
\label{eq:Gamma-Gamma*-pairing}
\langle \Gamma_1, \Gamma_2^* \rangle = 
\left\{
\begin{array}{ll}
\pm |\mathrm{Aut}(\Gamma_1)| & \text{ if } \Gamma_1 \cong \Gamma_2, \\
0 & \text{ otherwise },
\end{array}
\right.
\end{equation}
where the sign is determined by whether the isomorphism between $\Gamma_1$ and $\Gamma_2$ is orientation-preserving or orientation-reversing.
We may write this more concisely as $\langle \Gamma_1, \Gamma_2^* \rangle  = \delta_{\Gamma_1, \Gamma_2} |\mathrm{Aut}(\Gamma_1)|$, by a mild abuse of the Kronecker delta.  
The factor $|\mathrm{Aut}(\Gamma_1)|$ is the size of the group of automorphisms of $\Gamma_1$, meaning automorphisms of the underlying undirected, unlabeled graph which fix the $m$ segment vertices.
It simplifies both the diagrammatic description of the dual boundary map $\delta^*$ and formula \eqref{eq:Gamma_C} which relates $\delta^*$ to the Samelson product on $\Omega \Conf(m, \R^n)$.

We draw a diagram $\Gamma^*$ as the rotation of the diagram $\Gamma \in \Dm(m)$ by $90^\circ$ counter-clockwise.
Six examples of diagrams in $\Dm(3)^\ast$ are shown below. 
\[
\begin{split}
\includegraphics[scale=0.16]{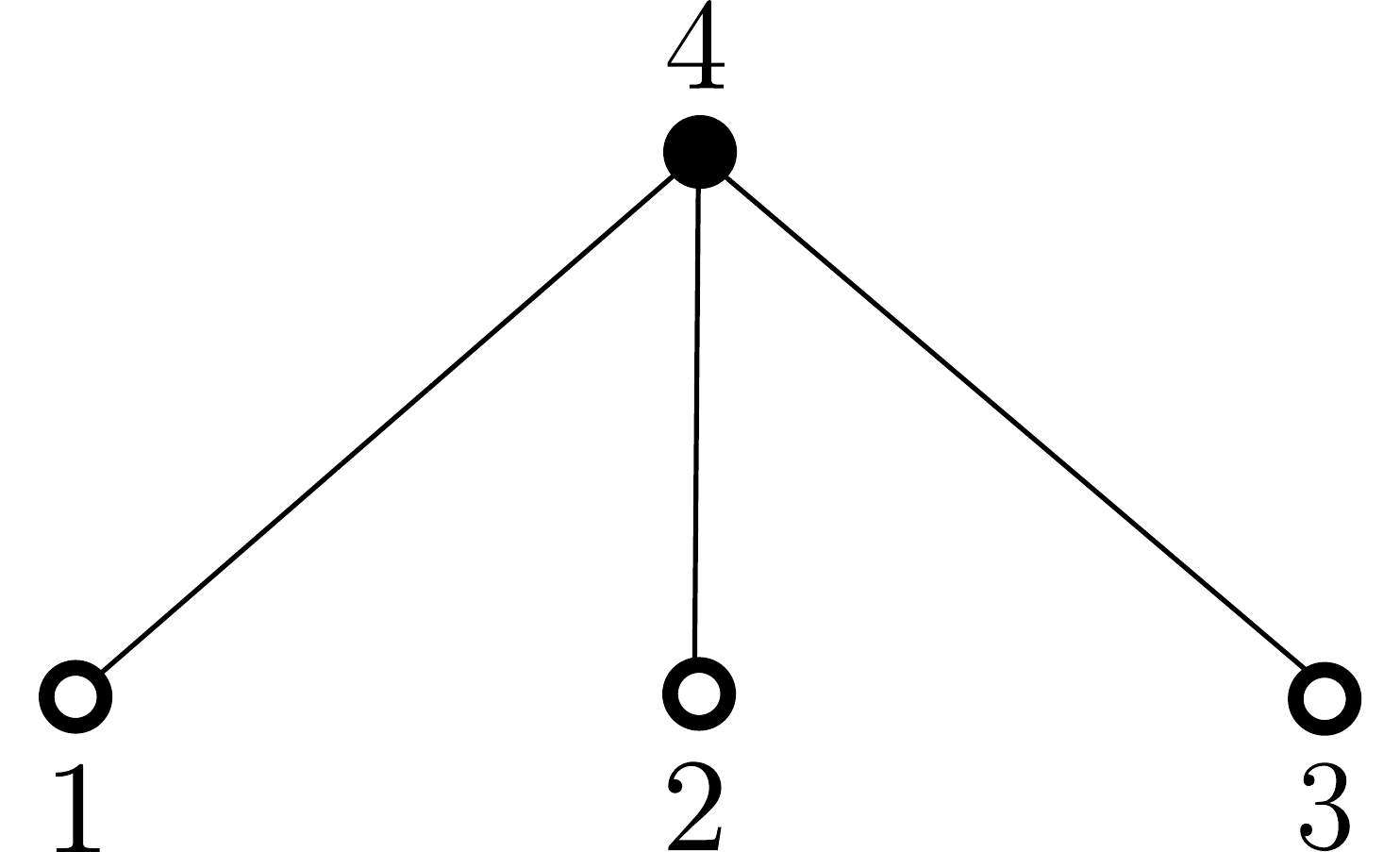}
\qquad \quad
\includegraphics[scale=0.16]{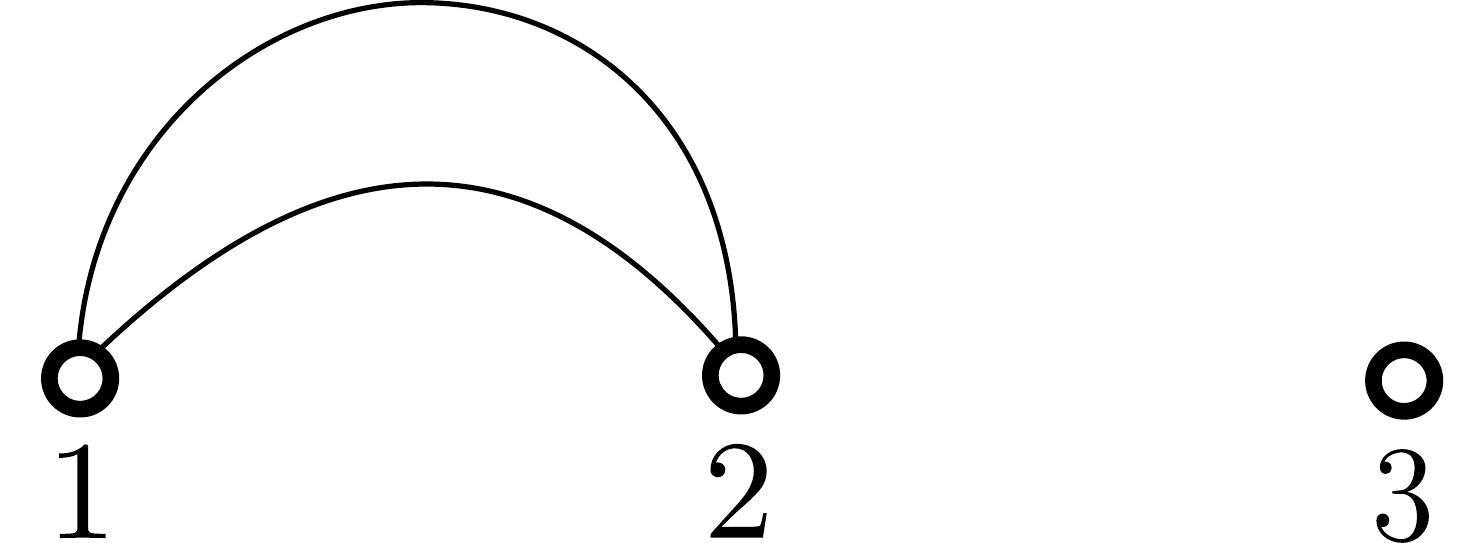}
\qquad \quad
\includegraphics[scale=0.16]{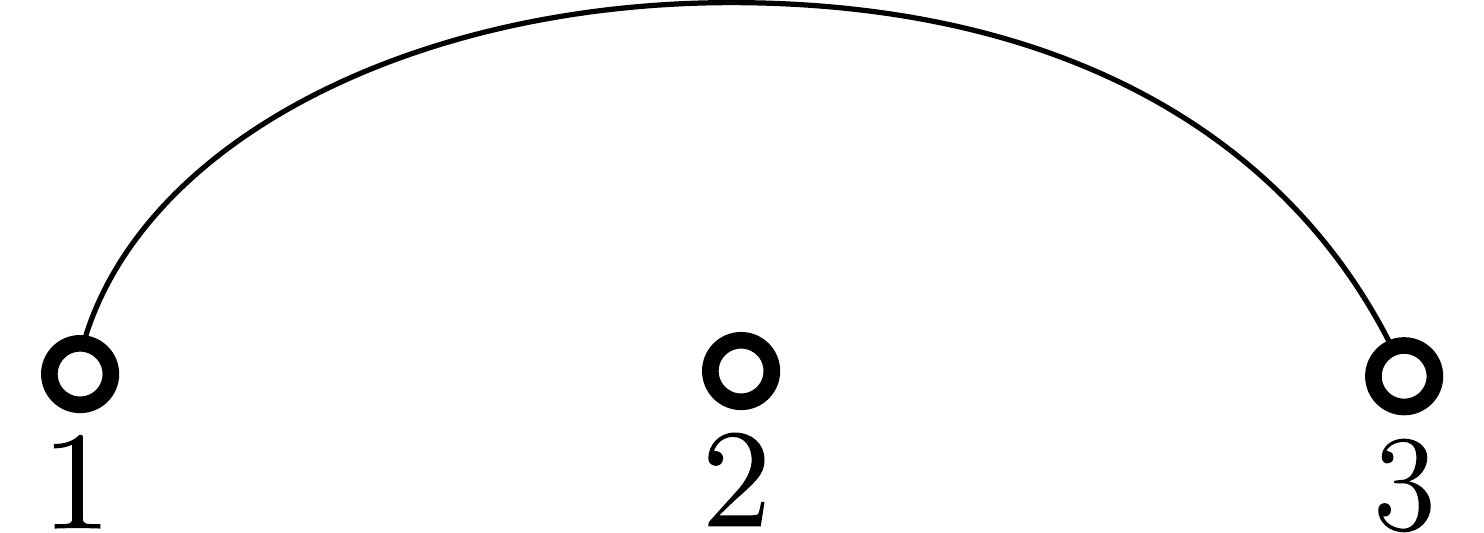}
\\
\includegraphics[scale=0.16]{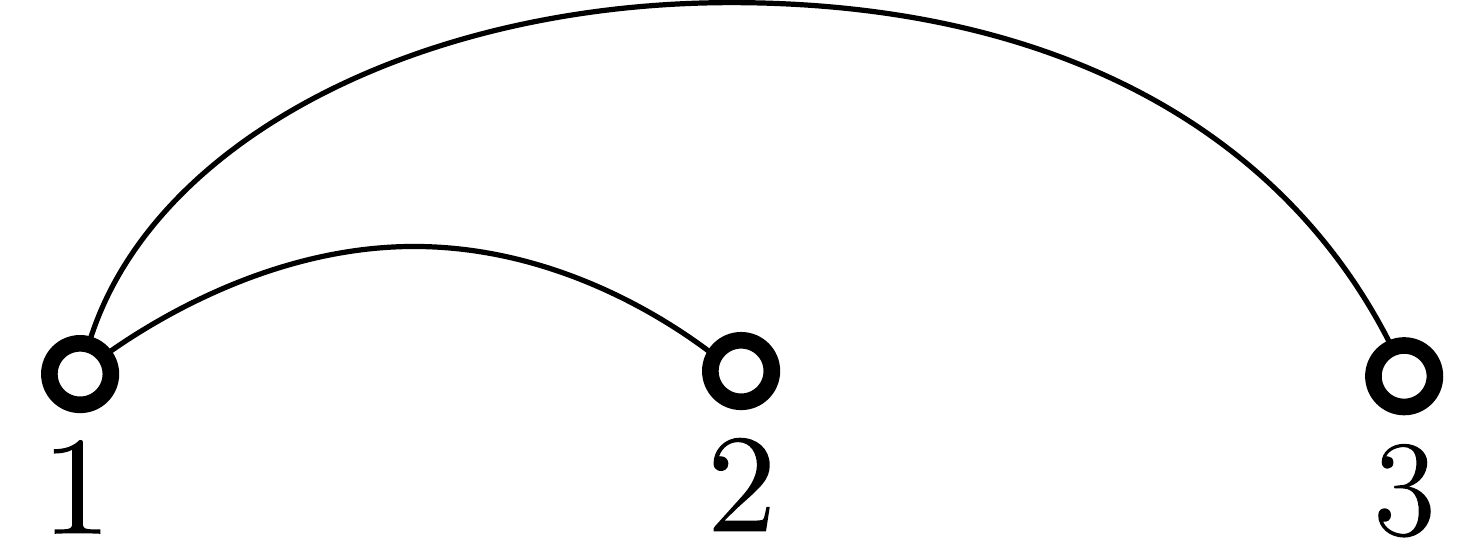}
\qquad \quad
\includegraphics[scale=0.16]{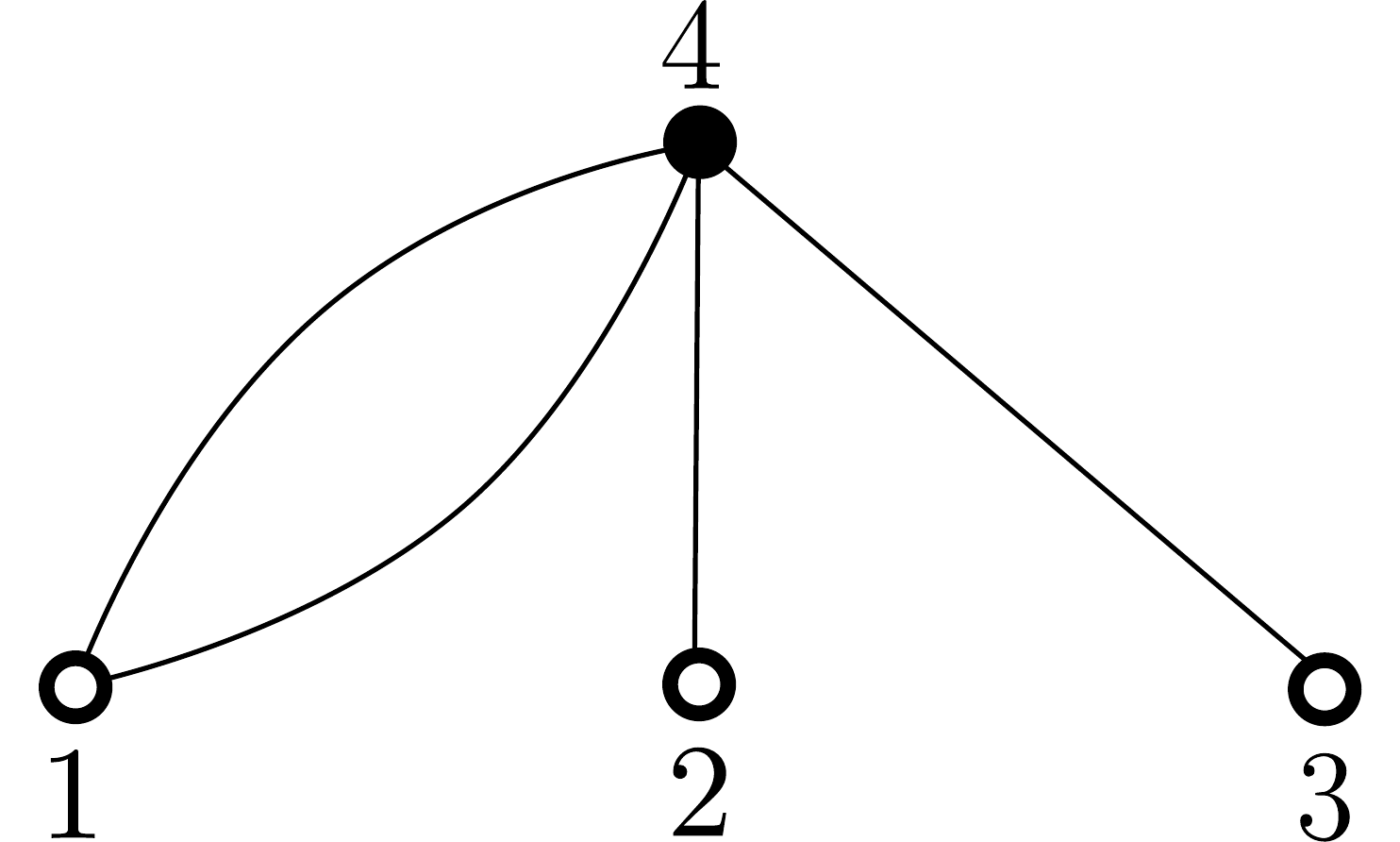}
\qquad \quad
\includegraphics[scale=0.16]{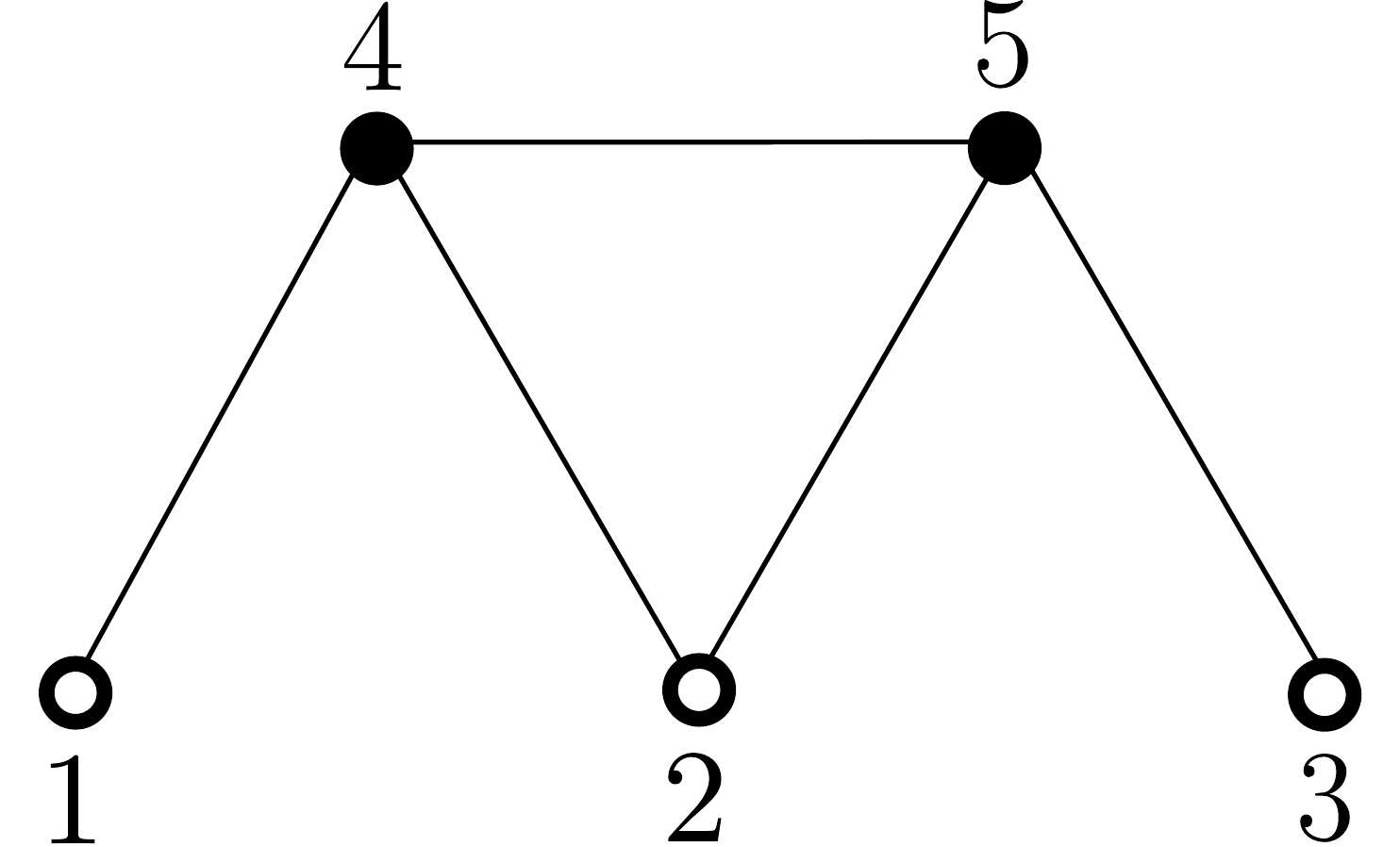}
\end{split}
\]

The differential $\delta^*$ dual to $\delta$ lowers degree by 1. On a basis element $\Gamma^*$ it is given by the signed sum of diagrams 
\[
\delta^\ast(\Gamma^*)  = \sum_{(\Gamma',e) \, : \ \Gamma'/e \cong \Gamma } \varepsilon(\Gamma', e) (\Gamma')^\ast\ .
\]  
The operation yielding such a graph $\Gamma'$ from $\Gamma$ is called the \emph{blowing up} or \emph{blow-up} of the basepoint vertex in $\Gamma'/e$; it replaces it by two vertices joined by an edge.
The above sum is taken over all segment vertices of valence $\geq 2$ and all free vertices of valence $\geq 4$.  For each such segment vertex $v$, one obtains a diagram for every ordered partition into two parts of the half-edges incident to $v$, 
such that the second part has cardinality $\geq 2$; the first part corresponds to the half-edges still incident to the segment vertex.  For each such free vertex $v$, one obtains a term for every unordered partition of the half-edges incident to $v$ into two parts such that each part has cardinality $\geq 2$. The blow-up of a $3$-valent segment vertex is illustrated below:
\[
\vvcenteredinclude{0.2}{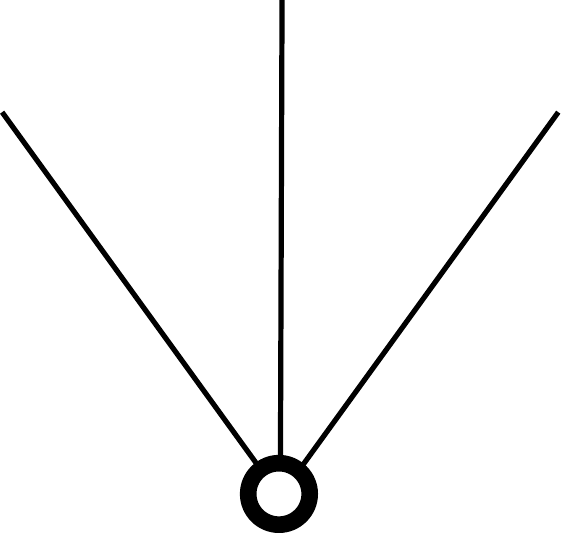}\quad \rightsquigarrow\quad \pm\vvcenteredinclude{0.2}{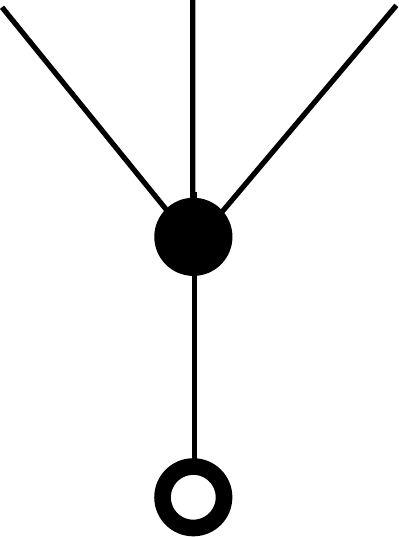}\quad+\quad\pm\vvcenteredinclude{0.2}{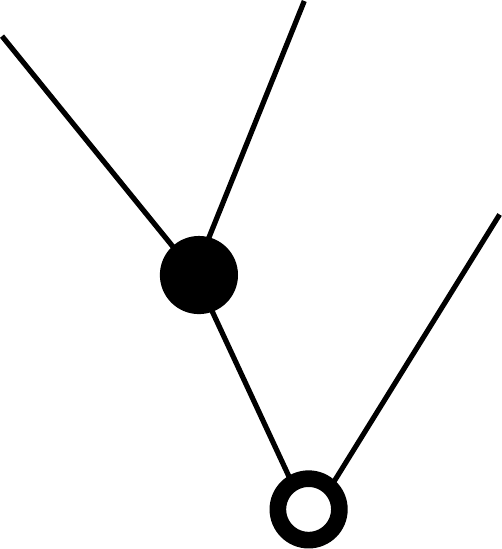}\quad+\quad\pm\vvcenteredinclude{0.2}{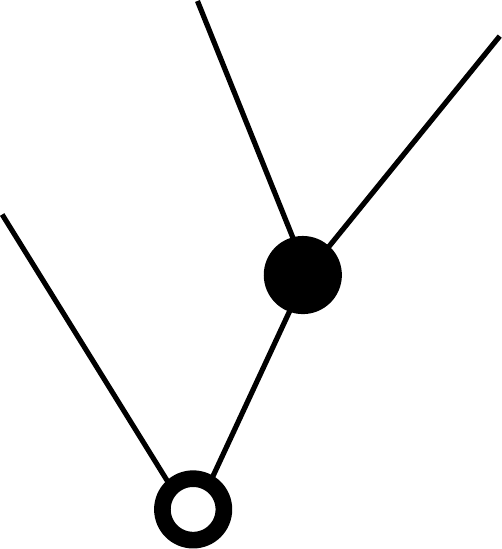}+\quad\pm\vvcenteredinclude{0.2}{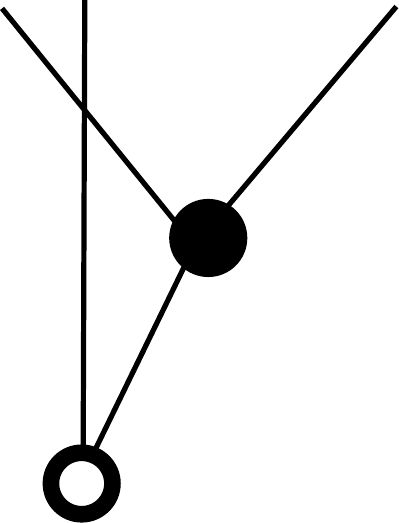}.
\]
As a special case, the sum of diagrams resulting from a $4$-valent free vertex is the well known IHX relation.
The factor $|\mathrm{Aut}(\Gamma_1)|$ in \eqref{eq:Gamma-Gamma*-pairing} ensures that $\delta^*$ of a $4$-valent vertex is always the IHX relation.

\subsection{Bar and cobar complexes and Hopf algebra structure}
\label{S:bar-and-cobar}

Recall that given a connected, augmented CDGA $(\mathcal{A}, \delta)$ over a field $k$, the bar construction $\Ba(\mathcal{A})$ on $\mathcal{A}$ is the tensor 
algebra $T(J\mathcal{A})$ on the augmentation ideal $J\mathcal{A}$ of $\mathcal{A}$. For elements $a_1$, $a_2$,\ldots, $a_r\in J\mathcal{A}$ of degrees $p_1$, $p_2$,\ldots, $p_r$, the element $a_1\otimes a_2\otimes \ldots\otimes a_r$ is denoted by $[a_1 | a_2 | \ldots | a_r]$ and has degree $-r+p_1+p_2+\cdots+p_r$.  
Because $\mathcal{A}$ is (graded-)commutative, $\Ba(\mathcal{A})$ is not just a differential graded coalgebra, but a differential graded Hopf algebra in the following way:
\begin{itemize}
	\item[(a)]  The product $\wedge$ (shuffle) is given by 
	\[
	[a_1 | \ldots | a_r]\wedge [a_{r+1} | \ldots | a_{r+s}]:=\sum_{\sigma\in \operatorname{Sh}(r,s)} \varepsilon(\sigma; p_1-1,p_2-1,\ldots, p_{r+s}-1)\, [a_{\sigma(1)} | a_{\sigma(2)} | \ldots | a_{\sigma(r+s)}]
	\]
	where $\sigma$ ranges over the shuffles $\operatorname{Sh}(r,s)$ of type $(r,s)$ of $\{1,2,\ldots,r+s\}$, and $\varepsilon(\sigma; p_1-1,p_2-1,\ldots, p_{r+s}-1)$ is the sign defined by
	\[
	a_{1}\wedge a_{2}\wedge \ldots \wedge a_{r+s}=\varepsilon(\sigma; p_1-1,p_2-1,\ldots, p_{r+s}-1)\, a_{\sigma(1)}\wedge a_{\sigma(2)}\wedge \ldots \wedge a_{\sigma(r+s)}.
	\]
	\item[(b)] The coproduct $\vartriangle$ (de-concatenation) is given by 
	\[
	\vartriangle([a_1 | a_2 | \ldots | a_r]):=\sum^{r}_{i=0} [a_1 | \ldots | a_i]\otimes [a_{i+1} | \ldots | a_r].
	\] 
	\item[(c)] The differential $d_{\Ba}$ is the (signed) sum of an internal differential $\delta_{\Ba}$ and a homological differential $D_{\Ba}$:\footnote{Here $\Ba(\mathcal{A})$ is essentially the same differential graded Hopf algebra as in our previous work \cite{KKV:2020}.  The only difference is that in \cite{KKV:2020} we defined $d_{\Ba}$ as $\delta_{\Ba} + D_{\Ba}$ rather than $\delta_{\Ba} - D_{\Ba}$.  In particular, formula (19) in \cite{KKV:2020} can be read as the sum from $1$ to $p-1$ rather than from $0$ to $p$ because our bar complex is the normalized one.  }
	\begin{align}
	\label{eq:bar-differential}
	\begin{split}
	d_{\Ba}\bigl([a_1 | a_2 | \ldots | a_r]\bigr)&:= (\delta_{\Ba} - D_{\Ba}) \bigl([a_1 | a_2 | \ldots | a_r]\bigr)\\
	& :=\sum^r_{i=1} (-1)^i [\varepsilon(a_1) a_1|\ldots|\varepsilon(a_{i-1}) a_{i-1}|\delta a_i| a_{i+1}|\ldots|a_r]\\
	&\quad - \sum^{r-1}_{i=1} (-1)^i [\varepsilon(a_1) a_1|\ldots|\varepsilon(a_{i-1}) a_{i-1}| \varepsilon(a_i) a_i\cdot a_{i+1}|a_{i+2}|\ldots|a_r]
	\end{split}
	\end{align}
where $\varepsilon(a_i)=(-1)^{p_i}$.
\end{itemize}
Furthermore, the product and coproduct descend to a Hopf algebra structure on cohomology $H^\ast(\Ba(\mathcal{A});k)$.
The graded dual $(\Ba(\mathcal{A}))^\ast$ is (canonically) isomorphic as a Hopf algebra to the cobar construction on $\mathcal{A}^*$.
Many authors denote this dual by $\Omega(A^\ast)$, but we denote it by $\Ba^\ast(\mathcal{A})$ to avoid overloading the notation $\Omega$.
Its Hopf algebra structure is as follows:
\begin{itemize}
	\item[(d)] The product $\cdot=\vartriangle^\ast$ (concatenation) is given by 
	\[
	[a^\ast_1 | \cdots | a^\ast_r]\cdot[a^\ast_{r+1} | \cdots | a^\ast_{r+s}]:=[a^\ast_1 | \cdots | a^\ast_r | a^\ast_{r+1} | \cdots | a^\ast_{r+s}].
	\]
	\item[(e)] The coproduct $\Delta=\wedge^\ast$ (co-shuffle) is given by 
	\begin{equation}\label{eq:Delta}
	\Delta([a^\ast_1 | \cdots | a^\ast_r]):=\sum^{r}_{i=0}\sum_{\sigma\in \operatorname{Sh}(i,r-i)} \varepsilon(\sigma; p_1-1,\ldots, p_{r+s}-1) [a^\ast_{\sigma(1)} | \cdots | a^\ast_{\sigma(i)}]\otimes [a^\ast_{\sigma(i+1)} | \cdots | a^\ast_{\sigma(r)}].
	\end{equation}
\end{itemize}	
The differential on $\Ba^\ast(\mathcal{A})$ is given by the dual $d^\ast_{\Ba}$. The product (d) and coproduct (e) again yield a Hopf algebra structure on homology $H_\ast(\Ba^\ast(\mathcal{A});k)$.

The above definitions applied to the algebras $\mathcal{A}=\Dm(m)$ and $\mathcal{A}=\ChdR^\ast(\Conf(m, \R^n))$ yield the bar complexes $\Ba(\Dm(m))$ and $\Ba(\ChdR^{\ast}(\Conf(m, \R^n)))$ and  their cohomology. In addition,  the coalgebra $\mathcal{A}^*$ gives the cobar complex $\Ba^\ast(\Dm(m))$ and its homology. These objects are of primary interest in the following sections.  Also keep in mind that the isomorphism $\Dm(m) \cong \Dm(m)^*$ coming from the basis of diagrams determines the module isomorphism $\Ba(\Dm(m)) \cong   \Ba^\ast(\Dm(m))$.

Recall that in an algebra $\mathcal{A}$ (with a product $\cdot\,$), an element $a \in J\mathcal{A}$ is \emph{indecomposable} if $a=b\cdot c$ implies $b=1$ or $c=1$.  We denote the submodule of indecomposables by $I\mathcal{A}$.  An element $x$ in the coalgebra $\mathcal{A}^*$ (with the coproduct $s=(\,\cdot\,)^\ast$) is \emph{primitive} if $s(x)=x \otimes 1 + 1 \otimes x$.  We denote the submodule of primitives by $P\mathcal{A}^\ast$. 
 The indecomposable elements of $\mathcal{A}$ are dual to the primitive elements of $\mathcal{A}^*$.
We will mainly consider these modules for the algebras $\Dm(m)$ and $H^*(\Ba(\Dm(m)))$ and the coalgebras $\Dm(m)^*$ and $H_*(\Ba^*(\Dm(m)))$.

In $\Dm(m)$, the indecomposable elements and the corresponding primitive elements in $\Dm(m)^*$ are precisely the linear combinations of nonempty diagrams which are connected after all the segment vertices are removed.  In \cite{KKV:2020}, we called such diagrams \emph{internally connected}. 
Finally recall that, for a monomial $[a^*] \in \Ba^*(\mathcal{A})$ of length 1, the homological differential on $[a^*]$ is zero if and only if $a^*$ is primitive in $\mathcal{A}^*$.

\section{The power series connection from the Kontsevich formality integral}\label{S:power-series}
Our main purpose now is to prove Theorem \ref{thm:Phi-pwr-series}.  In Section \ref{S:chen-formality-int}, we review Chen's iterated integrals and Kontsevich's formality integration map.  In Section \ref{S:connections}, we review Chen's method of formal power series connections, which informs our approach to Theorem \ref{thm:Phi-pwr-series}.  Its proof is given in Section \ref{S:proof-connection}

\subsection{Chen's integrals and the formality integration map} 
\label{S:chen-formality-int}
In \cite[Theorem 3.7, Theorem 4.1]{KKV:2020}, we showed that the map
\begin{equation}\label{eq:cohomology-iso-D(m)-2}
\Phi: H^\ast(\Ba(\Dm(m)))\longrightarrow H^\ast_{dR}(\Omega \Conf(m,\R^n)),\qquad \Phi=\varint_{\mathrm{Chen}}\circ \ \Ba(I).
\end{equation}
induces an isomorphism on cohomology additively.
On the cochain level, $\Phi$  is a composition of Chen's integration $\varint_{\mathrm{Chen}}$ and the homomorphism $\Ba(I)$ induced on the bar complex by the formality integration map $I$ of \cite{Kontsevich:1999, Lambrechts-Volic:2014}.  Next, we briefly explain these maps. 

Chen \cite{Chen:1973} defined his integration map by first defining de Rham cohomology of a loop space $\Omega M$  in the general setting of a differentiable space $M$.  More precisely, he constructed a double complex
\[
\Chen^{*,*}(M)=\bigoplus_{p,q\geq 0} \Chen^{-p,q}(M),\quad \Chen^{-p,q}(M)=\big\langle \varint \omega_1 \omega_2 \cdots \omega_p\ |\ \omega_i\in \ChdR^{\ast}(M)\big\rangle
\]
of so called {\em iterated integrals} which can be formally viewed as a subcomplex of differential forms $\ChdR^\ast(\Omega M)$ on $\Omega M$. Further, in \cite{Chen:1973, Hain:1984} he defined a natural chain map, which we call {\em Chen's integration map}, from the bar complex of $\ChdR^\ast(M)$ to $\Chen^{*,*}(M)$, as follows.
First, there are maps 
\[
M^p \overset{ev_p}{\longleftarrow} \Delta^p \x \Omega M \overset{pr}{\longrightarrow} \Omega M
\]
where $ev_p$ evaluates a loop at times $t_1 \leq \dots \leq t_p$ and $pr$ is the projection to $\Omega M$.
With $\pi_k: M^p \to M$ as the projection to the $k$-th factor, one writes $\omega_1 \omega_2 \cdots \omega_p$ as shorthand for $ev_p^*(\pi_1^*\omega_1 \wedge \dots \wedge \pi_p^*\omega_p)$.  One then defines
\begin{equation}\label{E:ChenIntegrationMap}
\varint_{\mathrm{Chen}}\colon\Ba(\ChdR^{\ast}(M)) \longrightarrow \Chen^{*,*}(M)\subset \ChdR^\ast(\Omega M) \quad \text{by} \quad
[\omega_1 | \omega_2 | \cdots | \omega_p] \longmapsto \varint \omega_1 \omega_2 \cdots \omega_p
\end{equation}
where the integration is along the fiber of the trivial bundle over $\Omega M$ given by $pr$.
This map is a quasi-isomorphism for simply connected $M$, as proved in \cite{Chen:1973, Hain:1984}. In our setting, the configuration space $\Conf(m,\R^n)$ is simply connected for $n\geq 3$, and the quasi-isomorphism \eqref{E:ChenIntegrationMap} thus becomes the quasi-isomorphism
\begin{equation}\label{E:ChenIntegrationMapBraids}
\varint_{\mathrm{Chen}}:\Ba(\ChdR^{\ast}(\Conf(m,\R^n))) \longrightarrow \ChdR^*(\Omega\Conf(m,\R^n)).
\end{equation}

The second map   in \eqref{eq:cohomology-iso-D(m)-2}, $\Ba(I)$, is induced by the formality integration map 
\begin{equation}\label{eq:I-formality-map}
I\colon \Dm(m) \longrightarrow \ChdR^\ast(\Conf(m,\R^n)), 
\end{equation}
which factors through the projection $p:\Dm(m)\longrightarrow \D(m)$; i.e. we can define 
\begin{equation}\label{eq:I'-formality-map}
I'\colon \D(m) \longrightarrow \ChdR^\ast(\Conf(m,\R^n)) 
\end{equation}
such that 
\begin{equation}\label{eq:I=I'p}
 I=I'\circ p.
\end{equation}

We now review the definition of $I$.
Given a diagram $\Gamma\in\Dm(m)$ with $v$ free vertices and $e$ edges, there are classical Gauss maps
\begin{equation}\label{eq:phi_ji}
\phi_{j,i} \colon \Conf(m+v, \R^n)\longrightarrow (S^{n-1})^{E(\Gamma)}
\end{equation}
given by 
\[
\phi_{j,i}\colon (x_1, x_2, ..., x_{m+v})\longmapsto \frac{x_j-x_i}{\vert  x_j-x_i\vert}.
\]
Each such map can be used to pull back the rotation-invariant unit volume form $\nu$ on the sphere $S^{n-1}$.  Denote this pullback by $\alpha_{j,i} := \phi_{j,i}^*(\nu)$, and let
\[
\alpha_\Gamma = \bigwedge_{\text{edges $i \to j$ of $\Gamma$}}\alpha_{j,i}.
\]
By the notation above, we mean that for $n$ odd, the order of the indices $i$ and $j$ is determined by the orientation of the edge.   (For even, the two orders gives cohomologous forms.)
Now let
\[
\pi\colon \Conf(m+v, \R^n)\longrightarrow \Conf(m, \R^n)
\]
be the projection onto the first $m$ configuration points.
Then we set, for $\Gamma\in \D(m)$ (or $\Dm(m)$):
\begin{equation}\label{E:FormalityIntegral}
I(\Gamma)\ (\text{or}\ I'(\Gamma)):=\pi_{*} (\alpha_\Gamma)\in \ChdR^*(\Conf(m, \R^n)).
\end{equation}
Here $\pi_{*}$ denotes the pushforward, or integration along the fiber, of the projection $\pi$.  Thus, for $(x_1,...,x_m)\in \Conf(m, \R^n)$, 
\[
I(\Gamma)(x_1,...,x_m)=\int_{\pi^{-1}(x_1,...,x_m)} \alpha_\Gamma.
\]
The degree of the form obtained this way is precisely the degree of $\Gamma$, as defined just after \refD{DmOrientations}.
Note that for a diagram $\Gamma$ with multiple edges, $\alpha_\Gamma=0$. 
Thus $I$ is simply a trivial extension of $I'$ to $\Dm(m)$, i.e. \eqref{eq:I=I'p} holds. 
\begin{rem}
 From this point on, we will suppress distinct notations $I$ and $I'$ and simply use $I$ for both, since the choice should be clear from the context.
\end{rem}
 To show that $I$ is a cochain map, one proves that the form $I(\Gamma)$ is closed using Stokes' Theorem, which amounts to checking that the restrictions of integrals to the codimension one boundary components of the compactified configuration space vanish.  For details, see \cite[Chapter 9]{Lambrechts-Volic:2014}.

Combining the above constructions, we conclude that $\Phi$ is induced from the quasi-isomorphism defined on monomials of $\Ba(\Dm(m))$ by 
\begin{equation}\label{eq:Phi-monomials}
 \Ba(\Dm(m))\ni [\Gamma_1 | \Gamma_2 | \ldots |\Gamma_k] \xrightarrow{\quad \Phi \quad } \varint I(\Gamma_1)I(\Gamma_2)\cdots I(\Gamma_k)\in \Chen^{*,*}(M).
\end{equation}

\subsection{Chen's power series connections} 
\label{S:connections}
Our Theorem \ref{thm:Phi-pwr-series} is loosely based on Chen's method of formal power series connections.  
We explain the original method  developed in \cite{Chen:1973, Chen:1977, Hain:1984} before describing the modification needed for our Theorem.  
Power series connections were first introduced in \cite[p. 223]{Chen:1973} for any manifold\footnote{More generally, $M$ can be any differentiable space in the sense of \cite{Chen:1973}.} $M$, simply as a formal power series 
\begin{equation}\label{eq:formal-connection}
 w=\sum w_i X_i+\sum w_{i,j} X_i X_j + \sum w_{i,j,k} X_i X_j X_k+\cdots \quad \in \quad \ChdR^{\ast}(M)\otimes \overline{T(V)}
\end{equation}
where $w_i$, $w_{i,j}$, $w_{i,j,k}$, \ldots $\in \ChdR^{\ast}(M)$ are differential forms on $M$ and $X=\{X_i\}_{1\leq i\leq m}$, is a basis of a real vector space $V$.\footnote{We work over $\R$, but Chen's method equally applies to complex-valued forms.} 
Here $\overline{T(V)}$ is the completion of the tensor algebra $T(V)$ of $V$, which can also be viewed as the (associative) non-commutative power series algebra in $X$ over $\R$.  

The {\em transport} $\mathcal{T}$ of the formal power series connection $w$ in \eqref{eq:formal-connection} is then defined as a formal power series:
\begin{equation}\label{eq:formal-transport}
\begin{split}
	\mathcal{T} & =1+\varint w+\varint w w+\varint w w w +\ldots\ \\
	& = 1 + \sum \varint w_i X_i + \sum \varint (w_{i,j} + w_i w_j) X_i X_j +\ldots\  \in \ \Chen^\ast(M)\otimes \overline{T(V)}.
	\end{split}
\end{equation}

In \cite{Chen:1977}, Chen sets $V=H_{\ast-1}(M;\R)$ for a simply connected $M$, with the $X_i$ a basis of $V$, and he assumes that the $w_i$ in \eqref{eq:formal-connection} are closed forms representing cohomology classes dual to the $X_i$. 
A graded derivation\footnote{Explicitly, $\partial (X\otimes Y)=(\partial X)\otimes Y + \varepsilon(X) X\otimes (\partial Y)$, where $\varepsilon(X)=(-1)^{|X|}$.} $\partial$ of degree $-1$ on $\overline{T(V)}$ is said to satisfy the {\em twisted cochain condition} if
\begin{equation}\label{eq:twisted-cochain-partial}
\partial w+d w-\varepsilon(w)\wedge w=0
\end{equation}
where $\partial w$ and $dw$ are shorthand for $(1 \otimes \partial)w$ and $(d \otimes 1)w$ respectively, and where $\wedge$ is the product on the algebra $\ChdR^{\ast}(M)\otimes \overline{T(V)}$.  The existence of such a $\partial$ guarantees the following result, where the transport map $\theta$ lands in $T(V)$ rather than $\overline{T(V)}$ because the degrees of the differential forms in \eqref{eq:formal-connection} are all greater than $1$; see \cite[p. 229]{Chen:1973}, \cite{Chen-2:1977}. 

\begin{thm}[Chen \cite{Chen:1977}]
\label{thm:chen-pwr-connection}
	With $M$, $w$, $\partial$, and $\mathcal{T}$ as above, 
	$(T H_{\ast-1}(M;\R),\partial)$ is a CDGA and we have an isomorphism $H_\ast(\Omega M;\R)\cong H_\ast(T H_{\ast-1}(M;\R),\partial)$ induced by
	the transport $\mathcal{T}$,
\begin{equation}\label{eq:general-theta}
\theta:C^\textrm{sing}_\ast(\Omega M;\R) \longrightarrow T H_{\ast-1}(M;\R),\qquad  c\longrightarrow \langle \mathcal{T}, c\rangle=\delta_{0,n}+\sum_r\langle\varint w^{(r)}, c\rangle,
\end{equation}
for any $c\in H_n(\Omega M;\R)$, where $\delta_{0,n}$ is the Kronecker delta, $w^{(r)}=\underbrace{ww\cdots w}_r$ and $C^{\operatorname{sing}}_\ast(\Omega M;\R)$ is the space of smooth singular chains on $\Omega  M$.
\end{thm}
Call a pair $(w,\partial)$ satisfying the assumptions of the above theorem a {\em homological power series connection}.
For any $M$, $(w,\partial)$ is uniquely determined by the first-order term in \eqref{eq:formal-connection} and can be constructed by an inductive procedure \cite[Theorem 1.3.1]{Chen-2:1977}. 
Chen \cite{Chen:1977} also showed that the chain map $\theta$ in \eqref{eq:general-theta} is multiplicative with respect to the Pontryagin product, and Hain \cite{Hain:1984} showed further that $\theta$ induces a Hopf algebra isomorphism.  
See Appendix \ref{apx:brackets} for more details on the Pontryagin product.

\begin{rem}
\label{R:kohno-connection}
	In the context of $\Conf(m,\C)$, Chen's power series connections were studied by Kohno in the series of works \cite{Kohno:1987,Kohno:2000,Kohno:2002} in relation to Vassiliev invariants of pure braids. In more recent work, Kohno  \cite[Theorem 6.2]{Kohno:2010} argued for $\Conf(m,\R^n)$, $n\geq 3$, that there is a homological power series connection $\partial$ as in Theorem \ref{thm:chen-pwr-connection} that is {\em quadratic}, i.e.~satisfies $\partial X_i = \sum_{k,l} c^i_{k,l} X_k X_l$.  This implies the formality of $\Conf(m,\R^n)$ for $n\geq 3$.  We instead use the formality integration map $I$ to determine a $\D(m)$-valued connection, which may be more useful for explicit geometric computations, while still yielding a Hopf algebra isomorphism.
\end{rem}

\subsection{Proof of Theorem \ref{thm:Phi-pwr-series} (power series connection)}
\label{S:proof-connection}
Our goal for this Section is to define an algebra quasi-isomorphism  
\begin{equation}\label{eq:Theta}
\Theta:C^\textrm{sing}_\ast(\Omega M) \longrightarrow \Ba^\ast(\Dm(m)),\qquad M=\Conf(m,\R^n),
\end{equation}
via a modification of the method of power series connections and thus establish the following result.

\begin{repthm}{thm:Phi-pwr-series}
Let $n \geq 3$.
	\begin{itemize}
	\item[(a)]  
	Let
	$\Theta:C^\textrm{sing}_\ast(\Omega \Conf(m,\R^n)) \longrightarrow \Ba^\ast(\Dm(m))$
	be the map induced by the transport of the formal power series connection $\omega$, valued in $\Ba^\ast(\D(m))\subset \Ba^\ast(\Dm(m))$ and defined by 
	\begin{equation}\label{eq:omega-form-0}
	\omega=\sum_{\Gamma\in \mathcal{B}(m)} I(\Gamma)\otimes \frac{[\Gamma^\ast]}{|\mathrm{Aut}(\Gamma)|} \ \in \ \Ch^*_{dR}(\Conf(m,\R^n)) \otimes \Dm(m)^*,
	\end{equation}
	where $I$ is the formality integration map and $\mathcal{B}(m)$ is the basis of diagrams  (which is well defined up to signs) for the subspace of $\D(m)$ spanned by nonempty diagrams.  
	Then at the level of homology, $\Theta$ agrees with the map $\Phi^\ast$ in \eqref{eq:Theta-homology-iso-D(m)}.
	\item[(b)] At the level of homology,  $\Theta$ is a Hopf algebra isomorphism.
	\end{itemize}
\end{repthm}
\begin{rem}
If $n \geq 4$, one can easily show that $\D(m)$ is of finite type \cite[Remark 6.21]{Lambrechts-Volic:2014}.  Thus for $n \geq 4$, 
the sum in \eqref{eq:omega-form-0} is finite, since $\Conf(m,\R^n)$ is finite-dimensional and thus supports differential forms of degree at most $\dim(\Conf(m,\R^n))=m n$.
\end{rem}

Our notation so far suggests that, in the right setting, the isomorphism $\Theta$ in \eqref{eq:Theta-homology-iso-D(m)} is defined much like $\theta$ in \eqref{eq:general-theta}, with $M=\Conf(m,\R^n)$. Indeed this is the case, but there are some important differences discussed next.

We begin by setting the vector space $V=\D(m)^\ast$ instead of $H_{\ast-1}(M;\R)$; note that this $V$ is not isomorphic to $H_{\ast-1}(M;\R)$. Consequently, in place of the basis $X$, we choose the canonical basis $\mathcal{B}(m)=\{\Gamma\}$ of diagrams for the subspace of $\D(m)$ spanned by nonempty diagrams. (This basis is well defined up to signs.)  The corresponding tensor algebra $T(V)=T(\D(m)^\ast)$ can be now identified with $\Ba^\ast(\D(m))$.

Consider the definition of our connection $\omega$ in \eqref{eq:omega-form-0}.
In contrast to the general formula \eqref{eq:formal-connection} (and to Kohno's connection in Remark \ref{R:kohno-connection}), it contains only the first-order part (i.e. $\sum w_i X_i$). One of the assumptions for Theorem \ref{thm:chen-pwr-connection} is that each $w_i$ is a closed form representing the cohomology class dual to $X_i$.  This assumption does not hold for $\omega$, since $I(\Gamma)$ is not a closed form in general (for instance when $\Gamma$ has free vertices). Therefore, we cannot apply Theorem \ref{thm:chen-pwr-connection} directly.  In addition, we need to define a derivation $\partial$ satisfisfying \eqref{eq:twisted-cochain-partial}. We will show in Lemma \ref{lem:twisted-cochain} that, not suprisingly, the dual differential $d^\ast_{\Ba}$ on $\Ba^\ast(\Dm(m))$ plays the role of $\partial$ in our setting.  This will allow us to show in Lemma \ref{lem:connection-chain-map} that the transport of $\omega$ induces a chain map.  To prove Theorem \ref{thm:chen-pwr-connection}, it will only remain to check that $\Theta$ induces a quasi-isomorphism and is a map of Hopf algebras.

\medskip

For $\omega$ in \eqref{eq:omega-form-0}, we introduce the following notation:
\begin{align*}
	\varepsilon(\omega) & =\sum_{\Gamma\in \mathcal{B}(m)}  \frac{1}{|\mathrm{Aut}(\Gamma)|} I(\Gamma)\otimes [\varepsilon(\Gamma)\Gamma^\ast], \\
	\varepsilon(\Gamma)  =(-1)^{|\Gamma|},\qquad |\Gamma| & = (n-1)|E(\Gamma)| - n |V_{\mathrm{free}}(\Gamma)|,\qquad \varepsilon(I(\Gamma))=\varepsilon(\Gamma).
\end{align*}
Following \cite{Hain:1984}, we write for any $\sum_i a_i\otimes A_i$ and $\sum_k b_k\otimes B_k$  in $\ChdR^{\ast}(\Conf(m,\R^n))\otimes \Ba^\ast(\Dm(m))$,
\begin{align}
	d\bigl(\sum_i a_i\otimes A_i\bigr) & = \sum_i da_i\otimes A_i, \notag\\
	d^\ast_{\Ba}\bigl(\sum_i a_i\otimes A_i\bigr) & = \sum_i a_i\otimes d^\ast_{\Ba} A_i,  \label{eq:tensor-differentials}\\
	\bigl(\sum_i a_i\otimes A_i\bigr)\wedge \bigl(\sum_k b_k\otimes B_k\bigr)& =\sum_{i,k} a_i\wedge b_k\otimes [A_i|B_k]. \notag
\end{align}
We interpret $\omega \wedge \omega$ by regarding $\Dm(m)^*$ as the subspace of length-one monomials in $\Ba^*(\Dm(m))$.

\begin{lem} 
\label{lem:twisted-cochain}
The connection form $\omega$ in \eqref{eq:omega-form-0} satisfies the twisted cochain condition
	\begin{equation}\label{eq:twisted-cochain}
		d^\ast_{\Ba} \omega+d\omega-\varepsilon(\omega)\wedge\omega=0.
	\end{equation}
\end{lem}
\begin{proof}
	Let us compute the differential $d^\ast_{\Ba}([\Gamma^\ast])$ of a length-one monomial $[\Gamma^\ast]$ in $(\Ba^\ast(\D(m)),d^\ast_{\Ba})$. 
	By directly dualizing the definition \eqref{eq:bar-differential} of $d_{\Ba}$,  we obtain the following (cf.~\cite[p.~307]{Felix-Halperin-Thomas}), where $(\Gamma_a, \Gamma_b)$ denotes an ordered pair in $\mathcal{B}(m) \x \mathcal{B}(m)$:
	\begin{equation}\label{eq:d-ast-B-gamma}
	d^\ast_{\Ba}\left( \frac{[\Gamma^\ast]}{|\mathrm{Aut}(\Gamma)|}  \right)  =
	\sum_{\substack{ (\Gamma_a,\Gamma_b) \\ \Gamma_a\cdot\Gamma_b \cong \Gamma}} 
	\frac{\frac{1}{2} \bigl([\varepsilon(\Gamma_a)\Gamma^\ast_a | \Gamma^\ast_b ]+(-1)^{|\Gamma_a||\Gamma_b|}[\varepsilon(\Gamma_b)\Gamma^\ast_b | \Gamma^\ast_a ]\bigr)}{|\mathrm{Aut}(\Gamma_a)| |\mathrm{Aut}(\Gamma_b)|}  - 
	\sum_{\Gamma':\,  \Gamma'/e=\Gamma} \frac{\varepsilon(\Gamma',e) [(\Gamma')^\ast]}{|\mathrm{Aut}(\Gamma')|}
	\end{equation}
	The above equality is most easily verified by recalling that $\Gamma^*/|\mathrm{Aut}(\Gamma)|$ is the functional which sends $\Gamma$ to 1, and thus by grouping together any $|\mathrm{Aut}(-)|$ factors together with the corresponding diagrams.
	This yields
	\begin{align}
	d^\ast_{\Ba}\omega 
	& =\sum_{\Gamma \in \mathcal{B}(m)} I(\Gamma)\otimes d^\ast_{\Ba}\left( \frac{[\Gamma^\ast]}{|\mathrm{Aut}(\Gamma)|}\right) \notag\\
	&  = \sum_{\Gamma \in \mathcal{B}(m)}  \sum_{\substack{ (\Gamma_a,\Gamma_b) \\ \Gamma_a\cdot\Gamma_b\cong\Gamma}} I(\Gamma)\otimes  \frac{\frac{1}{2} \bigl([\varepsilon(\Gamma_a)\Gamma^\ast_a | \Gamma^\ast_b ]+(-1)^{|\Gamma_a||\Gamma_b|}[\varepsilon(\Gamma_b)\Gamma^\ast_b | \Gamma^\ast_a ]\bigr)}{|\mathrm{Aut}(\Gamma_a)| |\mathrm{Aut}(\Gamma_b)|}  \label{eq:d^ast-omega-2-terms}\\
	&\qquad - \sum_{\Gamma \in \mathcal{B}(m)} \sum_{\substack{\Gamma' \\  \Gamma'/e=\Gamma}} \varepsilon(\Gamma',e)\, I(\Gamma)\otimes \frac{[(\Gamma')^\ast]}{|\mathrm{Aut}(\Gamma')|}. \notag
	\end{align}
	\no From \eqref{eq:omega-form-0}, $d\omega$ is the sum
	\begin{equation}\label{eq:d-omega-sum}
	\begin{split}
	d\omega
	& =\sum_{\Gamma \in \mathcal{B}(m)} d(I(\Gamma))\otimes \frac{[\Gamma^\ast]}{|\mathrm{Aut}(\Gamma)|}
	=\sum_{\Gamma \in \mathcal{B}(m)} I(d\Gamma)\otimes \frac{[\Gamma^\ast]}{|\mathrm{Aut}(\Gamma)|}\\
	& =\sum_{\Gamma \in \mathcal{B}(m)} \ \sum_{e\in E(\Gamma)} \varepsilon(\Gamma,e) I(\Gamma/e)\otimes \frac{[\Gamma^\ast]}{|\mathrm{Aut}(\Gamma)|}.
	\end{split}
	\end{equation}
	Rearranging \eqref{eq:d-omega-sum}, we obtain 
	\[
	d\omega=\sum_{\Gamma \in \mathcal{B}(m)} \sum_{\Gamma':\, \Gamma'/e=\Gamma} \varepsilon(\Gamma',e)\,  I(\Gamma)\otimes \frac{[(\Gamma')^\ast]}{|\mathrm{Aut}(\Gamma')|},
	\]
	which is the same as the second term in \eqref{eq:d^ast-omega-2-terms}. Next
	\[
	\begin{split}
	\varepsilon(\omega)\wedge\omega & =
	\sum_{(\Gamma_a, \Gamma_b)} I(\Gamma_a)\wedge I(\Gamma_b) \otimes \frac{[\varepsilon(\Gamma_a) \Gamma^\ast_a|\Gamma^\ast_b]}{|\mathrm{Aut}(\Gamma_a)| |\mathrm{Aut}(\Gamma_b)|}
	=\sum_{(\Gamma_a, \Gamma_b)} I(\Gamma_a\cdot\Gamma_b) \otimes \frac{ [\varepsilon(\Gamma_a) \Gamma^\ast_a|\Gamma^\ast_b]}{|\mathrm{Aut}(\Gamma_a)| |\mathrm{Aut}(\Gamma_b)|}\\
	& = \sum_{\Gamma \in \mathcal{B}(m)} \sum_{\substack{ (\Gamma_a, \Gamma_b) \\ \Gamma_a\cdot\Gamma_b \cong \Gamma}} I(\Gamma)\otimes  
	\frac{\frac{1}{2} \bigl([\varepsilon(\Gamma_a) \Gamma^\ast_a|\Gamma^\ast_b]+(-1)^{|\Gamma_a||\Gamma_b|}[\varepsilon(\Gamma_b)\Gamma^\ast_b | \Gamma^\ast_a ]\bigr)}{|\mathrm{Aut}(\Gamma_a)| |\mathrm{Aut}(\Gamma_b)|},
	\end{split}
	\]
	which is the same as the first term in \eqref{eq:d^ast-omega-2-terms}.  These computations yield \eqref{eq:twisted-cochain}.	 
\end{proof}

\begin{lem}\label{lem:connection-chain-map}
	The formal transport of $\omega$, defined as\footnote{The symbol 1 here is the 0-cochain that is the unity in $\Chen^\ast(M)\otimes \Ba^*(\Dm(m))$.  In particular, it evaluates to 0 on chains of positive dimension.}
	\begin{equation}\label{eq:T}
	\mathcal{T}=1+\varint \omega+\varint \omega \omega+\varint \omega  \omega \omega+\ldots 
	\ \in \ \Chen^\ast(M)\otimes \Ba^*(\Dm(m)),
	\end{equation}
	satisfies
	\begin{equation}\label{eq:T-d-commute}
	d_{\Chen}\mathcal{T}=-d^\ast_{\Ba}\mathcal{T}
	\end{equation}
	and yields a chain map
	\begin{align}
	\label{eq:Theta-via-T}
	\begin{split}
	\Theta: C^{\operatorname{sing}}_\ast(\Omega \Conf(m,\R^n)) &\longrightarrow \Ba^\ast(\D(m)), \\
	c &\longmapsto \langle  \mathcal{T}, c\rangle = \delta_{0,n}+\sum_r\langle\varint \omega^{(r)}, c\rangle.	
	\end{split}
	\end{align}
\end{lem}

\begin{proof}
    Using the bar differential  $d_{\Chen}$ of the iterated integral complex  $\Chen^\ast(\Conf(m,\R^n))$ defined in \eqref{eq:bar-differential}, we have
	\[
	\begin{split}
	d_{\Chen}\bigl(\varint \omega^{(r)}\bigr)   & = \sum^r_{i=1} (-1)^{i}\varint\varepsilon(\omega)^{(i-1)} (d\omega) \omega^{(r-i)} + \sum^{r-1}_{i=1} (-1)^{i+1}\varint\varepsilon(\omega)^{(i-1)}(\varepsilon(\omega) \wedge \omega)\omega^{(r-i-1)}.
	\end{split}
	\]
	Arranging the summation properly, we get
	\[
	\begin{split}
	d_{\Chen}\bigl(\mathcal{T}\bigr)=d_{\Chen}\bigl(\sum_{r\geq 0} \varint \omega^{(r)}\bigr) & = \sum_{r\geq 0}\sum^r_{i=1} (-1)^{i}\varint\varepsilon(\omega)^{(i-1)} (d\omega-\varepsilon(\omega) \wedge \omega) \omega^{(r-i)}\\
	& = \sum_{r\geq 0}\sum^r_{i=1} (-1)^{i}\varint\varepsilon(\omega)^{(i-1)} (-d^\ast_{\Ba}\omega) \omega^{(r-i)}=-d^\ast_{\Ba} \bigl(\mathcal{T}\bigr).
	\end{split}
	\]
	We used Lemma \ref{lem:twisted-cochain} in the third equality above.  
	The last equality holds because, by using \eqref{eq:d-ast-B-gamma}, we may write
	\[
	\begin{split}
	d^\ast_{\Ba}\bigl([\Gamma^\ast_1\, |\, \Gamma^\ast_2\, |\, \cdots \, |\, \Gamma^\ast_m]\bigr)   & =\sum^m_{i=1} (-1)^{i}[\varepsilon(\Gamma_1)\Gamma^\ast_1\, |\, \cdots \,| \varepsilon(\Gamma_{i-1})\Gamma^\ast_{i-1}\,| d^\ast_{\Ba} ([\Gamma^\ast_i]) |\, \Gamma^\ast_{i+1}\,|\, \cdots\, |\, \Gamma^\ast_m].
	\end{split}
	\]
	Since $\langle d_{\Chen} \mathcal{T}, c\rangle = \langle \mathcal{T}, \partial_{\mathrm{sing}} c \rangle$ and $ d_{\Ba}^*\langle \mathcal{T}, c \rangle = \langle d_{\Ba}^* \mathcal{T}, c \rangle$, formula \eqref{eq:T-d-commute} implies that $\Theta$ as defined above is a chain map.
\end{proof}

\begin{proof}[Proof of Theorem \ref{thm:Phi-pwr-series}]
		It remains only to check that $\Theta$ is dual to $\Phi$ and that $\Theta$ is a map of Hopf algebras.  
	Indeed, $\Theta$ is a chain map which fits in the commutative diagram of chain maps (which are also quasi-isomorphisms)
	\begin{equation}\label{eq:Theta-Phi-dual-diagram}
	\begin{tikzcd}[column sep=small]
	C^\textrm{sing}_\ast(\Omega \Conf(m,\R^n)) \arrow{rrrr}{i}\arrow{rrrrd}{\Theta} & & & & \operatorname{Hom}\bigl(\Ch^*_{dR}(\Omega \Conf(m,\R^n)), \R\bigr)\arrow{d}{\Phi^\ast},\\
	& & & &	\Ba^\ast(\Dm(m))	
	\end{tikzcd}
	\end{equation}
	i.e. $\Theta=\Phi^\ast\circ i$,  where 
	\[
	i(c)(\beta)=\langle \beta , c\rangle,\qquad \beta \in \Ch^*_{dR}(\Conf(m,\R^n)).
	\]
	Note  that $\Phi^\ast$ is a quasi-isomorphism by our previous work \cite{KKV:2020} and $i$ is a quasi-isomorphism by a result of Chen in \cite[p. 859]{Chen:1977}.
	 We next check that $\Theta$ is multiplicative with respect to the Pontryagin product, i.e., 
	\begin{equation}\label{eq:Theta-multiplicative}
	\Theta(a\cdot b)=\Theta(a) \Theta(b).
	\end{equation}
	We use the identity from \cite[p. 843]{Chen:1977} 
	\[
	\begin{split}
	\langle\varint \alpha_1\alpha_2\cdots\alpha_m, a\cdot b\rangle & =\sum^m_{i=0} \langle\varint \alpha_1\alpha_2\cdots\alpha_i, a\rangle\langle \varint \alpha_{i+1}\alpha_{i+2}\cdots\alpha_m, b\rangle.
	\end{split}
	\]
	Since
	\[
	\varint \omega^{(r)} =\sum_{\Gamma_1,\ldots, \Gamma_r} \varint I({\Gamma_1})\cdots I({\Gamma_r})\otimes [\Gamma_1 |\, \ldots | \Gamma_r], \]
	we have
\begin{align*}
\langle\varint \omega^{(r)},c\rangle & = \sum_{\Gamma_1,\ldots, \Gamma_r} \langle\varint I({\Gamma_1})\cdots I({\Gamma_r}),c\rangle \otimes [\Gamma_1 |\, \ldots | \Gamma_r],\\
\langle\varint \omega^{(t)},a\rangle \langle\varint \omega^{(r)},b\rangle  & = \sum_{\Gamma_1,\ldots, \Gamma_t,\Gamma'_1,\ldots, \Gamma'_r} \langle\varint I({\Gamma_1})\cdots I({\Gamma_t}),a\rangle   \langle\varint I({\Gamma'_1})\cdots I({\Gamma'_r}),b\rangle\otimes [\Gamma_1 |\, \ldots | \Gamma_t\, |\, \Gamma'_1 |\, \ldots | \Gamma'_r],\\
\langle\varint \omega^{(r)},a\cdot b\rangle & = \sum_{\Gamma_1,\ldots, \Gamma_r} \langle\varint I({\Gamma_1})\cdots I({\Gamma_r}),a\cdot b\rangle  \otimes  [\Gamma_1 |\, \ldots | \Gamma_r]\\
	& =  \sum_{\Gamma_1,\ldots, \Gamma_r}  \sum^r_{i=1}\langle\varint I({\Gamma_1})\cdots I({\Gamma_i}),a\rangle    \langle\varint I({\Gamma_{i+1}})\cdots I({\Gamma_r}),b\rangle  \otimes   [\Gamma_1 |\, \ldots | \Gamma_r],
\end{align*}	
and we obtain $\langle \mathcal{T},  a\cdot b\rangle = \langle \mathcal{T},  a\rangle \langle \mathcal{T}, b\rangle$ and \eqref{eq:Theta-multiplicative} as a result.

Finally, since $\Theta$ is defined via the power series connection, we can apply \cite[Lemma 6.17]{Hain:1984} to establish that $\Theta$ is a map of coalgebras.  The key point is that the map $\Phi=\Theta^*$ is an algebra map because it takes the shuffle product on $\Ba(\Dm(m))$ to the wedge product of differential forms.
\end{proof}

\begin{rem}	
The above modification of the power series technique applies to any situation where one has a quasi-isomorphism 
$I:(\mathcal{A},\delta)\longrightarrow (\ChdR^\ast(M),d)$ of CDGA's. Then a connection form $\omega$ may be defined as in \eqref{eq:omega-form-0} for a choice of a basis $\mathcal{B}(\mathcal{A})$ of the CDGA $\mathcal{A}$.
\end{rem}

\section{Samelson products, primitive diagrams, and trees}
\label{S:brackets-primitives-trees}

We will now prove Theorems \ref{thm:primitive-diagrams} and \ref{thm:theta-trees}, which describe homotopy classes in configuration space in terms of diagrams.  The main point of Section \ref{S:SamelsonBlowup} is Lemma \ref{lem:bracket-recursive}, which describes  Samelson products in terms of blow-ups of diagrams.  We use this fact in Section \ref{S:proof-primitives} to prove Theorem \ref{thm:primitive-diagrams}, which says that the map $\Theta$ induced by the transport of the power series connection takes real homotopy classes in $\Omega\Conf(m,\R^n)$ to primitive diagrams in $\D(m)^*$.  
In Section \ref{S:theta-trees-proof}, we prove Theorem \ref{thm:theta-trees}, which recasts these elements as trivalent trees modulo the IHX relations; we also obtain an explicit formula that completely describes $\Theta$ for non-repeating monomials.  In Section \ref{S:brunnian-spherical-links}, we observe a connection to Brunnian spherical links and provide examples of Theorem \ref{thm:theta-trees}.

\subsection{Samelson products and blow-ups of diagrams}
\label{S:SamelsonBlowup}

Recall that the Samelson product $[ - , - ]$ on $\pi_*(\Omega X)$ makes $\pi_*(\Omega X) \otimes \R$ into a graded Lie algebra over $\R$.  It differs only by a sign from the Whitehead product  $[ - , - ]_W$ of the corresponding elements in $\pi_*(X)$.  
Let $h: \pi_{*}(\Omega X)\otimes \R \to H_*(\Omega X; \R)$ be the Hurewicz map.
Restricting its codomain to the subspace of primitive elements gives an isomorphism of graded Lie algebras.
See Appendix \ref{apx:brackets} for further details.

The next Lemma is proven in Appendix \ref{apx:brackets}.
Lemma \ref{lem:left-normed-span-conf} below is its specialization to the graded Lie algebra $\mathcal{L}=\pi_*(\Omega \Conf(m,\R^n)) \otimes \R$.
The latter lemma helps us show in Corollary \ref{cor:PDm(m)-to-PD(m)} that our proof of Theorem \ref{thm:primitive-diagrams} can be adapted to $\D(m)$ rather than $\Dm(m)$ with only a little modification.
The spanning set obtained in Lemma \ref{lem:left-normed-span-conf} will also be used in Section \ref{S:high-dim-braids}.

\begin{replem}{lem:left-normed-span}
	Let $\mathcal{L}$ be a graded Lie algebra over $\R$, generated by elements $X_i$, $i\in \mathcal{I}$, of degrees $|X_i|$.  
	\begin{enumerate}
		\item[(a)]  The left-normed monomials in the $X_i$ span $\mathcal{L}$.  
		\item[(b)]  Powers $[ X_i, \dots, X_i ]$ of length greater than 2 vanish. 
		\item[(c)]  Squares $[X_i,X_i]$ are nontrivial only on generators in even grading.
		\item[(d)]  If $|X_i|$ and $|X_h|$ are odd, then $[[X_i,X_i],X_h] = - 2 [[X_i,X_h],X_i]$.
	\end{enumerate}
\end{replem}

Recall that additively, $\mathcal{L}_{m}(n-2)=\pi_* (\Omega \Conf(m,\R^n))$ is a direct sum of free Lie algebras \eqref{eq:L_m(n-2)}, with the $j$-th summand generated by the elements $B_{j,1},B_{j,2},\ldots, B_{j,j-1}$.
We thus deduce the following.

\begin{lem}\label{lem:left-normed-span-conf}
	  For any $n\geq 3$ and any $m\geq 1$, $\mathcal{L}_{m}(n-2)=\pi_*(\Omega \Conf(m,\R^n)) \otimes \R$	is additively generated by left-normed Samelson products
	  \begin{equation}\label{eq:left-normed-span-conf}
	   [ B_{j,i_1},B_{j,i_2},B_{j,i_3}\ldots, B_{j,i_k} ]=[[\ldots [[B_{j,i_1},B_{j,i_2}],B_{j,i_3}],\ldots ], B_{j,i_k}],
	  \end{equation}
	  for $2\leq j\leq m$ and $1\leq i_1, i_2, \ldots, i_k \leq j-1$.
	  For any $j$ with $2\leq j\leq m$ and any $i,h < j$, the identities in Lemma \ref{lem:left-normed-span}, parts (b), (c), and (d), hold with $X_i=B_{j,i}$ and $X_h=B_{j,h}$.
\qed
\end{lem}

We base the proofs of Theorem \ref{thm:primitive-diagrams} and part of Theorem \ref{thm:theta-trees} on the next Lemma, which says that we can represent an iterated bracket in the $B_{j,i}$ as an element in $P\Dm(m)^\ast$.

\begin{lem}\label{lem:bracket-recursive} Let $C$ be any iterated Samelson product (with repeats allowed) in the generators $B_{j,i}$ of length $k$. 
\begin{itemize}
	\item[(i)]  Then $\Theta\circ h (C) \in \Ba^\ast(\Dm(m))$ is homologous to a length-1 monomial 
	$[\Gamma^\ast_C]$ where $\Gamma^\ast_C$ is a cycle in $P\Dm(m)^\ast$.  
	\item[(ii)]  With $C$ as above, write $C=[A,B]$, where $A$ and $B$ are iterated Samelson products of length less than $k$.  Then $\Gamma^\ast_C$ can be determined recursively by using $\Gamma_{A}$ and $\Gamma_{B}$:
	\begin{equation}\label{eq:Gamma_C}
	\Gamma^\ast_{[A,B]}=
	(-1)^{a} \delta^\ast((\Gamma_A \cdot \Gamma_{B})^*)  \ \in \ P\Dm(m)^\ast, \qquad
	 a=|\Gamma^\ast_A|=|A|+1.
	\end{equation}
	\end{itemize}
\end{lem}

\begin{proof}[Proof of Lemma \ref{lem:bracket-recursive}] 
	We proceed by induction on $k$.  For $k=1$, $\Theta(B_{j,i})$ is the class $[\Gamma_{j,i}^*]$ of a diagram with a single chord between vertices $i$ to $j$ (oriented from $i$ to $j$ if $n$ is odd), which lies in $P\Dm(m)^*$. This proves the base case for statement (i).
	
	Now suppose that $A,B,$ and $C$ are as above, where $C$ has length $k\geq 2$, and that (i) holds for $A$ and $B$; that is $\Theta\circ h (A)$ and $\Theta\circ h (B)$ are homologous to $[\Gamma_A^*]$ and $[\Gamma_B^*]$ for some cycles $\Gamma_A^*, \Gamma_B^* \in P\Dm(m)$.  
	Let $a=|\Gamma^\ast_A|$ and $b=|\Gamma^\ast_B|$ be the degrees in $\Dm(m)^\ast$.  Note that $\Gamma_A\cdot \Gamma_B\neq 0$, since we do not quotient by diagrams with multiple edges in $\Dm(m)$.
	(More generally, $\Dm(m)$ is the free CDGA on its indecomposable elements, as observed in \cite{Lambrechts-Turchin}.)
	In general, the homological differential on length-1 monomials in the cobar complex on a CDGC is a sum of graded commutators \cite[p.~307]{Felix-Halperin-Thomas}.  We apply this to the product $\Gamma_A\cdot \Gamma_B$ of two indecomposable elements.  For the moment, suppose $\Gamma_A \not\cong \Gamma_B$.  We then have 
	\begin{equation}\label{eq:d_B(A.B)}
	\begin{split}
	d^\ast_{\Ba}[(\Gamma_A\cdot \Gamma_B)^*] & =  [(-1)^{a}\Gamma^\ast_A | \Gamma^*_B]+(-1)^{a b}[(-1)^b \Gamma^*_B | \Gamma^\ast_A]-[\delta^\ast((\Gamma_A\cdot \Gamma_B)^*)]\\
	& = (-1)^{a}\bigl([\Gamma^\ast_A | \Gamma^*_B]-(-1)^{(a-1)(b-1)}[ \Gamma^*_B | \Gamma^\ast_A]\bigr)-[\delta^\ast((\Gamma_A\cdot \Gamma_B)^*)]\\
	& = (-1)^{a}[ \Gamma^\ast_A, \Gamma^*_B ]-[\delta^\ast((\Gamma_A\cdot \Gamma_B)^*)].
	\end{split}
	\end{equation}
	The first equality implicitly uses the fact that $\Gamma^*$ is defined so that $\langle \Gamma_1^*, \Gamma_2\rangle = \delta_{\Gamma_1, \Gamma_2} |\mathrm{Aut}(\Gamma_1)|$, together with the fact that $|\mathrm{Aut}(\Gamma_A \cdot \Gamma_B)| = |\mathrm{Aut}(\Gamma_A)| |\mathrm{Aut}(\Gamma_B)|$ if  $\Gamma_A \not\cong \Gamma_B$. 
	In the last line above, $[ - , - ]$ denotes the graded commutator in the cobar complex, and we recall that there are suppressed desuspensions $\partial^{-1}_\ast$ throughout to get the last equality. 
	
	Now if instead $\Gamma_A \cong \Gamma_B$ (and $n$ is odd), then the homological differential gives only one term $(-1)^a[\Gamma_A^* | \Gamma_A^*]$, which can be rewritten as $\frac{1}{2}(-1)^a[\Gamma_A, \Gamma_A]$.  However, $|\mathrm{Aut}(\Gamma_A \cdot \Gamma_A)| = 2 |\mathrm{Aut}(\Gamma_A)|^2$, so the other two terms in \eqref{eq:d_B(A.B)} are also essentially multiplied by $\frac{1}{2}$, and \eqref{eq:d_B(A.B)} still holds.  
	
	Thus  $[\delta^\ast((\Gamma_A\cdot \Gamma_B)^*)]$ is always homologous to $(-1)^{a}[\Gamma^\ast_A, \Gamma^*_B]$ in $H_\ast(\Ba^\ast(\Dm(m)))$.  
	 Now $h$ sends a Samelson product $[A,B]$ to 
	 $(-1)^a[ h(A), h (B)]$.  Since $\Theta$ is a map of graded algebras, 
	 $\Theta\circ h [ A, B]$ is homologous to $(-1)^{a}[ \Gamma^\ast_A, \Gamma^*_B ]$, and 
	 statement (ii), i.e.~formula \eqref{eq:Gamma_C}, is proven.
	 	
	It remains to check that $\Gamma^\ast_{[A,B]}$ is a primitive cycle in $\Dm(m)^\ast$.  
	Indeed, $[\Gamma_{[A,B]}^*]$ is homologous to $\Theta\circ h [ A, B]$ and hence is a cycle in $\Ba^*(\Dm(m))$.  Since $[\Gamma_{[A,B]}^*]$ is a cycle of length 1 in $\Ba^*(\Dm(m))$, $\Gamma_{[A,B]}^*$ must be primitive and a cycle in $\Dm(m)^*$.
	So statement (i) follows as well.
\end{proof} 

\subsection{Proof of Theorem \ref{thm:primitive-diagrams} (homotopy classes as primitive diagrams)}
\label{S:proof-primitives} 
By a mild abuse of notation, we will sometimes write $\Theta$ to mean the composition $\Theta \circ h$ with the Hurewicz map.
\begin{repthm}{thm:primitive-diagrams}
	(a) Let $P\Dm(m)^\ast$ be the primitives of the coalgebra $\Dm(m)^\ast$.  Then $\Theta$ induces an isomorphism 
	\[
	\pi_\ast(\Omega \Conf(m,\R^n)) \otimes \R \overset{\cong}{\longrightarrow} 	H_\ast(P\Dm(m)^\ast, \delta^\ast|_{P\Dm(m)^*}).
	\]
	(b) The dual $\Phi$ of $\Theta$ induces an isomorphism 
	\[
	\mathrm{Hom}(\pi_\ast(\Omega \Conf(m,\R^n)), \, \R) \overset{\cong}{\longleftarrow} 
	H^\ast(I\Dm(m),\widetilde{\delta}),
	\]
	where $\widetilde{\delta}=\pi_{I\Dm}\circ\delta$, and $\pi_{I\Dm}$ is the projection onto the indecomposables $I\Dm(m)$ of $\Dm(m)$.
\end{repthm}

\begin{proof}
Pre-composing the map $\Theta$ in \eqref{eq:Theta-homology-iso-D(m)} with the Hurewicz map gives an isomorphism of graded Lie algebras 
	\[
	 \pi_\ast(\Omega \Conf(m,\R^n)) \otimes \R\longrightarrow PH_\ast(\Ba^\ast(\Dm(m)))
	\]
onto the space of primitive elements of the Hopf algebra $H_\ast(\Ba^\ast(\Dm(m)))$.
The primitive elements in a CDGC always form a subcomplex, so we consider $(P\Dm(m)^*, \delta^*|_{P\Dm(m)^*})$ as a subcomplex of $(\Dm(m)^*, \delta^*)$.
We just have to check that
\begin{equation}
\label{eq:prim-homology-vs-homology-prim}
PH_\ast(\Ba^\ast(\Dm(m))) = H_*(P\Dm(m)^*).
\end{equation}

The containment ``$\supset$'' in \eqref{eq:prim-homology-vs-homology-prim} holds even if $\Dm(m)$ is replaced by an arbitrary CDGA.
Indeed, if $\Gamma^* \in P\Dm(m)^*$ is a cycle in $(P\Dm(m)^*, \delta^*|_{P\Dm(m)^*})$, then $\Gamma$ is indecomposable and $\delta^*(\Gamma^*)=0$, so $d_B^*[\Gamma^*]=0$.  
Furthermore, $\Gamma^\ast_1$ and $\Gamma^\ast_2$ are homologous in  $(P\D(m)^\ast,\delta^\ast)$ if and only if $[\Gamma^\ast_1]$ and $[\Gamma^\ast_2]$ are homologous in $(\Ba^\ast(\Dm(m)),d^\ast_{\Ba})$.
Finally $[\Gamma^\ast]$ is primitive in $\Ba^\ast(\Dm(m))$ as a cobar cycle of length 1.

To show the inclusion ``$\subset$'' in \eqref{eq:prim-homology-vs-homology-prim}, note that by 
Lemma \ref{lem:bracket-recursive}, any bracket expression in the $B_{j,i}$ is homologous in $\Ba^\ast(\Dm(m))$ to a length-1 monomial $[\Gamma^*]$ such that $\delta^*\Gamma^*=0$,  which thus represents an element of $H_\ast(P\Dm(m)^\ast)$.  
So part (a) is proven.

For part (b), we dualize the composite given by $\Theta$ followed by the identification 
$PH_*(\Ba^*(\Dm(m))) \overset{\cong}{\to} H_*(P(\Dm(m)^*))$ to obtain an isomorphism 
\[
 \operatorname{Hom}(H_\ast(P\Dm(m)^\ast),\R) \overset{\cong}{\longrightarrow} \operatorname{Hom}(\pi_\ast(\Omega \Conf(m,\R^n)), \R).
 \]
By the universal coefficient theorem for cohomology \cite{Hatcher-book:2002} applied to the complex $(I\Dm(m),\widetilde{\delta})$,
\[
 \operatorname{Hom}(H^\ast(I\Dm(m)), \R)\cong H_\ast(P\Dm(m)^{\ast}, \delta^{\ast}|_{P\Dm(m)^*}).
\]
In each degree, this is an isomorphism of finite-dimensional vector spaces, so dualizing it gives an isomorphism $H^\ast(I\Dm(m)) \cong  \operatorname{Hom}(H_\ast(P\Dm(m)^{\ast}), \R)$.
So part (b) of Theorem \ref{thm:primitive-diagrams} is proven.
\end{proof}
 We can now ``diagrammatically'' define the bracket via the identity \eqref{eq:Gamma_C} in Lemma \ref{lem:bracket-recursive}:
\begin{equation}\label{eq:PD-bracket}
[ - , - ]: H_\ast(P\Dm(m)^\ast)\otimes H_\ast(P\Dm(m)^\ast)\longrightarrow H_\ast(P\Dm(m)^\ast).
\end{equation}
\begin{example}
\label{Ex:2HopfMap}
If $n=3$, 
$\Conf(2,\R^3)\simeq S^{2}$, and
$[B_{2,1},B_{2,1}] \in \pi_{3}(\Conf(2,\R^3))$ is twice the Hopf map \cite{Hatcher-book:2002}.  In this case the corresponding diagram is obtained as follows:
\begin{equation}\label{eq:diagram-2hopf}
\begin{split}
\Theta([B_{2,1},B_{2,1}])=[\Gamma_{2,1},\Gamma_{2,1}] & =	
	\Bigl[\vvcenteredinclude{0.15}{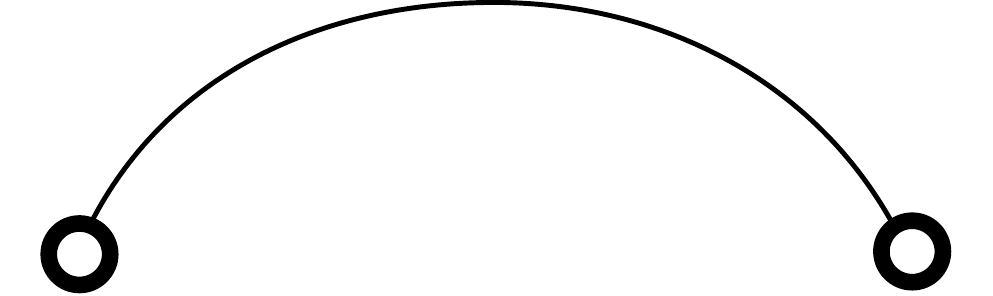},\vvcenteredinclude{0.15}{Dm-chord-G21.pdf}\Bigr]=(-1)^{2}\delta^\ast\Bigl(\vvcenteredinclude{0.15}{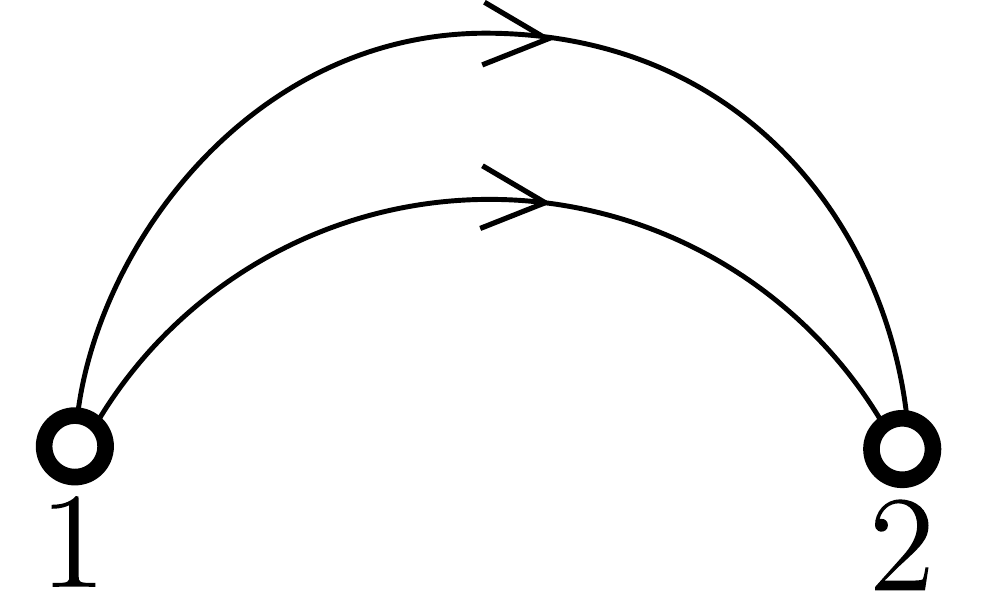}\Bigr)\\
	& = -\vvcenteredinclude{0.15}{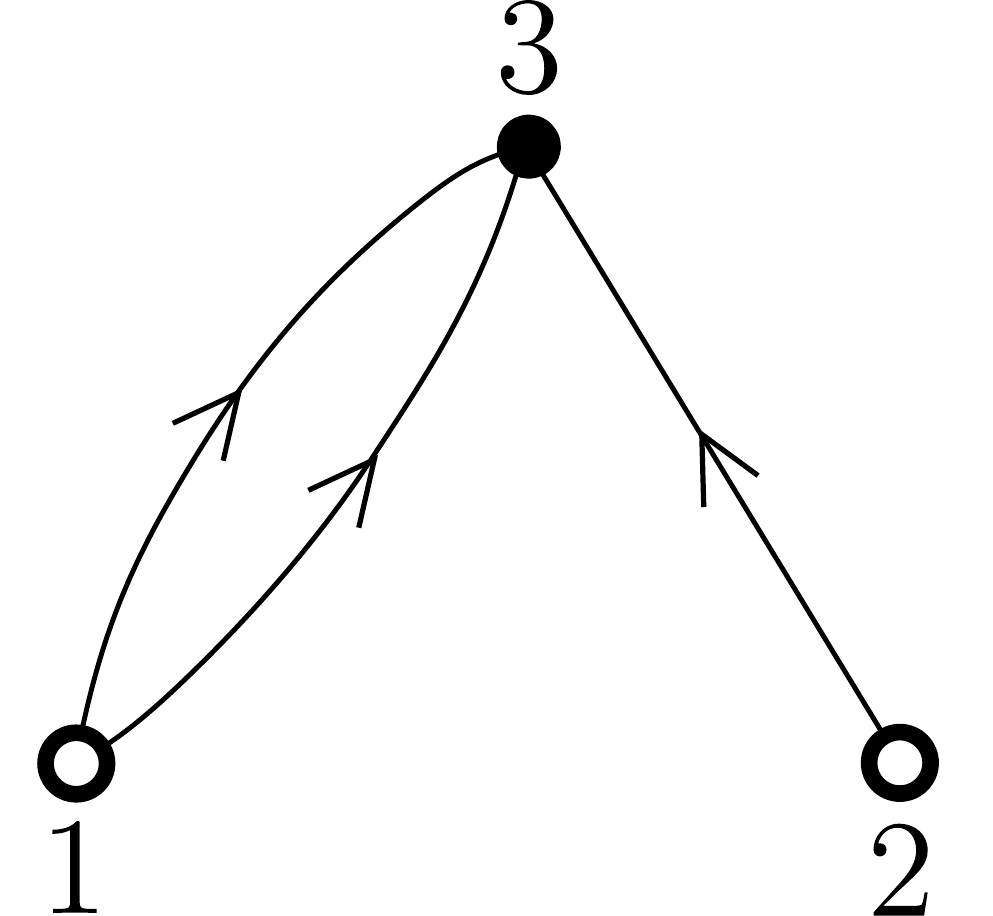}-\vvcenteredinclude{0.15}{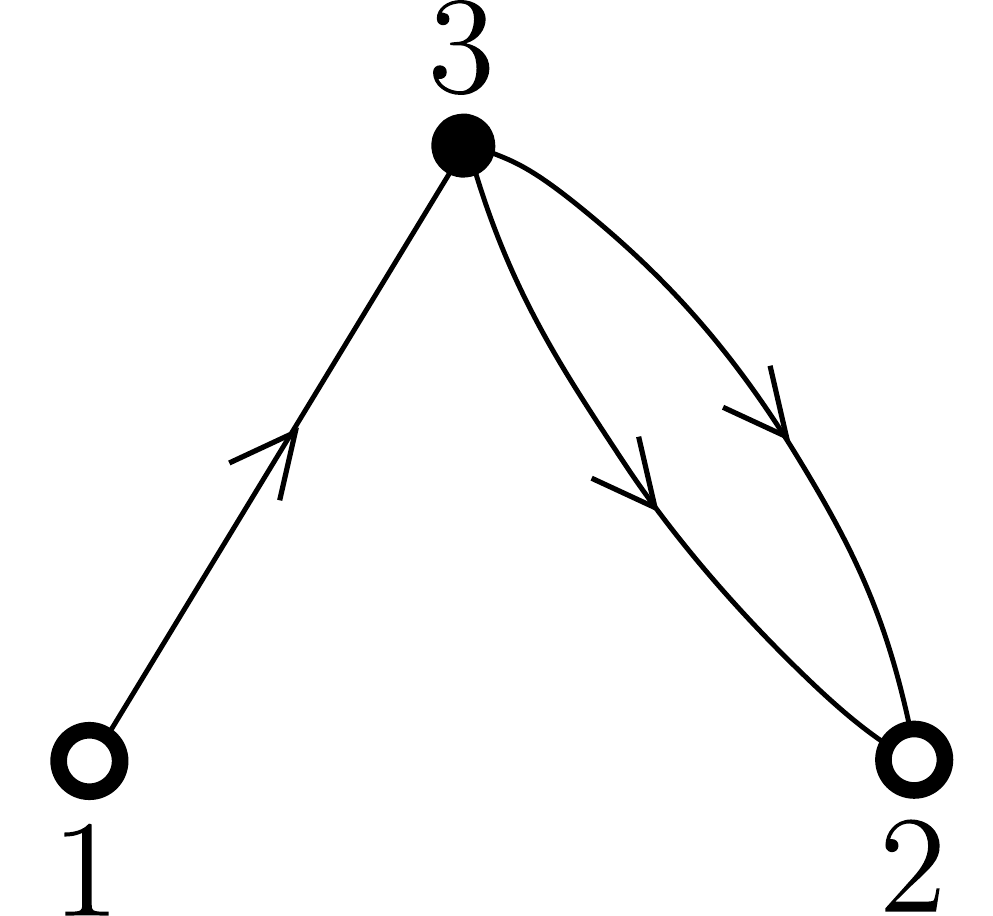},
	\end{split}
\end{equation}
where the sign comes from \eqref{eq:edge-sign}, applied to the choice of labels and arrows shown.  
If we call the above sum $-\Gamma_1^* - \Gamma_2^*$, then $\langle \Gamma_i, \Gamma_i^* \rangle=2$ for $i=1,2$ because each $\Gamma_i$ has 2 automorphisms.  The element in $\mathcal{T}^n(m)$ corresponding to each $\Gamma_i$ is a tripod (a uni-trivalent tree with 3 leaves), with leaves labeled 1,1,2 (for $i=1$) and 1,2,2 (for $i=2$).
\end{example}

We now describe the relationship between $H_\ast(P\Dm(m)^\ast)$ and $H_\ast(P\D(m)^\ast)$.
\begin{cor}\label{cor:PDm(m)-to-PD(m)}
	There are isomorphisms
	\begin{equation}\label{eq:to-H(PD(m))}
	H_\ast(P\Dm(m)^\ast)\cong\begin{cases}
	H_\ast(P\D(m)^\ast), & \text{for $n$ even},\\
	H_\ast(P\D(m)^\ast)\oplus \R^N, & \text{for $n$ odd}.
	\end{cases},\quad N={m\choose 2}.
	\end{equation}
\end{cor}
\begin{proof}
	The proof of Lemma \ref{lem:bracket-recursive} can be repeated for $\D(m)$ and 
	for left-normed bracket expressions with 
	distinct first two terms, i.e., for $C=[ B_{j,i_1,},B_{j,i_2},\ldots, B_{j,i_k} ]$ with $i_1\neq i_2$.  
	Then $\Theta\circ h (C) \in \Ba^\ast(\D(m))$ is homologous to a length-1 monomial 
	$[\Gamma^\ast_C]$ where $\Gamma^\ast_C \in P\D(m)^\ast$. In addition, \eqref{eq:d_B(A.B)} holds for these brackets. As in the proof of Theorem \ref{thm:primitive-diagrams}, each of the generators of Lemma \ref{lem:left-normed-span-conf} can be represented by an element of $H_\ast(P\D(m)^\ast)$, with the exception of the squares $[B_{j,i},B_{j,i}]$ for $n$ odd. We conclude that $\mathcal{L}_m$ is isomorphic to $H_\ast(P\D(m)^\ast)$ for $n$ even and $H_\ast(P\D(m)^\ast)\oplus \R^N$ for $n$ odd, where the $\R^N$ factor is spanned by the square brackets. Since $H_\ast(P\Dm(m)^\ast)\cong \mathcal{L}_m$ the claim follows.
\end{proof}

\subsection{Proof of Theorem \ref{thm:theta-trees} (homotopy classes as trees)}
\label{S:theta-trees-proof}

Recall from the Introduction that $T: \pi_*(\Omega \Conf(m,\R^n)) \otimes \R \to \mathcal{T}^n(m)$ is the map defined using the direct-sum decomposition \eqref{eq:L_m(n-2)} that sends a bracket expression $B$ to the corresponding tree.

\begin{repthm}{thm:theta-trees} 
Below, assume $j, j_1, \dots, j_k \leq m$ and $I = (i_1, i_2, \dots, i_k)$ with $1 \leq i_1, \dots, i_k < j$.
	\begin{itemize}
		\item[(a)] If 
		 $C$ is a length-$k$ Samelson product on the generators $B_{j_1,i_1}, B_{j_2,i_2},\ldots, B_{j_k,i_k}$  (with possible repeats), then $\Theta(C)$ is represented in $H_\ast (P\Dm(m)^\ast)$ by a linear combination, with $\pm 1$ coefficients,  of all successive vertex blow-ups of the diagram $(\Gamma_{j,i_1}\cdot \Gamma_{j,i_2}\cdot\ldots  \cdot \Gamma_{j,i_k})^*$,  performed in the order given by the parenthesization in $C$, from innermost to outermost brackets. 
				
		\item[(b)]  The map $\Theta$ gives rise to an injection $\TTheta$ of $\pi_* (\Conf(m,\R^n))\otimes \R$ into the space $\mathcal{T}^n(m)$ of trivalent trees with leaves labeled by $\{1,\dots, m\}$, modulo the (graded) AS and IHX relations.
		
		\item[(c)] For any such multi-index $I$  (with repeats of indices allowed), $\langle \TTheta(B_{j;I}), T(B_{j;I}) \rangle = \pm 1$, where $\langle \cdot, \cdot \rangle$ is any Kronecker pairing determined by a basis of trees for $\mathcal{T}^n(m)$.

		\item[(d)] If $I$ has no repeated indices, then $\TTheta(B_{j;I})= \pm T(B_{j;I})$.
		In particular, for any $k\leq j$,
		\begin{equation}\tag{\ref{eq:borromean-tree}}
		B_{j;1,2,\ldots,k-1} \overset{\Theta}{\longmapsto} \pm\vvcenteredinclude{0.15}{Dm-Gj123dotsk.pdf} \quad \in \quad H_\ast(P\Dm(m)^\ast).
		\end{equation}
		Thus $\Theta$ maps a basis element \eqref{eq:Lie(m-1)-basis} of $Lie(m-1)$ to the diagram obtained from \eqref{eq:borromean-tree} by setting $k=j=m$ and permuting the segment vertices by $\sigma\in \Sigma(2,\ldots,m-1)$.  

		\item[(e)]  (Signs)  If $I$ is a multi-index of length $m-1$ with no repeats, then the sign of the diagram \eqref{eq:borromean-tree} corresponding to a basis element \eqref{eq:Lie(m-1)-basis} of $Lie(m-1)$ is as follows.  
		For $n$ odd, orient this diagram by labeling the vertices as in \eqref{eq:borromean-tree} and orienting each edge from smaller to larger vertex label.  
		For $n$ even, orient this diagram by using the vertex labels shown to order the edges first by smallest endpoint label, then by largest endpoint label, i.e., $\{1, m+1\}, \{2, m+1\}, \{3, m+2\}, \{4, m+3\}, \dots,  \{m,2m-2\}, \{2m-1, 2m-2\}$.  Then the sign $\varepsilon(n,m)$ is given by 
		\[
		\varepsilon(n,m) = 
		\left\{ 
		\begin{array}{ll}
		+1 & \text{ if \quad $n$ is even and $m\equiv 0, 1 \ \mathrm{mod} \ 4$ \quad or \quad $n$ is odd and $m$ is odd}\\
		-1 & \text{ if \quad $n$ is even and $m\equiv 2,3 \ \mathrm{mod} \ 4$ \quad or \quad $n$ is odd and $m$ is even.}\\
		\end{array}
		\right.
		\]

	\end{itemize}
\end{repthm}

\begin{proof}
	Part (a) follows directly from the identity \eqref{eq:Gamma_C} in Lemma \ref{lem:bracket-recursive}. 

	\medskip 

For part (b), the desired injection $\TTheta$ can be obtained as the composition in the top row of the following commutative diagram, explained below.
\begin{equation}
\label{E:LD-diagram}
\xymatrix@R1pc@C0.5pc{
\pi_*(\Omega\Conf(m,\R^n))\otimes \R \cong PH_*(\Omega\Conf(m,\R^n)) \ar[rr] \ar@{^(->}[d]& & 
PH_{0,*} (\LDm^{n+1}_{\mathrm{forest}}(m)^*)  \cong \mathcal{T}^n(m) \ar@{^(->}[d] \\
H_*(\Omega \Conf(m,\R^n)) \ar[r]^-{\Theta}_-\cong & 
H_*(\Ba^* (\Dm(m))) \ar@{^(->}[r]^-{\varphi^*} & 
H_{0,*} (\LDm^{n+1}_{\mathrm{forest}}(m)^*) 
}
\end{equation}
There is a differential graded Hopf algebra $\LDm^{n+1}(m)$ of \emph{link diagrams} for long links in $\R^{n+1}$.  A link diagram is a graph with some vertices lying on $m$ line segments and some free vertices, together with an orientation depending on the parity of $n+1$, much like in Definition \ref{D:DmOrientations}.
The subspace $\LDm^{n+1}_{\mathrm{forest}}(m)$, which is also a differential graded Hopf algebra, consists of forests whose leaves lie on the segments.  
Link diagrams of bidegree $(0,*)$ are precisely uni-trivalent graphs.  
The CDGAs $\LDm^{n+1}(m)$ and $\LDm^{n+1}_{\mathrm{forest}}(m)$ differ from the CDGAs $\LD^{n+1}(m)$ and $\LD^{n+1}_{\mathrm{forest}}(m)$ that we studied in \cite{KKV:2020} only in that link diagrams with multiple edges are not set to zero.  
Imposing this relation gives maps $\LDm(m) \twoheadrightarrow \LD(m)$ and $\LDm_{\mathrm{forest}}(m) \twoheadrightarrow \LD_{\mathrm{forest}}(m)$ of differential graded Hopf algebras.

The map $\varphi^*$ is dual to a map $\varphi$ defined at the level of cochains in \cite{KKV:2020}, where we showed that $\varphi$ is surjective in cohomology.
Although in that paper we defined $\varphi$ as a map $\Ba(\D(m)) \to \LD_{\mathrm{forest}}(m)$, one can easily modify that definition to get a map  $\Ba(\Dm(m)) \to \LDm_{\mathrm{forest}}(m)$.
Our proof of the surjectivity of $\varphi$ relied on work of Habegger and Masbaum,\footnote{Habegger and Masbaum's $C^t(m)$ is our $\mathcal{T}^2(m)$ ($\cong \mathcal{T}^n(m)$ for any even $n$), and their $A_k(m)$ is our $H_{0,k}(\LDm_{\mathrm{forest}}(m))$.} 
which like various other authors' work on Vassiliev invariants of knots links in $\R^3$ actually used (the uni-trivalent part of) $\LDm_{\mathrm{forest}}(m)$.  
So more precisely, our application of their work yields surjectivity of the composition $H^{0,*}(\LDm_{\mathrm{forest}}(m)) \to H^{0,*}(\LD_{\mathrm{forest}}(m)) \overset{\varphi}{\to} H^*(\Ba(\D(m)))$.  Thus the dual $\phi^*$ in \eqref{E:LD-diagram} is indeed injective.

Moreover, $\varphi$ is a map of Hopf algebras.  
Thus $\Theta \circ \varphi^*$ is also a map of Hopf algebras, and
the top horizontal arrow can be defined as the restriction of $\varphi^* \circ \Theta$ to the primitive elements.

Finally, the isomorphism $\mathcal{T}^n(m) \to PH_{0,*} (\LDm^{n+1}_{\mathrm{forest}}(m)^*)$ is given by averaging over all the possible ways of attaching the leaves labeled $i$ to segment $i$, for $i=1,\dots, m$.  The case where $n$ is even is proven in \cite[Theorem 8]{Bar-Natan:1995}, and the proof for $n$ odd is similar.\footnote{One needs to use $\LDm_{\mathrm{forest}}(m)$ rather than $\LD_{\mathrm{forest}}(m)$ to get an isomorphism to $\mathcal{T}^n(m)$.  Indeed, $PH_{0,*}(\LD_{\mathrm{forest}}(m))$ has elements corresponding to connected graphs that are not trees, such as a graph in which a pair of edges joins two (free) vertices, each of which is connected to a univalent (segment) vertex.}
The right-hand vertical inclusion is thus essentially an inclusion of trees on $m$ line segments into forests on $m$ line segments.

By Theorem \ref{thm:primitive-diagrams}, $\Theta$ sends the image of the Hurewicz map to linear combinations of (length-1 monomials of) connected diagrams.  On such elements, the composition of $\phi^*$ followed by the isomorphism to $\mathcal{T}^n(m)$ is given by first killing diagrams with loops of free vertices and then separating any edges which meet at the same segment vertex, as one would do to get from the fourth picture to the third picture in \eqref{eq:bracket-to-trees-2}.
In summary, the desired injection $\TTheta: \pi_*(\Conf(m,\R^n)) \hookrightarrow \mathcal{T}^n(m)$ is given by the Hurewicz map, followed by $\Theta$, followed the quotient by diagrams with loops of free vertices and the splitting apart of all edge-ends which meet at multivalent segment vertices.

\medskip
	
	For part (c), if $I$ is any multi-index we want to show that the successive blow-ups for $B_{j;I}$ determined by formula \eqref{eq:Gamma_C} produce the diagram $T(B_{j;I})$ with coefficient $\pm 1$.  
We get a diagram for every sequence of blow-ups, where at each step we may have a choice of multivalent vertex to blow up at because repeated indices are possible.
The sequence where we blow up at the root $j$ at every step produces precisely this diagram.  Consider any other sequence, which at some step has a blow-up at some other segment vertex, say $v$.  The result of such a sequence of blow-ups cannot be $T(B_{j;I})$ because segment vertex $v$ will have fewer edges incident to it than it does in $T(B_{j;I})$.  Moreover, any linear combination of trees equivalent to it cannot contain a term $T(B_{j;I})$ because the IHX relations preserve leaf labels.
	
\medskip

	For part (d), write $\Gamma_{j;I}^*$ for $\Theta(B_{j;I})$, as in Lemma \ref{lem:bracket-recursive}.  
	First suppose $I=(1,2,\dots, k)$.  
	We proceed by induction on $k$.  We know the statement is true for $k=1$.  
	Suppose $2 \leq k <j$ and that $\Gamma_{j;1,\dots,k-1}^*= \pm T(B_{j;1,\dots,k-1})$, 
	which is the diagram below (where the free vertex labels are shown only to determine the sign in part (e)):
	\[
	\vvcenteredinclude{0.15}{Dm-Gj123dotsk.pdf}
	\]
	By Lemma \ref{lem:bracket-recursive} and the induction hypothesis, 
	\[
	\Theta ([ B_{j,1}, \dots, B_{j,k} ])  = \pm\delta^*((\Gamma_{j;1,\dots,{k-1}} \cdot \Gamma_{j;k})^*) 
	= \pm T(B_{j;1,\dots,k}).
	\]
	The case of an arbitrary non-repeating multi-index $I$ can be treated by applying the same argument but with $(1,\dots,k)$ replaced by $(i_1, \dots, i_{k-1}, i_k)$.  

\medskip	

	To get the signs claimed in part (e), consider the iterated blow-ups of  $(\Gamma_{m,1} \cdot \Gamma_{m,2} \cdots \Gamma_{m,m-1})^*$. One adds up the signs from \eqref{eq:edge-sign} in those blow-ups, the signs from Definition \ref{D:DmOrientations} incurred by relabeling the tree, and for $n$ odd, the signs from \eqref{eq:Gamma_C} coming from the Samelson product.  
	For odd $n$, we incur the same number of minus signs, modulo 2, from the blow-ups and the Samelson product, namely $\frac{1}{2}(m-1)(m-2)$, while we incur $m-3$ minus signs from edge reversals.  For even $n$, we incur $\frac{1}{2}(m-1)(m-2)$ minus signs from the blow-ups and $m-3$ minus signs from permuting the edge labels.  
\end{proof}

\begin{rem}
\label{R:signs-of-trees-whitehead}		
If one uses the Whitehead product instead of the Samelson product, the sign $(-1)^a$ in equation \eqref{eq:Gamma_C} in Lemma \ref{lem:bracket-recursive} disappears, and the sign for $n$ odd above becomes $-1$ if $m \equiv 0,3 \mod 4$, while the sign for $n$ even is unchanged.
\end{rem}

We postpone our Examples to the end of Section \ref{S:brunnian-spherical-links} below for the purpose of connecting them to Brunnian links.

\subsection{The Lie module
and Brunnian spherical links.}
\label{S:brunnian-spherical-links}  
A spherical $m$-component Brunnian link $L$ in $\R^n$ gives rise to a class in $\pi_*(\Conf(m,\R^n))$ and hence a class in the homotopy groups $\pi_*(\Omega \Conf(m,\R^n))$ of the space of braids in one dimension higher.  
Using Theorem \ref{thm:Phi-pwr-series} and Theorem \ref{thm:theta-trees}, we will deduce that the Milnor invariants of $L$ in the sense of Koschorke \cite{Koschorke:1997} can be recovered by applying $\Theta$ to the corresponding class in $\pi_*(\Omega \Conf(m,\R^n))$ in Corollary \ref{cor:brunnian-spherical-links}.
In more detail, a spherical $m$-component \emph{link map} is a smooth map 
\begin{equation}\label{eq:spherical-link-map}
	L:S^{p_1}\sqcup S^{p_2} \sqcup \ldots \sqcup  S^{p_m}\longrightarrow \R^n, \ n\geq 3 \qquad  L(S^{p_i})\cap L(S^{p_j})=\varnothing,\ i\neq j.
\end{equation}
A \emph{link homotopy} is a homotopy through link maps.
If each $(m-1)$-component sublink of $L$ is link homotopic to a trivial link, $L$ is called \emph{Brunnian} or \emph{Borromean}. 
Let $p=\sum^m_{i=1} p_i$.  We say $L$ is $\kappa$-\emph{Brunnian}, following \cite{Koschorke:1997}, if whenever $p\leq m(n-2)$ the evaluation map
\begin{equation}
\label{eq:kappa(L)}
\kappa(L): \prod^m_{i=1} S^{p_i}\longrightarrow \Conf(m,\R^n),\qquad \kappa(L)(t_1,\ldots t_m)=(L_1(t_1),\ldots,L_m(t_m)),\quad L_i=L\bigr|_{S^{p_i}},
\end{equation}
factors via the projection $\pi:\prod^m_{i=1} S^{p_i}\longrightarrow S^{p}$, which collapses all faces of $\prod^m_{i=1} S^{p_i}$ to a point. That is, $L$ is $\kappa$-Brunnian if there exists a map $\widetilde{\kappa}(L):S^{|p|}\longrightarrow \Conf(m,\R^n)$ completing the diagram
\begin{equation}\label{eq:borr-diagram}
	\begin{tikzcd}[column sep=small]
		\prod^m_{i=1} S^{p_i} \arrow{drr}{\pi} \arrow{rrrr}{\kappa(L)} & & & & \Conf(m,\R^n).\\
		& &  S^{p} \arrow[dashed]{urr}{\widetilde{\kappa}(L)} & &		
	\end{tikzcd}
\end{equation}
It is easy to show that if $L$ is Brunnian, then it is also $\kappa$-Brunnian.  
Moreover, Koschorke \cite{Koschorke:1997} defines a map $h =\bigoplus_{I \in \Sigma_{m-2}} h_I$ which for a certain value of $p$ (specified below) induces an isomorphism from link homotopy classes of $m$-component Brunnian spherical link maps to 
$\bigoplus_{(m-2)!} \Z$.
One can then define the Milnor invariants $\{\mu_{I;j}(L)\}_I$ of $L$ as the components of the resulting element in $\bigoplus_{(m-2)!} \Z$. 

We now digress to define the generalized Hopf invariants $h_I$.  Although we do not need it to deduce Corollary \ref{cor:brunnian-spherical-links}, it is essential for our proof of Theorem \ref{T:braids-as-links}.
First, $\Conf(m,\R^n)$ contains a subspace $\bigvee_{m-1} S^{n-1}$ which is homotopy equivalent to the fiber of the projection that forgets one configuration point.
Let $pr_{\, \widehat{i}}: \bigvee_{m-1} S^{n-1} \to \bigvee_{m-2} S^{n-1}$ be the map that collapses the $i$-th summand to the basepoint.  Define the \emph{reduced homotopy groups} of a wedge of spheres by
$\widetilde{\pi}_*\left( \bigvee_{m-1} S^{n-1} \right):= \bigcap_{i=1}^{m-1} \ker (pr_{\, \widehat{i}})$.  
Because $L$ is Brunnian, $\widetilde{\kappa}(L)$ is not only homotopic to a map with image image in $\bigvee_{m-1} S^{n-1}$, but also represents a class in the reduced homotopy groups.
This definition is due to Koschorke \cite{Koschorke:1997}, but it is also closely related to earlier work of Boardman and Steer \cite{Boardman-Steer:1967}:

\begin{defin}
\label{D:hopf-invt}
Let $I=(i_1, \dots, i_{m-1})$ be a permutation of $\{1, \dots, m-1\}$ which fixes 1, i.e., such that $i_1=1$.\footnote{Compared to Koschorke's conventions, we reverse the order of each multi-index because we consider left-normed iterated Whitehead products, whereas he considers the right-normed ones.}  
For each such multi-index $I$, we now define a 
{\em generalized Hopf invariant}
\[
h_I: \widetilde{\pi}_p\left( \bigvee_{m-1} S^{n-1} \right) \longrightarrow \pi_{p-(m-1)(n-1)+m-2}^s
\]
where the target is the stable homotopy groups of spheres.  Given a smooth map $f$ representing an element of the domain, take the preimages under $f$ of $(m-1)$ points, one in each wedge summand.  The result is an $(m-1)$ component link $P_1 \sqcup \dots \sqcup P_{m-1}$ in $S^{p}$.  Find a manifold $Q_1$ bounding $P_1$, then replace $P_1 \sqcup P_{i_2}$ by $Q_1 \cap P_{i_2}$.  Then find a manifold $Q_2$ bounding $Q_1 \cap P_{i_2}$, and replace $Q_1 \cap P_{i_2} \sqcup P_{i_3}$ by $Q_2 \cap P_{i_3}$.  Continue this procedure until $Q_{m-3} \cap P_{i_{m-2}} \sqcup P_{i_{m-1}}$ is replaced by $Q_{m-2} \cap P_{i_{m-1}}$.  Consider the result as a framed bordism class, which by the Pontryagin--Thom construction can be identified with a class in the stable homotopy groups of spheres.  
\end{defin}

The specialization of a result of Koschorke \cite[Theorem 3.1]{Koschorke:1997} to the case of a wedge of equidimensional spheres gives an isomorphism 
\begin{equation}
\label{eq:hopf-invt-iso}
h:=\bigoplus_{I \, \leftrightarrow\,  \sigma \in \Sigma_{m-2}} h_I: \widetilde{\pi}_{p}\left( \bigvee_{m-1} S^{n-1} \right)
\overset{\cong}{\longrightarrow} 
\bigoplus_{(m-2)!}  \pi^s_{p-(m-1)(n-1)+m-2}
\end{equation}
provided that $(m-1)(n-1)-(m-2) \leq p \leq m(n-1)-m$.  
Of particular interest to us, both in Corollary \ref{cor:brunnian-spherical-links} and Theorem \ref{T:braids-as-links}, is the case of the smallest possible value of $p$, namely  
$p=(m-1)(n-1)-m+2$.  
For this $p$, the target is $\bigoplus_{(m-2)!}  \pi_0^s$, and $h=\bigoplus h_I$ gives an isomorphism 
$\widetilde{\pi}_{p}\left( \bigvee_{m-1} S^{n-1} \right) \overset{\cong}{\longrightarrow} \bigoplus_{(m-2)!}  \Z$.
The domain is spanned by iterated Whitehead brackets $[\iota_{1}, \iota_{i_2} \dots \iota_{i_{m-1}}]$, where $\iota_j$ is the projection to the $j$-th wedge summand.
This allows us to explicitly identify it with $Lie(m-1)$.  

\begin{defin}[Koschorke \cite{Koschorke:1997}]
\label{D:milnor-invt}
Let $L: S^{p_1} \sqcup \dots \sqcup S^{p_m} \to \R^n$ be a Brunnian link, with $p:= \sum_{i=1}^m p_i$, and suppose that $p= mn-2m-n+3$.  Define the \emph{Milnor invariant} $\mu_{I;j}(L)$ as the Hopf invariant $h_I$ of $\widetilde{\kappa}(L) \in \widetilde{\pi}_{p}\left( \bigvee_{m-1} S^{n-1} \right)$.
 \end{defin}
 
Combining the results of \cite{Koschorke:1997} and Theorem \ref{thm:theta-trees} yields the following:

\begin{repcor}{cor:brunnian-spherical-links}
	Let $L: S^{p_1} \sqcup \dots \sqcup S^{p_m} \to \R^n$ be a Brunnian link such that $\sum_{i=1}^m p_i= mn-2m-n+3$.  
	Let  $\overline{\kappa}(L)$ be the adjoint of the map $\widetilde{\kappa}(L)$ defined in \eqref{eq:borr-diagram}. 
	Then 
	\[
	\Theta(\overline{\kappa}(L)) =\sum_{I \in \Sigma_{m-2}} \mu_{I;m}(L)\Gamma_{m;I}^*\qquad  \in Lie(m-1)\subset H_*(P\Dm(m)^*),
	\]
	where  $\Gamma_{m;I}^*=\Theta(B_{m;I})$ can be viewed as a trivalent tree.
	\qed
\end{repcor}

  We will use the maps $\kappa$ and $h$ to prove our last main result, Theorem \ref{T:braids-as-links}.  There we will apply them to basepoint-preserving embeddings (or basepoint-preserving link maps, as in Remark \ref{R:variations-of-braids-as-links}) rather than arbitrary link maps.  The definition of $h$ is given just after formula \eqref{eq:hopf-invt-iso}.

\begin{example}
	\label{Ex:B312}
	The basic case of $\Theta([B_{3,1}, B_{3,2}])$ is obtained by a single vertex blow-up of the product
	$\Gamma_{3,1}\cdot \Gamma_{3,2}$, which gives a tripod diagram:
	\begin{equation}\label{eq:G_[B31,B32]}
	(\Gamma_{3,1}\cdot \Gamma_{3,2})\quad \rightsquigarrow\quad \vvcenteredinclude{0.15}{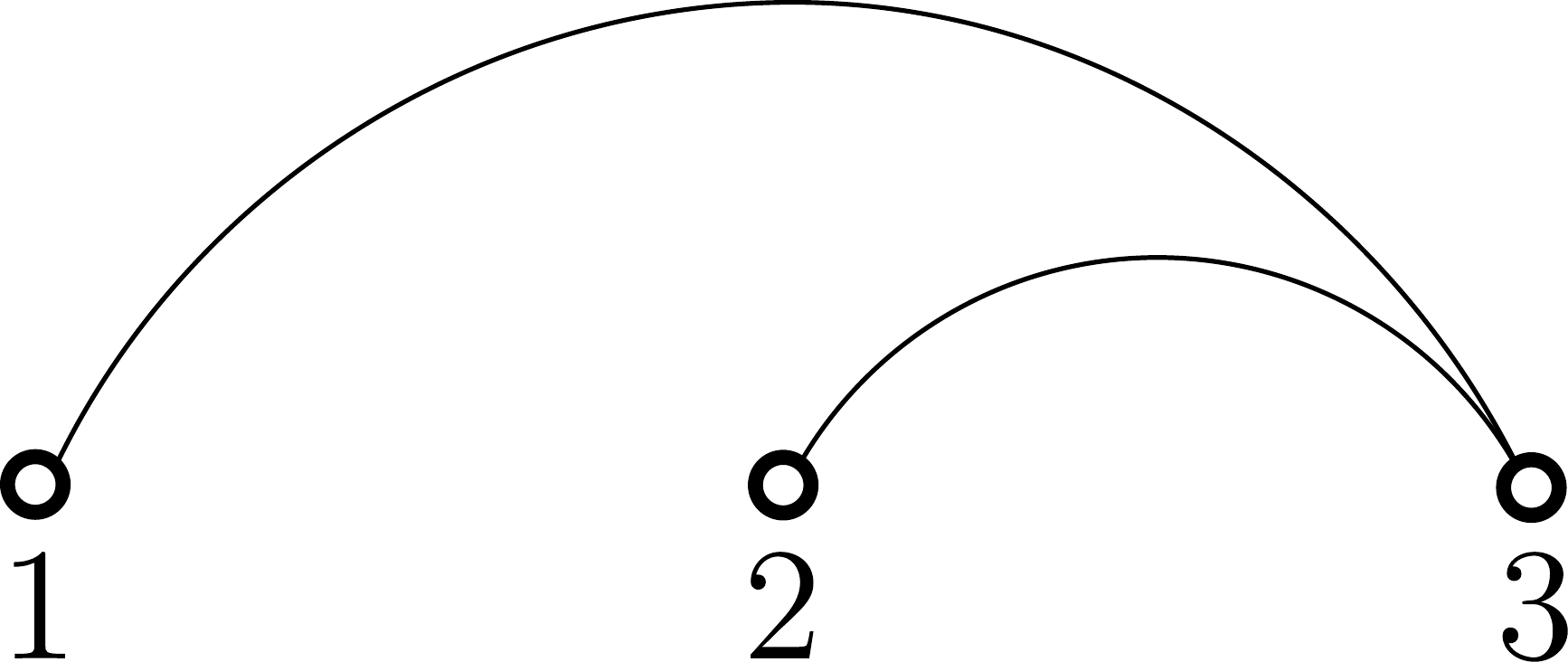}\quad\rightsquigarrow\quad \pm \vvcenteredinclude{0.15}{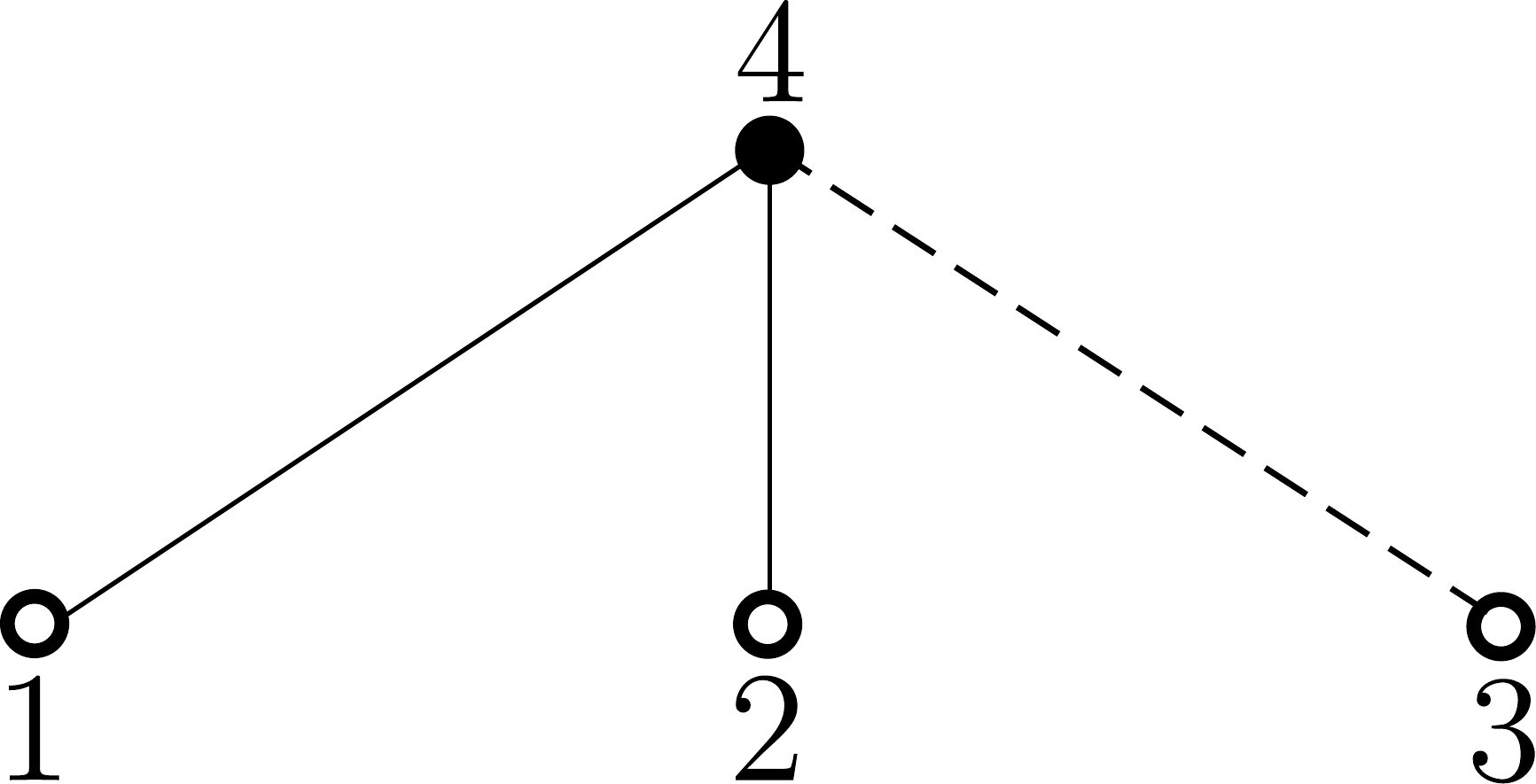}.
	\end{equation}
	The dashed decoration on an edge shows the edge resulting from a blow-up and is not part of the data of the diagram. 
	Since $[B_{3,1}, B_{3,2}]\in Lie(2)\cong\Z$, 
	any Brunnian spherical link map $L:S^{n-2}\sqcup S^{n-2}\sqcup S^{n-2}\longrightarrow \R^n$ satisfies $\overline{\kappa}(L)=\mu_{1,2;3} [B_{3,1}, B_{3,2}]$, in $\pi_{3(n-2)}(\Omega\Conf(3,\R^n))$ for some integer $\mu_{1,2;3}$. By Corollary \ref{cor:brunnian-spherical-links}, we have
	\[
	 \Theta(\overline{\kappa}(L))=1\pm\mu_{1,2;3} \vvcenteredinclude{0.15}{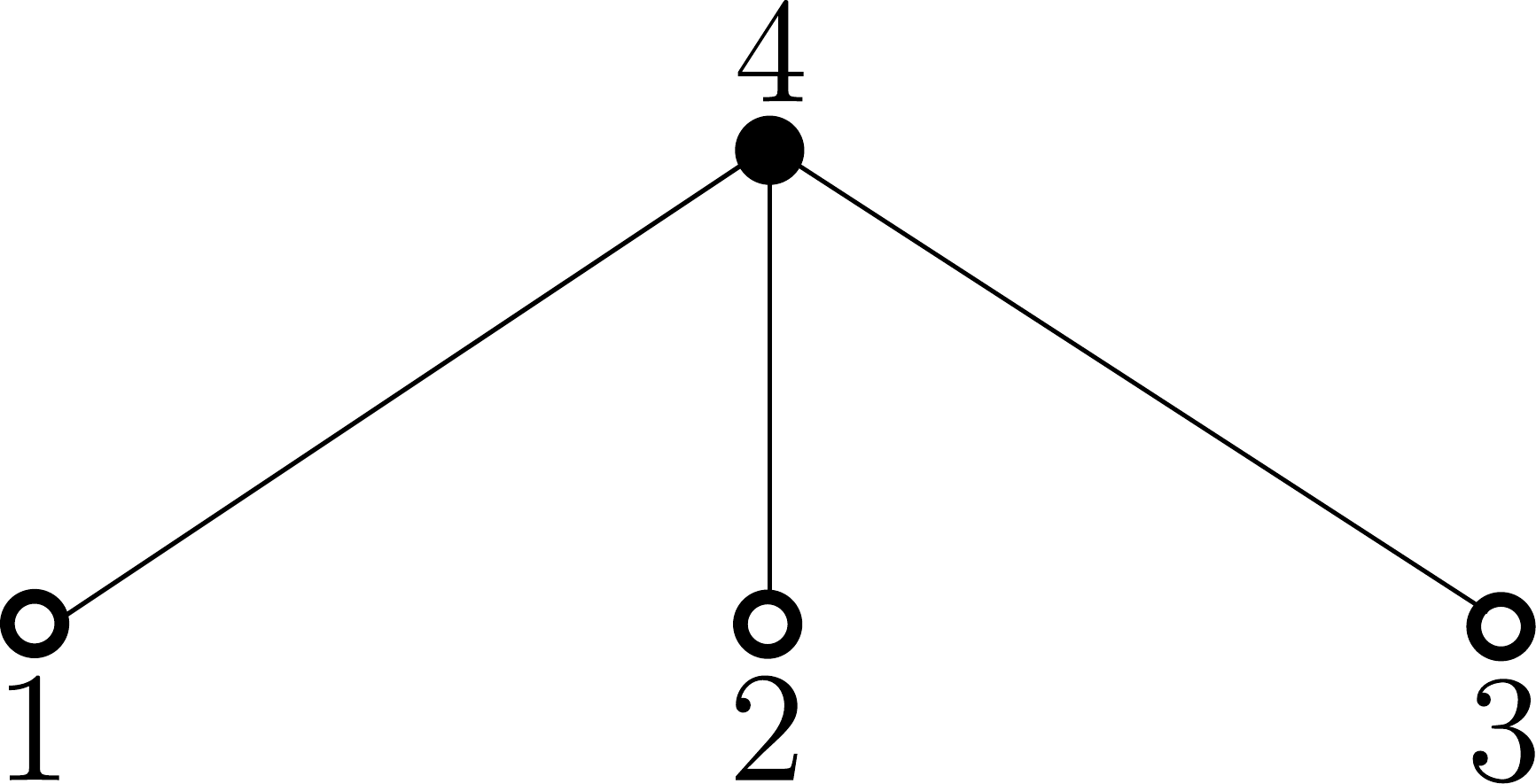}.
	\]
	Therefore, $\mu_{1,2;3}=\mu_{1,2;3}(L)$ is (up to sign) the triple linking number of $L$ (cf.~\cite{DeTurck:2013} for the classical case of $n=3$). Taking this one step further, we may easily generalize this example to any $m$-component Brunnian spherical link $L:S^{n-2}\sqcup \ldots \sqcup S^{n-2}\longrightarrow \R^n$, and by part (d) of Theorem \ref{thm:theta-trees} write 
	\[
	\Theta(\overline{\kappa}(L))=1+\sum_{\sigma\in \Sigma(2,\ldots,m-1)}\pm\mu_{1,\sigma(2),\ldots,\sigma(m-1);m}(L) \vvcenteredinclude{0.15}{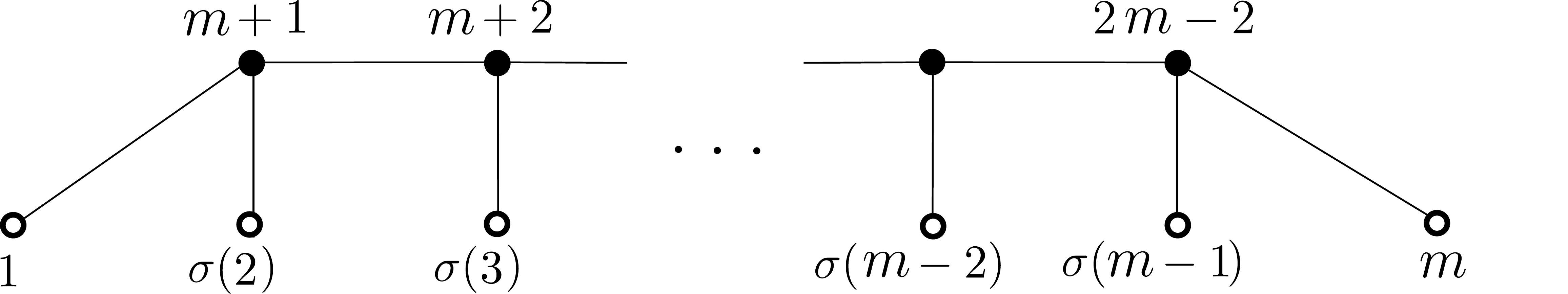},
	\]
	where $\mu_{1,\sigma(2),\ldots,\sigma(m-1);m}(L)$ are the generalized Milnor invariants of Definition \ref{D:milnor-invt}
\end{example}
\begin{example}
\label{Ex:B3121}
As in the previous example we compute $\Theta([ B_{3,1}, B_{3,2}, B_{3,1} ])$ via successive vertex blow-ups of the product
 $\Gamma_{3,1}\cdot \Gamma_{3,2} \cdot \Gamma_{3,1}$, parenthesized as $((\Gamma_{3,1}\cdot \Gamma_{3,2})\cdot \Gamma_{3,1})$, where the dashed decoration on some edges is not part of the data of the diagram, but is simply given to indicate that those edges result from a blow-up:
\begin{equation}\label{eq:G_[B31,B32,B31]}
\begin{split}
((\Gamma_{3,1}\cdot \Gamma_{3,2})\cdot \Gamma_{3,1})\quad & \rightsquigarrow\quad \vvcenteredinclude{0.15}{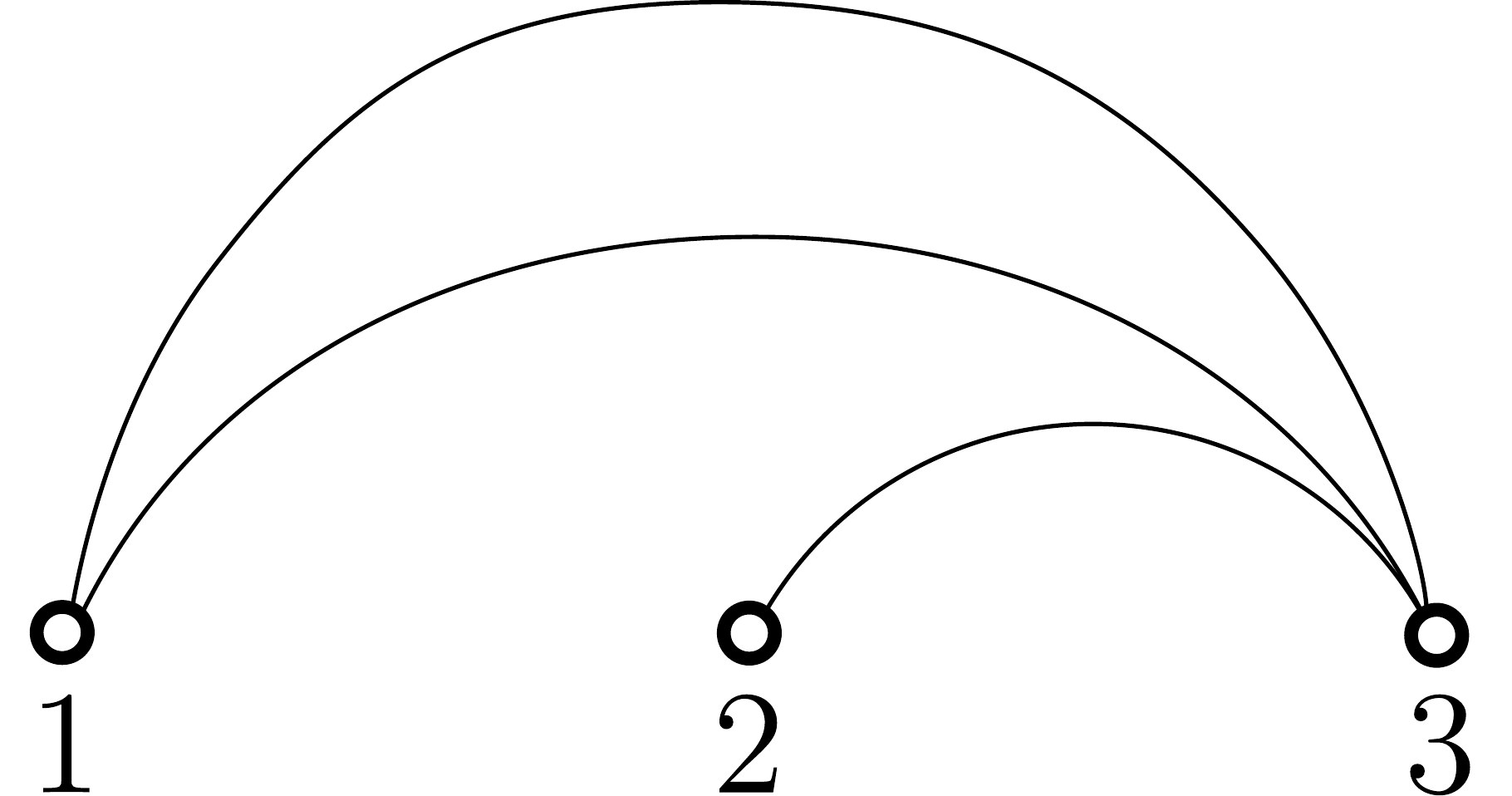}\quad\rightsquigarrow\quad \pm \vvcenteredinclude{0.15}{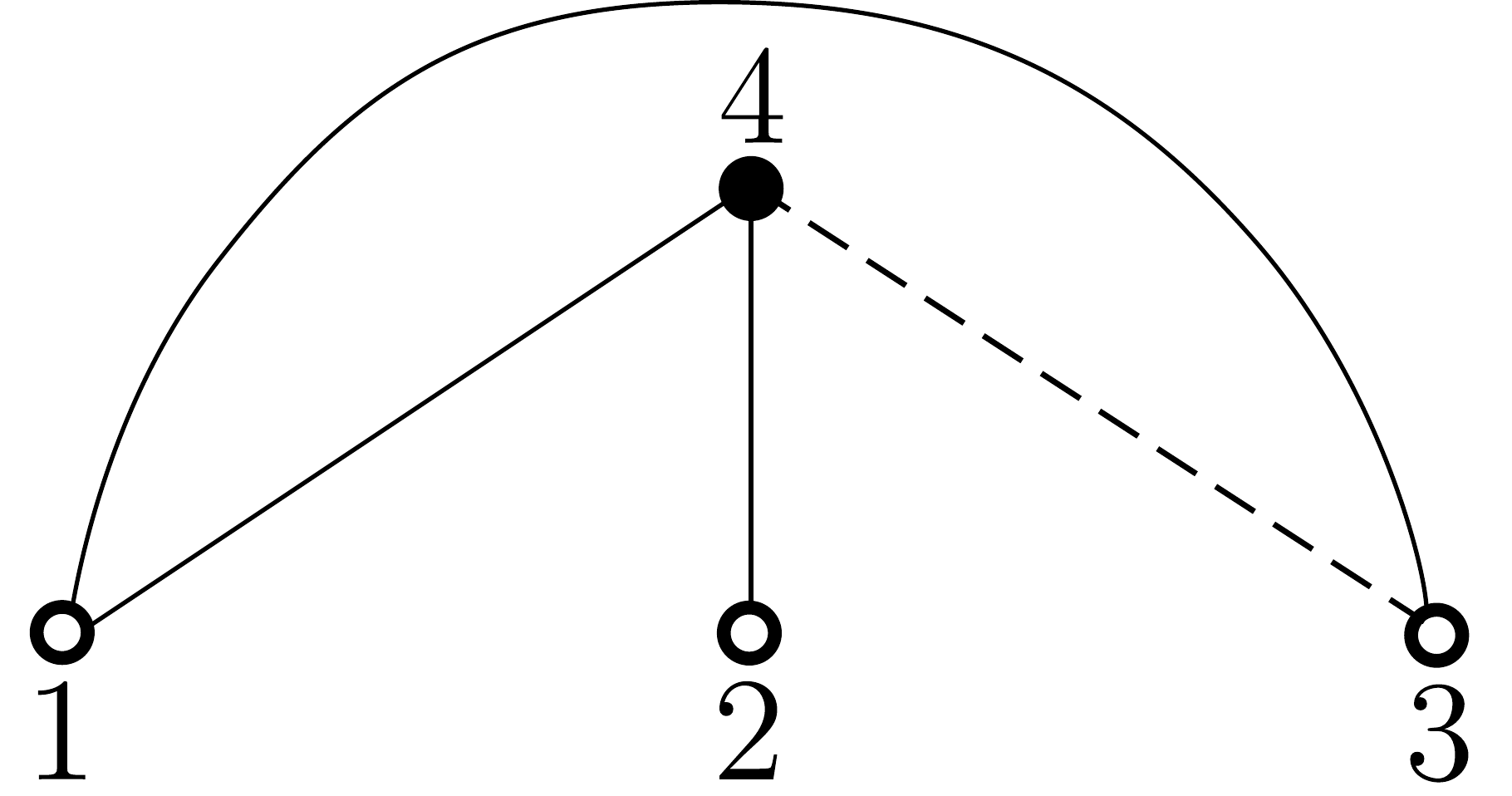}\\
& \rightsquigarrow\quad \pm \vvcenteredinclude{0.13}{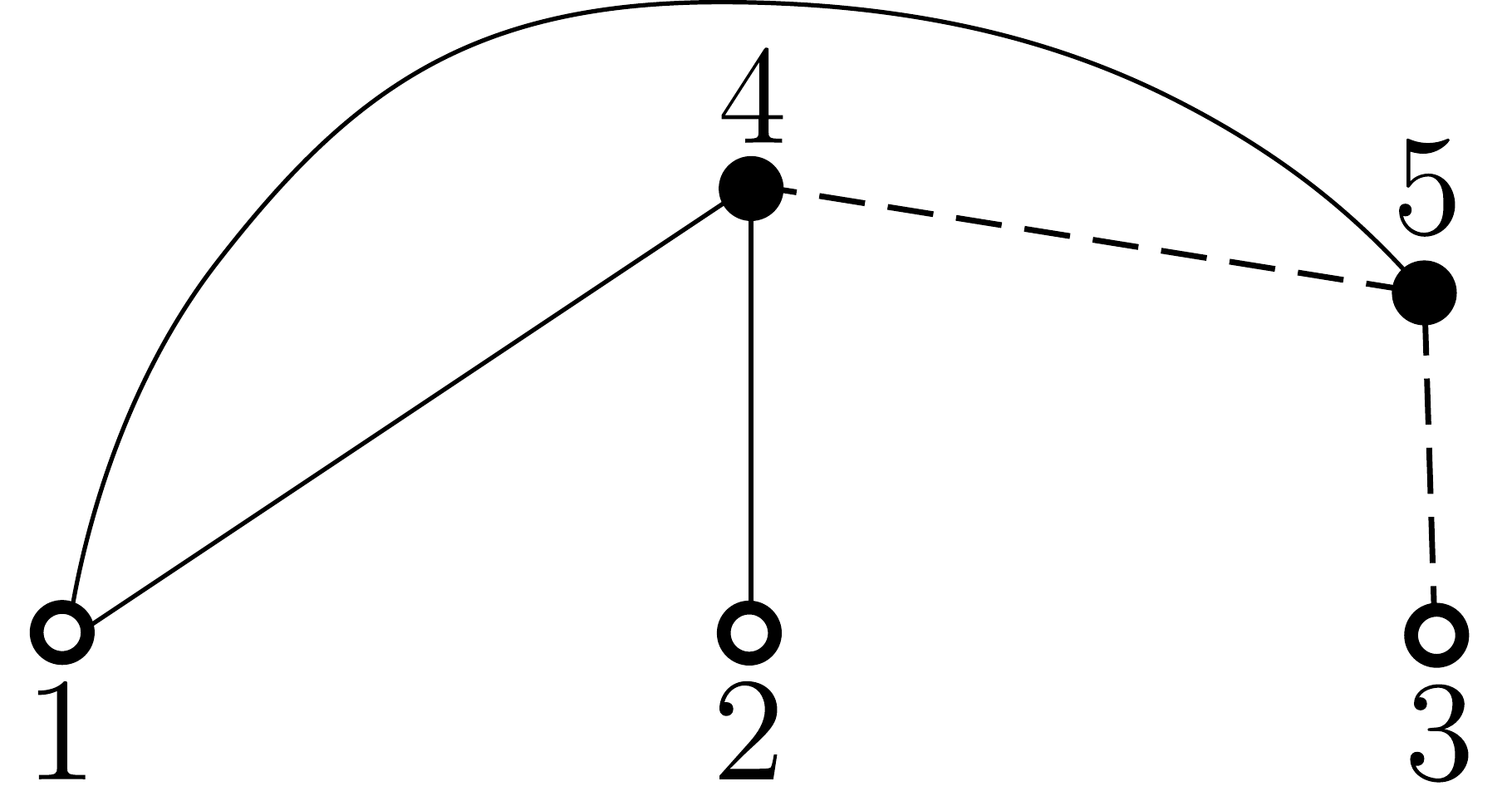}\quad+\quad \pm \vvcenteredinclude{0.15}{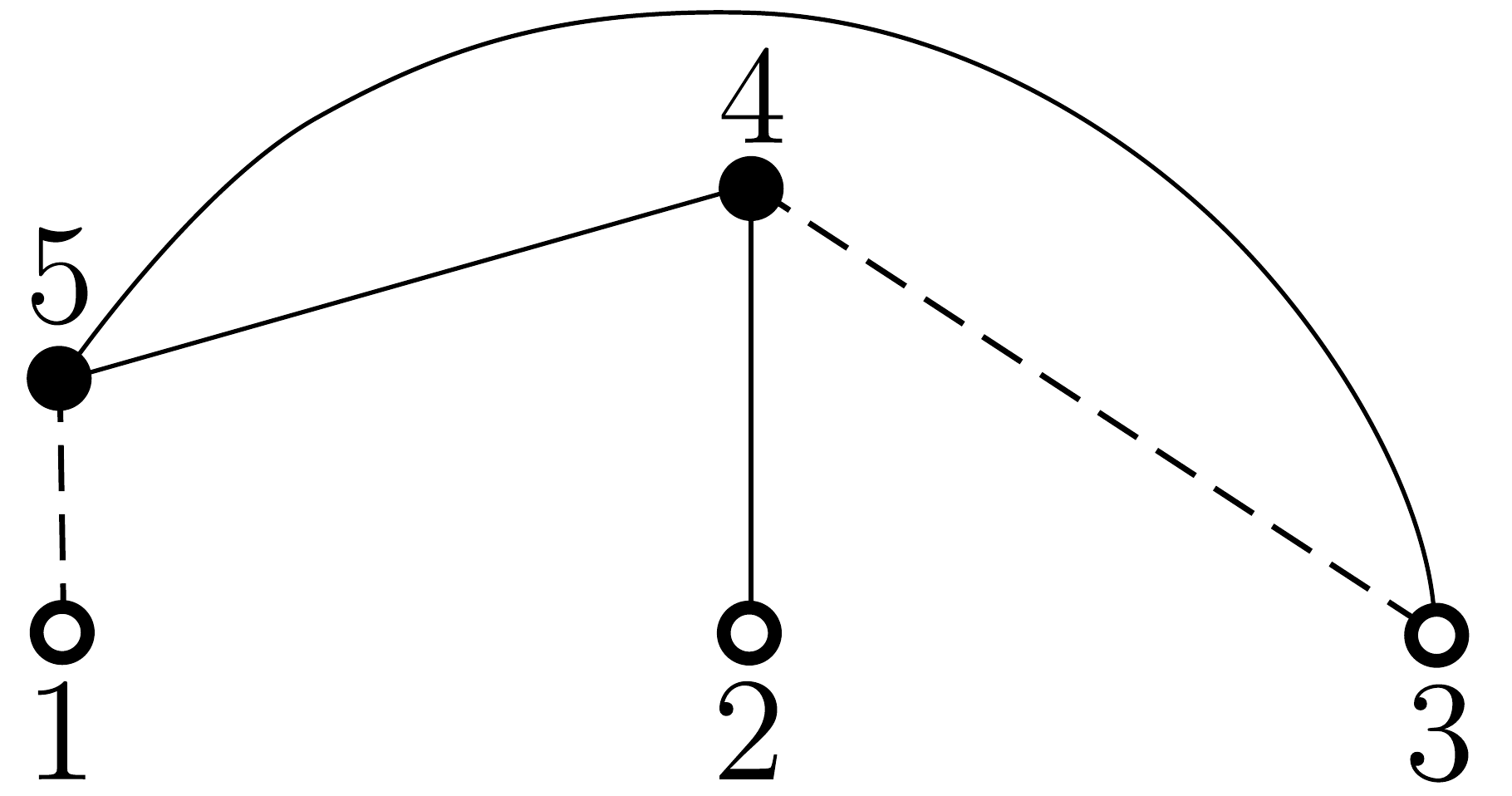} 
\end{split}
\end{equation}
Thus for $B_{3;1,2,1}:=[ B_{3,1}, B_{3,2}, B_{3,1} ]$, we get that $\Theta(B_{3;1,2,1})$ is given by the diagram corresponding to $\pm T(B_{3;1,2,1})$ plus one extra term.  
\end{example}

\begin{example}
\label{Ex:B31222}
Although the second term in the last line of the example above can be obtained from the first term by blowing up segment vertex $1$ and contracting the edge incident to the root $3$, the extra terms become more complicated as the multiplicity of a repeated index increases.  For instance, the successive blow-ups for $\Theta(B_{3;1,2,2,2})$ ultimately yield the six terms shown below after simplifying, not all of which are obtained in this manner.  The labelings and therefore the signs depend on the parity of $n$ and are not shown.
As in the previous example, the dashed decoration on some edges is not part of the data of the diagram, but is simply shown to indicate that those edges result from blow-ups.
\[
\begin{split}
\Theta(B_{3;1,2,2,2}) & =\pm\vvcenteredinclude{0.15}{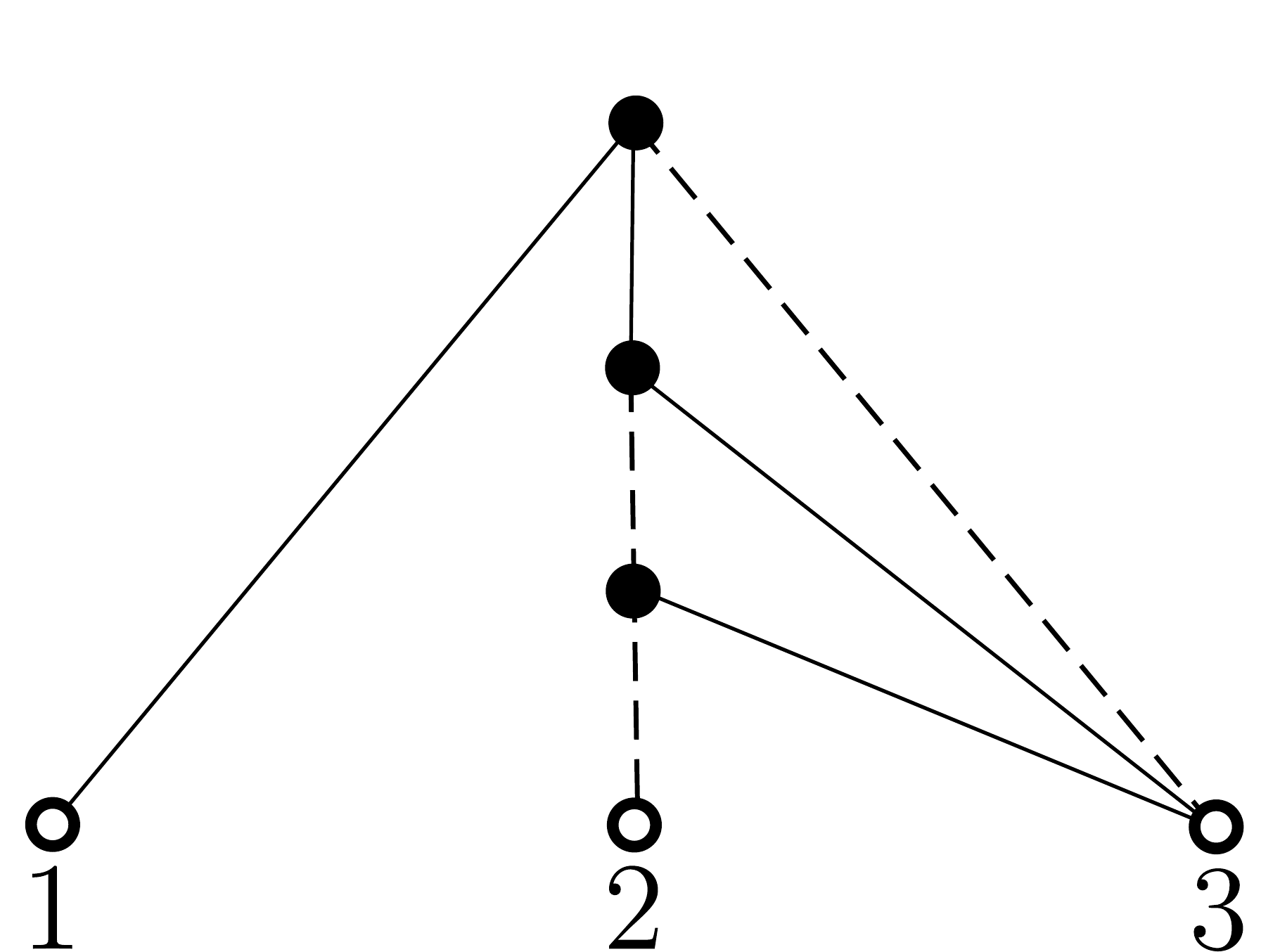}\quad+\quad \pm\vvcenteredinclude{0.15}{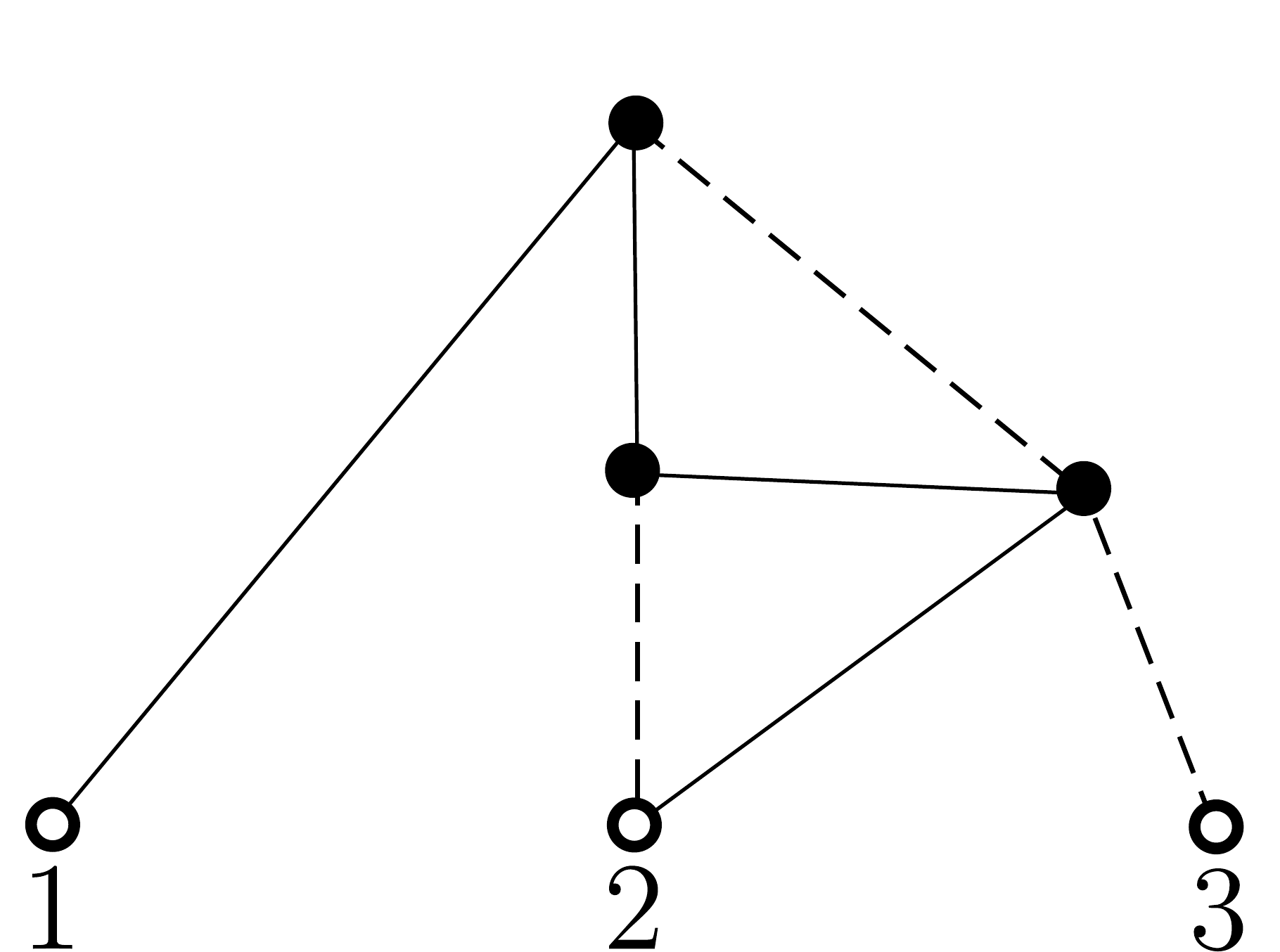}\quad+\quad \pm\vvcenteredinclude{0.15}{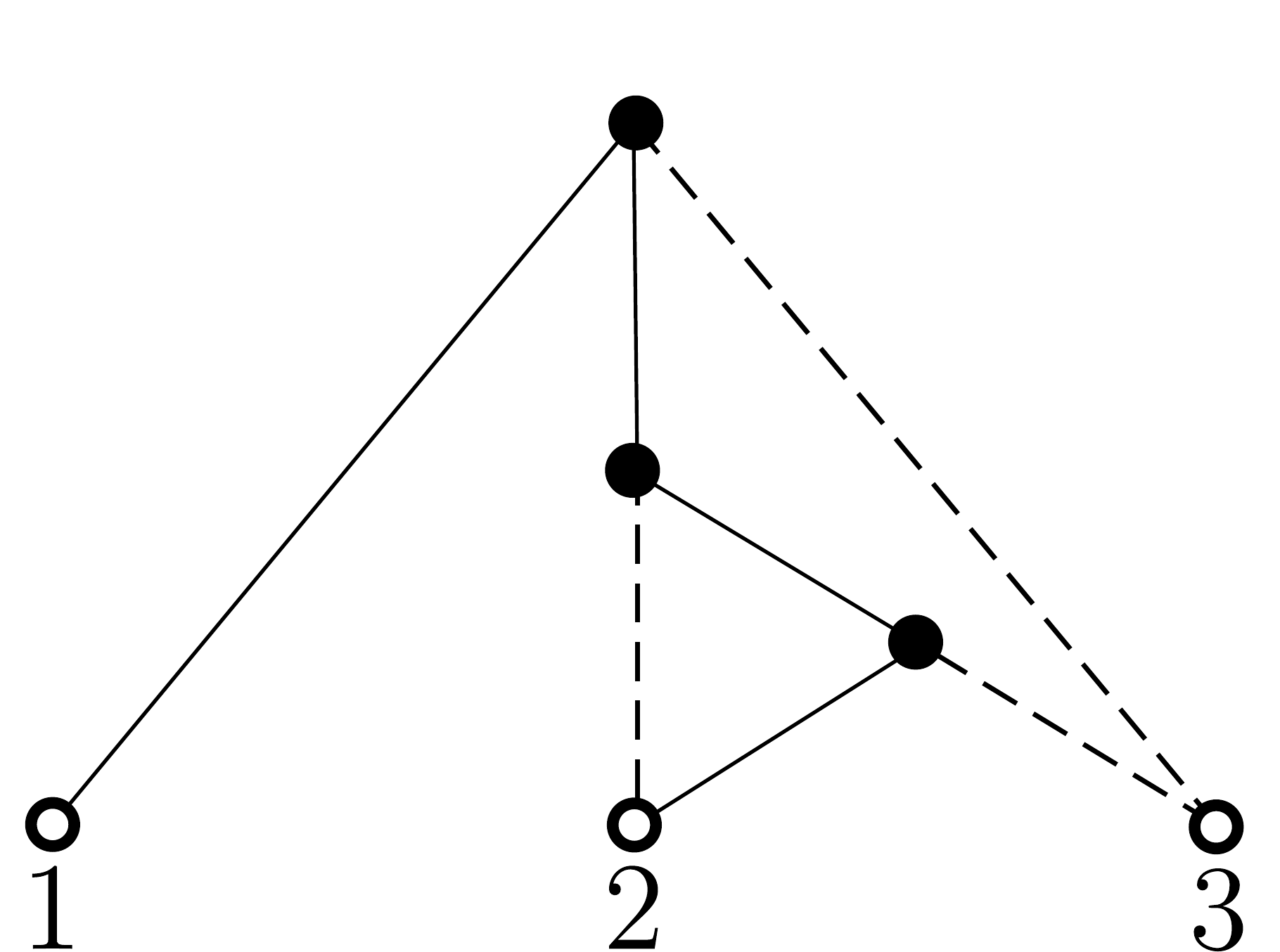}\\
& \quad+\quad \pm\vvcenteredinclude{0.15}{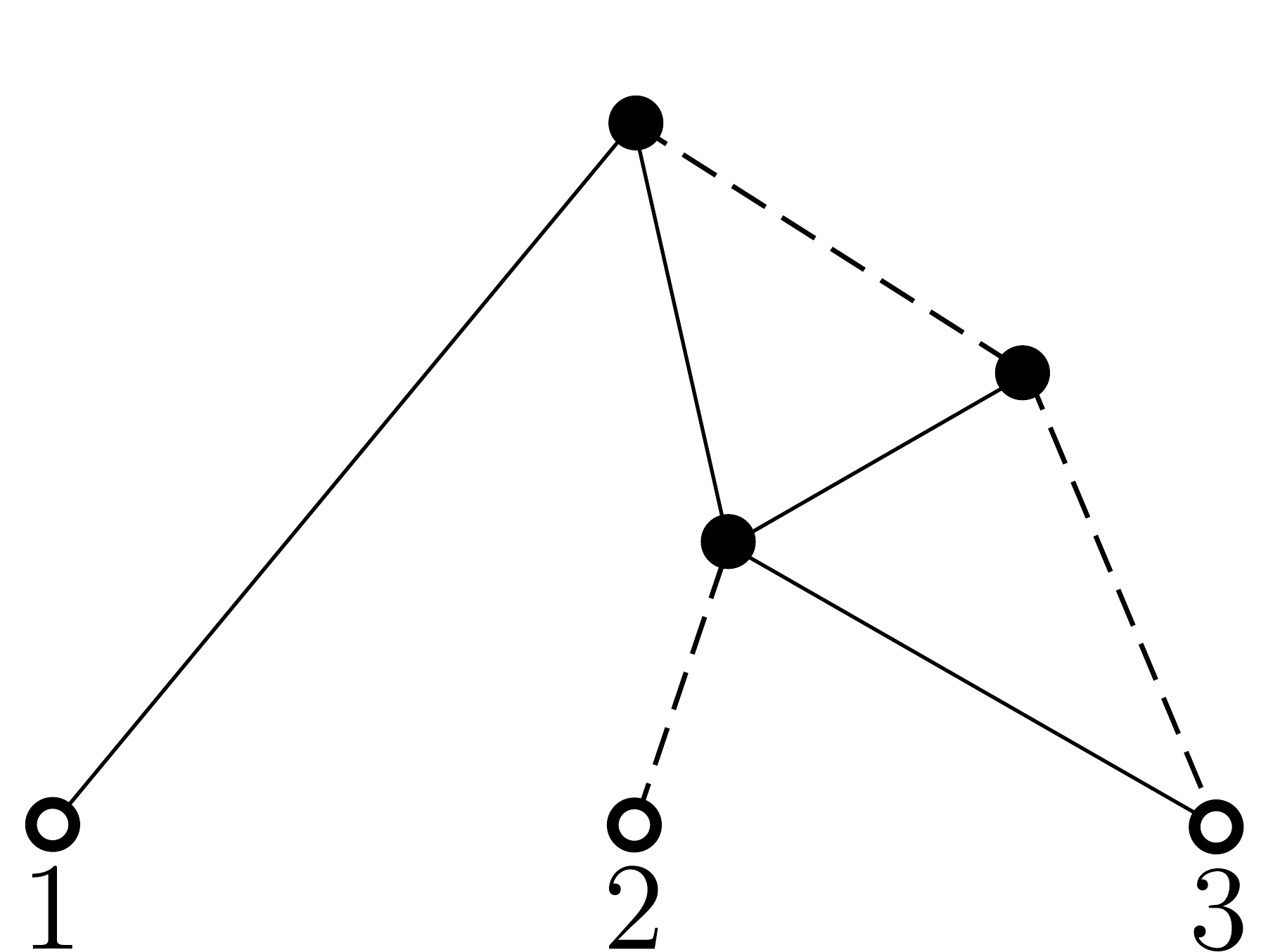}\quad+\quad \pm\vvcenteredinclude{0.15}{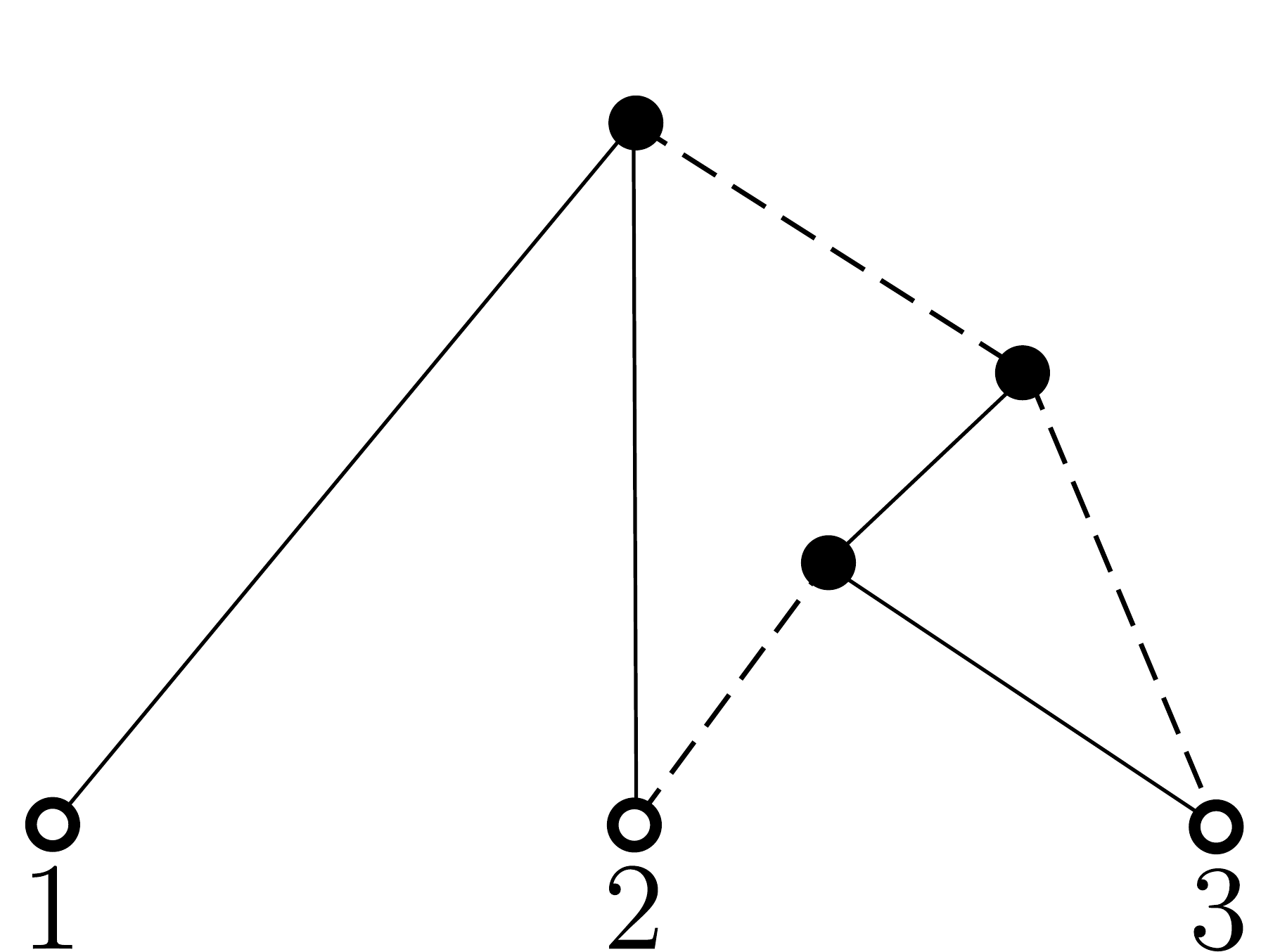}\quad+\quad \pm\vvcenteredinclude{0.15}{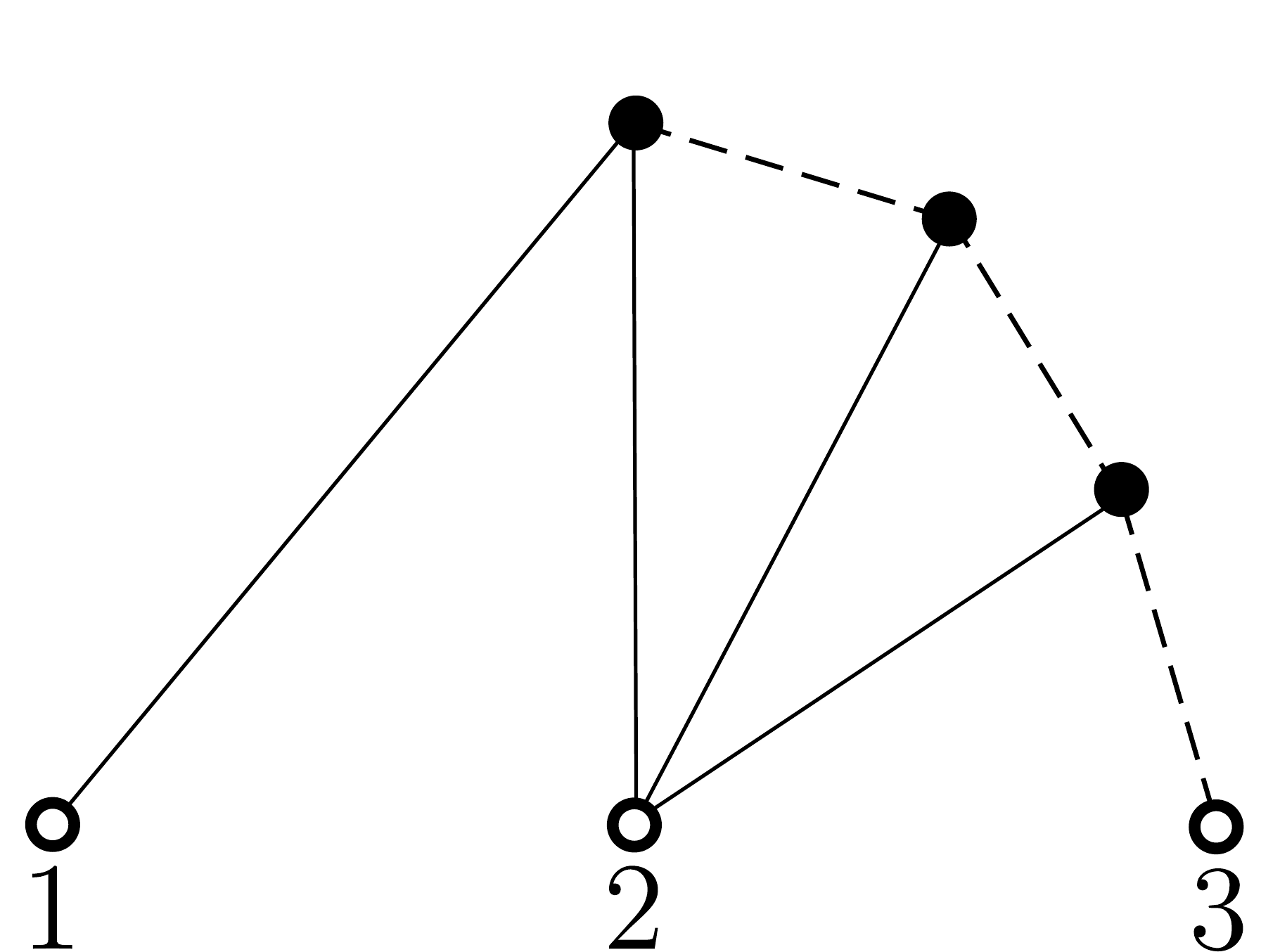}.
\end{split}
\]
The first term corresponds to $\pm T(B_{3;1,2,2,2})$.
The third, fifth, and sixth terms can be viewed as trivalent trees by splitting each multivalent segment vertex into multiple leaves.
The second and fourth terms however are not trees, even after splitting multivalent segment vertices into multiple leaves.
\end{example}

\section{High-dimensional braids as high-dimensional string links}
\label{S:high-dim-braids}
We will now prove Theorem \ref{T:braids-as-links}, which says that via a certain graphing map, any bracket expression without repeated indices in $\pi_*(\Conf(m,\R^n))$ gives rise to a nontrivial class in spaces of $k$-dimensional string links in $\R^{n+k}$ for many values of  $k \geq 1$ and $n\geq 3$.  By Theorem \ref{thm:theta-trees}, we may identify such classes with trivalent trees with distinctly labeled leaves.  The construction of these classes of string links is fairly explicit, via Whitehead products of the generators in Remark \ref{rem:B_ji} and the graphing map.

For $k=1$ and any $n\geq 2$, the result actually holds for all real homotopy classes by classical results \cite{Artin:Braids} and our previous work \cite{KKV:2020} .   
This motivates Conjecture \ref{conj:graphing-injects-homotopy-Q-homology}, which says that Theorem \ref{T:braids-as-links} generalizes to arbitrary bracket expressions in rational homotopy.  That is, we expect that it generalizes to brackets with repeated indices or equivalently, trees with repeated leaf labels.

Section \ref{S:proof-braids-as-links} mainly contains the proof of Theorem \ref{T:braids-as-links}, while Section \ref{S:ex-conj-discussion} contains some examples, Conjecture \ref{conj:graphing-injects-homotopy-Q-homology}, and related discussion.
Throughout these Sections, we will use iterated Whitehead products of the classes $b_{j,i} \in \pi_{n-1}(\Conf(m,\R^n))$ adjoint to $B_{j,i} \in \pi_{n-2}(\Omega \Conf(m,\R^n))$.  We will simply write $[ \ , \ ]$ instead of $[ \ , \ ]_W$ for the Whitehead product, which should cause no confusion, since we will not use the Samelson product in these Sections.
\begin{figure}[ht]
	\centering
	\includegraphics[width=0.4\linewidth]{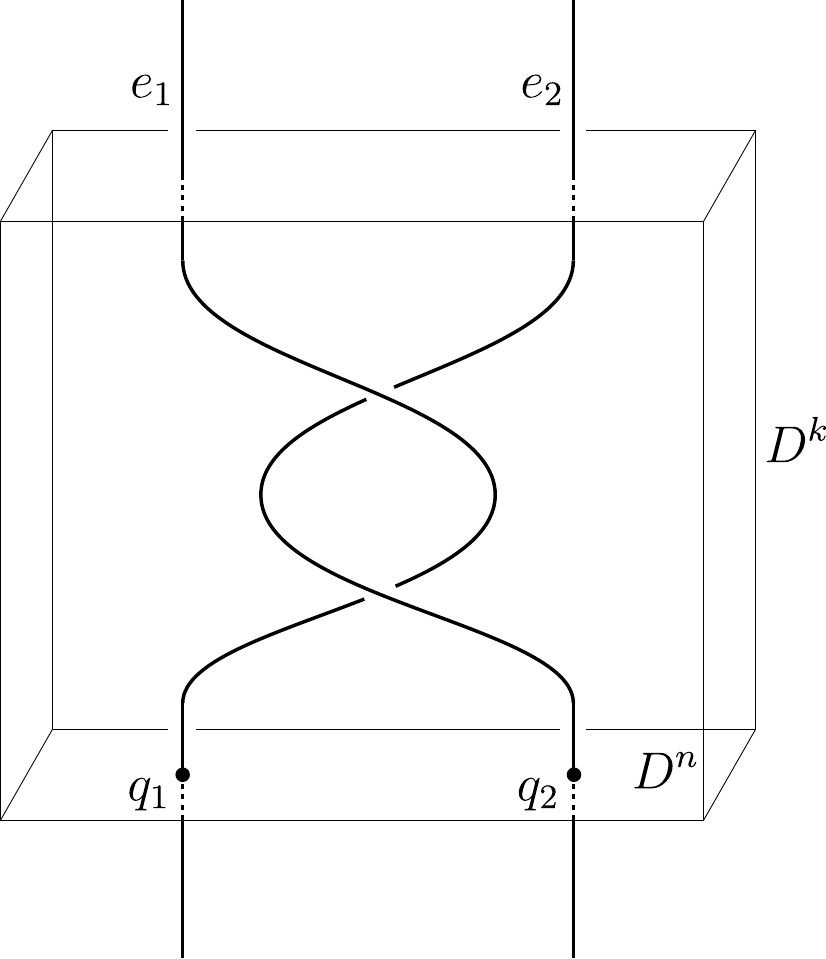}
	\caption{A long embedding in $\Emb_c\left(\coprod_m \R^k, \ \R^{n+k}\right)$ that is the graph $G(f)$ of a braid $f$ in $\Omega^k \Conf(m, \R^n)$, for $m=2$, $k=1$, $n=2$.}
	\label{fig:long-embedding}
\end{figure}
\subsection{Proof of Theorem \ref{T:braids-as-links} (string links by graphing braids)}
\label{S:proof-braids-as-links}
Let $k\geq1$ and $n\geq 2$, and let $\Emb_c\left(\coprod_m \R^k, \ \R^{n+k}\right)$ be the space of string links, namely smooth embeddings $g$ of $m$ disjoint copies of $\R^k$ into $\R^{n+k}$ with fixed behavior outside a compact set.  More precisely, let $D^k$
be the unit disk in $\R^k$.  
Then 
for each $i=1,\dots, m$, 
we require that at any $\vec{t}=(t_1, \dots, t_k)$ outside the interior of $D^k$, the $i$-th component $g_i$ and all its partial derivatives of all orders agree with those of the embedding 
\[
e_i:(t_1, \dots, t_{k}) \longmapsto (q_i,0, \dots,0, t_1, \dots, t_k),
\]
where $q_i = -1 + \frac{2i}{m+1}$ so that $-1, q_1, q_2, \dots, q_m, 1$ are evenly spaced points in $[-1,1]$ (see Figure \ref{fig:long-embedding}).

We will now define a map $\Omega^k \Conf(m,\R^n) \to \Emb_c\left(\coprod_m \R^k, \ \R^{n+k}\right)$ from $k$-dimensional braids in $\R^n$ to $k$-dimensional string links in $\R^{n+k}$.
For each $i=1,\dots, m$, let $p_i \colon \Conf(m,\R^n) \to \R^n$ be the map which remembers only the $i$-th configuration point.
Write an element $f \in \Omega^k \Conf(m,\R^n)$ as $f=(f_1,\dots, f_m)$ where  $f_i = p_i \circ f : D^k \to \R^n$.
Write an element $g \in \Emb_c\left(\coprod_m \R^k,\  \R^{n+k}\right)$ as $g=(g_1, \dots, g_m)$ where $g_i \in \Emb_c(\R^k, \, \R^{n+k})$.  (We are just using the universal properties of the product $(\R^n)^m \supset \Conf(m,\R^n)$ and the coproduct $\coprod_m \R^k$.) 
Define the \emph{graphing map}
\begin{align*}
G: \Omega^k \Conf(m,\R^n) &\longrightarrow \Emb_c \left(\coprod_m \R^k, \ \R^{n+k}\right) \\
f &\longmapsto G(f)
\end{align*}
by defining, for each $i=1, \dots, m$ the component $G(f)_i$ as
\[
(G(f)_i)(t_1, \dots, t_k) := (f_i(t_1, \dots, t_k), t_1, \dots, t_k) \in\R^{n+k}.
\]
\begin{repthm}{T:braids-as-links}
Let $m\geq 2$, $k\geq 1$, and $n\geq 3$.
\begin{itemize}
\item[(a)]
View $Lie(m-1)$ as the submodule of $\pi_{(m-1)(n-2)+1} (\Conf(m,\R^n))$ spanned by length-$(m-1)$ iterated Whitehead brackets in which each generator $b_{m,1}, \dots, b_{m,m-1}$ appears exactly once.  
Let $\ell=mn-2m-n-k+3$.  If $\ell \geq 0$, then the following composition is injective:
\[
Lie(m-1) \hookrightarrow \pi_{\ell+k} \Conf(m,\R^n) \overset{\cong}{\to} \pi_\ell \Omega^k \Conf(m,\R^n) \overset{G_*}{\longrightarrow} \pi_\ell \Emb_c \left(\coprod_m \R^k, \ \R^{n+k}\right)
\]
\item[(b)]
Let $\mathcal{H}_{m}^{n,k}$ be the submodule of $\pi_{*} (\Conf(m,\R^n))$ spanned by iterated Whitehead brackets  
on distinct generators $ b_{j,i_i}, \dots, b_{j,i_p}$ where $2\leq j \leq m$ and 
where $p (n-2) \geq {n+k-3}$.
Then the following composition is injective:
\[
\mathcal{H}_{m}^{n,k} \hookrightarrow \pi_{\ast+k} \Conf(m,\R^n) \overset{\cong}{\to} \pi_{\ast} \Omega^k \Conf(m,\R^n) \overset{G_*}{\longrightarrow} \pi_{\ast} \Emb_c \left(\coprod_m \R^k, \ \R^{n+k}\right).
\]
\end{itemize}
\end{repthm}

Part (b) generalizes and will follow rather quickly from part (a).
 The inequality involving $p,k,$ and $n$ in part (b) ensures precisely that the index $*$ is nonnegative.
If $k=1$ and $n\geq 2$, then that inequality is satisfied for all $p\geq 1$.

In the classical case where $k=1$ and $n=2$, the domain of ${G}_*$ is nonzero only for $\ell=0$, where it is the pure braid group $\mathcal{PB}_m$.  In this case, the map ${G}_*$ is injective on all homotopy classes by work of Artin \cite{Artin:Braids}.  
This setting motivates the notation $\mathcal{H}_{m}^{n,k}$, since $\mathcal{H}_{m}^{2,1}$ is a subspace of the associated graded Lie algebra of $\mathcal{PB}_m$ that is dual to Milnor {homotopy} invariants of string links or equivalently, to additive Vassiliev homotopy invariants of string links; cf.~Remark \ref{R:classical-lie-alg}.  

If $k=1$ and $n\geq 3$, then $G_*$ is injective on all real homotopy classes.  This holds because $G$ is surjective on real cohomology \cite[Corollary 5.21]{KKV:2020}, hence injective on real homology, and because the real homotopy of a loop space is a subspace of its real homology.  

The proof of the theorem uses two key maps, the first of which we now define.

\begin{defin} 
For any $n\geq 2$ and $k\geq 1$, we will construct a \emph{closure} map 
\begin{align*}
\widehat{\cdot} \ : \Emb_c \left(\coprod_m \R^k, \ \R^{n+k}\right) &\longrightarrow 
\Emb_{*} \left(\coprod_m S^k, \ \R^{n+k}\right)\\
g &\longmapsto \widehat{g}
\end{align*}
 where the subscript $*$ on the right-hand side indicates embeddings that take a prescribed value at a basepoint in each component $S^k$.
Since $n \geq 2$, we can first fix an embedding of $D^k$ into $[q_i-\frac{1}{2m}, q_i+\frac{1}{2m}] \x \mathbb{\R}^{n+k-1} - D^n \x D^k$; moreover, we prescribe the behavior on a collar of $\del D^k$ so that this embedding together with $g_i|_{D^k}$ can be glued together to give a smooth embedding $\widehat{g}$ of $S^k=D^k \cup D^k$ for any $g \in \Emb_c \left(\coprod_m \R^k, \ \R^{n+k}\right)$.  
For the basepoint in each component $S^k$, we choose any point in the equator $S^{k-1}$ along which the two copies of $D^k$ are glued.  
\end{defin}

The second key map is 
\begin{align*}
\kappa: \Emb_{*} \left(\coprod_m S^k, \ \R^{n+k}\right) &\longrightarrow \Map_* ((S^k)^{\x m}, \Conf(m, \R^{n+k})) \\
g=(g_1,\dots, g_m) &\longmapsto (\kappa(g):(s_1, \dots, s_m) \mapsto (g_1(s_1), \dots, g_m(s_m))
\end{align*}
previously shown in \eqref{eq:kappa(L)}.
It evaluates $g$ at all configurations of $m$ points in $\coprod_m S^k$ such that one point lies on each component.

\begin{defin}
Define the space of \emph{Brunnian long links} $\BrEmb_c \left(\coprod_m \R^k, \ \R^{n+k}\right)$ 
as the subspace of long links $g$ in $\Emb_c \left(\coprod_m \R^k, \ \R^{n+k}\right)$ 
such that the restriction of $g$ to any $m-1$ of its components is isotopic to a trivial link, meaning a sublink of $e=(e_1,\dots, e_m)$.  
Similarly define the space of \emph{Brunnian links} $\BrEmb_{*} \left(\coprod_m S^k, \ \R^{n+k}\right)$ 
as the subspace of links in $\Emb_{*} \left(\coprod_m S^k, \ \R^{n+k}\right)$ whose restrictions to any $m-1$ components are trivial.
\end{defin}

Thus the Brunnian (long) links form a union of path components of the space of all (long) links.  In many of the cases we consider, the latter space is path-connected, in which case all embeddings are Brunnian.  
The key reason for introducing the subspace of Brunnian links is that the restriction of $\kappa$ to it descends to a map
\begin{equation}
\label{Eq:kappaBrunnian}
\kappa: \mathrm{BrEmb}_* \left(\coprod_1^m S^k, \, \R^{n+k}\right)  \to 
\Map_*\left( S^{mk}, \, \Conf(m, \R^{n+k}) \right).
\end{equation}
Indeed, by the Brunnian property, the restriction of $\kappa(g)$ to each $k(m-1)$-dimensional cell given by fixing the basepoint in one coordinate is nullhomotopic.  
The proof of part (a) of Theorem \ref{T:braids-as-links} requires replacing long links by Brunnian long links, but this is sufficient because the latter are a union of path components of the former.

\begin{proof}[Proof of Theorem \ref{T:braids-as-links}, part (a)]
Let $\widehat{G}$ denote the composition of $G$ followed by the closure map.  That is, $\widehat{G}(f):=G(\widehat{f})$.
%
%
Suppose $b \in \pi_{\ell+k}\Conf(m,\R^n)$ comes from $Lie(m-1)$, 
where we will use the same symbol $b$ for a representative map and its homotopy class.
Then $\ell +k = mn-2m-n+3=(m-1)(n-2)+1$ and $\ell\geq0$ (as in the Theorem statement), so we can write $b \in \pi_\ell\Omega^k\Conf(m,\R^n)$.
Notice that $\widehat{G}(b)$ is Brunnian, i.e., forgetting any of the $m$ components yields the trivial element in $\pi_\ell$.  
Thus $\kappa\circ \widehat{G}(b)$ lies in the codomain of \eqref{Eq:kappaBrunnian}
and represents an element of $\pi_{\ell+mk}\Conf(m, \R^{n+k})$.

We will check that $\kappa\circ \widehat{G}(b)$ is represented by the same bracket expression as $b$, but with the dimensions of the generators shifted up by $k$.  
To distinguish between spherical generators in configuration spaces of points in Euclidean spaces of different dimensions, write $b_{j,i}^n$ for such a generator in $\pi_{n-1}\Conf(m, \R^n)$ (see Remark \ref{rem:B_ji}).
The restriction of $\kappa_* \circ \widehat{G}_*$ to $Lie(m-1)$ induces a homomorphism of abelian groups $Lie(m-1)\to \pi_{\ell+mk}\Conf(m, \R^{n+k})$.
Therefore, it suffices to check that if $b=b_{m;I} = [ b_{m,i_1}^n, \dots, b_{m,i_{m-1}}^n ]$, then
\begin{equation}
\label{eq:brackets-preserved}
\kappa_*\circ \widehat{G}_*(b) =\pm [ b_{m,i_1}^{n+k}, \dots, b_{m,i_{m-1}}^{n+k} ].
\end{equation}  
Since changing signs of basis elements always produces a basis, it suffices to check \eqref{eq:brackets-preserved} up to a sign.

To prove \eqref{eq:brackets-preserved}, we will use Koschorke's generalized Hopf invariants $h_I$ and in particular their relation to Whitehead products.
Recall from Definition \ref{D:hopf-invt} that for any multi-index $I=(1,i_2,\dots, i_{m-1})$, $h_I$ is an invariant of classes in the reduced homotopy groups of a wedge of $(m-1)$ spheres.
We consider these invariants for $\bigvee_{m-1} S^d$, viewed as a subspace of $\Conf(m,\R^{d+1})$ (i.e.~the fiber of the projection $\Conf(m,\R^{d+1})\longrightarrow \Conf(m-1,\R^{d+1})$ forgetting the $m$-th configuration point), for two different values of $d$, namely $d=n-1$ or $d=n+k-1$.
We will therefore sometimes write $h_I^d$ to indicate this value.  

The main point is to establish that each $h_J^{n-k+1}$ of $(\kappa \circ \widehat{G})(b_{m;I})$ coincides with $h_J^{n-1}(b_{m;I})$ for all $J$.  Recall also from Definition \ref{D:hopf-invt} that $h_J(f)$ is a framed bordism class obtained from an $(m-1)$-component link associated to $f$, 
in an iterative manner determined by $J$.
One can use fairly standard link components and bounding manifolds to compute that $h^{n-1}_J (b_{m;I})$ is the Kronecker delta $\pm \delta_{JI}$, where $\pm 1 = \pm[ \{\text{point}\}]$ and $0=[\varnothing]$.  The punchline is that all the manifolds involved in computing $h^{n+k-1}_J$ of $\kappa_*\circ \widehat{G}_*(b_{m;I})$ are essentially products of these manifolds with various numbers of factors of $D^k$, but the iterated intersection is still either a point or empty.  
The lengthier details of explicitly checking this follow.  They are based on the ideas in Koschorke's proof of \cite[Theorem 3.1]{Koschorke:1997}.
We sacrifice some efficiency for the sake of clarity, by first explaining the cases $m=2$ and $m=3$.  We then cover the case of arbitrary $m$, at which point part (a) will be proven.

{\bf The case $m=2$:}  Here $\ell=n-k-1$, and the composition of the last two maps below is the identity:
\begin{equation}
\label{eq:composite,m=2,n}
\xymatrix@R4pt{
D^{\ell} \x D^k \ar@{->>}[r]^-{q} & S^{n-1}  \ar[r]^-{b^n_{2,1}} & \Conf(2,\R^n) \ar[r]^-{\phi_{2,1}^n} & S^{n-1}\\
(s,t) \ar@{|->}[rr]^-{b_{2,1}^n\circ q} & & (b_{2,1}(s,t))
}
\end{equation}
Above, $\phi_{j,i}^n$ is the Gauss map $\phi_{j,i}$ defined in \eqref{eq:phi_ji}, with the superscript $n$ added to indicate the dimension of the Euclidean space, just as for $b_{j,i}^n$. 
As shown above, quotients by the boundary of a disk such as $q$ are sometimes omitted from the formulas.
We want to verify that the composition of the last two maps below has degree $\pm 1$:  
\begin{equation}
\label{eq:composite,m=2,n+k}
\xymatrix@C2pc@R4pt{
D^{n-k-1} \x D^k \x D^k \ar@{->>}[r]^-{q'} &S^{n+k-1}  \ar[r]^-{(\kappa \circ \widehat{G})(b^n_{2,1})} & \Conf(2,\R^{n+k}) \ar[r]^-{\phi_{2,1}^{n+k}} & S^{n+k-1}.\\
(s,t,u) \ar@{|->}[rr]^-{(\kappa \circ \widehat{G})(b^n_{2,1}) \circ q'} && \bigl( ((b_{2,1}^n)_1(s,t), t), \ ((b_{2,1}^n)_2(s,u),u) \bigr) &
}
\end{equation}
Here $(b_{2,1}^n)_i$ denotes the $i$-th configuration point of $b_{2,1}^n$.  
This degree corresponds to the invariant $h_1=h^{n+k-1}_1$.

Let $p^n \in S^{n-1}$ be a regular value of the composite \eqref{eq:composite,m=2,n} such that the basepoint induced by the quotient $q$ is not in its preimage.  
Define $p^{n+k}:=p^n \x (0,\dots,0) \in S^{n+k-1} (\subset \R^{n+k})$.  
From the formula in \eqref{eq:composite,m=2,n+k}, we can see that $p^{n+k}$ is a regular value of this composite and that the basepoint induced by $q'$ does not map to it. 
Moreover, we use that formula to determine the pre-image of $p^{n+k}$ under 
\eqref{eq:composite,m=2,n+k}, 
by solving $\bigl(((b_{2,1}^n)_1(s,t), t)-((b_{2,1}^n)_2(s,u),u))\bigr) 
\parallel
p^{n+k}$.
The result is the single point $(s_0,t_0,t_0)$ such that $\phi_{2,1}^{n} \circ b_{2,1}^{n} \circ q(s_0,t_0) (=q(s_0,t_0))=p^n$.
So the composition of the last two maps has degree $\pm 1$,  $(\kappa \circ \widehat{G})(b^n_{2,1}) = \pm b^{n+k}_{2,1}$, and  $h_1(\phi_{2,1}^{n+k}\circ ((\kappa \circ \widehat{G})(b^n_{2,1})))= \pm 1$.

{\bf The case $m=3$:} Here $\ell=2n-k-3$.  The first relevant diagram is 
\begin{equation}
\label{eq:composite,m=3,n}
\xymatrix@C3pc@R0.5pc{
D^{\ell} \x D^k \ar@{->>}[r]^-q & S^{2n-3} \ar^-{[\iota_1^{n-1}, \,\iota_2^{n-1}]}[r] &  S^{n-1} \vee S^{n-1} \ar[r]^-{b^n_{3,1} \vee b^n_{3,2}} & \Conf(3,\R^n)
\ar[r]^-{\phi^n_{3,1}} \ar[dr]^-{\phi^n_{3,2}} & S^{n-1} \\
(s,t) \ar@{|->}[rrr]^-{(b^n_{3,1} \vee b^n_{3,2})\circ [\iota_1^{n-1}, \,\iota_2^{n-1}]\circ q} &&& \left([b_{3,1}, b_{3,2}]_i(s,t)\right)_{i=1}^3 & S^{n-1}
}
\end{equation}
where $\iota_j^{n-1}$ is the projection $S^{n-1} \vee S^{n-1} \to S^{n-1}$ onto the $j$-th summand and where $[b_{3,1}, b_{3,2}]_i$ denotes the $i$-th configuration point of $[b_{3,2},b_{3,1}]$.
The second relevant diagram is  
\begin{equation}
\label{eq:composite,m=3,n+k}
\xymatrix@C2pc{
D^{\ell} \x (D^k)^{\x3} \ar@{->>}[r]^-{q'} & S^{2n+2k-3} \ar@{..>}_-{f}[dr] \ar[rr]^-{(\kappa \circ \widehat{G})[b^n_{3,1}, \,b^n_{3,2}]} & &\Conf(3,\R^{n+k}) \ar^-{\phi^{n+k}_{3,1}}[r] \ar^-{\phi^{n+k}_{3,2}}[dr] & S^{n+k-1} \\
&& S^{n+k-1} \vee S^{n+k-1} \ar@{^(->}_-{b_{3,1}^{n+k} \vee b_{3,2}^{n+k}}[ur]& & S^{n+k-1}
}
\end{equation}
The composition $((\kappa \circ \widehat{G})[b^n_{3,1},\, b^n_{3,2}]) \circ q'$ of the first two horizontal maps in \eqref{eq:composite,m=3,n+k} is given by
\begin{equation}
\label{eq:kappa-G-formula}
(s,t,u,v)
\longmapsto \left( \
([b_{3,1}^n, b_{3,2}^n]_1(s,t), t), \
([b_{3,1}^n, b_{3,2}^n]_2(s,u), u), \
([b_{3,1}^n, b_{3,2}^n]_3(s,v), v) \
\right).
\end{equation}
Up to homotopy, $(\kappa \circ \widehat{G})[b^n_{3,1},\, b^n_{3,2}]$ factors through a map $f$ as shown in \eqref{eq:composite,m=3,n+k}.  Indeed, it lies in the kernel of the projection that forgets the third configuration point because $[b^n_{3,1},\, b^n_{3,2}]$ does.  
Explicitly, if $(r,s,t) \mapsto (H_i(r,s,t))_{i=1,2}$, is a nullhomotopy of $([b^n_{3,1},\, b^n_{3,2}]_i)_{i=1,2}$ with $r\in[0,1]$, then the homotopy 
\[
(r,s,t,u,v) \mapsto ((H_1(r,s,t), t), \, (H_2(r,s,u),u))
\]
starts at $(\kappa \circ \widehat{G})[b^n_{3,1},\, b^n_{3,2}]$ and ends at a map that is  nullhomotopic because it factors through a sphere of dimension $2k (<\ell + 3k)$.

We want to check that $f$ is homotopic to the Whitehead product $[\iota_1^{n+k-1}, \iota_2^{n+k-1}]$.
We take $f$ to be smooth and use the invariant $h_{12}$.  
To see that the homotopy class of $f$ is determined by $h_{12}(f)$, it suffices by the isomorphism \eqref{eq:hopf-invt-iso} to check that $f$ lies in the reduced homotopy group, i.e., the intersection of the kernels of the projections $\pi_{2n+2k-3} (S^{n+k-1} \vee S^{n+k-1}) \to \pi_{2n+2k-3} (S^{n+k-1})$ onto each summand.  These projections come from the projections $\Conf(3,\R^n) \to \Conf(2,\R^n)$ which forget the first and second configuration points.  
Since $[b_{3,1}^n, b_{3,2}^n]$ maps to zero under either of these two forgetting maps, 
so does $(\kappa \circ \widehat{G})[b^n_{3,1},\, b^n_{3,2}]$, as desired, by a nullhomotopy as in the previous paragraph.

We next describe the calculation of $h_{12}[\iota^{n-1}_1, \iota^{n-1}_2]$ using the framed bordism class of its preimages.  
Since this calculation is independent of the dimension $n-1$, it also applies to $h_{12}[\iota^{n+k-1}_1, \iota^{n+k-1}_2]$, at least up to  sign.  At the same time, this calculation will help us determine the value of $h_{12}(f)$, showing that it equals $\pm h_{12}[\iota^{n+k-1}_1, \iota^{n+k-1}_2]$.
Represent $[\iota_1^{n-1}, \iota_2^{n-1}]$ by a map under which the preimage of a pair of regular values $\{p^n \} \sqcup \{ p^n\}$
is
\[
P:=P_1 \sqcup P_2 := \{0\} \x S^{n-2} \sqcup S^{n-2} \x \{0\} 
\]
which lies 
in 
\begin{equation}
\label{eq:generalized-heegaard}
 D^{n-1} \x S^{n-2} \cup S^{n-2} \x D^{n-1} = S^{2n-3}.
\end{equation}
Each $P_i$ is the preimage in $S^{2n-3}$ of a point under one of the compositions to $S^{n-1}$ in \eqref{eq:composite,m=3,n}.
Then $h_{12}^{n-1}(P) = \pm1$, the linking number of $P_1$ and $P_2$.  It is represented by the intersection of say the second sphere $P_2$ with a disk $Q_1$ bounded by the first sphere $P_1$, which is a single point: $Q_1 \cap P_2 =\pm[\{\text{point}\}]$. 
Let $\infty \in S^{2n-3}$ be the basepoint, i.e.~the pre-image of the wedge-point in $S^{n-1}\vee S^{n-1}$.  We can arrange for it to be the image of the boundary under $q$ and for it to lie in the boundary $S^{n-2} \x S^{n-2}$ of the two summands in \eqref{eq:generalized-heegaard}, 
Then $q^{-1}(P)$ lies in the interior of $D^{2n-k-3} \x D^k$, and so does $q^{-1}(Q_1 \cap P_2)$.   
We thus identify the point $Q_1 \cap P_2$ with some $(s_0, t_0)\in D^{2n-k-3} \x D^k$, and we can view the $P_i$ as submanifolds of $D^{\ell} \x D^k$ in \eqref{eq:composite,m=3,n+k}.

Finally, to calculate $h_{12}(f)$, we use the framed bordism class of the preimage under $f$ of two points, one in each sphere summand.
Since $\phi^{n+k}_{m,i} \circ b^{n+k}_{m,i} = \mathrm{id}_{S^{n+k-1}}$, we can instead consider the preimages of $p^{n+k}:=p^n \x (0,\dots, 0)\in S^{n+k-1}$ under $\phi^{n+k}_{3,i} \circ ((\kappa \circ \widehat{G})[b^n_{3,1}, \,b^n_{3,2}])$ for $i=1,2$.
Define $\Delta: D^{2n-k-3} \x D^k \to D^{2n-k-3} \x D^k \x D^k$ as the (two-fold) diagonal map on $D^k$, namely $\Delta(s,t)=(s,t,t)$.
For $1 \leq i<j \leq 3$, let $pr_{ij}$ be the projection $D^{2n-k-3} \x (D^k)^{\x 3} \to D^k \x D^k$ onto the $i$-th and $j$-th factors of $(D^k)^{\x 3}$.  
Using formula \eqref{eq:kappa-G-formula}, as well as diagram \eqref{eq:composite,m=3,n} together with the fact 
that $\phi^n_{j,i} \circ b^n_{j,i}=\mathrm{id}_{S^{n-1}}$, we see that 
the union of the preimages of $p^{n+k}$ under the compositions across \eqref{eq:composite,m=3,n+k} is

\begin{align*}
\{(s,t,u,t)\colon (s,t) \in P_1 \} \sqcup  \{(s,t,u,u)\colon  (s,u) \in P_2\}= &  \ 
 pr_{13}^{-1}\Delta(P_1) \sqcup pr_{23}^{-1}\Delta(P_2) \\  &  \subset D^{2n-k-3} \x (D^k)^{\x 3}.
\end{align*}
Then  $pr_{13}^{-1}\Delta(Q_1)$ bounds $pr_{13}^{-1}\Delta(P_1)$, and
the intersection $(pr_{13}^{-1}\Delta(Q_1)) \cap (pr_{23}^{-1}\Delta(P_2))$ is the single point $\{(s_0,t_0,t_0,t_0)\}$.  So $h_{12}(f) = \pm 1$, which completes the verification for $m=3$.

\textbf{The case of arbitrary $m\geq 2$:} In general, $\ell=mn-2m-n-k+3$, as in the theorem statement.  
The space of bracket expressions in question has dimension $(m-2)!$.  However, we have a basis of monomials obtained by permuting the indices in $b=[[b_{m,1}, b_{m,2}], \dots, b_{m,m-1}]$, so it suffices to consider only this monomial.
The relevant diagrams are 
\begin{equation}
\label{eq:composite,m,n}
\xymatrix@C2pc@R0.5pc{
& & & & & S^{n-1} \\
D^{\ell} \x D^k \ar@{->>}[r]^-q & S^{\ell + k} \ar[rr]^-{[\iota_1^{n-1}, \dots, \iota_{m-1}^{n-1}]}
& &  \bigvee_{m-1} S^{n-1}  \ar[r]^-{b^n_{m,1} \vee \dots \vee b^n_{m,m-1}} 
& \Conf(m,\R^n)
\ar[ur]^-{\phi^n_{m,1}} \ar@{}[r]|-{\raisebox{1pc}{\vdots}} \ar[dr]^-{\phi^n_{m,m-1}} & \vdots \\
(s,t) \ar@{|->}[rrrr] &&&& \left([b^n_{m,1}, \dots b^n_{m,m-1}]_i(s,t)\right)_{i=1}^m & S^{n-1}
}
\end{equation}
and
\begin{equation}
\label{eq:composite,m,n+k}
\xymatrix@C3pc@R1pc{
 & & & & S^{n+k-1}\\
D^{\ell} \x (D^k)^{\x m} \ar@{->>}[r]^-{q'} & S^{\ell + mk} \ar@{..>}_-{f}[d] 
\ar[rr]^-{(\kappa \circ \widehat{G})[b^n_{m,1}, \dots, b^n_{m,m-1}]} & & 
\Conf(m,\R^{n+k}) \ar[ur]^-{\phi^{n+k}_{m,1}} \ar@{}[r]|-{\raisebox{1pc}{\vdots}} \ar^-{\phi^{n+k}_{m,m-1}}[dr] & \vdots \\
& \bigvee_{m-1} S^{n+k-1} \ar@{^(->}_-{\qquad b_{m,1}^{n+k} \vee \dots \vee b_{m,m-1}^{n+k}}[urr] && & S^{n+k-1}
}
\end{equation}

We claim that the homotopy class of $f$ is determined by the Hopf invariants $h_{I}(f)$ of \eqref{eq:hopf-invt-iso}, where $I$ runs over all permutations of $\{1,\dots, m-1\}$ which fix 1.  
The projection maps 
$\bigvee_{m-1} S^{n+k-1} \to \bigvee_{m-2} S^{n+k-1}$ come from maps which forget one of the first $m-1$ configuration points.  
Since $b$ lies in the kernel of any such projection, so does $\kappa_* \circ \widehat{G}_*(b)$, and therefore $f$ lies in the reduced homotopy groups. Therefore, by the isomorphism \eqref{eq:hopf-invt-iso}, the $(m-2)!$ Hopf invariants $h_{I}(f)$ completely determine $f$ up to homotopy.

We represent $[\iota^{n-1}_1, \dots, \iota^{n-1}_{m-1}] \in \pi_{\ell+k} (\bigvee_{m-1}S^{n-1})$ by a map such that the preimage of a collection of regular values $\{p^n\} \sqcup \dots \sqcup \{p^n\}$ is a certain $(m-1)$-component manifold 
\[
P:=P_1 \sqcup \dots \sqcup P_{m-1} \subset \mathrm{int}(D^\ell \x (D^k)^{\x m}) \subset S^{\ell +mk}
\]
of codimension $n-1$, obtained by a handle decomposition of $S^{\ell+mk}$ similar to \eqref{eq:generalized-heegaard} but iterated $m-2$ times.  
Each $P_i$ is the preimage in $S^{\ell+k}$ of a point under one of the compositions to $S^{n-1}$ in \eqref{eq:composite,m,n}.
Its Hopf invariant $h_{1,2,\dots,m-1}$ is calculated by bounding $P_1$ by $Q_1$, replacing $P_1 \sqcup P_2$ by $Q_1 \cap P_2$, bounding this intersection by $Q_2$, and so on, until we are left with $Q_{m-2} \cap P_{m-1}$, a single point $(s_0,t_0)$.
To calculate $h_J[\iota^{n-1}_1, \dots, \iota^{n-1}_{m-1}]$ for any other $J$, note that if $J \neq (1,2,\dots, m-1)$, then $J=(1, 2 \dots, i-1, i, j, \dots )$ for some $j>i+1$ and some $i\in \{1, \dots, m-3\}$.  Then $Q_i$ intersects only $P_{i+1}$, and $Q_i \cap P_j =\varnothing$.  


\begin{example}[for $m=4$] We represent $[[\iota^{n-1}_1, \, \iota^{n-1}_2],\, \iota^{n-1}_3] \in \pi_{3n-5} (\bigvee_3 S^{n-1})$ by a map under which the preimage of $\{p^n\} \sqcup \{p^n\} \sqcup \{p^n\}$ is 
\begin{align*}
P:= P_1 \sqcup P_2 \sqcup P_3 
&:= \{0\} \x S^{n-2} \x S^{n-2} \ \sqcup \ S^{n-2} \x \{0\} \x S^{n-2} \ \sqcup \ S^{2n-4} \x \{0\} \\
&\subset D^{2n-3} \x S^{n-2} \ \cup \ S^{2n-4} \x D^{n-1} \quad = \quad S^{3n-5} 
\end{align*}
Its Hopf invariant $h_{123}$ can be obtained by taking a manifold $Q_1$ bounded by $P_1$,
 replacing $P_1 \sqcup P_2$ by $Q_1 \cap P_2$, and 
calculating $h_{12}$ of the resulting 2-component submanifold.  Here we can take $Q_1 \cong D^{n-1} \x S^{n-2}$, using a $D^{n-1}$ bounded by $\{0\} \x S^{n-2}$ in $D^{2n-3}$.
Then $Q_1 \cap P_2 \cong S^{n-2}$, which bounds a disk $Q_2$, and $Q_2 \cap P_3$ is a single point.   
Let $\infty \in S^{3n-5}$ be the basepoint, i.e., the image of the boundary under $q$ and the preimage of the wedge-point in $\bigvee_3 S^{n-1}$.
We can take $\infty$ to lie in the $S^{2n-4} \x S^{n-2}$ which bounds the summands in the decomposition of $S^{3n-5}$.
We thus identify $Q_2 \cap P_3$ with a point $(s_0, t_0) \in D^{3n-k-5} \x D^k$.
So $h_{123}[[\iota_1, \iota_2], \iota_3]=\pm 1$.
On the other hand, when we calculate $h_{132}$ of the manifold $P$ in the same way, we find that $Q_1 \cap P_3$ is empty because $Q_1 \subset D^{2n-3} \x S^{n-2}$; that is, $Q_1$ is nested deep enough in the decomposition of $S^{3n-5}$ to miss $P_3$.  Thus $h_{132}([[\iota_1, \iota_2], \iota_3])=0$.
\end{example}
  
We now calculate $h_J(f)$, first for $J=(1,2,\dots, m-1)$.
For a subset $S \subseteq \{1,\dots,m\}$, let $pr_S$ be the projection $D^{\ell} \x (D^k)^{\x m} \to D^{\ell} \x (D^k)^S$ onto the factors of $D^k$ indexed by $S$.  
(We omit braces in the subscript, writing for example $pr_{i,j}$ for $pr_S$ with $S=\{i,j\}$.)
For $i\geq 2$, define $\Delta_i: D^{\ell} \x D^k \to D^{\ell} \x (D^k)^{\x i}$ via the $i$-fold diagonal on $D^k$, that is, $\Delta_i(s,t) = (s,t, \dots, t)$.
The union of preimages of $p^{n+k}:=p^n\x(0,\dots,0)$ under the compositions all the way across diagram 
\eqref{eq:composite,m,n+k} is 
\[
pr_{1,m}^{-1}\Delta_{2}(P_1) \  \sqcup \
pr_{2,m}^{-1}\Delta_{2}(P_1)
 \  \sqcup \ \dots \ \sqcup \ 
pr_{m-1,m}^{-1}\Delta_{2} (P_{m-1}).
\]
Then $pr_{1,m}^{-1}\Delta_{1,m}(Q_1)$ bounds $pr_{1,m}^{-1}\Delta_{1,m}(P_1)$, and 
\[
pr_{1,m}^{-1}\Delta_{2}(Q_1) \ \cap \
pr_{2,m}^{-1}\Delta_{2}(P_2) = 
pr_{1,2,m}^{-1}\Delta_{3} (Q_1 \cap P_2).
\]
The latter term bounds $pr_{1,2,m}^{-1}\Delta_{3} (Q_2)$, and 
\[
pr_{1,2,m}^{-1}\Delta_{3} (Q_2) \ \cap \
pr_{3,m}^{-1}\Delta_{2}(P_3) = 
pr_{1,2,3,m}^{-1}\Delta_{4} (Q_2 \cap P_3).
\]
Continuing, we 
are led to consider
\[
pr_{1,2,\dots, m-2,m}^{-1}\Delta_{m-1} (Q_{m-3} \cap P_{m-2})
\]
which bounds $pr_{1,2,\dots, m-2,m}^{-1}\Delta_{m-1}(Q_{m-2})$.  Finally, we 
are ultimately led to consider 
\begin{align*}
pr_{1,2,\dots, m-2,m}^{-1}\Delta_{m-1}(Q_{m-2}) \ \cap \
pr_{m-1,m}^{-1}\Delta_{2} (P_{m-1}) 
&= \Delta_{m}(Q_{m-2} \cap P_{m-1}) \\
&=\Delta_{m}\{(s_0,t_0)\} \\
&= \{(s_0, t_0, \dots, t_0)\} \subset D^{\ell} \x (D^k)^{\x m}.
\end{align*}
So $h_{1,2,\dots,m-1}(f)=\pm 1$, as desired.

If $J\neq (1,2,\dots,m-1)$, write $J=(1,2,\dots, i-1, i, j, \dots)$ where $1 \leq i \leq m-3$ and $j>i+1$.  Then, calculating as above, we ultimately obtain 
\[
pr_{1,2,\dots,i-1,i,m}^{-1} \Delta_{i+1} (Q_i) \cap pr_{j,m}^{-1} \Delta_{2} (P_j) = 
pr_{1,2,\dots,i-1,i,j,m}^{-1} \Delta_{i+2} (Q_i \cap P_j) = 
pr_{1,2,\dots,i-1,i,j,m}^{-1} \Delta_{i+2} (\varnothing) = 
\varnothing.
\]
Thus for such $J$, $h_J(f)=0$.   This completes the proof of part (a).
\end{proof}



\begin{proof}[Proof of Theorem \ref{T:braids-as-links}, part (b)]
We need only an extension of the proof of part (a) to all (unordered) subsets of $\{1,\dots, m\}$.  
The key idea is that an element in $\mathcal{H}_{m}^{n,k}$ is a sum of monomials each of which is Brunnian (i.e., comes from $Lie(p)$ for some $p<m$) after restricting to the appropriate subset $S \subset \{1,\dots,m\}$.  
For each $S$, restricting to the appropriate submodule of $\mathcal{H}_{m}^{n,k}$ will give a homomorphism into $\pi_*(\Conf(m,\R^n))$.

For any subset $S \subseteq \{1,\dots, m\}$ of cardinality at least 2, define 
\begin{align*}
\kappa_S: \BrEmb_{*} \left(\coprod_m S^k, \ \R^{n+k}\right) &\longrightarrow \Map_* ((S^k)^{S}, \Conf(S, \R^{n+k}))\\
g=(g_1,\dots, g_m) &\longmapsto (\kappa_S(g):(s_i)_{i \in S} \mapsto (g_i(s_i))_{i \in S})
\end{align*}
where $\Conf(S, \R^{n+k})$ is the space of injections $S\to \R^{n+k}$.

If $b$ is a left-normed bracket $b_{j;I}=[ b_{j,i_1}, \dots, b_{j,i_p} ]$, call $S:= S(j;I):=\{j,i_1,\dots,i_p\}$ 
 the \emph{index set} of $b$.
Then 
$\kappa_S \circ \widehat{G}(b)$ factors through the quotient $(S^k)^{S} \to S^{k|S|}$.  
So we may view $\kappa_S \circ \widehat{G}(b)$ as an element of
\[
\mathcal{M}_S := \Map_*(S^{k|S|}, \Conf(S,\R^{n+k})).
\]
Note that
\[
\pi_\ell \mathcal{M}_S \cong \pi_{\ell+ k|S|}\Conf(S,\R^{n+k}) \subset \pi_{\ell + k|S|}\Conf(m,\R^{n+k})
\]
where the above inclusion is induced by the inclusion $S \hookrightarrow \{1, \dots, m\}$.

Now suppose $b$ is any element in $\mathcal{H}_{m}^{n,k}$.   
We use the direct-sum decomposition \eqref{eq:L_m(n-2)}, which also applies over $\Z$ to the non-torsion part of $\pi_{*}(\Conf(m,\R^n))$.  
Fix a basis of left-normed monomials $b_{j;I}$ for each free graded Lie algebra summand; such a basis exists by for example \cite{Walter:2010arXiv}.  Let $\mathcal{B}$ be the union of these bases, with elements corresponding to multi-indices $(j;I)$.
We can then write $b$ as a unique $\Z$-linear combination 
\begin{equation}
\label{eq:lin-comb-basis-monomials}
\begin{split}
b=&\sum_{(j;I) \in \mathcal{B}} \alpha_{j;I} b_{j;I},  
\qquad  b_{j;I} = [ b_{j,i_1}, \dots, b_{j,i_p} ],  
\end{split}
\end{equation}
where $\alpha_{j;I}$ is nonzero only if $I=(i_1,\dots, i_p)$ consists of distinct indices and $p (n-2) \geq {n+k-3}$.
Now $b \in \pi_\ell \Omega^k \Conf(m,\R^n)$ where $\ell = pm - 2p - n - k +3 \geq 0$. 
If $S=S(j;I)$ is the index set of $b_{j;I}$, then $\kappa_S \circ \widehat{G} (b_{j;I})$ can be mapped to $\mathcal{M}_S$.  
On the image $\widehat{G}_*(\mathcal{H}_{m}^{n,k})$, we define the following map, where the sum ranges over all subsets $S\subset \{1,\dots,m\}$:
\begin{align*}
\sum_{S} \kappa_{S}: \widehat{G}_*(\mathcal{H}_{m}^{n,k}) &\longrightarrow \sum_{S} \pi_*(\mathcal{M}_{S}) \subset \pi_*(\Conf(m, \R^{n+k}))\\
\widehat{G}_*\left( \sum_{(j;I)\in \mathcal{B}} \alpha_{j;I}  b_{j;I} \right)& \longmapsto \sum_{(j;I)\in \mathcal{B}} \alpha_{j;I} \, \left(\kappa_{S(j;I)}\right)_* \circ \widehat{G}_*(b_{j;I})
\end{align*}
The sum $\sum_{S} \pi_*(\mathcal{M}_{S})$ is not direct, and the assignment $(j;I) \mapsto S(j;I)$ is not injective because it forgets the order of the indices, but this map suffices to show the desired injectivity.  
Indeed, by a similar argument as for part (a), a Whitehead bracket with index set $S=\{j, i_1, \dots, i_p\}$ is mapped to the corresponding left-normed Whitehead bracket by $(\kappa_S)_* \circ \widehat{G}_*$:
\[
\xymatrix@C4pc{
\pi_{\ell+k}\Conf(m,\R^n) \ni
[ b^n_{j,i_1}, \dots, b^n_{j,i_p} ]
\ar@{|->}[r]^-{(\kappa_S)_* \circ \widehat{G}_*}
&
[ b^{n+k}_{j,i_1}, \dots, b^{n+k}_{j,i_p} ]
\in \pi_{\ell+pk}\Conf(m,\R^{n+k}).
}
\]  
As a result, 
$\left(\sum_{S} \kappa_S\right)_* \circ \widehat{G}_*$ maps
$\mathcal{H}_{m}^{n,k} \subset \pi_*(\Conf(m,\R^n))$  isomorphically onto the corresponding submodule of $\pi_*(\Conf(m,\R^{n+k}))$.  
\end{proof}

\begin{rem}[Variations of Theorem \ref{T:braids-as-links}]
\label{R:variations-of-braids-as-links}
Both parts of Theorem \ref{T:braids-as-links} hold if we replace the space of long links by the corresponding space of based closed links, with very little change to the proofs.  They similarly hold --- in both the long and closed settings --- if we replace embeddings by link maps, meaning smooth maps of the disjoint union such that the images of the components are pairwise disjoint.  
\end{rem}

\subsection{Examples and a conjecture}
\label{S:ex-conj-discussion}
Part (a) of Theorem \ref{T:braids-as-links} shows that graphing yields nontrivial classes in $\pi_0$ of spaces of  $m$-component $k$-dimensional string links in $\R^{n+k}$, for any $n$ and $k$ such that $m=\frac{n+k-3}{n-2}$.
To list all the possible values of $n$ and $k$, we introduce another integer parameter $j\geq 0$ and put $k=j(m-1)+1$ and $n=j+2$.   
\begin{cor}[Isotopy classes of high-dimensional string links]
\label{C:pi0-string-links}
For each $j \geq 0$, there is an injection 
\begin{align*}
Lie(m-1) \hookrightarrow 
\pi_0(\Omega^{j(m-1)+1} \Conf(m, \R^{j+2})) 
\overset{G_*}{\longrightarrow} \pi_0\left(\Emb_c\left( \coprod_m \R^{j(m-1)+1}, \ \R^{jm+3}\right)\right).
\end{align*}
\qed
\end{cor}
Setting $j=0$ gives an inclusion of a subspace of pure braids into string links.  For $m \geq 2$ and $j\geq 1$, we get $k$-dimensional string links in $\R^N$ with $k\geq 2$ and $N\geq 5$.

\begin{example} \ 
\begin{itemize}
\item[(a)]  Fix $m=3$.  If $(k,n)=(2j-1,j+1)$), the class $[b_{3,1}, b_{3,2}]$ corresponding to the tripod diagram from Example \ref{Ex:B312} gives rise via ($G$ or) $\widehat{G}$ to elements in $\pi_0$ of spaces of $(2j-1)$-dimensional (string) links in $\R^{3j}$.  Such embeddings are the well known higher-dimensional analogues $S^{2j-1} \sqcup S^{2j-1} \sqcup S^{2j-1} \hookrightarrow \R^{3j}$  of the Borromean rings.  
The smallest possible values of $k$ and $n$ in this case are $k=3$ and $n=3$, for which we get an embedding $S^3 \sqcup S^3 \sqcup S^3 \hookrightarrow \R^6$.
\item[(b)]
Fix $m=4$.  
If $(k,n)=(3j-2,j+1)$, then ($G$ or) $\widehat{G}$ gives rise to a certain 2-dimensional subspace of isotopy classes of $(3j-2)$-dimensional (string) links in $\R^{4j-1}$.  
It is the image of the subspace spanned by say $[[b_{4,1}, b_{4,2}], b_{4,3}]$ and $[[b_{4,1}, b_{4,3}], b_{4,2}]$, which is isomorphic to the subspace of $\mathcal{T}^n(4)$ spanned by trees with 4 distinctly labeled leaves.  It is 2-dimensional for either parity of $n$.
 The smallest possible values of $k$ and $n$ in this case are $k=4$ and $n=3$, for which we get embeddings $\coprod_4 S^{4} \hookrightarrow \R^7$.
\end{itemize}
\end{example}

In this setting of equidimensional string links, we can compare our work to a result of Songhafouo Tsopm\'en\'e and Turchin \cite[Theorem 3.2]{Songhafouo-Turchin:Forum} (further developed by these authors \cite{Songhafouo-Turchin:HHA} and by Fresse, Turchin, and Willwacher \cite[Section 5]{FTW:2020arXiv}).
That result identifies $\pi_0\left(\Emb_c\left(\coprod_m\R^k, \R^{n+k}\right)\right) \otimes \Q$ with the $\Q$-vector space 
of trivalent trees with leaves labeled by $\{1,\dots,m\}$, modulo IHX, defined like $\mathcal{T}^n(m)$ in the Introduction but over $\Q$ instead of $\R$.  
The orientations on graphs in those cited works depend a priori on the parities of $k$ and $n$.  However, one can check that for trees of odd valence in this equidimensional setting, they depend only on the parity of $n$ and agree with the orientations in $\mathcal{T}^n(m)$, by adapting \cite[Section 3.1]{Kuperberg-Thurston} to the setting of two types of vertices.
Theorem \ref{T:braids-as-links} realizes (over $\Z$) the subspace $Lie(m-1)$ of distinctly labeled trees with $m$ leaves in $\mathcal{T}^n(m)$.  We suspect that this identification agrees with the one given by \cite{Songhafouo-Turchin:HHA}, since both use graph complexes associated to configuration spaces.
Some but not all trees with repeated leaf-labels could also arise from graphing braids, as explained in Conjecture \ref{conj:graphing-injects-homotopy-Q-homology} and Remark \ref{R:braids-vs-milnor-invts} below.  

\begin{conj}
\label{conj:graphing-injects-homotopy-Q-homology}
The graphing map $G: \Omega^k \Conf(m,\R^n) \to \Emb_c\left(\coprod_m \R^k, \R^{n+k}\right)$ induces injections on  rational homotopy  $\pi_*(-) \otimes \Q$ and rational homology $H_*( - ; \Q)$.  
\end{conj}

It seems likely that $G$ is a map of algebras over the little $k$-cubes operad, in which case an injection in rational homotopy would imply an injection in rational homology by F.~Cohen's thesis results \cite{HoILS}.  
Theorem \ref{T:braids-as-links} essentially uses the fact that the space of ``homotopy string link'' classes $\mathcal{H}_{m}^{n,k}$ survives the evaluation map $ev_{1,1,\dots,1}=\kappa$ to the multi-linear stage $T_{1,1,\dots, 1}\mathrm{Link}\left(\coprod_m \R^k, \R^{n+k}\right)$ of the Taylor tower for the space of link maps.  The latter space records information about configurations where one point lies on each component, though it records no tangential data.  Equivalently, it is the multi-linear stage $T_{1,1,\dots, 1}\overline{\mathrm{Emb}}\left(\coprod_m \R^k, \R^{n+k}\right)$ where $\overline{\mathrm{Emb}}:=\operatorname{hofiber}(\mathrm{Emb} \to \mathrm{Imm})$ is the space of embeddings modulo immersions.
The nontriviality of bracket expressions with repeats may be detected by a more refined analogue of $\kappa$ where in the target, multiple configuration points are allowed on each component.  See for example \cite{BCKS:2017}, \cite{BCSS:2005}, \cite{Sinha:Cosimplicial}.  
Such an investigation may also be informed by the recent work of Kosanovi\'c \cite{Kosanovic:arXiv2020}, who realized certain trees in the Taylor tower for knots in $\R^3$ using grope cobordisms.

If $n$ and $k$ are both odd, a method of proof would have to go beyond the Whitehead bracket expressions used in \eqref{eq:brackets-preserved}.
Indeed, let $d^{\mathrm{even}}_\ell(m)$ (respectively $d^{\mathrm{odd}}_\ell(m)$) be the dimension of the subspace of length-$\ell$ monomials in the free graded Lie algebra over $\Q$ on generators $X_1,\dots, X_m$ which all have even (respectively odd) degree.  By for example the first formula in Corollary 1.1 in \cite{Petrogradsky:2003}, $d^{\mathrm{even}}_{4j-2}(m) < d^{\mathrm{odd}}_{4j-2}(m)$.  For instance, at length $\ell=2$, $[X,X]$ is nontrivial only if $|X|$ is odd.
On the other hand, for $\ell \not\equiv 2 \mod 4$, the same formula shows that $d^{\mathrm{even}}_\ell(m) = d^{\mathrm{odd}}_\ell(m)$.

\begin{rem}[Braids vs.~trivalent trees modulo IHX]
\label{R:braids-vs-milnor-invts}
Even if $G_*$ is injective on all of rational homotopy, we would realize only a proper subspace of the space $\mathcal{T}^n(m)$ of leaf-labeled trivalent trees modulo IHX.  
We explain why in the case where $n$ is even, though we expect this statement to hold also for odd $n$.

For any fixed  $m$, the subspace $\mathcal{M}_r(m)$ of trees in $\mathcal{T}^{\mathrm{even}}(m)$ with $r$ leaves is dual to the space of Milnor invariants of $m$-component string links of finite type $r$ 
\cite{Habegger-Masbaum:2000}.  (See also \cite{Levine:Addendum, Habegger-Pitsch, CST:Tree-Homology} for statements over $\Z$ and at the level of Lie algebras.)  
On the other hand, the subspace $\mathcal{P}_r(m)$ of $\pi_*(\Conf(m, \R^n))$ corresponding to trees in $\mathcal{T}^{\mathrm{even}}(m)$ with $r$ leaves is dual to the space of indecomposable invariants of $m$-component pure braids of finite type $r$.
Over $\Q$ or $\R$, both $\mathcal{P}_r(m)$ and $\mathcal{M}_r(m)$ can be further identified with spaces of type-$r$ concordance invariants of $m$-strand pure braids and $m$-strand string links respectively.
In general, $\mathcal{P}_r(m) \subset \mathcal{M}_r(m)$ \cite[Theorem 16.1]{Habegger-Masbaum:2000}.
More precisely, their respective dimensions
  are \cite[Theorem 15]{Orr:1989}
\[
\dim \mathcal{M}_r(m) = mN_r(m) - N_{r+1}(m),
\qquad
\text{where}
\qquad
N_r(m) = \frac{1}{r}\sum_{d|r}\mu(r/d) m^d
\]
and \cite[Corollary 29]{Willerton:PhD}, \cite[Theorem 5.11]{Magnus-Karrass-Solitar}
\[
\dim \mathcal{P}_r(m) = \frac{1}{r} \sum_{d|r} \mu(r/d) \sum_{i=1}^{m-1} i^d.
\]
Although for fixed $r$, $P_r(m)/M_r(m)\to 1$ as $m\to\infty$, it appears that for fixed $m$, $P_r(m)/M_r(m)\to 0$ as $r\to\infty$ and that $P_r(m)/M_r(m) \approx 1/2$ for large and roughly equal values of $m$ and $r$.  
\end{rem}

\appendix

\section{Graded Lie algebras and Whitehead and Samelson products}
\label{apx:brackets}

Here we review the Whitehead and Samelson products in homotopy of a space and the Pontryagin product on the homology of its based loop space.  We begin with some facts about graded Lie algebras.

First, recall that a \emph{graded Lie algebra}\footnote{Some authors call this a ``graded \emph{quasi}-Lie algebra.''  A \emph{quasi-Lie algebra} (vis-a-vis a Lie algebra) is a structure where anti-symmetry rather than the stronger condition of alternativity is satisfied.  Thus one could say that a quasi-Lie algebra is a graded quasi-Lie algebra concentrated in degree 0.  
However, omitting the prefix ``quasi'' should cause no confusion, especially over $\R$ where anti-symmetry and alternativity are equivalent.}
over $\R$ is an  $\R$-vector space with a bilinear operation 
$[ - ,  - ]$ satisfying the graded anti-symmetry and Jacobi relations, where 
$|a|$, $|b|$, and $|c|$ are the degrees of the elements $a$, $b$, and $c$:
\begin{equation}
\label{eq:Lie-alg-AS-Jacobi}
\begin{split}
[ a,b ] & =-(-1)^{|a||b|} [ b, a ]\\
(-1)^{|a||c|} [ [ a,b ], c ] & +(-1)^{|b||a|}[ [ b,c ], a ] +(-1)^{|c||b|}[ [ c,a], b] =0.
\end{split}
\end{equation}
When all the elements have degree zero, one gets an ordinary Lie algebra.  One can replace ``$\R$-vector space'' by ``$\Z$-module'' to get the definition of a graded Lie algebra (or graded Lie ring) over $\Z$.  Lemma \ref{lem:left-normed-span} applies in that setting too.

\begin{defin}
The \emph{length} of a Lie monomial (i.e.~bracket expression) in the generators of a graded Lie algebra is the number of generators which appear, counted with multiplicity.
A Lie monomial is \emph{left-normed} if it is of the form 
\[
[ X_{i_1}, X_{i_2}, X_{i_3}\dots, X_{i_k} ] :=
[[ \dots [[X_{i_1}, X_{i_2}], X_{i_3}], \dots], X_{i_k}].
\]
\end{defin}

\begin{lem}
\label{lem:left-normed-span}
Let $\mathcal{L}$ be a graded Lie algebra, generated by elements $X_i$, $i\in \mathcal{I}$.  
\begin{enumerate}
\item[(a)]  The left-normed monomials in the $X_i$ span $\mathcal{L}$.  
\item[(b)]  Powers $[ X_i, \dots, X_i ]$ of length greater than 2 vanish. 
\item[(c)]  Squares $[X_i,X_i]$ are nontrivial only on generators in odd grading.
\item[(d)]  If $|X_i|$ and $|X_h|$ are odd, then $[[X_i,X_i],X_h] = - 2 [[X_i,X_h],X_i]$.
\end{enumerate}
\end{lem}

\begin{proof}
We show that any monomial of length $\ell$ is in the span of left-normed monomials by induction on $\ell$. 
The statement is obvious for $\ell=1$.  Suppose it holds for all monomials of length at most $\ell\geq 1$.
Given a monomial of length $\ell+1$, write it as $[V_0,W_0]$, where $V_0$ and $W_0$ are monomials of lengths $\ell(V_0), \ell(W_0) < \ell$. 
If  $\ell(V_0)=1$ or $\ell(W_0)=1$, then $[V_0,W_0]$ can be written as a linear combination of left-normed monomials by applying possibly anti-symmetry and then the induction hypothesis.
If $\ell(V_0), \ell(W_0)\geq 2$, we may assume by induction that both $V_0$ an $W_0$ are left-normed monomials.  So we may write $V_0=[V_1,X_i]$ for some monomial $V_1$ and some generator $X_i$, and by the Jacobi identity 
	\[
	 [V_0,W_0]=[[V_1,X_i],W_0]=\pm [[X_i,W_0],V_1]\pm [[W_0,V_1],X_i].
	\]
	The second term can be written as a linear combination of left-normed monomials.  If $\ell(V_1)=1$, so can the first term.  Otherwise, rewrite it as $\pm [V_1, [X_i,W_0]]=[V_1,W_1]$, where $W_1=[X_i,W_0]$.  Note that $\ell(V_1)=\ell(V_0)-1$ and $\ell(W_1)=\ell(W_0)+1$.  Therefore repeating this calculation $\ell(V_1)-1$ more times yields an expression for $[V_0,W_0]$ as a linear combination of left-normed monomials, proving part (a). 
	
	Part (b) follows from the graded Jacobi identity, since all the signs are the same and $1/3 \in \R$.  Part (c) follows from graded anti-symmetry, since $1/2 \in \R$.  Part (d) follows by using both relations.
\end{proof}

The homotopy groups $\pi_*(X)$ of a space $X$ equipped with the Whitehead product $[ - , - ]_W$ satisfy
the relations below, where the second is the graded Jacobi identity as above, but the first differs from graded anti-symmetry:
\begin{equation}
\label{eq:Whitehead-AS-Jacobi}
\begin{split}
[ A,B]_W & =(-1)^{|A||B|} [ B,A]_W\\
(-1)^{|A||C|} [[ A,B]_W, C]_W & +(-1)^{|B||A|}[[ B,C]_W, A]_W + (-1)^{|C||B|}[[ C,A]_W, B]_W=0.
\end{split}
\end{equation}
The Samelson product $[ - , - ]$ on $\pi_*(\Omega X)$ is defined by 
\[
[ a , b ] = (-1)^{|a| + 1} \del_* [\del_*^{-1} a, \, \del_*^{-1} b]_W
\]
where $\del_*: \pi_{*+1}(X) \to \pi_*(\Omega X)$.
So if $a=\del_*A$, then $|a|+1=|A|$.
The Samelson product makes $\pi_*(\Omega X)$ and $\pi_*(\Omega X)\otimes \R$ into graded Lie algebras over $\Z$ and $\R$ respectively.
In Section \ref{S:brackets-primitives-trees}, we apply Lemma \ref{lem:left-normed-span} to it.

Concatenation of loops induces the Pontryagin product on $H_\ast(\Omega X;\R)$, which makes it a graded algebra, with its graded commutator given by
\begin{equation}\label{eq:bracket-H_ast}
	\begin{split}
		[ a,b ] & = a b - (-1)^{|a||b|} b a\ .
	\end{split}
\end{equation}
The coalgebra structure given by the diagonal map 
\[
\Delta: H_\ast(\Omega X;\R) \longrightarrow H_\ast(\Omega X;\R)\otimes H_\ast(\Omega X;\R).
\]
then makes $H_\ast(\Omega X;\R)$ into a Hopf algebra.
By the Milnor--Moore Theorem, the Hurewicz map $h\colon \pi_{*}(\Omega X)\otimes \R \to H_*(\Omega X; \R)$ maps $\pi_{*}(\Omega X)\otimes \R$ isomorphically onto the subspace of primitive elements.  (This statement holds only over a field of characteristic zero.)   
Restricting the codomain thusly gives an isomorphism $\pi_{*}(\Omega X) \otimes \R \to PH_*(\Omega X; \R)$ of graded Lie algebras.  Explicitly, if $A \in \pi_p(X)$ and $B \in \pi_q(X)$, then  
 \begin{equation}\label{eq:hurewicz-bracket}
 	\begin{split}
	h \partial_* ([A,B])  &= (-1)^{p} \left( h \partial_* (A) h \partial_*(B) - (-1)^{(p-1)(q-1)}h  \partial_* (A) h \partial_*(B) \right).
	\end{split}
\end{equation}

\section{Computations of some cocycles in spaces of pure braids}
\label{A:cocycles}

Our goal here is to construct cocycles in $(\Ba(\Dm(m)),d_{\Ba})$ which detect primitive homology classes in $H_\ast(\Ba^\ast(\Dm(m)))$.  The map $\Phi$ sends these to cohomology classes in $H^\ast_{dR}(\Omega \Conf(m,\R^n))$ which detect classes in $\pi_*(\Omega \Conf(m,\R^n))\otimes \R$.  The general setup here applies when $\Dm(m)$ replaced by $\D(m)$, by Corollary \ref{cor:PDm(m)-to-PD(m)}. In Example \ref{Ex:z3122} below, we thus simplify the calculation by eliminating diagrams with multiple edges from consideration.

We can assign a grading to $\Ba(\Dm(m))$ by $(-p,q)$,  i.e. the total degree $q$ of factors in the monomial and  the monomial's length $p$:
\[
\Ba(\Dm(m))=\bigoplus_{p, q} \Ba(\Dm(m))^{(-p,q)}.
\]
Recall that the differential $d_{\Ba}$ decomposes as $d_{\Ba} =\delta_{\Ba}-D_{\Ba}$  as in formula \eqref{eq:bar-differential}, where
\[
\begin{split}
\delta_{\Ba}: \Ba(\Dm(m))^{(-p,q)}\longrightarrow \Ba(\Dm(m))^{(-p,q+1)},  \\
D_{\Ba}: \Ba(\Dm(m))^{(-p,q)}\longrightarrow \Ba(\Dm(m))^{(-p+1,q)}.
\end{split}
\]
For any cochain $z\in \Ba(\Dm(m))$, we write 
\begin{equation}\label{eq:length-decomp}
z=z^{(1)}+\cdots+ z^{(p)},
\end{equation}
a decomposition with respect to the monomial length, where $p=p(z)$ is the maximal monomial length in $z$. Next $d_{\Ba} z=0$ can be written as the following system of equations for the factors $z^{(i)}$, $i=1,\ldots, p$ 
(cf.~\cite[p.~163]{Bott:1982}):
\begin{equation}\label{eq:zigzag}
\begin{cases}
& \delta_{\Ba} z^{(i)}  =  D_{\Ba} z^{(i+1)},\qquad i=1,\ldots,p-1,\\
& \delta_{\Ba} z^{(p)} = 0.
\end{cases}
\end{equation}

\begin{lem}\label{lem:zigzag}
The map $j^\ast\colon H^\ast(\Ba(\Dm(m)),d_{\Ba})\longrightarrow H^\ast(I\Dm(m),\widetilde{\delta})$, induced by the subspace inclusion $j:P\Dm(m)^\ast\longrightarrow \Ba^\ast(\Dm(m))$ is an epimorphism.
\end{lem}
\begin{proof}
	Following the proof of Theorem \ref{thm:primitive-diagrams}, we observe that the inclusion $j$ induces a monomorphism 
	\[
	 j_\ast:H_\ast(P\Dm(m)^\ast,\delta^\ast|_{P\Dm(m)^\ast})\longrightarrow H_\ast(\Ba^\ast(\Dm(m)),d^\ast_{\Ba}),
	\]
	whose image is $PH_\ast(\Ba^\ast(\Dm(m)),d^\ast_{\Ba})$. Dualizing and applying the Universal Coefficient Theorem yields the claim. \qedhere
\end{proof}
\no	As a consequence, for any cocycle in $H^\ast(\Ba(\Dm(m)),d_{\Ba})$ represented by $z:\Ba^\ast(\Dm(m))\longrightarrow \R$, the top term $z^{(1)}$ satisfies
	\begin{equation}\label{eq:zigzag-ini}
\widetilde{\delta} z_0=0,\qquad z_0=j^\ast (z)=z^{(1)}|_{P\Dm(m)^\ast}.
\end{equation}
In other words, $z_0$ represents a cocycle in $H^\ast(P\Dm(m),\widetilde{\delta})$. As a result we may attempt to recover $z$ from a representative cocycle $z_0$ by solving the system \eqref{eq:zigzag}. We end this Appendix by performing such a calculation for two examples. The reader may consult \cite{Haefliger:1978}, \cite{Komendarczyk:2010}, and \cite{Sinha:2012} for other approaches to cocycle computations. In the following examples we denote by $z_{j;I}$ the expressions for cocycles dual to the left--normed brackets $B_{j;I}$ defined in \eqref{eq:left-normed-brackets}.

\begin{example}[Computation of $z_{2;1,1}$, a cocycle detecting $B_{2;1,1}$]	
Let $n$ be odd.  A representative of $B_{2;1,1}=[B_{2,1},B_{2,1}]$, $n$ odd, in $H_\ast(P\Dm(2)^\ast)$
	was already computed in Example \ref{Ex:2HopfMap}.
	 Thus we may try $z_0=-\vvcenteredinclude{0.1}{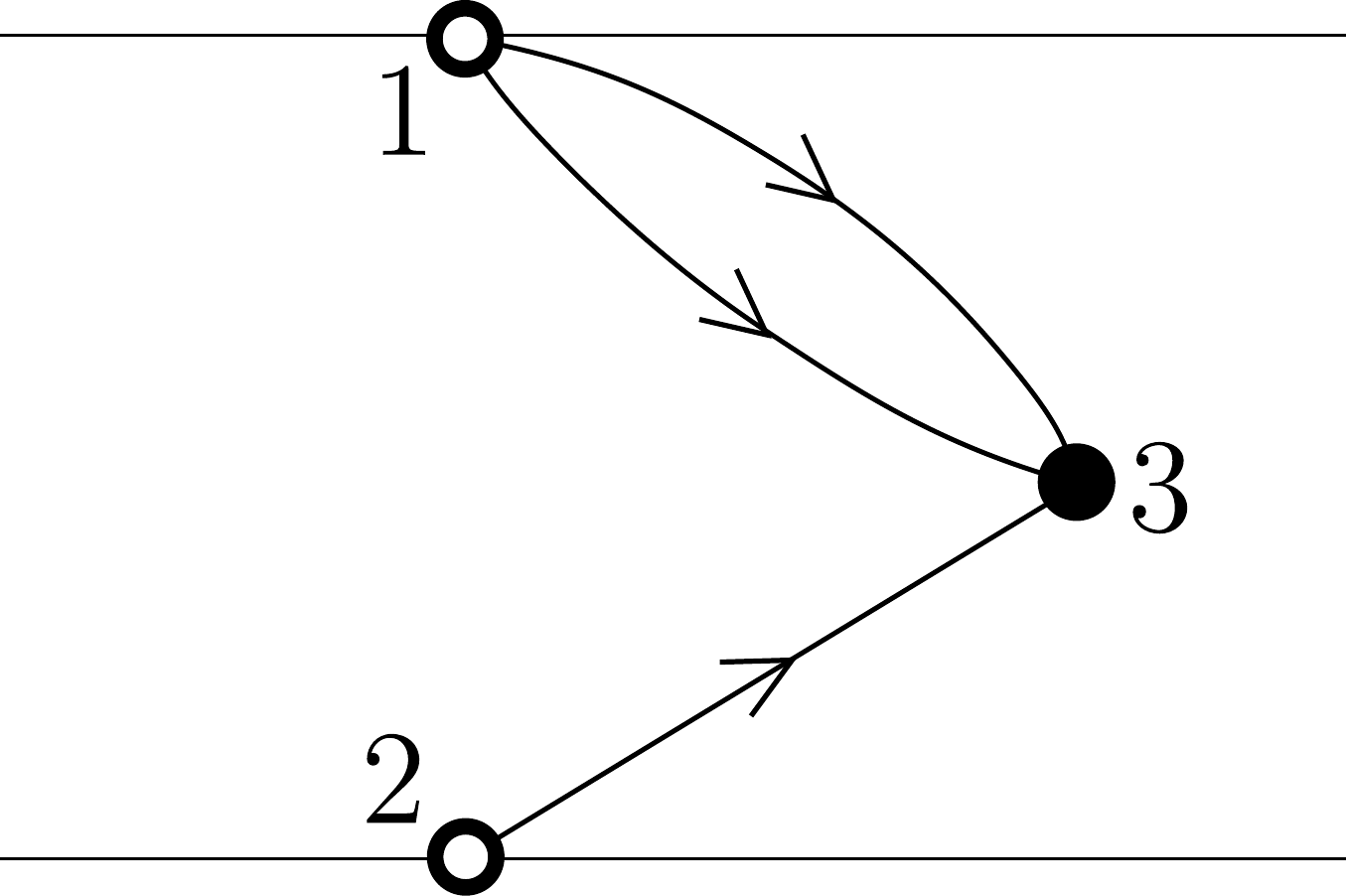}$ as a candidate in \eqref{eq:zigzag-ini}. Since $\delta z_0=\vvcenteredinclude{0.1}{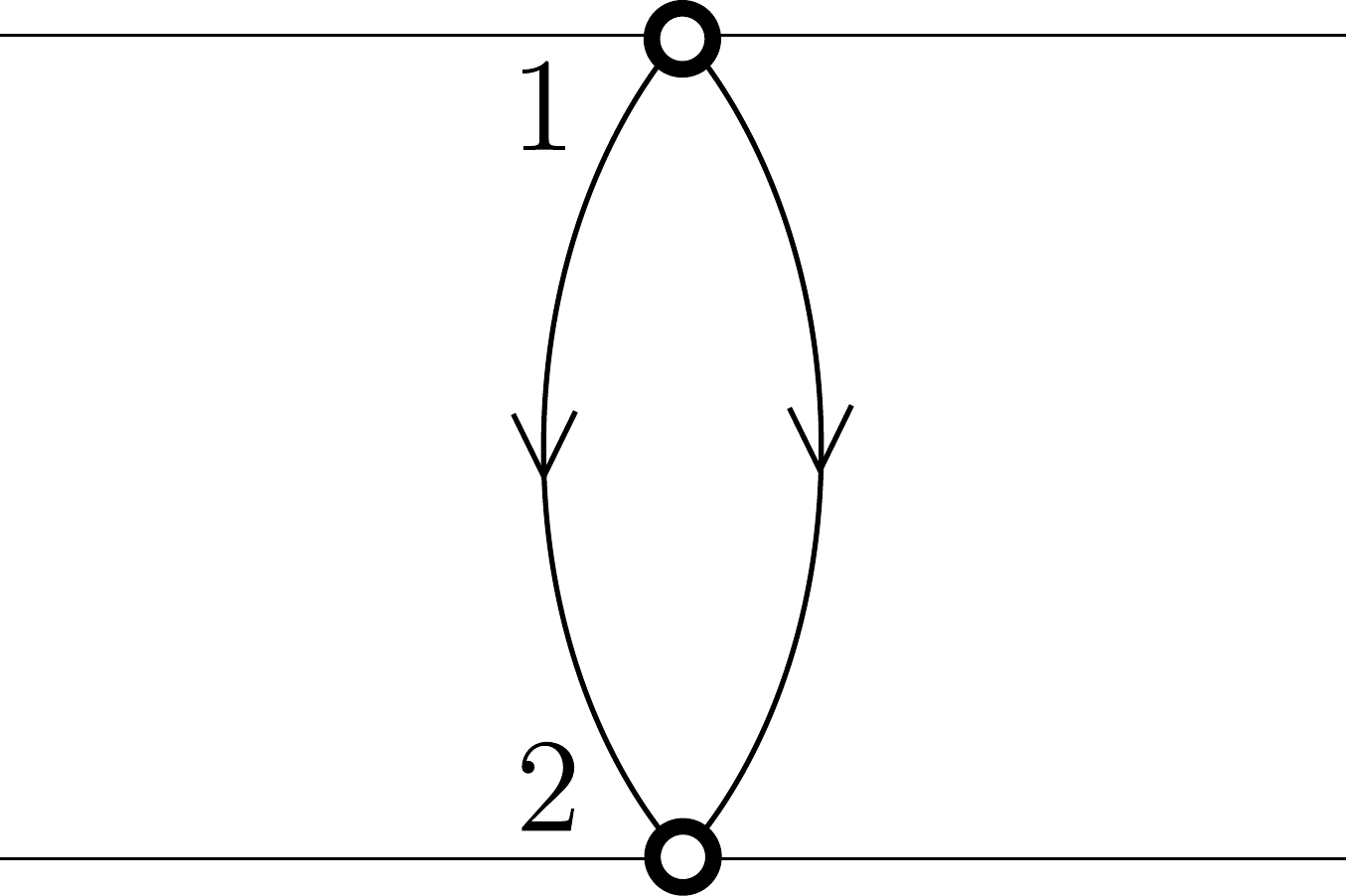}$, we have $\widetilde{\delta} z_0=0$. We may try solving \eqref{eq:zigzag} with $z^{(1)}=z_0$, which yields just a single equation 	
\[
	\delta_{\Ba} z_0=D_{\Ba} z^{(2)},
\]	
with $z^{(2)}$ to be determined. Applying the definitions we have
\[
\delta_{\Ba} \left(-\vvcenteredinclude{0.14}{sqCo-z0.pdf}\right)=-\vvcenteredinclude{0.14}{sqCo-G21-sq.pdf}=D_{\Ba} \left(\vvcenteredinclude{0.14}{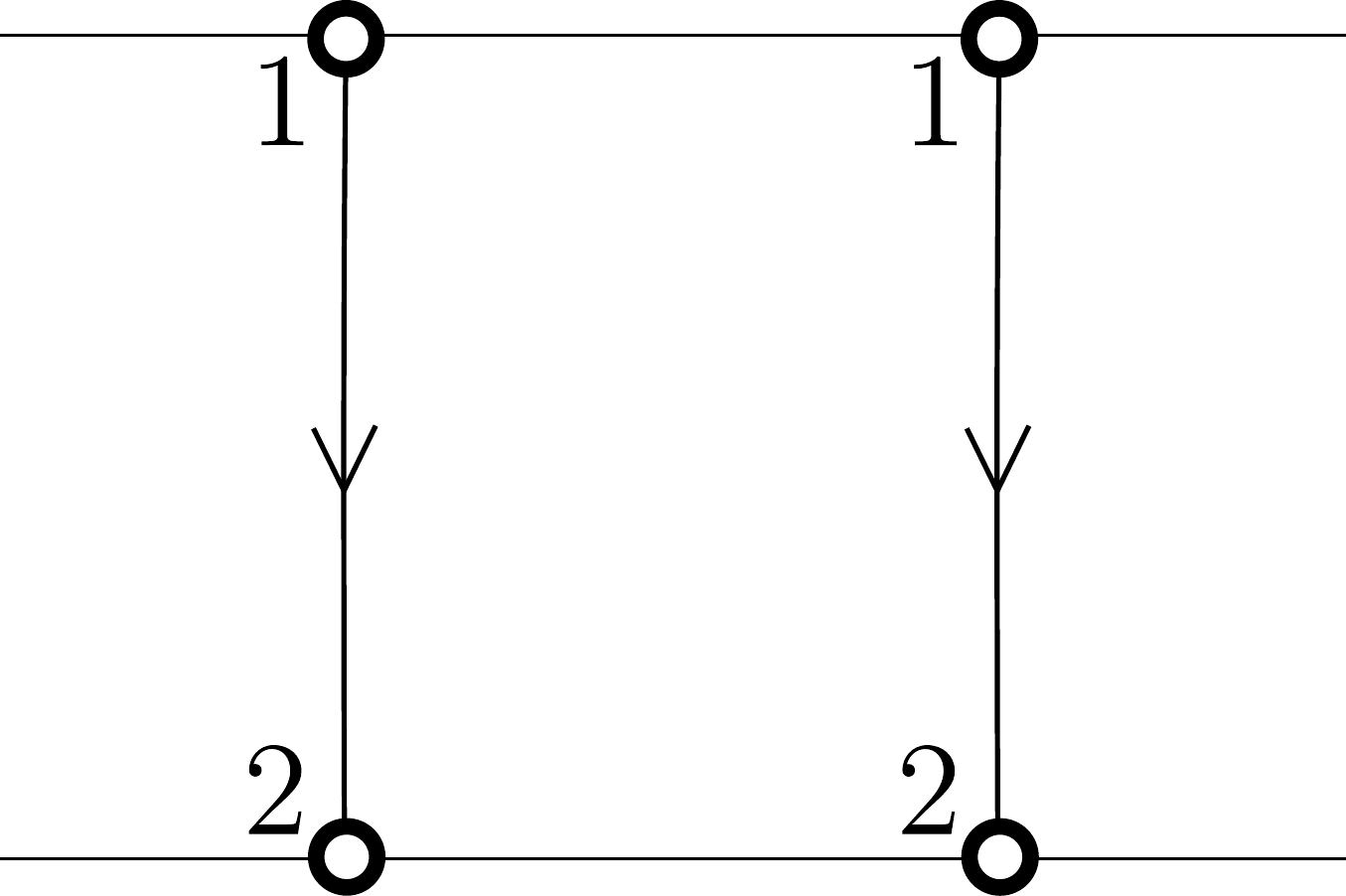}\right),
\]
and therefore 
\[
z_{2;1,1}=z^{(1)}+z^{(2)}=-\vvcenteredinclude{0.14}{sqCo-z0.pdf}+\vvcenteredinclude{0.14}{sqCo-G21-G21.pdf}.
\]
The corresponding Chen iterated integral\footnote{The $z^{(1)}$ term vanishes because the diagram has a double edge and $I$ vanishes on such diagrams.} is
\[
 \Phi(z_{2;1,1})=\varint \alpha_{2,1}\alpha_{2,1}.
\]
One may recognize that, up to a sign, the above integral is just the classical Whitehead integral for the Hopf invariant (cf.~\cite{Hain:1984}), so it yields $\pm 2$ on $[B_{2,1}, B_{2,1}]$.  (If $n-1$ is $2,4,$ or $8$, this cycle is twice the Hopf map, and otherwise it corresponds to a generator of $\pi_{2n-3}(S^{n-1})$.)
Alternatively (and to get the correct sign), 
we can use \eqref{eq:diagram-2hopf} to get $\langle \Theta(B_{2;1,1}), z_{2;1,1} \rangle= 2$ because the first diagram in $z_{2;1,1}$, which corresponds to the first diagram in \eqref{eq:diagram-2hopf}, has 2 automorphisms.
\end{example}

\begin{example}[Computation of $z_{3;1,2,2}$, a cocycle detecting $B_{3;1,2,2}$]
\label{Ex:z3122}
Next we show the computation for a  cocycle detecting $B_{3;1,2,2}=[[B_{3,1},B_{3,2}],B_{3,2}]$.  
The class $B_{3;1,2,2}$ is nontrivial in either parity of $n$, but we show the calculation only for $n$ odd.
Also, to avoid lengthy expressions we calculate in $\D(3)$ rather than in $\Dm(3)$.

As in the previous example, we can deduce that  
$z_0=\vvcenteredinclude{0.1}{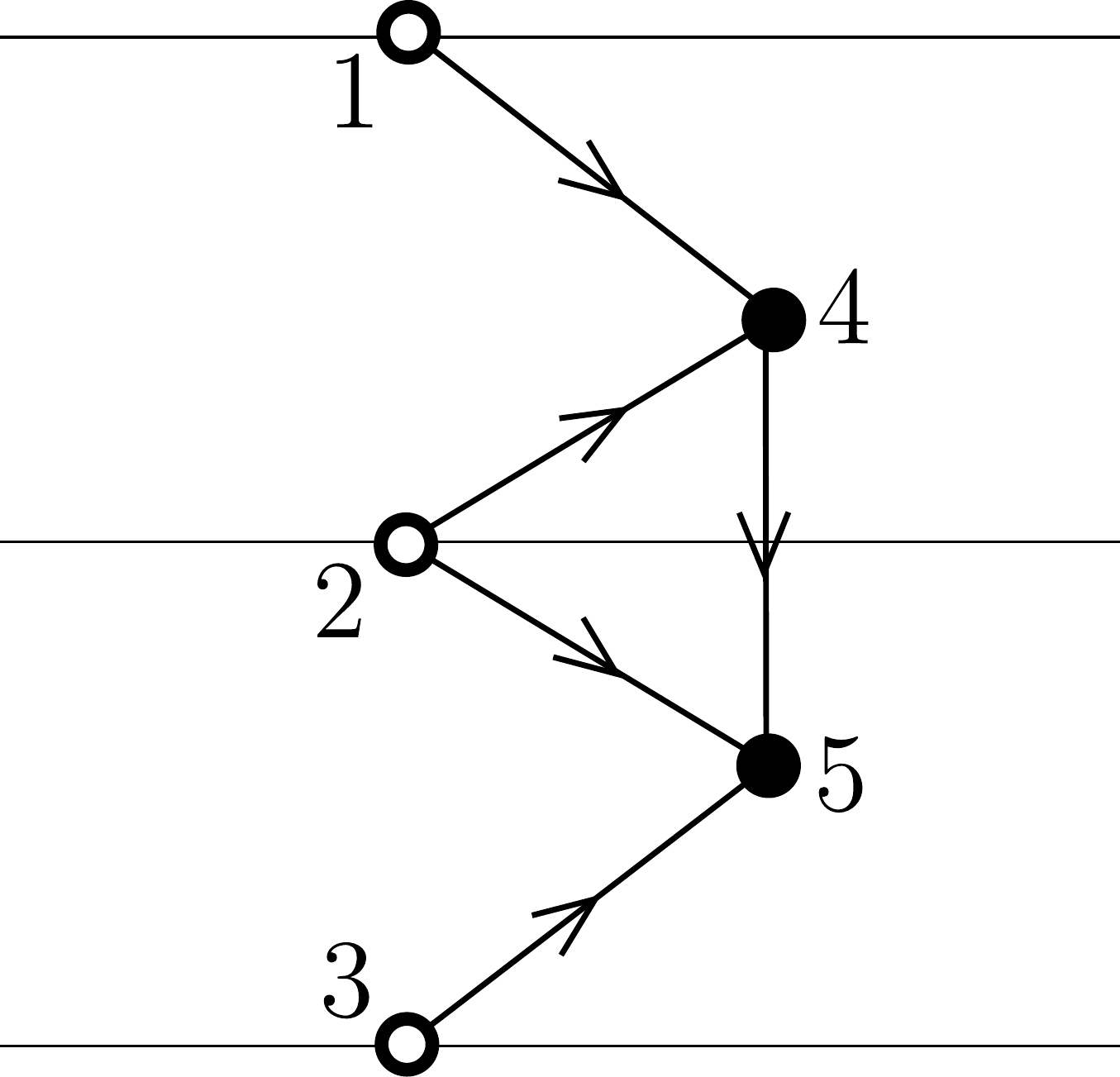}$ represents a cocycle in\\ $H^{5(n-1)-2n}(P\D(3),\widetilde{\delta})$, dual to $\Gamma^\ast_{3;1,2,2}=\Theta(B_{3;1,2,2})$. Next, we solve \eqref{eq:zigzag} with $z^{(1)}=z_0$:
\[ 
\delta_{\Ba}\left(\raisebox{-1.6pc}{\includegraphics[scale=0.12]{Co-W.pdf}}\right)=-\left(-\raisebox{-1.6pc}{\includegraphics[scale=0.12]{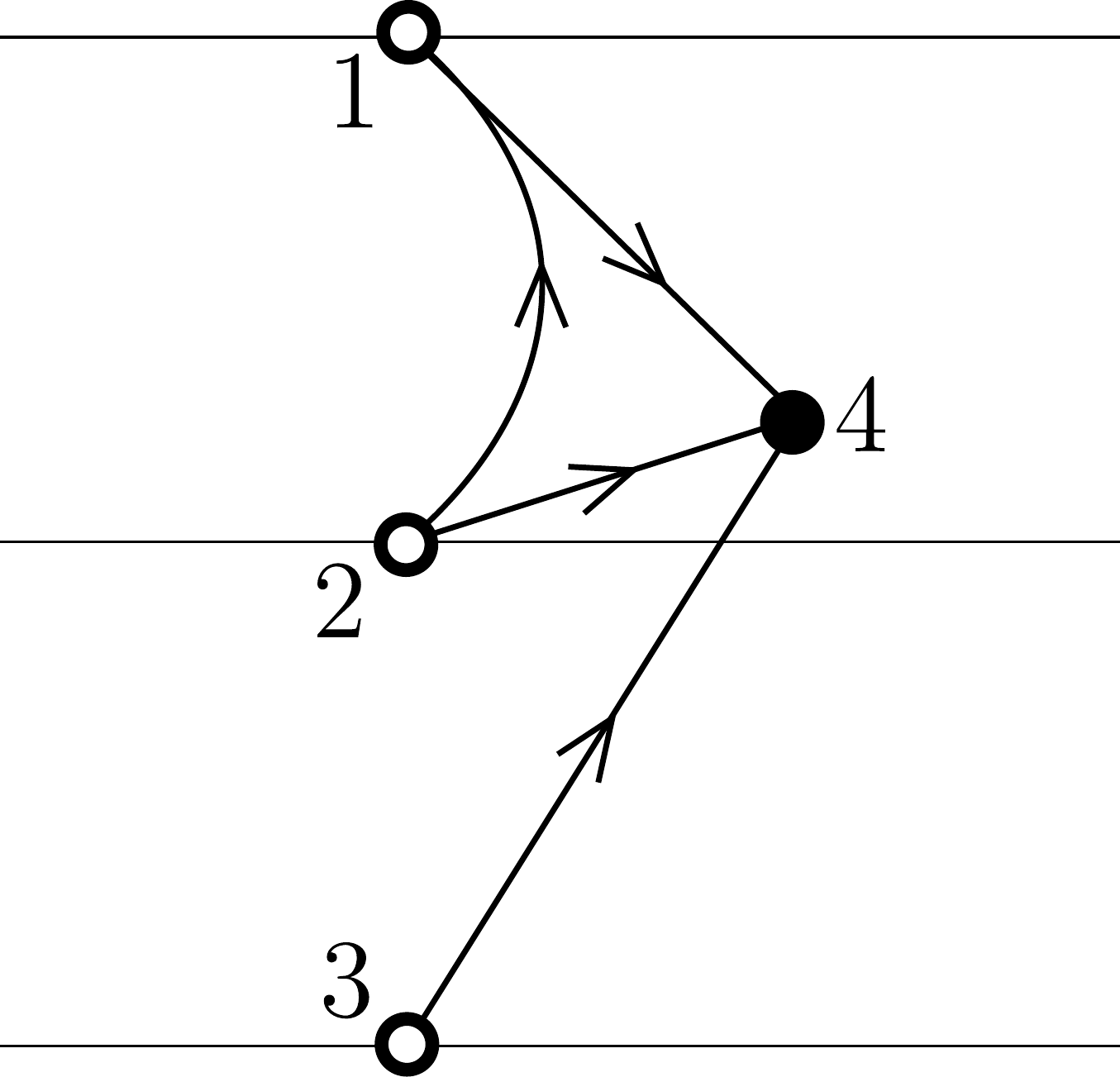}}+\raisebox{-1.6pc}{\includegraphics[scale=0.12]{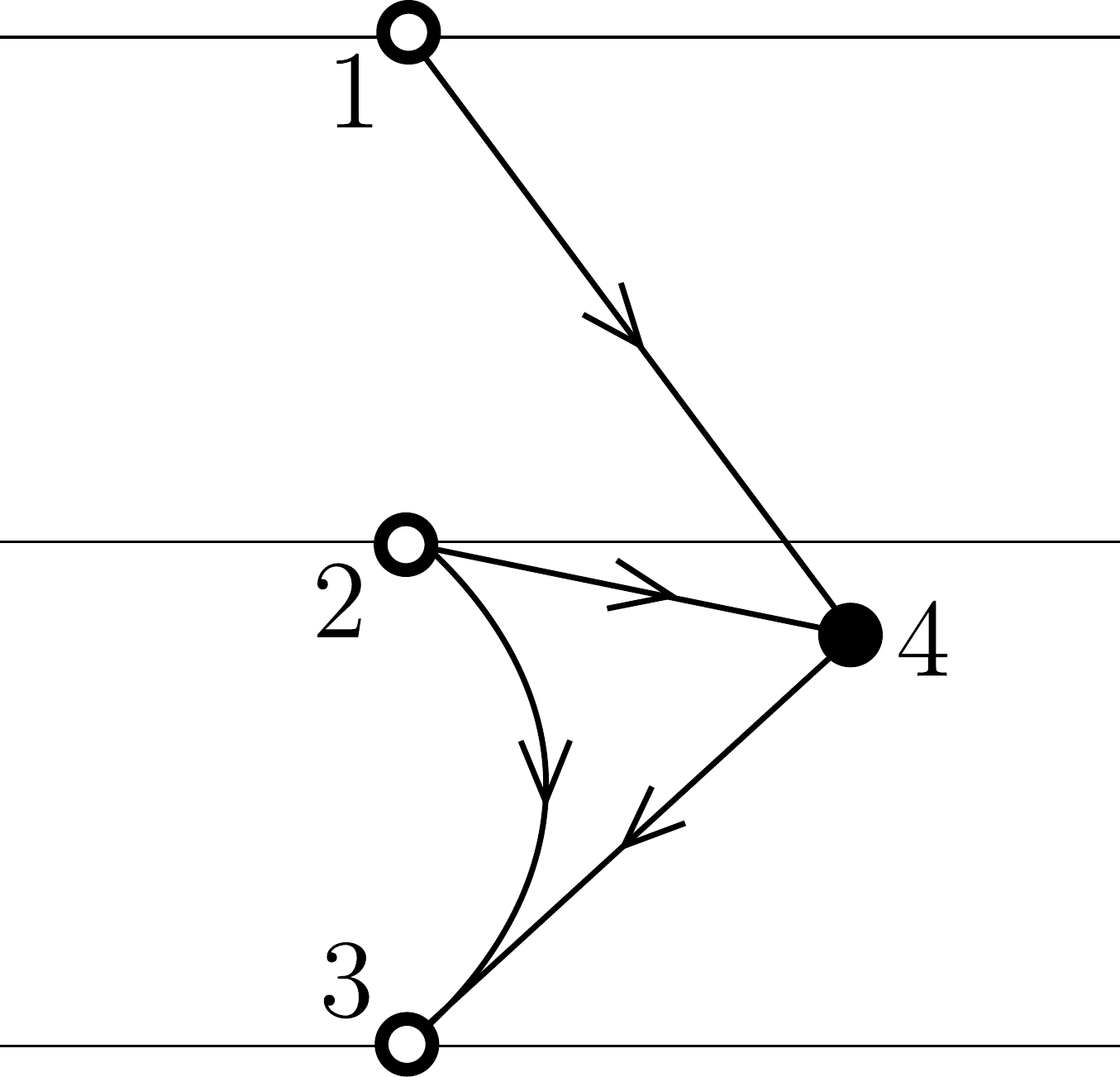}}\right)=\raisebox{-1.6pc}{\includegraphics[scale=0.12]{Co-W1.pdf}}-\raisebox{-1.6pc}{\includegraphics[scale=0.12]{Co-W2.pdf}}.
\]
We match the right side with a value under $D_{\Ba}$:
\[
D_{\Ba}\left(\raisebox{-1.6pc}{\includegraphics[scale=0.12]{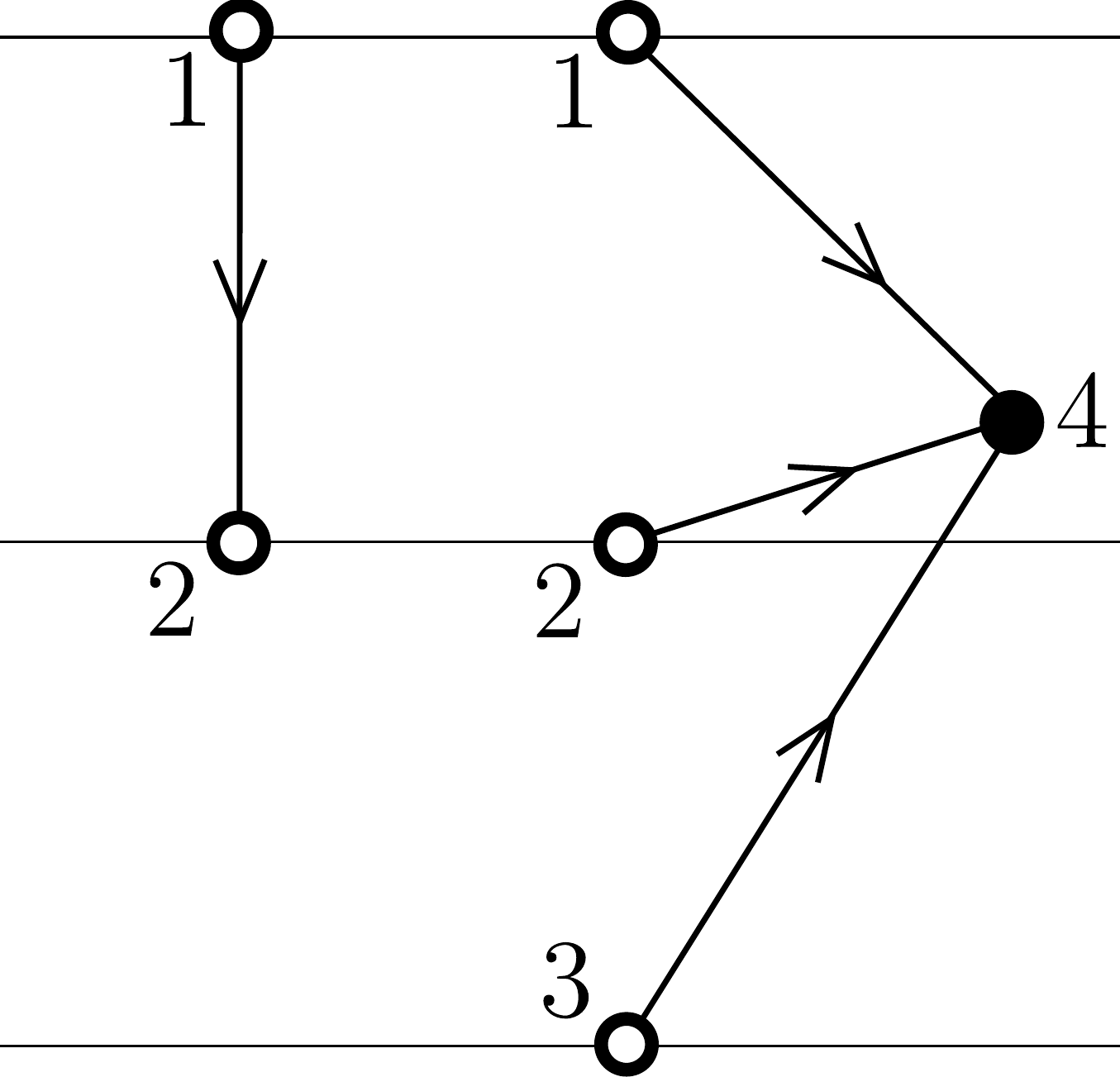}}\right)=-\raisebox{-1.6pc}{\includegraphics[scale=0.12]{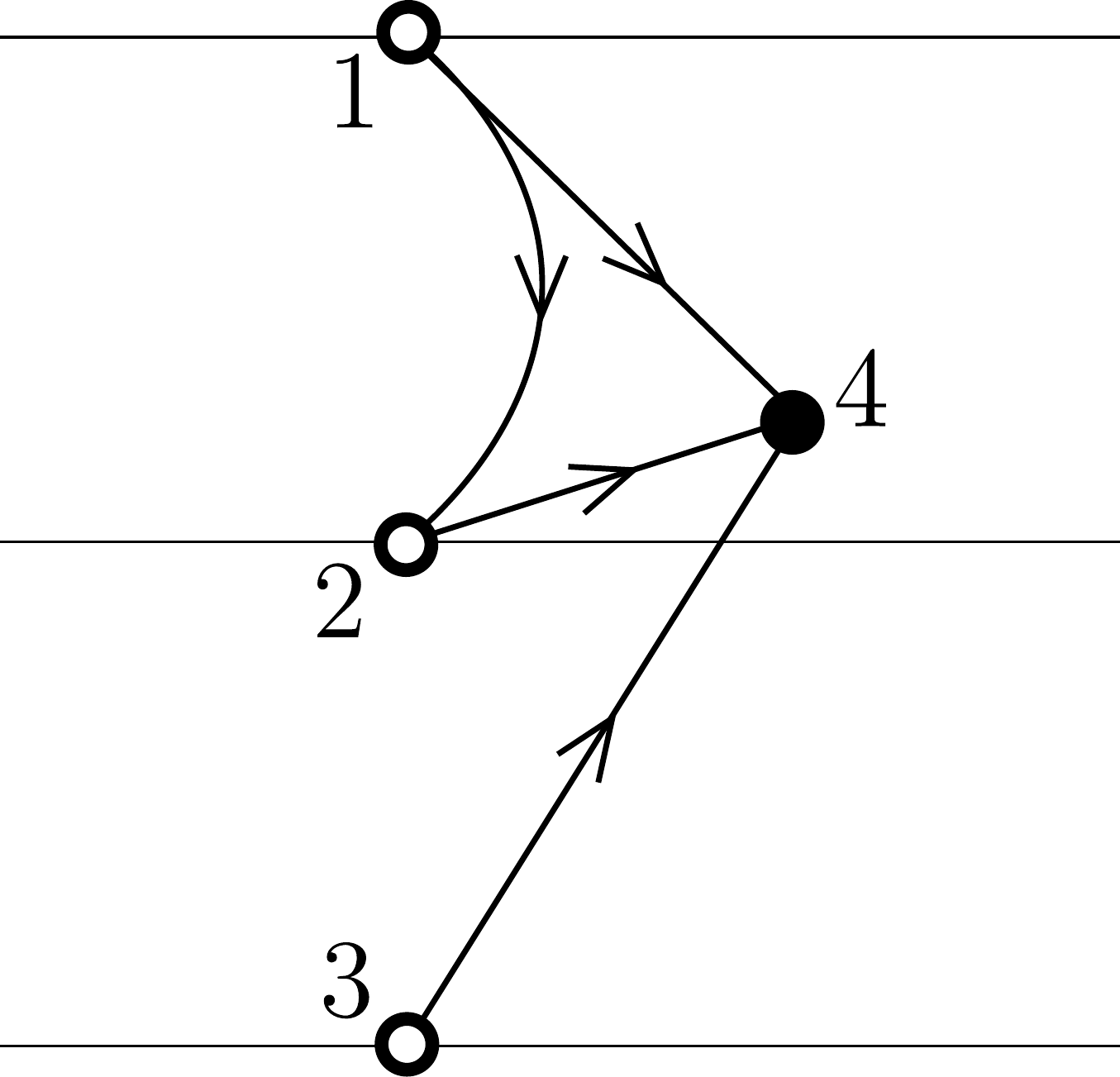}},\qquad D_{\Ba}\left(\raisebox{-1.6pc}{\includegraphics[scale=0.12]{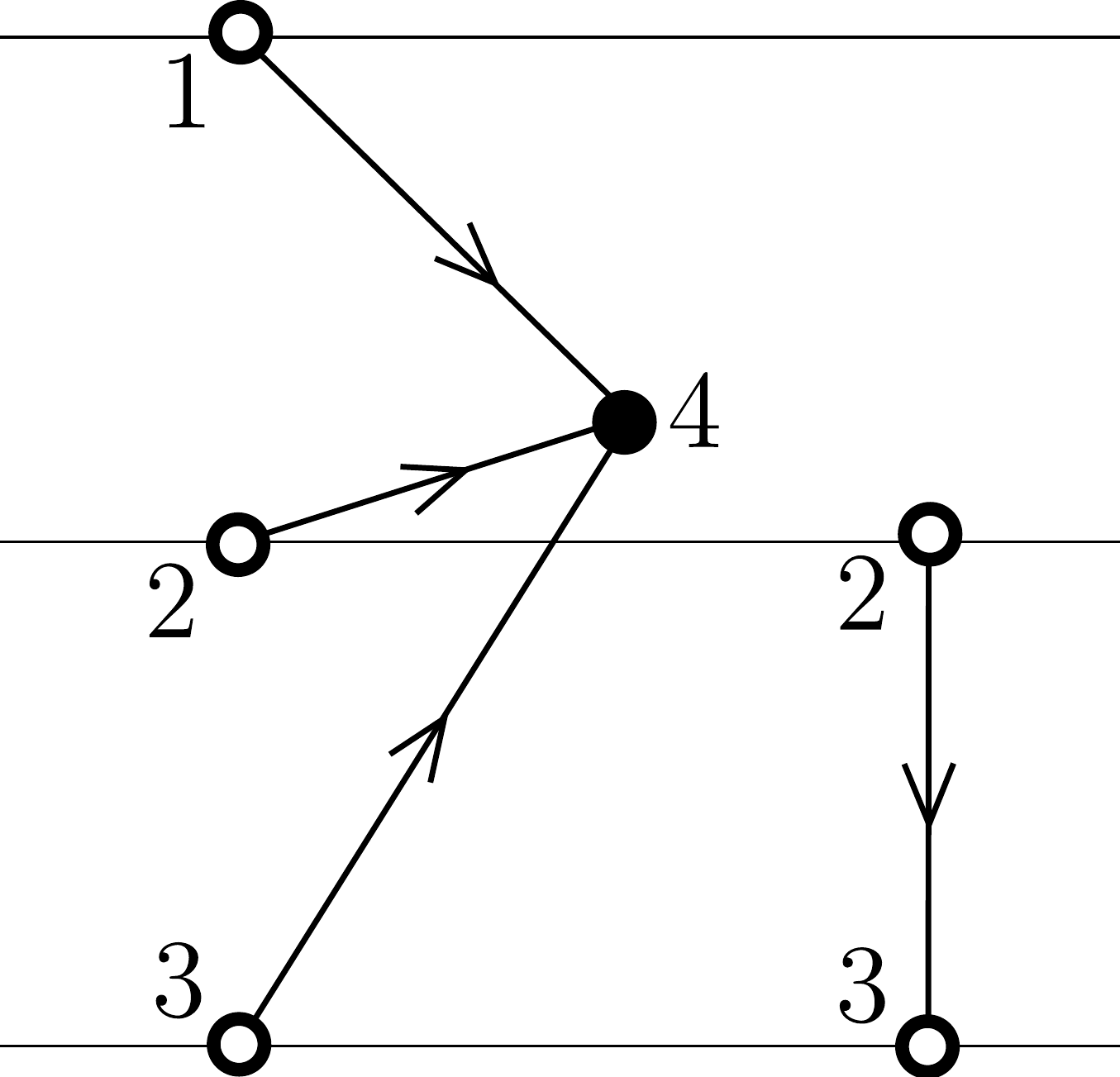}}\right)=\raisebox{-1.6pc}{\includegraphics[scale=0.12]{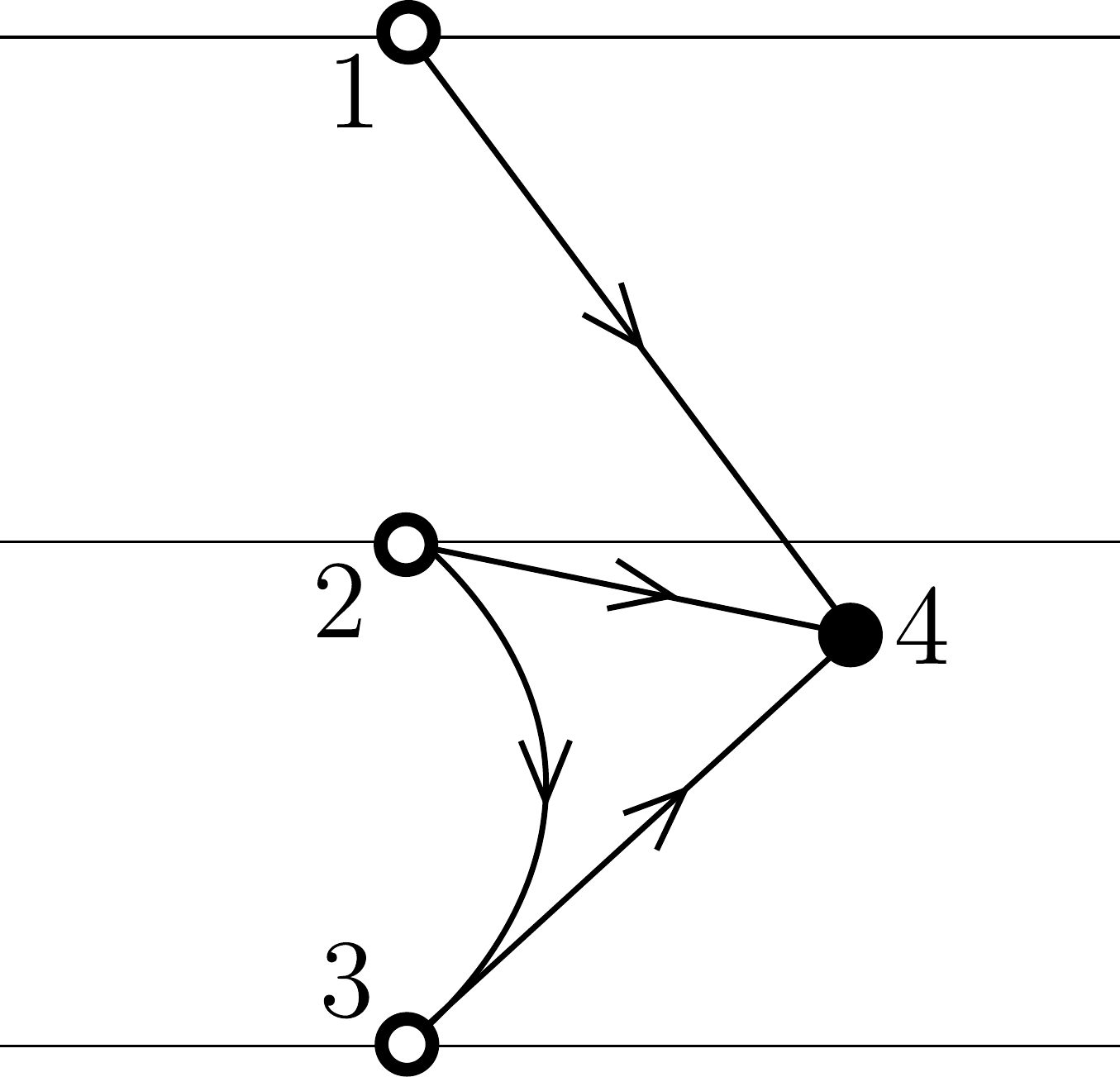}}. 
\]
Further,
\[
\begin{split}
\delta_{\Ba}\left(\raisebox{-1.6pc}{\includegraphics[scale=0.12]{Co-G12-T1234.pdf}}\right) & =\left(-\raisebox{-1.6pc}{\includegraphics[scale=0.12]{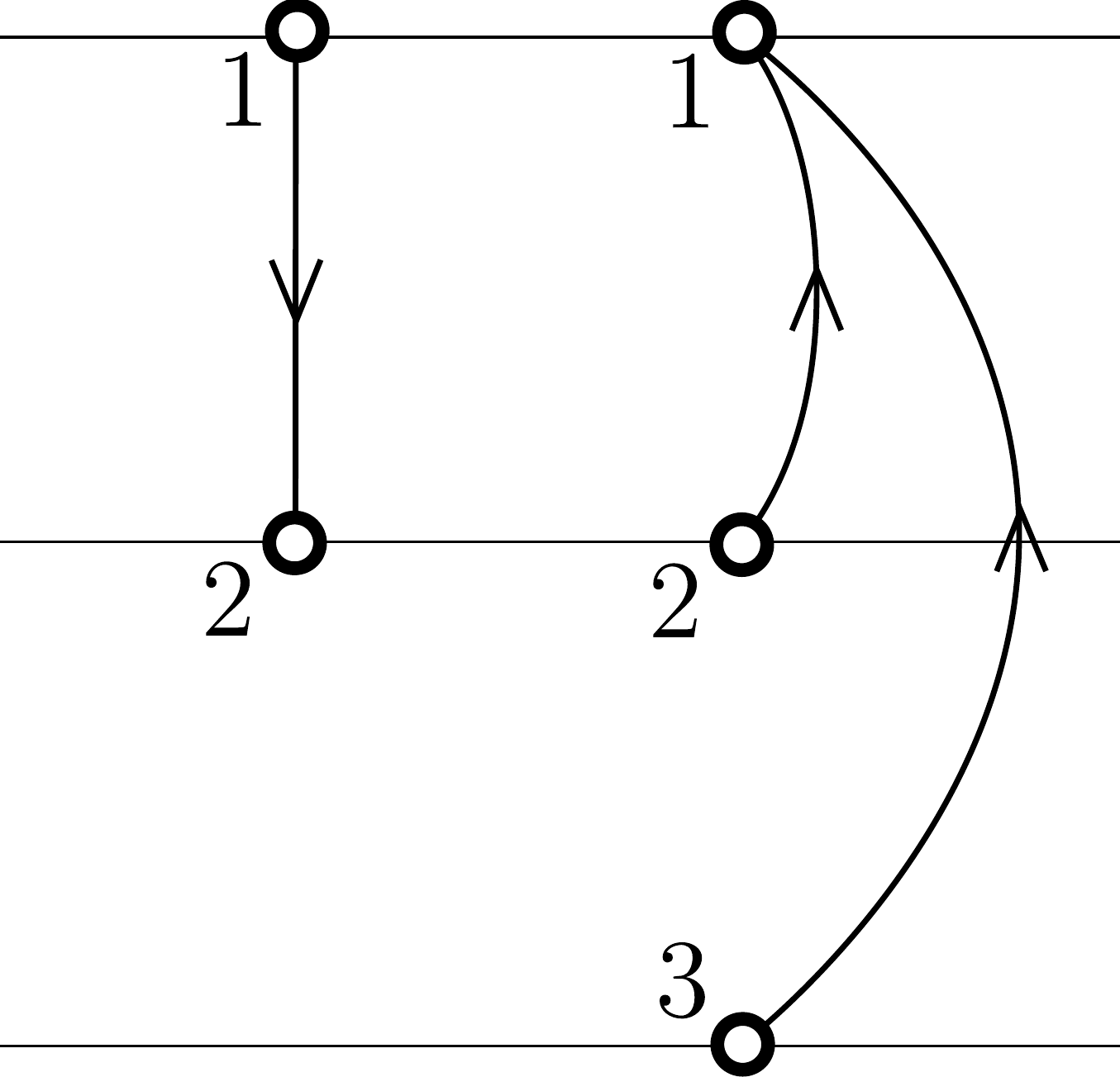}}-\raisebox{-1.6pc}{\includegraphics[scale=0.12]{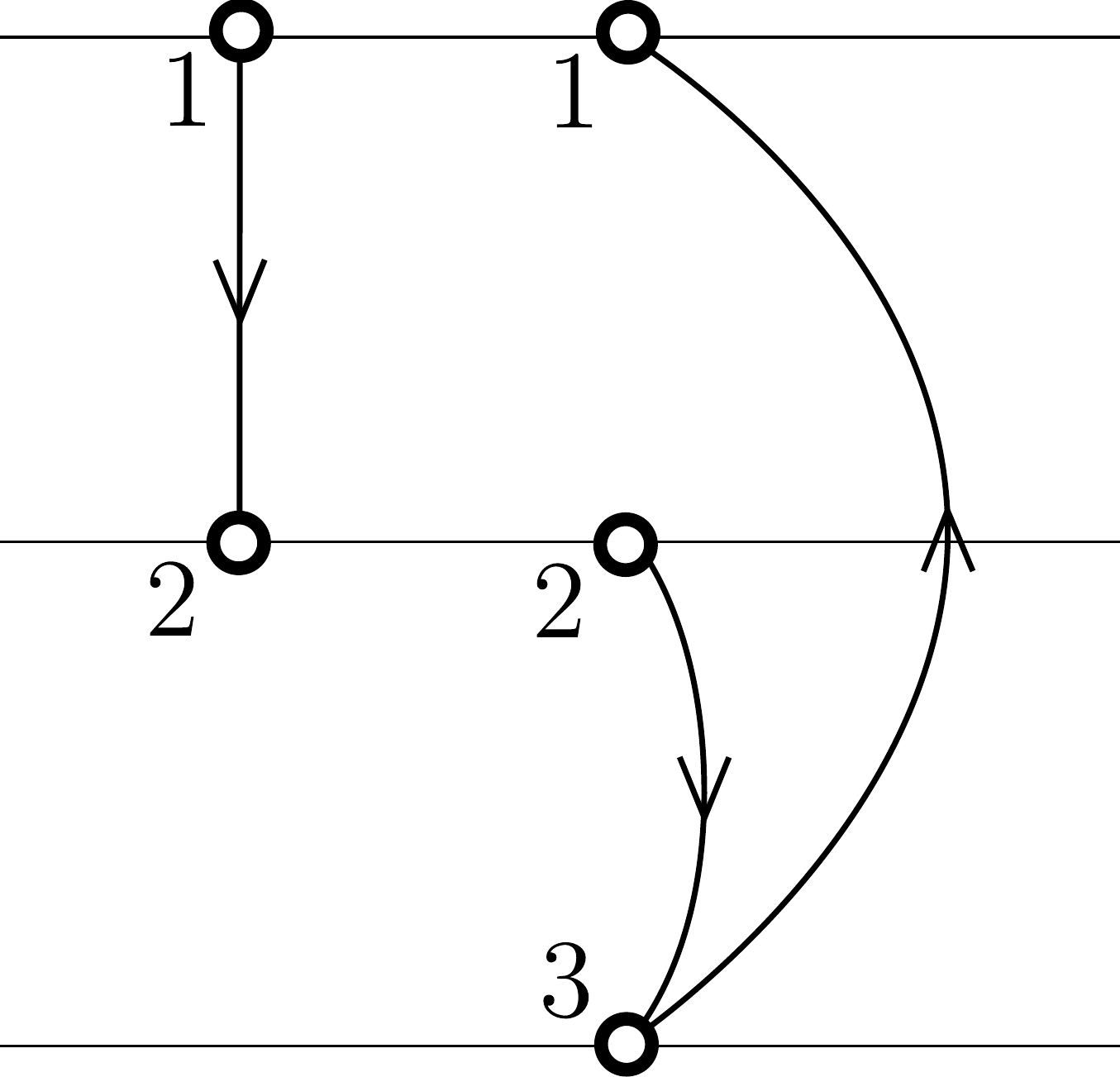}}-\raisebox{-1.6pc}{\includegraphics[scale=0.12]{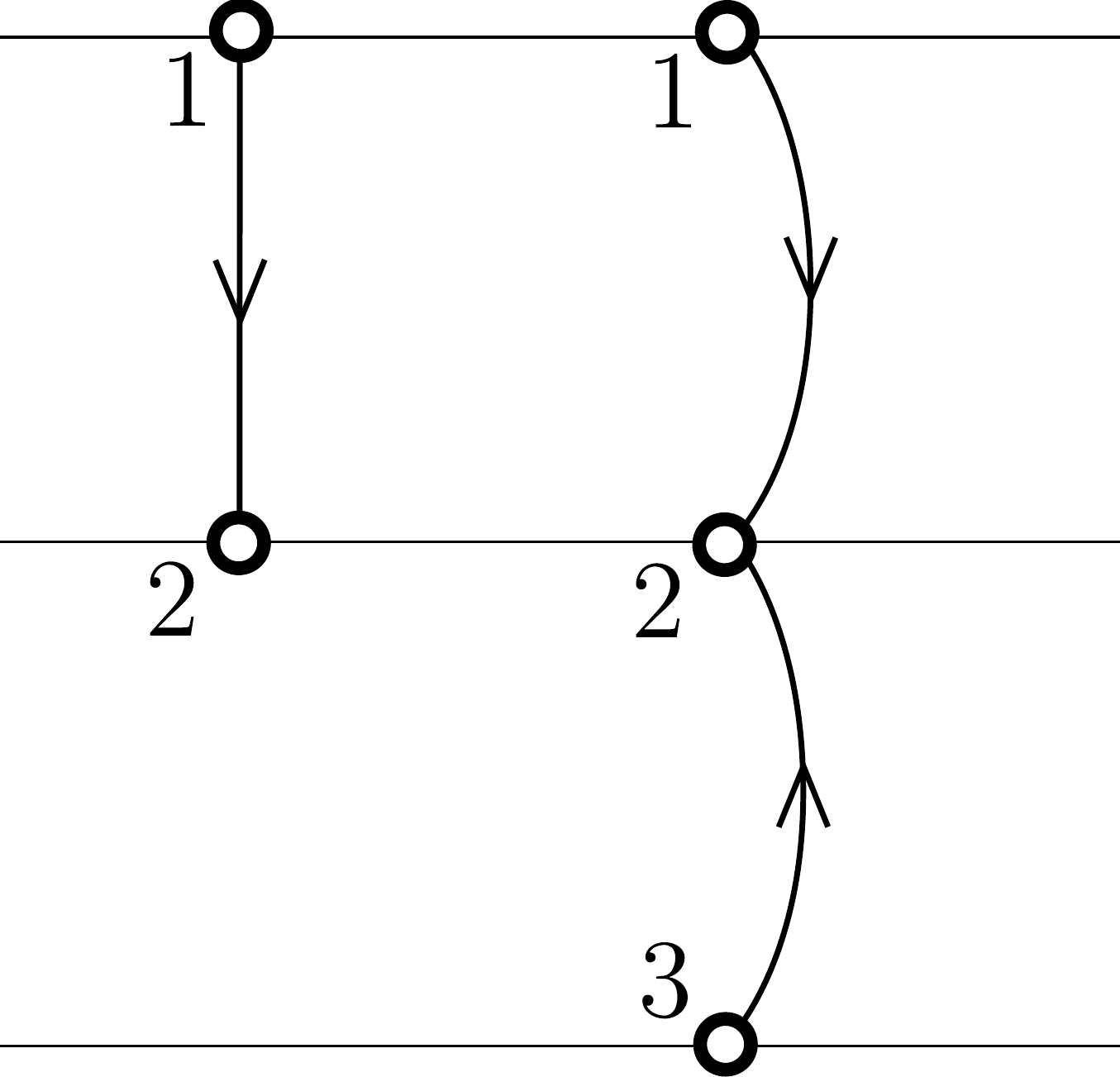}}\right),\\
\delta_{\Ba}\left(\raisebox{-1.6pc}{\includegraphics[scale=0.12]{Co-T1234-G23.pdf}}\right) & =-\left(-\raisebox{-1.6pc}{\includegraphics[scale=0.12]{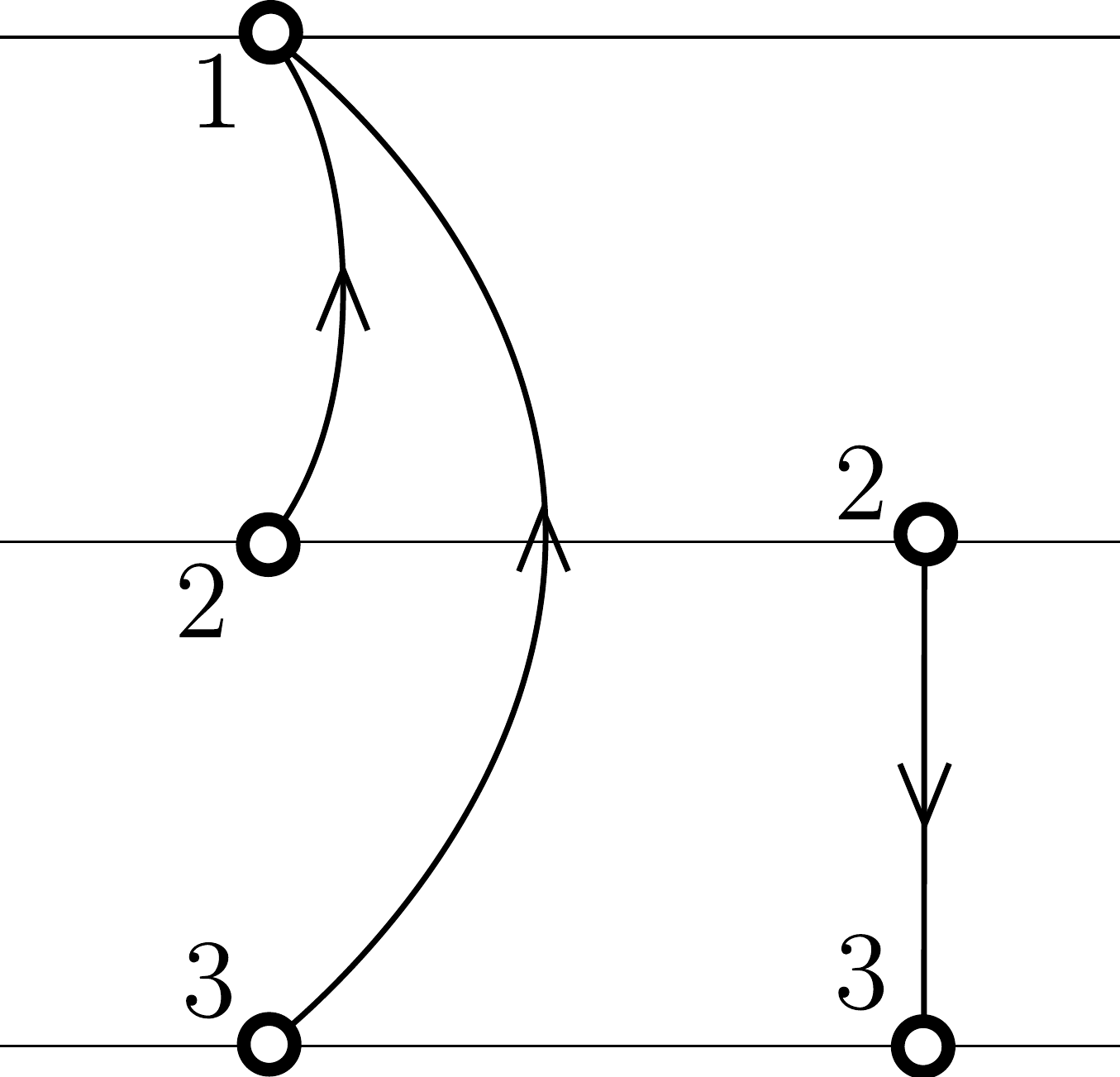}}-\raisebox{-1.6pc}{\includegraphics[scale=0.12]{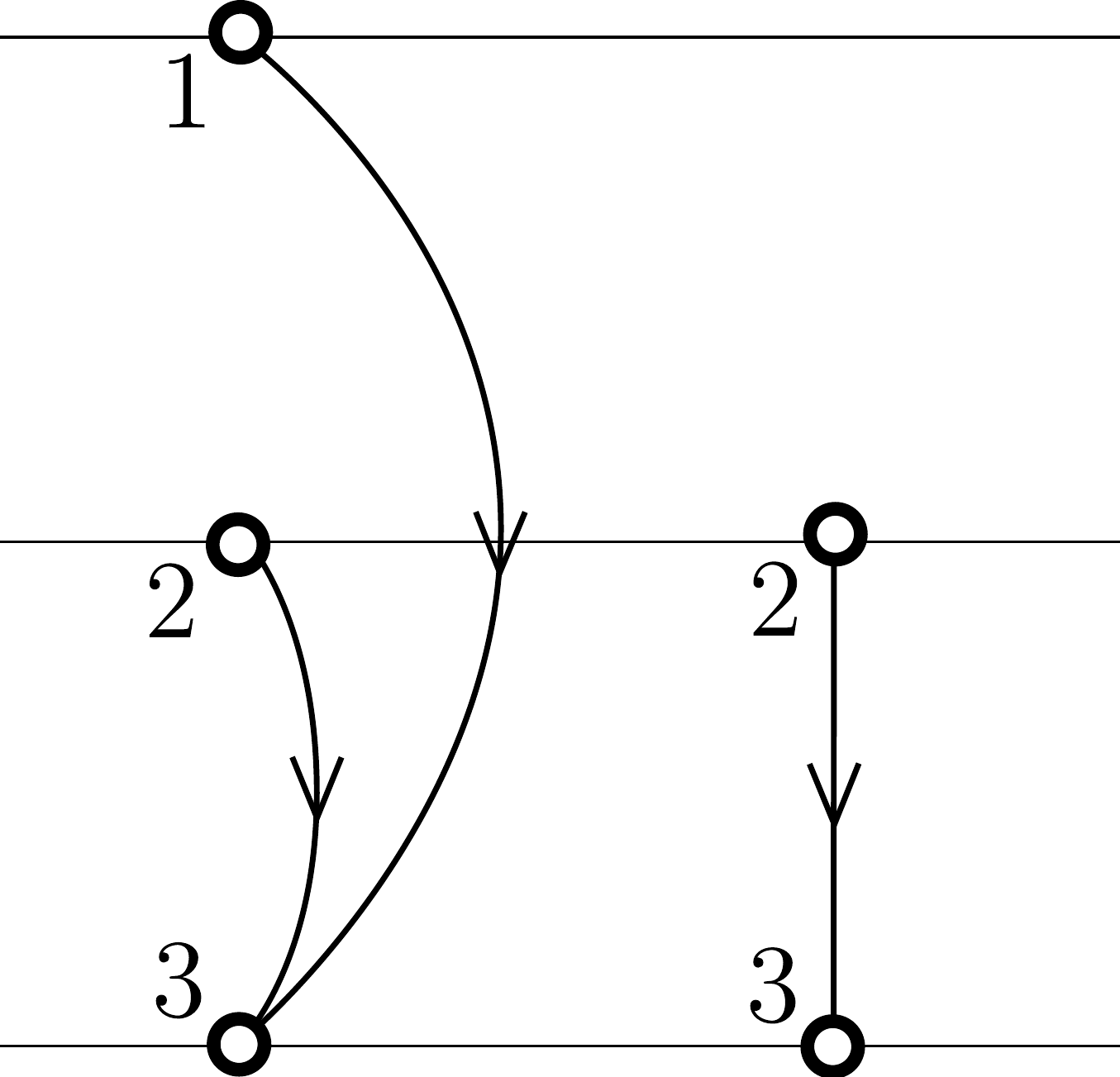}}-\raisebox{-1.6pc}{\includegraphics[scale=0.12]{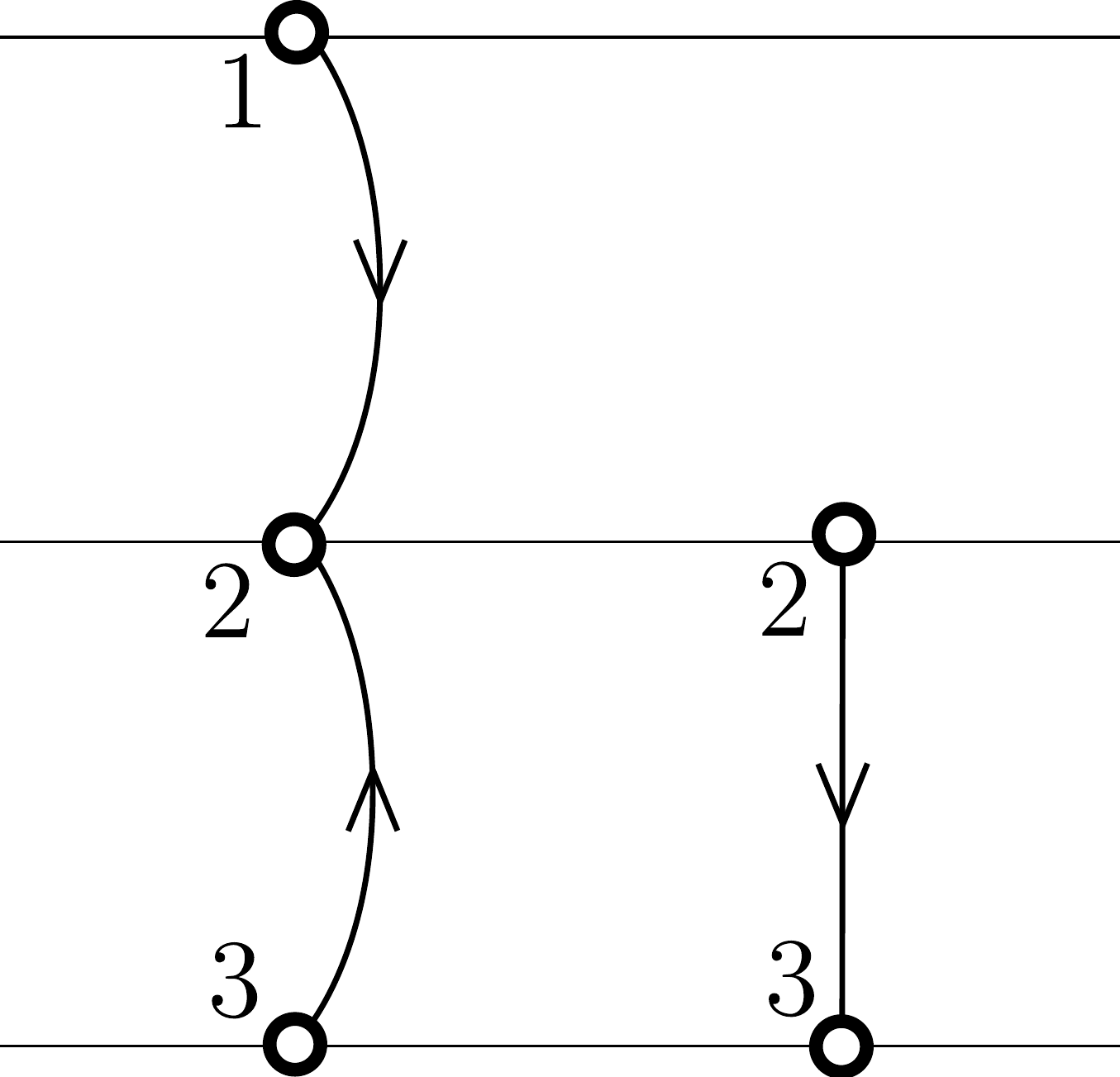}}\right),
\end{split}
\]
and
\[
D_{\Ba}\left(\raisebox{-1.6pc}{\includegraphics[scale=0.12]{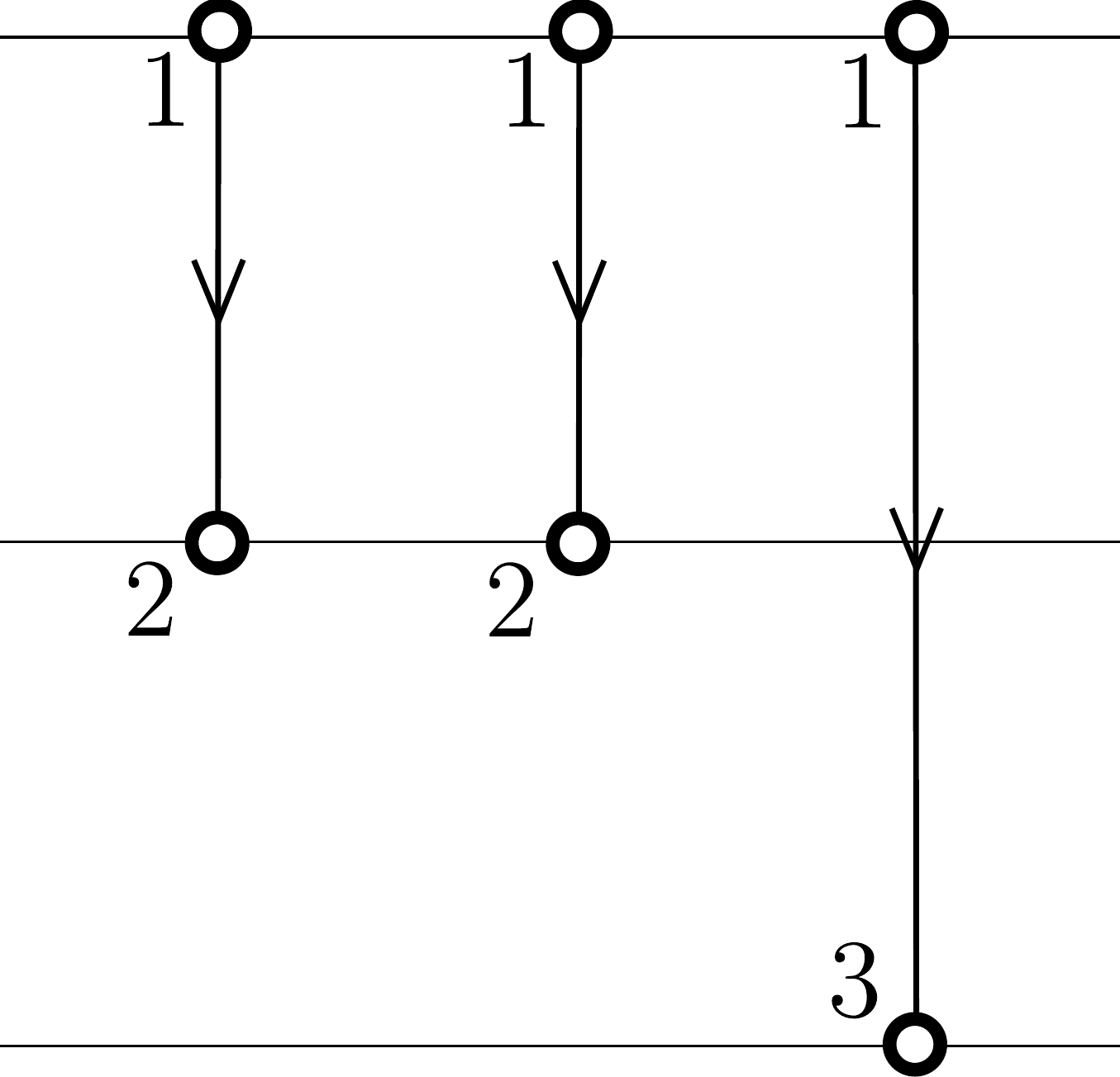}}\right)=\raisebox{-1.6pc}{\includegraphics[scale=0.12]{Co-G12-G12G13.pdf}},\quad D_{\Ba}\left(\raisebox{-1.6pc}{\includegraphics[scale=0.12]{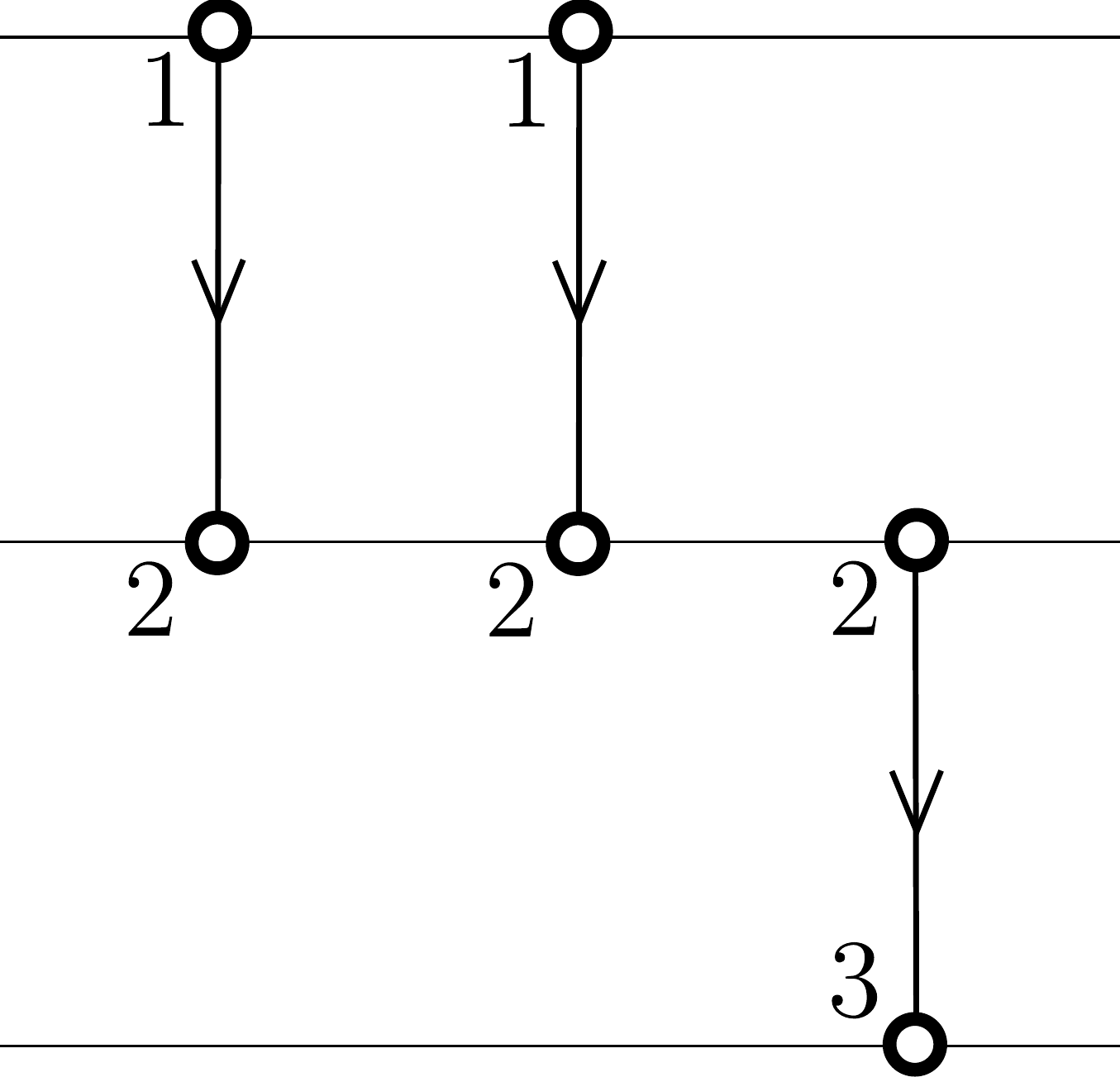}}\right)=-\raisebox{-1.6pc}{\includegraphics[scale=0.12]{Co-G12-G23G12.pdf}},
\]
\[
D_{\Ba}\left(\raisebox{-1.6pc}{\includegraphics[scale=0.12]{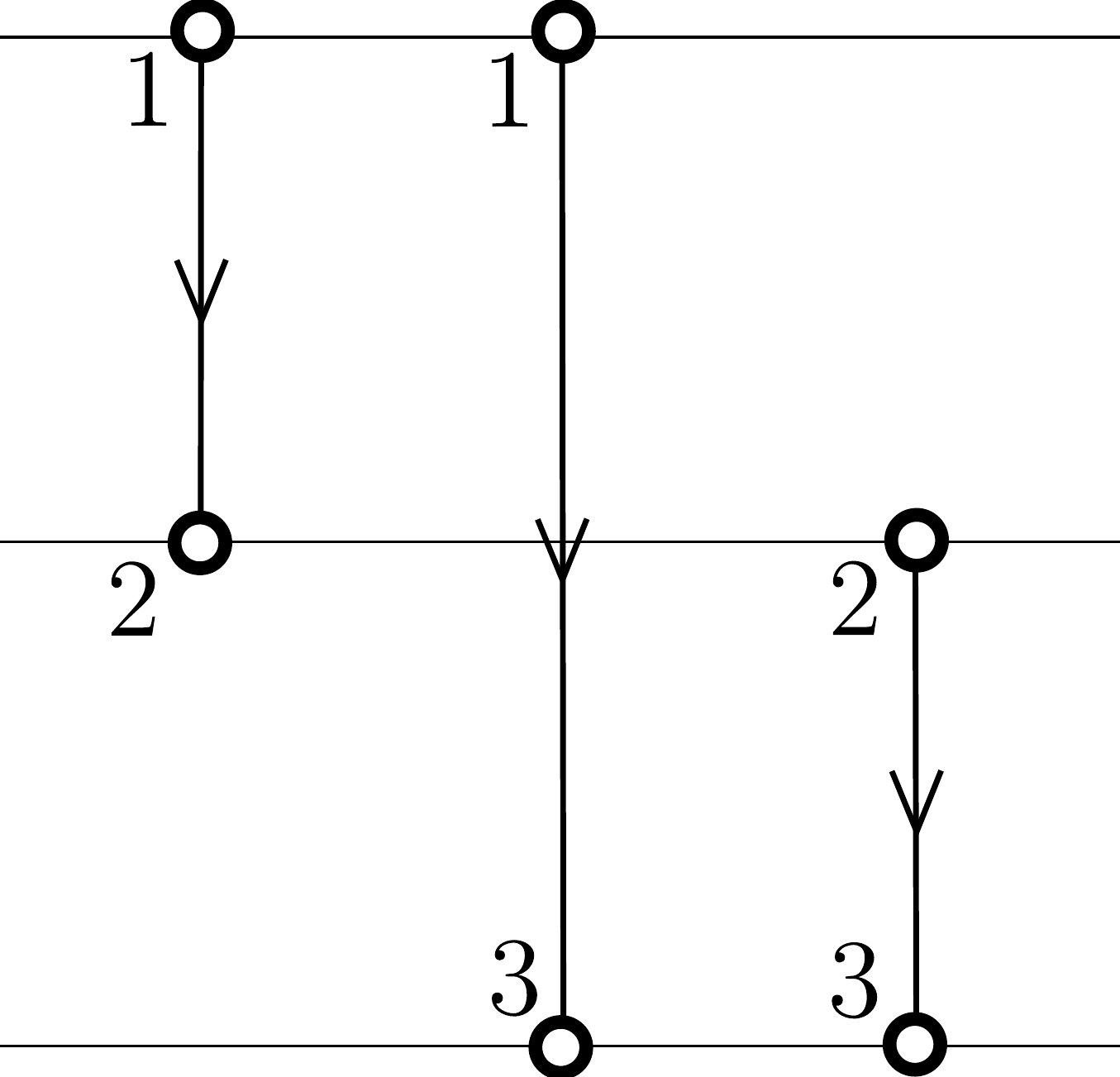}}\right)=-\raisebox{-1.6pc}{\includegraphics[scale=0.12]{Co-G12G13-G23.pdf}}+ \raisebox{-1.6pc}{\includegraphics[scale=0.12]{Co-G12-G23G13.pdf}},
\]
\[
D_{\Ba}\left(\raisebox{-1.6pc}{\includegraphics[scale=0.12]{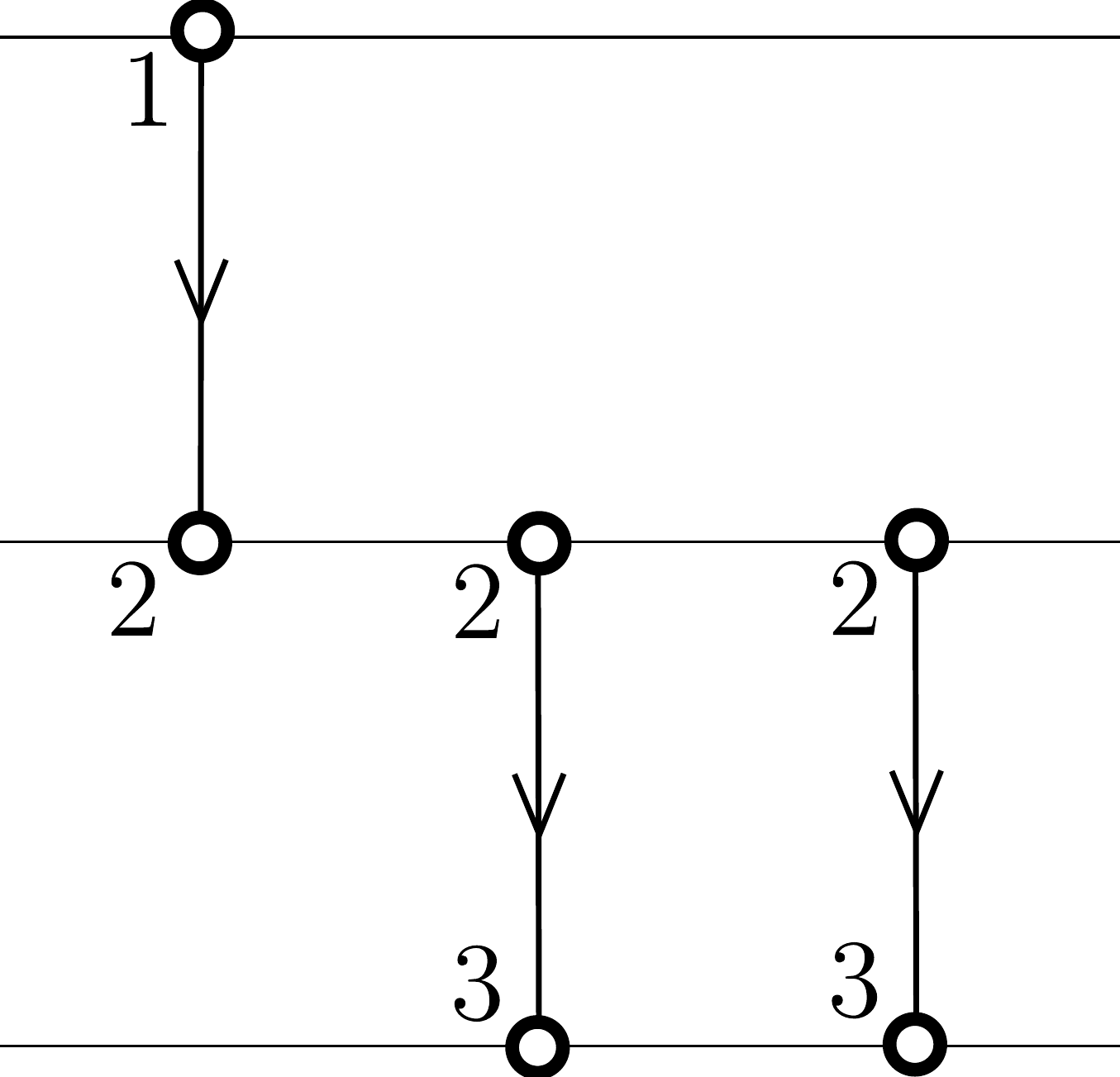}}\right)=\raisebox{-1.6pc}{\includegraphics[scale=0.12]{Co-G23G12-G23.pdf}},
\qquad 
D_{\Ba}\left(\raisebox{-1.6pc}{\includegraphics[scale=0.12]{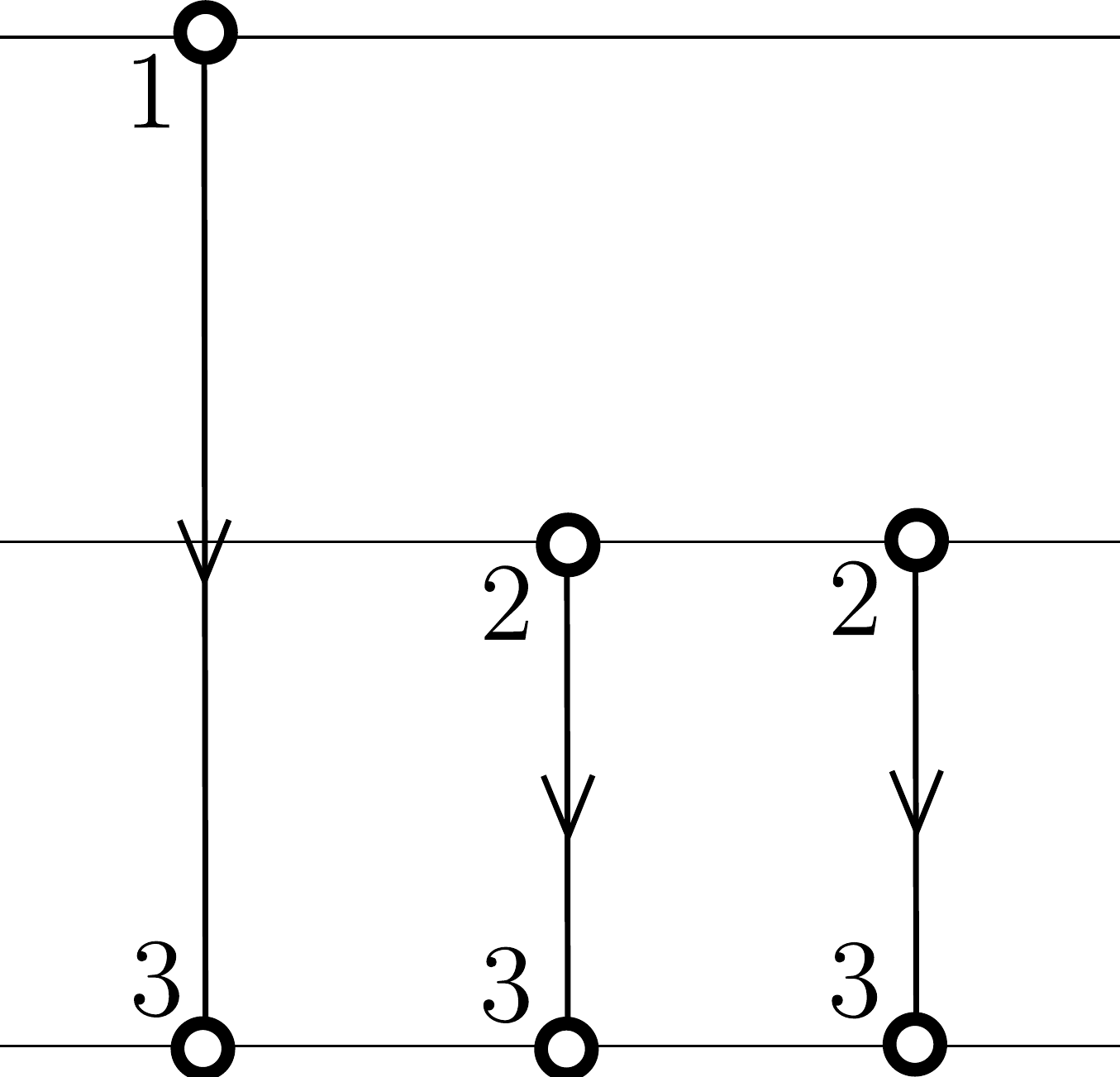}}\right)= 
- \raisebox{-1.6pc}{\includegraphics[scale=0.12]{Co-G23G13-G23.pdf}}.
\]
This yields the cocycle
\[
\begin{split}
\raisebox{-1.6pc}{\includegraphics[scale=0.12]{Co-G12-G12-G13.pdf}} 
- \raisebox{-1.6pc}{\includegraphics[scale=0.12]{Co-G12-G12-G23.pdf}}
+\raisebox{-1.6pc}{\includegraphics[scale=0.12]{Co-G12-G13-G23.pdf}}
-\raisebox{-1.6pc}{\includegraphics[scale=0.12]{Co-G12-G23-G23.pdf}}
+ \raisebox{-1.6pc}{\includegraphics[scale=0.12]{Co-G13-G23-G23.pdf}} 
- \raisebox{-1.6pc}{\includegraphics[scale=0.12]{Co-G12-T1234.pdf}} 
- \raisebox{-1.6pc}{\includegraphics[scale=0.12]{Co-T1234-G23.pdf}}
+\raisebox{-1.6pc}{\includegraphics[scale=0.12]{Co-W.pdf}}.
\end{split}
\]
This cocycle was previously found in Example 5.12 of \cite{KKV:2020}.
The corresponding  Chen's iterated integral reads
\[
\begin{split}
\Phi(z_{3;1,2,2}) =&
\varint \alpha_{2,1}\alpha_{2,1}\alpha_{3,1}
+\varint \alpha_{2,1}\alpha_{2,1}\alpha_{3,2}
+\varint \alpha_{2,1}\alpha_{3,1}\alpha_{3,2}
+\varint \alpha_{2,1}\alpha_{3,2}\alpha_{3,2}
+\varint \alpha_{3,1}\alpha_{3,2}\alpha_{3,2}\\
&-\varint \alpha_{2,1}\pi_{\ast}(\alpha_{4,1}\times\alpha_{4,2}\times\alpha_{4,3})
-\varint \pi_{\ast}(\alpha_{4,1}\times\alpha_{4,2}\times\alpha_{3,4})\alpha_{3,2}\\
&+\varint \pi_{\ast}(\alpha_{4,1}\times\alpha_{4,2}\times\alpha_{5,2}\times\alpha_{5,3}),
\end{split}
\]
where the pushforward $\pi_{\ast}$ was previously defined in \eqref{E:FormalityIntegral} and 
\[
 \pi_{\ast}(\alpha_{4,1}\times\alpha_{4,2}\times\alpha_{5,2}\times\alpha_{5,3})=I\left(\vvcenteredinclude{0.1}{Ex-5--trapez.pdf}\right),\quad
 \pi_{\ast}(\alpha_{4,1}\times\alpha_{4,2}\times\alpha_{4,3})=I\left(\vvcenteredinclude{0.1}{Ex-0-tripod.pdf}\right).
\]
\end{example}

We end this appendix by proposing that all cocycles $z_{j;I}$ and the corresponding iterated integrals should be computable 
via the method outlined here, and one may hope for a more tractable algorithm for computing such cocycle expressions. We aim to investigate this further in the context of geometric problems as studied, for example, in  \cite{Arone-Krushkal, Freedman-Krushkal:2014, Komendarczyk:2009, Elliott:arXiv2020, DeTurck:2013}.

\bibliographystyle{amsplain}
\bibliography{cochains}

\end{document}